\documentclass[reqno, final,  twoside, a4paper, 11pt]{amsart}
\usepackage[english]{babel}
\usepackage{amsmath}
\usepackage{amsfonts}
\usepackage{amssymb}
\usepackage{amsthm}
\usepackage[latin1]{inputenc}
\usepackage{soul}
\sethlcolor{red}

\usepackage{units}
\usepackage{fixme}
\usepackage{adjustbox}
\usepackage{MnSymbol}
\usepackage{stmaryrd}
\usepackage{mathtools}
\usepackage{color}

\usepackage{mathrsfs}
\usepackage{enumitem}

\usepackage{tikz}
\usepackage{tikz-cd}
\usetikzlibrary{matrix, arrows}
\tikzset{commutative diagrams/.cd, column sep/normal=3ex}
\usepackage[colorlinks=true,draft=false,hyperindex]{hyperref}

\newcount\hh
\newcount\mm
\mm=\time
\hh=\time
\divide\hh by 60
\divide\mm by 60
\multiply\mm by 60
\mm=-\mm
\advance\mm by \time
\def\hhmm{\number\hh:\ifnum\mm<10{}0\fi\number\mm}
\renewcommand{\phi}{\varphi}
\DeclareMathOperator{\image}{im}

\DeclareMathOperator{\Sw}{Sw}
\DeclareMathOperator{\Swan}{Swan}

\DeclareMathOperator{\supp}{supp}

\DeclareMathOperator{\unr}{unr}

\DeclareMathOperator{\pr}{pr}

\DeclareMathOperator{\sep}{sep}
\DeclareMathOperator{\sw}{sw}
\DeclareMathOperator{\Aug}{Aug}

\DeclareMathOperator{\rank}{rank}

\DeclareMathOperator{\tame}{tame}

\renewcommand{\epsilon}{\varepsilon}

\DeclareMathOperator{\card}{card}

\DeclareMathOperator{\coker}{coker}
\DeclareMathOperator{\Spec}{Spec}

\DeclareMathOperator{\Char}{char}
\DeclareMathOperator{\ord}{ord}

\DeclareMathOperator{\Aut}{Aut}

\DeclareMathOperator{\Hom}{Hom}
\renewcommand{\hom}{\Hom}
\DeclareMathOperator{\IntHom}{\underline{\Hom}}
\DeclareMathOperator{\HHom}{\mathscr{H}\!\!\mathit{om}}

\DeclareMathOperator{\can}{can}
\DeclareMathOperator{\End}{End}

\DeclareMathOperator{\alg}{alg}

\DeclareMathOperator{\Repf}{Repf}

\DeclareMathOperator{\Cu}{Cu}
\DeclareMathOperator{\Gal}{Gal}
\DeclareMathOperator{\id}{id}

\DeclareMathOperator{\RES}{Res}
\DeclareMathOperator{\GL}{GL}
\DeclareMathOperator{\opp}{opp}

\DeclareMathOperator{\rad}{rad}

\DeclareMathOperator{\Ind}{Ind}

\DeclareMathOperator{\Tr}{Tr}

\DeclareMathOperator{\Frac}{Frac}
\DeclareMathOperator{\Fib}{Fib}

\DeclareMathOperator{\sh}{sh}

\DeclareMathOperator{\sTor}{\mathcal{T}\!\!\mathit{or}}
\newcommand{\TopGroup}{\mathbf{TopGroup}}

\DeclareMathOperator{\FinSet}{\mathbf{FiniteSet}}

\DeclareMathOperator{\Set}{\text{\bf{Set}}}
\DeclareMathOperator{\FEt}{\mathbf{F\acute{E}t}}

\renewcommand{\div}{\mathop{\mathrm{div}}}

\DeclareMathOperator{\an}{\text{an}}
\newcommand{\ab}{\text{ab}}

\DeclareMathOperator{\et}{\text{\'et}}

\DeclareMathOperator{\rk}{rk}
\newcommand{\Z}{\mathbb{Z}}
\newcommand{\G}{\mathbb{G}}
\newcommand{\R}{\mathbb{R}}
\newcommand{\Q}{\mathbb{Q}}
\newcommand{\N}{\mathbb{N}}
\newcommand{\C}{\mathbb{C}}
\newcommand{\A}{\mathbb{A}}
\newcommand{\F}{\mathbb{F}}
\newcommand{\Qlb}{{\bar{\Q}_\ell}}
\newcommand{\llparen}{(\!(}
\newcommand{\rrparen}{)\!)}
\renewcommand{\P}{\mathbb{P}}

\newcommand{\fm}{\mathfrak{m}}
\newcommand{\fp}{\mathfrak{p}}
\newcommand{\fq}{\mathfrak{q}}
\renewcommand{\O}{\mathcal{O}}

\newcommand{\sC}{{\mathcal C}}

\newcommand{\sE}{{\mathcal E}}
\newcommand{\sF}{{\mathcal F}}
\newcommand{\sG}{{\mathcal G}}

\newcommand{\sI}{{\mathcal I}}

\newcommand{\sL}{{\mathcal L}}
\newcommand{\sM}{{\mathcal M}}

\newcommand{\sO}{{\mathcal O}}

\newcommand{\sR}{{\mathcal R}}

\newcommand{\sV}{{\mathcal V}}

\renewcommand{\subset}{\subseteq}
\renewcommand{\supset}{\supseteq}
\renewcommand{\subsetneq}{\subsetneqq}
\renewcommand{\supsetneq}{\supsetneqq}

\numberwithin{equation}{section}

\makeindex

\newtheorem{theorem}{Theorem}[section]
\newtheorem{proposition}[theorem]{Proposition}
\newtheorem{lemma}[theorem]{Lemma}
\newtheorem{conjecture}[theorem]{Conjecture}
\newtheorem{corollary}[theorem]{Corollary}

\theoremstyle{definition} 
\newtheorem{definition}[theorem]{Definition}
\newtheorem{convention}[theorem]{Convention}
\newtheorem{notation}[theorem]{Notation}
\newtheorem{example}[theorem]{Example}
\newtheorem{exercise}[theorem]{Exercise}
\newtheorem{remark}[theorem]{Remark}

\newtheorem*{claim}{Claim}
\newtheorem{question}[theorem]{Question}

\newcommand{\xr}[1] {\xrightarrow{#1}}

\newcommand{\inj}{\hookrightarrow}

\newcommand{\eq}[2]{\begin{equation}\label{#1}#2 \end{equation}}
\newcommand{\ml}[2]{\begin{multline}\label{#1}#2 \end{multline}}
\newcommand{\mlnl}[1]{\begin{multline*}#1 \end{multline*}}

	\setenumerate{label=\textnormal{(\alph*)}, ref=(\alph*)}
\title{Introductory course on $\ell$-adic sheaves and their ramification
theory on curves}
\author{Lars Kindler}
\author{Kay R\"ulling}
\address{Freie Universit\"at Berlin, Mathematisches Institut, Arnimallee 3,
D-14195 Berlin, Germany}
\email{kindler@math.fu-berlin.de, kay.ruelling@fu-berlin.de}
\thanks{\emph{Email}: \texttt{kindler@math.fu-berlin.de,
kay.ruelling@fu-berlin.de}\\Both authors were supported by ERC Advanced Grant
226257. The second author is supported by the DFG Heisenberg Grant RU 1412/2-1.}
\begin{document}
	\begin{abstract}
	These are the notes accompanying 13 lectures given by the
	authors at
	the Clay Mathematics Institute Summer School 2014 in Madrid. They
	give an
	introduction into the theory of $\ell$-adic sheaves with emphasis on their
	ramification theory on curves.
	\end{abstract}
\maketitle
\tableofcontents

\section{Introduction}
These are the  notes accompanying 13 lectures given by the
authors at
the Clay Mathematics Institute Summer School 2014 in Madrid. 
The goal of this lecture series is to introduce the audience to the theory
of $\ell$-adic sheaves with emphasis on their ramification theory. 
Ideally, the lectures and these notes will equip
the audience with the necessary background knowledge to read
current literature
on the subject, particularly \cite{EK12}, which is the focus of a second
series of lectures at the same summer school. We do not attempt to give a
panoramic exposition of recent research in the subject.

Before giving an outline of this document, the authors wish to stress that
there is \emph{no original mathematical content in these notes, and that
inaccuracies, omissions and errors are solely their responsibility.} 


To introduce ideas, let $p$ be a prime number and $U$ a smooth,
connected curve over the finite field $\F_p$. If $\ell\neq p$ is a
second prime number, an $\ell$-adic sheaf can be understood as a continuous
representation $\rho:\pi_1^{\et}(U,u)\rightarrow \GL_r(E)$, where $E$ is a
finite extension of the field of $\ell$-adic numbers $\Q_{\ell}$, and where
$\pi_1^{\et}(U,u)$ is the \'etale fundamental group of $U$ with respect to the
base point $u$. This group is an
algebraic variant of the usual fundamental group of a topological space;
additionally, it carries a topology. We summarize its construction and
properties in Section
\ref{sec:fundamentalGroup}. 

A representation $\rho$ as above has a geometric interpretation (hence the word ``sheaf''),
which is explained in Section \ref{sec:l-adic-sheaves}. If $X$ is the unique smooth projective
curve over $\F_p$ containing $U$   as an open dense subvariety, then one might
ask whether $\rho$ extends to a representation of $\pi^{\et}(X,u)$,
i.e.~whether $\rho$
factors through the canonical map $\pi_1^{\et}(U,u)\rightarrow
\pi_1^{\et}(X,u)$. If it does, $\rho$ is called unramified. 
If it does not, then it is natural to ask whether it can be
measured ``how bad'' the ramification of $\rho$ is. The two important definitions
in this context are \emph{tame ramification} and \emph{wild ramification}.
Tamely ramified $\ell$-adic sheaves are much better understood than wildly
ramified ones and they behave similarly to 
regular singular local systems on Riemann surfaces. 

One of the main results presented in these notes is the construction and analysis of an invariant of
$\rho$, locally at the finitely many closed points of $X\setminus U$, which
measures how wild the ramification of $\rho$ is; this is the so called Swan
conductor. In the course of the construction, we give a proof of the Hasse-Arf
theorem and we show that the Swan conductor arises from the character of a
projective $\Z_{\ell}$-representation, the Swan representation.
See Sections \ref{Ram-groups-Hasse-Arf} and \ref{sec:swan}.

The second main result gives a
global, cohomological interpretation of the Swan conductor: This is the 
formula of Grothendieck-Ogg-Shafarevich, see Section \ref{sec:GOS}.

The final sections survey generalizations of these notions to higher
dimensional varieties. One approach based on ideas of Wiesend, further developed by Kerz, Schmidt, 
Drinfeld and Deligne is presented in Section \ref{RTvC}. The idea is to
study an $\ell$-adic sheaf via its ``skeleton'': If $U$ is an algebraic
variety over $\F_p$, consider the family  $\{\phi_C:C\hookrightarrow U\}$ of
all curves lying on $U$, and let $\phi^N_C:C^N\rightarrow C\hookrightarrow U$
be the composition with the normalization of $C$. If $\rho$ is an
$\ell$-adic sheaf on $U$, then, roughly, its skeleton is the family of
$\ell$-adic
sheaves $(\phi^N_C)^*\rho$ on $C^N$. Using the ramification theory on curves one obtains in this way
a ramification theory for $\ell$-adic sheaves in higher dimensions. 
Recently, Deligne proved that on a smooth connected scheme over a finite field, there are only
finitely many irreducible lisse $\Qlb$-representations with bounded rank and ramification, at least up to twist
with a character of the Weil group of the ground field,
see Theorem \ref{thm-finiteness-ram-sheaves}. One of the aims of these notes
is to give the background necessary to understand
the statement of this theorem. Notice, however, that its proof lies far beyond the scope of the material
presented here and we refer the reader to \cite{EK12} for details.

In the final section we give a very brief outlook on the higher-dimensional generalizations of the
Grothedieck-Ogg-Shafarevich formula due to Kato-Saito. There are further generalizations of this formula 
due to Abbes and Saito who also develop a ramification theory in higher
dimensions. We can say nothing about this,
but give some references for further reading at the end of Section
\ref{higher-GOS}.\\

\emph{Acknowledgements:} The authors wish to thank the participants of the
summer school for their attention and input. In particular they wish to thank
Pedro \'Angel Castillejo, Javier Fres\'an and Shahab Rajabi for pointing out numerous
misprints. The authors are also grateful to the referees for making several valuable comments.

\part{Ramification theory of local fields}

\section{Infinite Galois theory}\label{sec:galois}
In this section we first define the notion of a profinite group and then briefly summarize Galois theory for infinite algebraic
extensions. We will give more details in the more general situation of Section
\ref{sec:fundamentalGroup}.
\subsection{Profinite groups}
We begin by recalling the notion of a projective limit\index{projective
limit}; we mainly follow \cite[Ch.~V]{CasselsFroehlich}.

\begin{definition}
	\begin{enumerate}[label=(\alph*)]
		\item A \emph{directed set}\index{directed set} is a set $I$
			together with a partial ordering $\leq$, such that for
			every pair $i,j\in I$ there exists $n\in I$ such that
			$i\leq n$ and $j\leq n$.
		\item If $\mathcal{C}$ is a category, then a \emph{projective
				system in $\mathcal{C}$ indexed by a directed set
		$I$}\index{projective system} consists
			of the following data: For every $i\in I$ an object
			$X_i$ of $\mathcal{C}$ and for every $j\in I$ with $i \leq j$, a
			morphism $\phi_{ji}: X_j\rightarrow X_i$, such that if
			$i\leq j\leq n$, the diagram
			\begin{equation*}
				\begin{tikzcd}[row sep=.4cm, column sep=.7cm]
					X_n\ar[swap]{dd}{\phi_{nj}}\ar{dr}{\phi_{ni}}\\
					&X_i\\
					X_j\ar[swap]{ur}{\phi_{ji}}
				\end{tikzcd}
			 \end{equation*}
			commutes, and such that $\phi_{ii}=\id_{X_i}$.
%
		\item If 
			$\{ (X_i)_{i\in I}, (\phi_{ji})_{i\leq j\in I}\}$ is a
			projective system,
			let $\mathcal{C}'$ be the following category: Its 
			objects are tuples $(X,(q_i)_{i\in I})$, where
			$X$ is an object of $\mathcal{C}$ and
			$q_i:X\rightarrow X_i$ morphisms
			in $\mathcal{C}$, such that for every $i\leq
			j\in I$, the diagram
			\begin{equation*}
				\begin{tikzcd}[row sep=.4cm, column sep=.7cm]
					&X_j\ar{dd}{\phi_{ji}}\\
					X\ar{ur}{q_j}\ar[swap]{dr}{q_i}\\
					&X_i
				\end{tikzcd}
			 \end{equation*}
			 commutes.  A morphism $(X,(q_i)_{i\in
			 I})\rightarrow (X',(q'_i)_{i\in I})$ in
			 $\mathcal{C}'$ is a morphism $f:X\rightarrow X'$ in
			 $\mathcal{C}$, such that $q_i=q'_if$ for all $i\in
			 I$. 

			 If $\mathcal{C}'$ has a final object, then it is
			 unique up to unique isomorphism and this
			 final object is called the \emph{projective limit} of
			 the projective system $\{(X_i)_{i\in I},
			 (\phi_{ji})_{i\leq j\in I}\}$, and usually denoted
			 $(\varprojlim_I X_i, (\pr_i)_{i\in I})$. Often the
			 morphisms $\pr_i$ are omitted from the notation.
	\end{enumerate}
\end{definition}
\begin{example}\label{example:projectiveLimits}
	\begin{enumerate}[label=(\alph*)]
		\item The simplest example of a projective system is a
			constant projective system: If $X$ is an object of
			$\mathcal{C}$, and $I$ a directed set, consider the projective system given by
			$X_i=X$, $\phi_{ji}=\id_X$. Clearly, $\varprojlim_I X_i
			\cong X$, and $\pr_i=\id$. 
		\item\label{example:padics} If $\mathcal{C}$ is the category of groups, let $I=\N$
			with its usual order, $p$ a prime and $G_n:=\Z/p^n\Z$.
			If $m\geq n$, then the projection map $\phi_{mn}:\Z/p^{m}\rightarrow
			\Z/p^n$ makes $\{G_n, \{\phi_{mn}\}_{n\leq m\in N}\}$ into a projective
			system, and the projective limit $\varprojlim_{\N}
			G_n$ of this system is an abelian
			group, the \emph{$p$-adic integers} $\Z_p$. 
			
			Moreover, as the
			maps $\phi_{mn}$ are maps of rings, the group $\Z_p$
			also carries a ring structure, which makes it a
			projective limit in the category of commutative rings.
		\item\label{example:zhat} Let $\mathcal{C}$ be the category of groups and equip $I=\N$
			with the partial order defined by $(n$ ``$\leq$'' $m)
			:\Leftrightarrow n|m$. If $n,m$ are integers such that
			$n|m$, then again we have a projection morphism
			$\phi_{mn}:\Z/m\twoheadrightarrow \Z/n$, and
			$\{(\Z/n)_{n\in N}, (\phi_{mn})_{n|m\in \N}\}$ is a
			projective system. Its projective limit is an abelian
			group denoted $\widehat{\Z}$. It is not difficult to
			check that $\widehat{\Z}\cong \prod_{p\text{ prime}}
			\Z_p$.

			As in the previous example, $\widehat{\Z}$ is also a
			projective limit in the category of commutative rings.
		\item More generally: If $\mathcal{C}$ is the category of
			groups, $I$ a directed set and $\{(G_i)_{i\in I},
			(\phi_{ji})_{j\geq i\in I})\}$ a projective system in
			$\mathcal{C}$, then $\varprojlim_I G_i$ exists and it
			is a subgroup of the product $\prod_{i\in I} G_i$:
			\begin{equation}\label{eqn:product}\varprojlim_I G_i=\left\{\left. (g_i)\in \prod_{i\in I}
				G_i\right|\forall i\leq j:
				\phi_{ji}(g_j)=g_i\right\}\subset \prod_{i\in
				I} G_i.\end{equation}
			The morphisms $\pr_j:\varprojlim_I G_i\rightarrow G_j$
			belonging to the datum of the projective limit are
			induced by the projection maps $\prod_{i\in I}
			G_i\rightarrow G_j$.
	\end{enumerate}
\end{example}

We are now interested in the category $\TopGroup$ of topological groups. 
\begin{proposition}
	The projective limit of a projective system of topological groups
	exists. If the groups in the projective system are Hausdorff and
	quasi-compact, then so is
	the inverse limit of the system.  
\end{proposition}
Indeed, this follows from the description \eqref{eqn:product}.

If
$G$ is a finite group, we consider it as a topological group by equipping it
with the discrete topology.

\begin{proposition}
	Let $I$ be a directed set and $\{ (G_i)_{i\in I},
	(\phi_{ji})_{i\leq j\in I}\}$ a projective system of finite groups.
	Its projective limit $\varprojlim_I G_i$ computed in the category of
	abstract groups exists and it is a subgroup of the product
	$\prod_{i\in I} G_i$. If we equip each $G_i$ with the discrete
	topology and $\varprojlim_I G_i$ with the topology induced by
	$\prod_{i\in I} G_i$, then $\varprojlim_I G_i$ is also a projective
	limit of the system $\{ (G_i)_{i\in I},(\phi_{ji})_{i\leq j \in I}\}$
	computed in the category of topological groups.

	Moreover, the topology on $\varprojlim_I G_i$ is the coarsest topology
	such that the projections $\pr_i$ are continuous.
\end{proposition}

In other words: We get the same result if we equip the abstract group $\varprojlim_I G_i$ with the
topology induced by the product, or if we compute the projective limit
in the category of topological groups, by considering the $G_i$ as discrete
groups.

\begin{definition}
	A topological group $G$ is called \emph{profinite}\index{profinite
	group}, if it is
	isomorphic to the projective limit of a projective system of finite
	(discrete) groups in the category of topological groups.

	Similarly, if $p$ is a prime number, $G$ is called
	\emph{pro-$p$-group}\index{pro-$p$-group} if $G$ is isomorphic to the projective limit of a
	projective system of finite $p$-groups.
\end{definition}

\begin{example}
	\begin{enumerate}[label=(\alph*)]
		\item The groups $\Z_p$ and $\widehat{\Z}$ from Example
			\ref{example:projectiveLimits} are profinite groups. 
		\item The group $\Z_p$ is a pro-$p$-group.
	\item If $G$ is an abstract group, then we write
		\[\widehat{G}:=\varprojlim_{\begin{subarray}{c}H\lhd
				G\\(G:H)<\infty\end{subarray}}G/H,\]
			where $H$ runs through the directed set of normal
			subgroups of finite index of $G$. The profinite group $\widehat{G}$ is
			called the \emph{profinite completion of
			$G$}\index{profinite completion}.
		\item Similarly, if $p$ is a prime number,  we write 
			\[\widehat{G}^{(p)}:=\varprojlim_{\begin{subarray}{c}H\lhd
					G\\(G:H)=p\text{-power}\end{subarray}}G/H,\]
					where $H$ runs through the directed
					set of normal
					subgroups of finite $p$-power index.
					The pro-$p$-group
					$\widehat{G}^{(p)}$ is called the
					\emph{pro-$p$-completion of
					$G$}\index{pro-$p$-completion}.
	\end{enumerate}
\end{example}

In fact, there is a purely topological description of profinite groups:

\begin{proposition}[{\cite[V, \S1.4]{CasselsFroehlich}}]
	A topological group is profinite if and only if it is
	compact and
	totally disconnected.
\end{proposition}
\begin{exercise}
	If $G$ is a topological group and $U$ an open subgroup, show
	that $U$ is also closed. Moreover, if $G$ is compact, show that $U$
	has finite index in $G$.
\end{exercise}
\begin{corollary}
	If $G$ is profinite, then 
	\[G\cong \varprojlim G/U,\]
	where $U$ runs through the set of open normal subgroups of $G$.
\end{corollary}
\begin{corollary}
	If $G$ is a profinite group and $H\subset G$ a closed normal subgroup,
	then $H$ and $G/H$ are both profinite.

	More precisely,
	\[ H\cong \varprojlim H/H\cap U\text{ and }G/H\cong \varprojlim G/UH,\]
	where in both cases $U$ runs through the set of open normal
	subgroups of $G$.
\end{corollary}

\begin{remark}
	Some caution is required: If $G$ is a profinite group, then it is not
	always true  that the canonical map  $G\rightarrow \widehat{G}$ of $G$
	into its profinite completion is an
	isomorphism. In other words, it is not true that
	every finite index subgroup is open. However, if $G$ is a topologically
	finitely generated profinite group (i.e.~if there exists a finitely
	generated dense subgroup of $G$), then every subgroup of finite index
	is open (\cite{Segal}). 
	
	For example, if $p$ is a prime, consider the projective system
	$((\Z/p\Z)^n, \phi_{ji})$ indexed by $\N$, where for $j\geq i$,
	$\phi_{ji}:(\Z/p\Z)^j\rightarrow (\Z/p\Z)^i$ is the projection onto the
	first $i$ factors. Then $\varprojlim_n ((\Z/p\Z)^n, \phi_{ji})$ is the product of countably many copies of $\Z/p\Z$. This is a profinite group
	which has infinitely many finite index subgroups which are not open.
	This can be used to show that there are
	finite index subgroups of $\Gal(\overline{\Q}/\Q)$, which are not
	open. In fact this is true for the subfield
	of $\overline{\Q}$ spanned by $\sqrt{-1}$ and $\{\sqrt{p}\;|\;p\text{
	 prime}\}$ (exercise or \cite[\S7.]{Milne/FieldTheory}). 
\end{remark}
\subsection{The Galois correspondence for infinite algebraic extensions}
In this section we recall the basic statements of Galois theory for infinite
extensions, as it can be found, for example, in \cite{Neukirch/1999},
\cite{Milne/FieldTheory} or \cite{Szamuely}. For a more general treatment see Section
\ref{sec:fundamentalGroup}. 

We assume that the reader is familiar with Galois
theory of finite extensions.

Let $K$ be a field and $L$ an algebraic extension of $K$. The extension $L/K$
is called Galois if $L^{\Aut_K(L)}=K$. In this case we also write
$\Gal(L/K):=\Aut_K(L)$. Equivalently, $L/K$ is Galois if and only if the
minimal polynomial of every element $\alpha\in L$ splits into distinct linear
factors over $L$. In particular, Galois extensions are separable.
\begin{proposition}
	If $L/K$ is a Galois extension of fields, write 
	\[\mathcal{G}_{L/K}:=\{F\subset L\;|\; F/K\text{ is a finite Galois
	extension}\}.\] Then $L=\bigcup_{F\in \mathcal{G}_{L/K}}F$, and the inclusion
	relation makes $\mathcal{G}_{L/K}$ into a directed set. If 	$F_1\subset
	F_2\in \mathcal{G}_{L/K}$, then restriction
	of automorphisms
	induces a homomorphism $\Gal(F_2/K)\rightarrow \Gal(F_1/K)$. This
	makes the set of groups $\Gal(F/K)$, $F\in \mathcal{G}_{L/K}$, into a projective
	system, and we have
	\[\Gal(L/K)=\varprojlim_{F\in \mathcal{G}_{L/K}}\Gal(F/K).\]
	
	In particular, $\Gal(L/K)$ is a profinite group.
\end{proposition}

\begin{exercise}
	\begin{enumerate}[label=(\alph*)]
		\item Let $p$ be a prime, $n\in \N$, $q=p^n$  and $\F_{q}$ the
			field with $q$
	elements. Fix an algebraic closure $\overline{\F_q}$. Show that the
	profinite group $\Gal(\overline{\F}_q/\F_{q})$ is isomorphic to
	$\widehat{\Z}$.
\item Let $\Q^{\text{cyc}}\subset \C$ be the smallest subfield containing all roots of unity.
	Is $\Q^{\text{cyc}}$ a Galois extension of $\Q$? If so,  
	what is $\Gal(\Q^{\text{cyc}}/\Q)$?
\end{enumerate}
\end{exercise}

\begin{theorem}[Galois correspondence\index{Galois correspondence}]
	Let $L/K$ be a Galois extension and $\Gal(L/K)$ its Galois group. There is an order
	reversing bijection
\begin{equation*}
	\begin{tikzcd}[row sep=1ex]
			\left\{\text{subextensions of
			}L/K\right\}\rar{\cong}&\left\{\text{closed subgroups
			of }\Gal(L/K)\right\}\\
			F\rar[mapsto]&\Gal(L/F)
		\end{tikzcd}
	\end{equation*}

	Under this correspondence,
	\begin{enumerate}
		\item finite subextensions of $L/K$ correspond to
	open subgroups, 
\item subextensions which are Galois over $K$ correspond
	to closed normal subgroups.
	\end{enumerate}
\end{theorem}
\begin{remark}
	In Section \ref{sec:fundamentalGroup} we will see a vast
	geometric generalization of this correspondence. If $F/K$ is a finite
	separable extension, then the associated morphism $\Spec F \rightarrow
	\Spec K$ will be interpreted as a (nontrivial!) covering space. The
	Galois group of $F/K$ will then act as the group of deck
	transformations.
\end{remark}


\section{Ramification groups and the theorem of Hasse-Arf}\label{Ram-groups-Hasse-Arf}
In this rather lengthy section we summarize the ramification theory
of finite separable extensions of a complete discretely valued field. Our main
references are \cite{Serre/LocalFields} and \cite{Neukirch/1999}. If $L/K$
is such an extension, assume that the associated extension of residue fields is
also separable (which is automatic if the residue field of $K$ is perfect). In this case, if $L/K$ is Galois,  we will construct two descending filtrations on the Galois group
$G:=\Gal(L/K)$: the \emph{lower numbering filtration}  $\{G_u\}_{u\in \Z_{\geq
-1}}$ and the \emph{upper
numbering} $\{G^v\}_{v\in \R_{\geq -1}}$.  The subgroups appearing in 
both filtrations are the same, but the lower numbering is adapted to
taking subgroups, while the upper numbering is adapted to taking quotients of
$G$, i.e.~if $H\lhd G$ is a normal subgroup, then $H_u=G_u\cap H$, and
$(G/H)^v=G^v/(H\cap G^v)$.

The numbers $\lambda\in \R_{\geq -1}$, for which  $G^{\lambda}\neq
G^{\lambda+\epsilon}$ for all $\epsilon>0$, are called \emph{jumps} or
\emph{breaks} of the
filtration. The
theorem of Hasse-Arf (Theorem \ref{thm:HasseArf}) states that the jumps are integers if
$G$ is abelian. This theorem is a crucial ingredient for the ramification
theory of $\ell$-adic representations developed in the following sections.

\subsection{Discretely valued fields and discrete valuation rings}
We first recall some basic material on discrete valuations.

\begin{definition}
	Let $K$ be a field. A \emph{discrete valuation}\index{discrete
	valuation on $K$} is a surjective homomorphism of abelian groups
	\[v: K^\times \rightarrow \Z\]
such that $v(x+y)\geq \min\{v(x),v(y)\}$, for all $x,y\in K^\times$ with $x+y\neq 0$. We
	extend $v$ to $K$, with the convention that $v(0)=\infty$.

	If $K$ is equipped with a discrete valuation $v$, then $(K,v)$ is
	called a \emph{discretely valued field}\index{discretely valued
	field}.
\end{definition}
\begin{remark}
	Some authors do not require a valuation to be surjective, but only
	nontrivial.
\end{remark}

\begin{exercise}\label{ex:nonArchimedianTriangleInequality}
	Let $(K,v)$ be a discretely valued field and $x,y\in K$. Show that
	$v(x+y)=\min\{v(x),v(y)\}$ if $v(x)\neq v(y)$.
\end{exercise}

If $(K,v)$ is a discretely valued field, then the subset $A:=A_K:=\{x\in K|
v(x)\geq 0\}$ is a commutative ring with unique maximal ideal $\mathfrak{m}:=\mathfrak{m}_K:=\{x\in
A| v(x)>0\}$. The local ring $A$ is called \emph{valuation ring of $K$}, and
it is in fact a \emph{discrete valuation
ring}\index{discrete valuation ring}, i.e.~a local principal ideal domain,
which is not a field. A generator $\pi$ of $\mathfrak{m}_K$ is
called \emph{uniformizer of $A$, or uniformizer of $K$}\index{uniformizer}. With our definitions, we
always have $v(\pi)=1$. The
\emph{residue field}\index{residue field} $A/\mathfrak{m}_K$ is also said to be the residue field of
$K$.

\begin{example}
	The three main examples to keep in mind are as follows:
	\begin{enumerate}[label=(\alph*)]\label{ex:dvr}
		\item Let $K=\Q$ be the field of rational numbers and $p$ a
			prime number. An nonzero integer $a\in \Z$ can be uniquely
			written as $a=a'p^r$, with $a'\in \Z$ prime to $p$,
			and $r\in
			\Z_{\geq 0}$. We define $v_p(a):=r$, and for
			$\frac{a}{b}\in
			\Q^\times$, $v_p(\frac{a}{b}):=v_p(a)-v_p(b)$. This defines a
			discrete valuation on $\Q$, which is called the
			\emph{$p$-adic valuation}. The valuation ring
			associated with this valuation is $\Z_{(p)}$: the ring
			of rational numbers with denominators prime to $p$,
			and clearly $p$ is a uniformizer of $(\Q,v_p)$. The
			residue field of $(\Q,v_p)$ is $\F_p$, the finite field with
			$p$ elements.
		\item\label{ex:dvr:curve} Let $k$ be a field and $C$ a
			connected, affine, normal, algebraic curve
			over $k$. This means that $C$ is of the form $C=\Spec
			A$, with $A$ a normal, finite type $k$-algebra of
			Krull dimension $1$. Let $K:=k(C)$ be its function field,
			i.e., the local ring of $C$ at its generic point, which
			amounts to saying that $K$ is the field of fractions
			of $A$.  If
			$c\in C$ is a closed point, we can associated with it
			a discrete valuation $v_c$ on $K$: $c$ corresponds to
			a maximal ideal $\mathfrak{m}_c\subset A$, and if
			$f\in A\setminus \{0\}$, there exists an integer $r\in\Z_{\geq 0}$,
			such that $f\in \mathfrak{m}_c^r\setminus
			\mathfrak{m}_c^{r+1}$. We define $v_c(f):=r$, and for
			$\frac{f}{g}\in K^\times$, $v_c(\frac{f}{g})=v_c(f)-v_c(g)$.
			The discrete valuation ring associate with $v_c$ is
			$A_{\mathfrak{m}_c}=\mathcal{O}_{C,c}$, and its
			residue field $k(c)=A_{\mathfrak{m}_c}/\mathfrak{m}_c$ is a finite extension of $k$.

			If $C$ is a normal, proper curve over $k$, the converse is also true: Every discrete valuation of
			$K$ which is trivial on $k$ is of the shape $v_c$ for some closed point $c\in
			C$ (valuative criterion of properness).
		\item More generally, if $k$ is a field and $X$ a connected, finite type,
			normal $k$-scheme, to every codimension $1$ point
			$\eta$ of $X$ (i.e.~to every point $\eta$ such that the closure of $\{\eta\}$ 
			has codimension $1$), one can attach a discrete valuation
			$v_{\eta}$ in the same way as in the previous example.
			However, if $\dim X>1$, there are discrete valuations
			on the function field $K=k(X)$ which do not arise from
			codimension $1$ points of $X$\footnote{There are even
				discrete valuations which do not arise from
				codimension $1$ points on a different model of
				$k(X)$, see \cite[Ch.~8, Thm.~3.26]{Liu}.}.
	\end{enumerate}
\end{example}

The globalization of the notion of a discrete valuation ring is called
\emph{Dedekind domain}\index{Dedekind domain}. 

\begin{definition} A \emph{Dedekind
	domain} $A$ is a noetherian integral domain, which is not a field, such that the localization at each
maximal ideal is a discrete valuation ring. This is equivalent to $A$ being
normal, noetherian and $1$-dimensional (see \cite[Ch.~I,
Prop.~4]{Serre/LocalFields}).
\end{definition}
Of course $\Z$ is a Dedekind domain,
and so is the ring $A$ from Example \ref{ex:dvr}, \ref{ex:dvr:curve}. In fact, geometrically, one can think of Dedekind domains as rings
of functions on a smooth curve.

\subsection{Completion}
A discretely valued field $(K,v)$ is equipped with a natural topology, given
by the ``ultrametric'' norm $\|\cdot\|_v$ defined by $\|x\|_v=\exp(-v(x))$. In
fact, one could also take $\|x\|:=a^{v(x)}$ for $a$ any real number in
$(0,1)$, but the induced topology on $K$ is independent of the choice of $a$.
We will mostly be interested in discretely valued fields $(K,v)$ which are
complete with respect to this topology. Such fields are, naturally, called
\emph{complete discretely valued fields}\index{complete discretely valued
field}. If $(K,v)$ is complete, then its valuation
ring $A$ is complete in the $\mathfrak{m}_K$-adic topology, i.e.~$A\cong
\varprojlim_n A/\mathfrak{m}_K^n$.

There is a standard process of completing a discretely valued field $(K,v)$: Let
$K_v$ denote the ring of Cauchy
sequences in $K$ modulo the maximal ideal of sequences converging to
$0$. The valuation $v$ can be  extended uniquely to $K_v$ according to the exercise below,
so that $(K_v, v)$ becomes a complete discretely valued field. If $A$ is the valuation ring of $K$ with maximal ideal $\mathfrak{m}_K$, then $K_v$ is the fraction field of the complete discrete valuation ring $\varprojlim_n A/\mathfrak{m}^n$.

\begin{exercise}
	Let $(a_n)_{n\in \N}$ be a Cauchy sequence in $K$ not converging to
	$0$. Show that the sequence of
	integers  $(v(a_n))_{n\in \N}$  becomes stationary.
\end{exercise}

\begin{example}We extend Example \ref{ex:dvr}:
	\begin{enumerate}[label=(\alph*)]
		\item If $K=\Q$ is given the $p$-adic valuation $v_p$, then
			its completion is the field of $p$-adic numbers
			$\Q_p$. Its valuation ring is the ring of $p$-adic
			integers $\Z_p$. Note that the additive group
			underlying $\Z_p$ coincides with the
			pro-$p$-completion of the abelian group $\Z$.
		\item If $k$ is a field and $C$ a normal, connected, algebraic curve over
			$k$, and $c\in C(k)$ a rational point, then the choice
			of a uniformizer $t$ of $\mathcal{O}_{C,c}$ gives rise
			to a canonical isomorphism between the completion
			$(K_c, v_c)$ of $k(C)$ with respect to
			$v_c$ and the
			field $k\llparen t\rrparen$ of Laurent series over $k$.
			Geometrically, $K_c$ should be thought of as an
			infinitesimally small, punctured disc around $c$.
	\end{enumerate}
\end{example}

\subsection{Extensions of discretely valued fields}\label{sec:extensions}
If $(K,v)$ is a complete discretely valued field with valuation ring $A$, and $L/K$ a finite
extension of fields, then there exists a unique discrete valuation $v_L$ on $L$
and an integer $e> 0$, such that for all $x\in K$, $v_L(x)=ev(x)$, and $L$ is
complete with respect to the discrete valuation $v_L$. Moreover, 
the integral closure $B$ of $A$ in $L$ is a complete discrete valuation
ring and a free $A$-module of rank $[L:K]$ (\cite[Ch.~II,
Prop.~3]{Serre/LocalFields}).

\begin{definition}\label{defn:localRamification}Keep the notations from the previous paragraph.
\begin{itemize}
	\item By abuse of notation we say that \emph{$v_L$ extends $v$}.
	\item The integer $e$ is called the \emph{ramification index of $L$ over $K$}\index{ramification index}.
	\item The extension $L/K$ induces a finite extension of their residue fields. Its degree is denoted by $f$ and called the \emph{residue degree of $L$ over $K$}.
	\item We say that $L/K$ is \emph{totally ramified}\index{totally ramified} if $f=1$.
	\item If $e=1$ and if the extension of residue fields is separable, then $L/K$ is
called \emph{unramified}\index{unramified}.
\end{itemize}
\end{definition}

In the situation of the definition, if $n=[L:K]$, then
$n=ef$.\\

We globalize:
Let $(K,v)$ be a discretely valued field (not necessarily complete) with
valuation ring $A$ and $L/K$  a finite extension. Denote by $B$ the integral
closure of $A$ in $L$, then according to the Krull-Akizuki theorem
(\cite[Ch.~I, Thm.~12.8]{Neukirch/1999}) $B$ is a Dedekind domain, integral
over $A$, but not necessarily local. In other words: The discrete valuation
$v$ can be extended to $L$, but there are finitely many distinct valuations
$w_1,\ldots, w_r$ which extend it. 
\begin{lemma}\label{lemma:japanese}
	The Dedekind domain $B$ is a finitely generated $A$-algebra, and hence
	a finitely generated $A$-module, in each of the
	following situations:
	\begin{enumerate}
		\item $A$ is a complete discrete valuation ring
			(see the discussion
			before Definition \ref{defn:localRamification}).
		\item $L/K$ is separable
			\emph{(\cite[Ch.~I, Prop.~8]{Serre/LocalFields})}.
		\item $A$ is the localization of a finitely generated algebra
			over a field\footnote{or
			more generally, if $A$ is a \emph{Japanese/N-2 ring}, see
			\cite[(31.H)]{Matsumura/CommutativeAlgebra}.}.
	\end{enumerate}
\end{lemma}
In our applications, $B$ will always be a finitely  generated, hence finite, $A$-algebra, so from now on
we tacitly assume this is the case.

\begin{proposition}[{\cite[Ch.~II, Thm.~1]{Serre/LocalFields}}]Let $(K,v)$ be a
	discretely valued field with valuation ring $A$ and $L/K$ a finite
	extension of degree $n$. Assume that the integral closure $B$ of $A$ in $L$ is
	a finite $A$-algebra. Write $w_1,\ldots, w_r$ for the valuations of
	$L$ extending $v$.
	\begin{enumerate}[label=\emph{(\alph*)}]\label{prop:extensionAndCompletion}
		\item The canonical map $L\otimes_K K_{v}\rightarrow
			\prod_{i=1}^rL_{w_i}$ is an isomorphism.
		\item $L_{w_i}/K_{v}$ is a finite extension.
	\end{enumerate}
\end{proposition}
Thus, for each $i$ we obtain numbers $n_i$, $e_i$, and $f_i$: the degree,
ramification index and residue degree of the extension $L_{w_i}/K_{v}$ of
complete discretely valued fields. As before we have $n_i=e_if_i$ and
\begin{equation}\label{eq:nef-formular}n= \sum_{i=1}^r n_i=\sum_{i=1}^re_if_i.\end{equation}

There is also an ideal theoretic interpretation of these numbers. The Dedekind
domain $B$ is a finite extension of the discrete valuation ring $A$, and the
discrete valuations $w_1,\ldots, w_r$ correspond to the finitely many nonzero prime 
ideals $\mathfrak{P}_i$ of $B$, i.e.~$\mathfrak{P}_i=\{b\in B|w_i(b)>0\}$. Since $B$
is a Dedekind domain, it has the pleasant property that any ideal $I$ can be uniquely
written as a product
\[I=\mathfrak{P}_1^{s_1}\cdot\ldots\cdot\mathfrak{P}_r^{s_r}.\]
In particular, if
$\mathfrak{m}$ is the maximal ideal of $A$, then  
$\mathfrak{m}B$ is a proper ideal in $B$, as the extension is integral, and
hence $\mathfrak{m}B$ can be written as such a product of prime ideals.
Indeed, 
\[\mathfrak{m}B=\mathfrak{P}_1^{e_1}\cdot\ldots\cdot\mathfrak{P}_r^{e_r},\]
where the $e_i$ are precisely the ramification indices from above. Moreover,
since $B$ is $1$-dimensional, the (nonzero) prime ideals $\mathfrak{P}_i$ are
maximal ideals, so we get finite extensions $k(v)=A/\mathfrak{m}\subset
B/\mathfrak{P}_i$, and these extensions are precisely the degree $f_i$
extensions of the residue fields of $L_{w_i}/K_{v}$.

\begin{definition}\label{def:unramified}
	Keep the notation from the previous paragraph. 
	\begin{enumerate}[label=(\alph*)]
		\item If $r=1$ and $f_1=1$, i.e.~if there is only one
			valuation $w_1$ on $L$
			extending $v$ and the residue extension is trivial,
			then $L/K$ is called \emph{totally
			ramified}\index{totally ramified}.
		\item If $e_i=1$ and if the residue extension
			$A/\mathfrak{m}_K\subset B/\mathfrak{P}_i$ of
			$L_{w_i}/K_{v_K}$ is separable, then $L/K$ is called
			\emph{unramified at $\mathfrak{P}_i$ or at
		$w_i$}\index{unramified}.
	\item If $L/K$ is unramified at $w_1,\ldots, w_r$, then $L/K$ is
		called \emph{unramified}\index{unramified}.
\end{enumerate}
\end{definition}
\subsection{Discriminant and different}
We start with a definition.
\begin{definition}\label{defn:norm}
	Let $L/K$ be a finite extension of fields, or, more generally, let $K$
	be a ring and $L$ a $K$-algebra which is free of finite rank as a
	$K$-module.
	\begin{itemize}
		\item For $x\in L$ define $N_{L/K}(x)\in K$ as the determinant of the $K$-linear endomorphism of $L$ given by multiplication with $x$.  
		The map $N_{L/K}:L^\times\rightarrow K^\times$  is called the	\emph{norm of $L/K$}\index{norm}.

	\item Similarly we define the \emph{trace map of $L/K$}\index{trace},
	$\Tr_{L/K}:L\rightarrow K$, by defining $\Tr_{L/K}(x)$ to be the trace
	of the $K$-linear endomorphism of $L$ given by multiplication with
	$x$.
	\end{itemize}
\end{definition}
An easy computation shows the following standard facts.
\begin{proposition}
	Let $L/K$ be a finite separable extension, and fix an algebraic closure
	$\overline{K}$ of $K$. Let $\sigma_1,\ldots, \sigma_r$ be the
	$K$-embeddings of $L$ into $\overline{K}$.
	Then for $x\in L$, we have
	\[N_{L/K}(x)=\prod_{i=1}^r\sigma_i(x)\]
	and
	\[\Tr_{L/K}(x)=\sum_{i=1}^r\sigma_i(x).\]
	Moreover if $K$ is complete with respect to the discrete valuation $v$, and $v_L$ discrete valuation on $L$ extending $v$, then for $x\in L$,
	\[v_L(x)=\frac{1}{f}v(N_{L/K}(x)),\]
	where $f$ is the degree of the residue extension of $L/K$.
\end{proposition}
Recall the following fact.
\begin{proposition}[{\cite[IX, \S2, Prop.~5]{Bourbaki/AlgebreIX}}]
	\label{prop:separability}
	If $R$ is a finite $K$-algebra, then $R$ is a product of finite separable field
	extensions of $K$ if and
	only if $R\times R\rightarrow R$, $(x,y)\mapsto \Tr_{L/K}(xy)$ is a nondegenerate bilinear form. 
\end{proposition}

\begin{definition}\label{defn:Bdual}
Let $A$ be a Dedekind domain with fraction field $K$ and $L$ a finite
separable extension. Write $B$ for the integral closure of $A$ in $L$. Define the
$B$-module
\[B^{\vee}:=\{x\in L|\Tr_{L/K}(xy)\in A\text{ for all }y\in B\}.\]
\end{definition}
\begin{proposition}
	We have $B\subset B^\vee$, and $B^{\vee}$ is a finitely
	generated $B$-submodule of $L$, i.e.~a fractional ideal of $B$.
	The map
	\[B^\vee\rightarrow \Hom_{A}(B,A),\quad x\mapsto (y\mapsto
		\Tr_{L/K}(xy))\]
	is an isomorphism of $B$-modules. The inverse of the fractional
	ideal $B^{\vee}$ is an ideal in $B$.
\end{proposition}
\begin{proof}
	Since $L/K$ is separable, Proposition \ref{prop:separability} shows
	that the map of $K$-vector spaces
	\begin{equation}\label{eq:nondegenerateIso}L\rightarrow \Hom_K(L,K),\quad x\mapsto
		(y\mapsto \Tr_{L/K}(xy))\end{equation}
	is an isomorphism.
	On the other hand, every $A$-morphism $B\rightarrow A$ extends
	uniquely to a $K$-morphism $L\rightarrow K$, so $\Hom_{A}(B,A)$
	is a sub-$A$-module of $\Hom_K(L,K)$, and it is precisely the image
	of $B^\vee$
	under \eqref{eq:nondegenerateIso}. Since $B$ is a finitely generated
	$A$-module, the same is true for $\Hom_{A}(B,A)$ and hence for
	$B^{\vee}$.

	From $\Tr_{L/K}(B)\subset A$,  it follows that $B^{\vee}\supset B$, so
	$B=B^{\vee}\cdot(B^{\vee})^{-1}\supset (B^{\vee})^{-1}$. This finishes
	the proof.
\end{proof}
Here is another characterization of $B^{\vee}$:
\begin{lemma}\label{lemma:maxCharOfDifferent}
	With the notations from above, $B^{\vee}$ is the largest
	sub-$B$-module of $L$ with the property that $\Tr_{L/K}(B^{\vee})\subset A$.
	In particular, if $\mathfrak{b}\subset L$ and $\mathfrak{a}\subset K$
	are fractional ideals of $B$, $A$, then $\Tr_{L/K}(\mathfrak{b})\subset
	\mathfrak{a}$ if and only if $\mathfrak{b}\subset \mathfrak{a}
	B^{\vee}$.
\end{lemma}
\begin{proof}
	If $E$ is a sub-$B$-module of $L$ such that $\Tr_{L/K}(E)\subset A$, then
	$E\subset B^{\vee}$, so the maximality statement is clear.

	For the second statement, note that
	$\Tr_{L/K}(\mathfrak{b})\subset\mathfrak{a}$ if and only if
	$\Tr_{L/K}(\mathfrak{a}^{-1}\mathfrak{b})\subset A$, if and only if
	$\mathfrak{b}\subset \mathfrak{a}B^{\vee}$.
\end{proof}

\begin{definition}\label{def:different}
	Keeping the notations from Definition \ref{defn:Bdual}, we write
	$\mathfrak{D}_{B/A}:=\left(B^{\vee}\right)^{-1}$. This ideal of
	$B$ is called the \emph{different of $B/A$}\index{different}.

	The ideal $N_{L/K}(\mathfrak{D}_{B/A}):=\mathfrak{d}_{B/A}$ of $A$ is called
	the \emph{discriminant of $B/A$}\index{discriminant}.
\end{definition}
\begin{remark}
	Here the norm of an ideal is defined as follows: In a Dedekind domain,
	every ideal can be uniquely written as a product of prime ideals. If
	$\mathfrak{P}$ is a nonzero prime ideal of $B$, then we define
	$N_{L/K}(\mathfrak{P}):=\prod_{\mathfrak{P}|\mathfrak{p}}\mathfrak{p}^{f_{\mathfrak{p}}}$,
	where $f_{\mathfrak{p}}$ is the degree of the residue extension
	$(B/\mathfrak{P})/(A/\mathfrak{p})$. Since we will only use these
	definitions in the case where $L/K$ is an extension of \emph{complete}
	discretely valued fields, we will only care about Dedekind rings which
	are discrete valuation rings, so we can forget the more complicated
	definition from above, and take the usual norm.
\end{remark}
\begin{proposition}[Transitivity of different and
	discriminant]\label{prop:transitivityOfDifferent}
	Let $L/K$ and $M/L$ be finite separable extensions, $A\subset K$ a
	Dedekind domain with $\Frac(A)=K$, and $B, C$ the integral closures of
	$A$ in $L$, resp.~$M$.
	 Then
	\[\mathfrak{D}_{C/A}=\mathfrak{D}_{C/B}\mathfrak{D}_{B/A}\]
	and
	\[\mathfrak{d}_{C/A}=N_{L/K}(\mathfrak{d}_{C/B})\mathfrak{d}_{B/A}^{[M:L]}.\]
\end{proposition}
\begin{proof}
	The formula for the discriminant follows from the formula for the
	different by applying $N_{M/K}$. The formula for the different follows
	from the transitivity of the trace, i.e.~from the fact that
	$\Tr_{M/K}=\Tr_{L/K}\circ\Tr_{M/L}$ (exercise).
\end{proof}

The different and discriminant indicate whether an extension is ramified or
not. To see this, we now put ourselves in the complete local situation. 
\begin{lemma}\label{lem:DiscriminantFormula}
	Let $L/K$ be a finite separable extension of complete discretely
	valued fields, $A$ the valuation ring of $K$, $B$ its integral closure
	in $L$. Let $x_1,\ldots, x_n$ be a
	basis of $B$ as an $A$-module. Let $K^{\sep}$ be a separable closure of
	$K$. Then
	\[\mathfrak{d}_{B/A}=(\det( (\Tr_{L/K}(x_ix_j))_{ij}))=\left(\det(
		(\sigma_ix_j)_{ij})\right)^2,\]
	where $\sigma_i$ runs through the finite number of $K$-linear embeddings of $L$ in
	$K^{\sep}$. 
\end{lemma}
\begin{proof}First note that for any $a_1,\ldots, a_n, b_1,\ldots, b_n\in L$, we have
	\[\Tr_{L/K}(a_ib_j)=\sum_{k=1}^n\sigma_k(a_i)\sigma_k(b_j)=(\sigma_1(a_i),\ldots, \sigma_n(a_i))\cdot\begin{pmatrix}\sigma_1(b_j)\\\vdots\\\sigma_n(b_j)
		\end{pmatrix},\]
	so $\det(\Tr_{L/K}(a_ib_j)_{ij})=\det( (\sigma_j(a_i))_{ij})\det((\sigma_i(b_j))_{ij})$.
	
	The $A$-module $B^\vee=\Hom_A(B,A)$ is spanned by the basis $x_1',\ldots,
	x_n'$, where $x_i'(x_j)=\Tr_{L/K}(x_i'x_j)=\delta_{ij}$. But $B^{\vee}$ is a
	fractional ideal, so $B^{\vee}=(\beta)$ for some $\beta\in L^\times$, and $\beta x_1,\ldots, \beta
	x_n$ is also an $A$-basis for $B^{\vee}$. 

	Let $M=(m_{ij})\in \GL_n(A)$ be the base change matrix such that $\beta x_i = \sum_{j=1}^n m_{ij}x'_j$. Then $M
	\cdot(\sigma_jx_i')_{ij}=(\sigma_j(\sum_{r=1}^nm_{ir}x'_r))_{ij}=(\sigma_j(\beta x_i))_{ij}$.
	It follows that 
	\begin{align*}\det(\Tr_{L/K}(\beta x_ix_j)) &= \det( M\cdot(\sigma_j(x'_i))_{ij}\cdot (\sigma_i(x_j))_{ij})\\
		&=\det(M\cdot(\Tr_{L/K}(x'_ix_j))_{ij})\\&
		=\det(M)\in A^\times.\end{align*}
	On the other hand, $\det(\Tr_{L/K}(\beta x_ix_j))_{ij}=N_{L/K}(\beta)\det(\Tr_{L/K}(x_ix_j)_{ij})$. It follows that $\mathfrak{d}_{B/A}=(N_{L/K}(\beta^{-1}))=(\det(\Tr_{L/K}(x_ix_j)_{ij}))=(\det((\sigma_i(x_j))_{ij}))^2$.
	
\end{proof}

\begin{proposition}
	Let $L/K$ be a finite separable extension of complete discretely valued fields.  Then
	$L/K$ is unramified if and only if $\mathfrak{D}_{B/A}=B$, if and
	only if $\mathfrak{d}_{B/A}=A$.
\end{proposition}
\begin{proof}
As $\mathfrak{d}_{B/A}=N_{L/K}(\mathfrak{D}_{B/A})$, we see that  
	$\mathfrak{D}_{B/A}=B$ if and only if
	$\mathfrak{d}_{B/A}=A$.

	By definition,  the extension $L/K$ is unramified if and only if $\mathfrak{m}_KB=\mathfrak{m}_L$,
	and if the residue extension is separable. 
	Let $x_1,\ldots, x_r$ be a basis of $B$ over $A$. Then the residue
	classes $\bar{x}_i$ form a basis of $B/\mathfrak{m}_KB$ over
	$A/\mathfrak{m}_K$. Since $B/\mathfrak{m}_KB$ is a local ring, $B$ is
	unramified over $A$, if and only if $B/\mathfrak{m}_KB$ is a separable
	$A/\mathfrak{m}_K$-algebra. By Proposition \ref{prop:separability}
	this is equivalent to $\det(\Tr_{L/K}(\bar{x}_i\bar{x}_j))\not\equiv 0 \mod
	\mathfrak{m}_K$.  According to the lemma this is equivalent to 
	$\mathfrak{d}_{B/A}=A$.
\end{proof}

We need two more facts about the different.
\begin{theorem}\label{thm:ableitung}
	Let $L/K$ be a finite separable extension of complete discretely valued fields.
	Assume that there exists $\alpha\in B$ which 
	generates $B$ as an $A$-algebra. If $f\in
	K[X]$ is the monic minimal polynomial of $\alpha$, then $f\in A[X]$, and
	$\mathfrak{D}_{B/A}=(f'(\alpha))$.
\end{theorem}
\begin{proof}
	Let $f(X)\in K[X]$ be the monic minimal polynomial of $\alpha$. The
	coefficients of $f(X)$ are symmetric functions in the roots of $f(X)$.
	Since $\alpha$ is integral over $A$, the same is true for all roots of
$f(X)$ (in a sufficiently large extension of $K$), so $f(X)\in A[X]$.
Let $b_0,\ldots,b_{n-1}\in L$ such that
\begin{equation}\label{eq:defnOfb}\frac{f(X)}{X-\alpha}:=b_0+b_1
	X+\ldots+b_{n-1}X^{n-1}.\end{equation}
The same argument as above shows that $b_0,\ldots, b_{n-1}$ are integral over
$B$, hence elements of $B$.

Note that $1,\alpha,\ldots, \alpha^{n-1}$ is a basis of $B$ over $A$. We
compute its dual basis of $B^{\vee}$.

Let $\alpha=\alpha_1,\ldots, \alpha_n$ be the distinct roots of $f(X)$ (in a
sufficiently large extension of $K$). Then $f'(\alpha_i)=\prod_{j\neq
i}(\alpha_i-\alpha_j)$. Hence, for $r\geq 0$ we  compute
\[\left(\frac{f(X)}{X-\alpha_i}\right)(\alpha_j)\cdot\frac{\alpha_i^r}{f'(\alpha_i)}=\delta_{ij}\alpha_i^r.\]
Consequently, if $r=0,\ldots, n-1$, then 
\[\Tr_{L[X]/K[X]}\left(\frac{f(X)}{X-\alpha}\frac{\alpha^r}{f'(\alpha)}\right)=\sum_{i=1}^n\frac{f(X)}{X-\alpha_i}\frac{\alpha_i^r}{f'(\alpha_i)}=X^r,\]
as the difference of both sides is a polynomial of degree $<n$ with zeroes $\alpha_1,\ldots,
\alpha_n$. In turn, we get 
\[X^r=\Tr_{L[X]/K[X]}\left(\frac{f(X)}{X-\alpha}\frac{\alpha^r}{f'(\alpha)}\right)=\sum_{j=0}^{n-1}\Tr_{L/K}\left(\alpha^r\frac{b_j}{f'(\alpha)}\right)X^j,\]
which shows that $\frac{b_0}{f'(\alpha)},\ldots, \frac{b_{n-1}}{f'(\alpha)}\in
B^{\vee}$ is the dual basis of $1,\alpha, \ldots, \alpha^{n-1}$, and hence
\[B^{\vee}=\frac{1}{f'(\alpha)}\cdot(b_0A+\ldots+ b_{n-1}A)\subset \frac{1}{f'(\alpha)}B.\]

Finally, if we write $f(X)=X^n+\sum_{i=0}^{n-1}a_iX^i$ with $a_0,\ldots, a_{n-1}\in K$, then from \eqref{eq:defnOfb} we know that 
\begin{align*}
	b_{n-1}&=1\\
	b_{n-i}-\alpha b_{n-i+1}&=a_{n-i+1}\text{ for }1<i<n\\
	-b_0\alpha&=a_0.
\end{align*}
Inductively it follows that $\alpha^j\in \bigoplus_{i=0}^{n-1}b_iA$ for
$j=0,\ldots, n-1$, so $B=\bigoplus_{i=0}^{n-1}b_iA$. It follows that
\[B^{\vee}=(f'(\alpha))^{-1},\]
which proves the claim.
\end{proof}
The hypothesis of the theorem is always satisified when the residue extension of $L/K$ is
separable.
\begin{theorem}[{\cite[Ch.~III, Prop.~12]{Serre/LocalFields}}]\label{thm:generator}
	Let $L/K$ be a finite extension of complete discretely valued fields
	with separable residue extension. Then there exists an element
	$x\in B$, generating $B$ as an $A$-algebra. Moreover, $x$ can be
	chosen such that  there exists a
	monic polynomial $R\in A[X]$ of degree $f$ such that $R(x)$ is a
	uniformizer of $B$. If $L/K$ is totally ramified, any uniformizer $x$
generates $B$ as an $A$-algebra.\end{theorem}

For future reference, we also prove:
\begin{proposition}\label{valuation-of-discriminant}
	Let $L/K$ be a finite separable extension of complete discretely valued fields
	with separable residue extension. Write $A$ for the valuation ring of
	$K$ and $B=A[\alpha]$ for its integral closure in $L$. 
	\begin{enumerate}
		\item If the
			ramification index of $L/K$ is $e$, then
			\[v_L(\mathfrak{D}_{B/A})\geq e-1\] with equality if and only
			if $L/K$ is tamely ramified (Definition \ref{defn:tameRamification}).
		\item\label{item:valuation-of-discriminant:b} 
			If $L/K$ is Galois with group $G$, then
			\[v_L(\mathfrak{D}_{B/A})=\sum_{\sigma\in G\setminus \{1\}}i_G(\sigma)\]
			where $i_G(\sigma):=v_L(\alpha-\sigma(\alpha))$, see Section
			\ref{sec:lowerNumbering}.
	\end{enumerate}
\end{proposition}
\begin{proof}
After reducing to the totally ramified
			case, there exists a uniformizer $\alpha \in B$ such that
			$B=A[\alpha]$. Let $f$ be the minimal polynomial of
			$\alpha$. For both formulas we have to compute
			$v_L(f'(\alpha))$.
	\begin{enumerate}
		\item We only sketch the proof, as we will not use this result
			in the sequel.  We already saw that $L/K$
			is unramified if and only if $\mathfrak{D}_{B/A}=B$ if
			and only if $v_L(\mathfrak{D}_{B/A})=0$.

			The minimal polynomial of $\alpha$ can be seen to
			be an Eisenstein polynomial
			$f(X)=X^e+a_{e-1}X^{e-1}+\ldots + a_0$ with $a_i\in
			A$.
			We want to compute the valuation of
			\[f'(\alpha)=e\alpha^{e-1}+(e-1)a_{e-1}\alpha^{e-2}+\ldots+a_1.\]
			The valuation of each summand is
			\[v_i:=v_L(
				(e-i)a_{e-i}\alpha^{e-i-1})=ev_K(e-i)+ev_K(a_{e-i})+e-i-1
				\]
			for $i=0,\ldots, e-1$. Note that the $v_i$ are
			pairwise distinct, because they are modulo $e$. Hence
			\[v_L(f'(\alpha))=\min_{i=0,\ldots, e-1}
				\{v_i\}\]
			Note that for $i>0$,
			$v_i=ev_K(e-i)+ev_K(a_{e-i})+e-i-1\geq e$, as
			$f$ is an Eisenstein polynomial, and hence
			$v_K(a_{e-i})>0$.
			It follows that $v_L(f'(\alpha))=e-1$ if and
			only if $v_0=e-1$. 
			But $v_0=ev_K(e)+e-1$, so $v_0=e-1$ if and
			only if $v_K(e)=0$ if and only if $L/K$ is
			tamely ramified.

			If $L/K$ is not tamely ramified, then $v_K(e)>0$, so
			$v_0=v_L(e)+e-1\geq e$, so $e\leq v_L(f'(\alpha))\leq
			v_L(e)+e-1$. This completes the proof.
		\item This is much easier: $f'(\alpha)=\prod_{\sigma\in G\setminus
				\{\id\}}(\alpha-\sigma\alpha)$; the formula
				for
				$v_L(f'(\alpha))=v_L(\mathfrak{D}_{B/A})$
				follows.
	\end{enumerate}
\end{proof}
\begin{remark}\label{rem:differentAnnihilatesDifferentials}
	Theorem \ref{thm:ableitung} allows us to characterize the different
	$\mathfrak{D}_{B/A}$ geometrically as follows: Let $\mu:B\otimes_A
	B\rightarrow B$ be the multiplication map $\mu(b_1\otimes
	b_2)=b_1b_2$ and $I$ its kernel. Then $I/I^2$ is a  $B\otimes_A
	B$-module, so it also carries two $B$-module structures, obtained via
	the two maps $B\rightarrow B\otimes_A B$, $b\mapsto b\otimes 1$ and
	$b\mapsto 1\otimes b$. It is not difficult to check that these two
	$B$-module structures on $I/I^2$ coincide. The $B$-module $I/I^2$ is
	denoted by $\Omega_{B/A}^1$, and it is called the module of
	\emph{K\"ahler differentials}\index{K\"ahler differentials}. The map
	$d:B\rightarrow \Omega^1_{B/A}$, $d(b)=b\otimes 1-1\otimes b$ is an
	$A$-derivation. 

	If $B=A[X]/(f(X))$, write $x$ for the image of $X$ in $B$. Then
	$\Omega^1_{B/A}$ is generated by $dx$, and the annihilator of
	$\Omega^1_{B/A}$ is precisely the ideal $(f'(x))$. Indeed, there is an
	exact sequence of $B$-modules
	\[(f(X))/(f(X))^2\xrightarrow{d\otimes 1} \Omega^1_{A[X]/A}\otimes_A B\rightarrow
	\Omega^1_{B/A}\rightarrow 0,\]
	and $df(X)=f'(X)dX$.
\end{remark}

\subsection{Galois theory of discretely valued fields}
We are now going to study the Galois theory of finite extensions $L/K$, where
$K$ is equipped with a discrete valuation $v_K$. We start with the noncomplete
setting.

Let $(K,v_K)$ be a discretely valued field, $L/K$ a finite
separable extension. Let $A$ be the valuation ring of $K$ and $B$ its integral
closure in $L$. The separability of $L/K$ implies that $B$ is finitely generated over
$A$ (Lemma \ref{lemma:japanese}), so we can apply the results from Section
\ref{sec:extensions}.

Write $w_1,\ldots, w_r$ for the discrete valuations on $L$
extending $v_K$. We saw in Proposition \ref{prop:extensionAndCompletion} that
we have a canonical isomorphism of $K_{v_K}$-algebras
$L\otimes_K K_{v_K}\xrightarrow{\cong}\prod_{i=1}^rL_{w_i}$. From this it follows
that the extensions $L_{w_i}/K_{v_K}$ are separable as well. Indeed, the
extension $L/K$ is
defined by separable polynomials with coefficients in $K$, hence the finite
$K_{v_K}$-algebra
$L\otimes_K K_{v_K}$ is defined by separable polynomials with coefficients in
$K_{v_K}$.

If $L/K$ is in addition a Galois extension with group $G:=\Gal(L/K)$, then $G$
acts on the set of valuations $\{w_1,\ldots, w_r\}$ over $v_K$: If $\sigma \in
G$, then $w_i\circ \sigma$ is a discrete valuation lying over $v_K$. Moreover
this action is transitive (\cite[Ch.~I, Prop.~19]{Serre/LocalFields}, or
\cite[Ch.~II, Prop.~9.1]{Neukirch/1999}), so the
numbers $e_i$ and $f_i$ are independent of $i\in\{1,\ldots, r\}$; we just
write $e$ and $f$. The formula \eqref{eq:nef-formular} becomes $n=efr$, where
$n=[L:K]$.

Next, we define subgroups $D_{w_i}:=\{\sigma\in G| w_i\circ \sigma =
w_i\}$. By construction they are all conjugate subgroups of order $ef$ of $G$,

\begin{definition}
	The subgroup $D_{w_i}$ is called \emph{decomposition group of
	$w_i$}\index{decomposition group}.
\end{definition}

In ideal theoretic language, if $\mathfrak{P}_i$ is the prime ideal of $B$
corresponding to $w_i$, then $G$ acts transitively on
the set $\{\mathfrak{P}_1,\ldots, \mathfrak{P}_r\}$, and $D_{w_i}=\{\sigma\in G|
	\sigma(\mathfrak{P}_i)=\mathfrak{P}_i\}\subset G$.

\begin{proposition}
	If $L/K$ is Galois, then $L_{w_i}/K_{v_K}$ is Galois with group
	$D_{w_i}$.
\end{proposition}
\begin{proof}Indeed, we have seen that $[L_{w_i}:K_{v_K}]=ef=|D_{w_i}|$, and
	${D}_{w_i}\subset \Aut_{K_{v_K}}(L_{w_i})$.\end{proof}

\vspace{.3cm}

Since we are mainly interested in the decomposition group, we now put
ourselves in a complete local situation. Let $L/K$ be a finite Galois
extension of complete discretely valued fields. Write $v$ for the
valuation of $K$, $v_L$ for the valuation of $L$, $A$ for the
valuation ring of $K$, and $B$ for its integral closure in $L$. The Galois group
$G=\Gal(L/K)$ is equal to the decomposition group of $v_L$. 

Next, consider the residue extension $k(v_L)/k(v)$. Since $v_L$ is the unique
valuation of $L$ lying over $v$,  we get a homomorphism
of groups $G\rightarrow \Aut_{k(v)}(k(v_L))$.
\begin{proposition}[{\cite[Ch.~I, Prop.~20]{Serre/LocalFields}}]
	If the extension  $k(v_L)/k(v)$ is separable, then it is Galois, and 
	\[G\rightarrow\Gal(k(v_L)/k(v))\]
	is a surjective homomorphism of groups.
\end{proposition}
Its kernel is denoted by $I:=I_{L}$.

\begin{definition}
	The normal subgroup $I$ of $G$ is called \emph{inertia group of
	$v_L$}\index{inertia group}.
\end{definition}

The Galois correspondence shows that $I$ corresponds to a field
$L^I$ lying between $L$ and $K$; the fixed field of $I$. It is Galois over
$K$, and by construction this is the maximal subextension of $L/K$ which is
unramified. Indeed, $|G/I|=f=[L^I:K]$, so the ramification index of $L^I/K$ is
$1$. On the other hand, the residue extension of $L/L^I$ is trivial, so
$L/L^I$ is totally ramified.

The inclusion $I\subset G$ is the first step in
the so called \emph{ramification filtration}.

\subsection{The ramification filtration --- lower numbering}\label{sec:lowerNumbering}
We continue with the previous setup: $L/K$ is a finite Galois extension of
complete discretely valued fields with separable residue extension. Let
$G=\Gal(L/K)$.  We also write $p=\Char(k(v))\geq 0$,
$A, B$ for the valuation rings of $K$, $L$, respectively, and
$\mathfrak{m}_K,\mathfrak{m}_L$ for their maximal ideals.

\begin{definition}
	Let $i\geq -1$ be an integer. Note that $G$ acts on
	$B/\mathfrak{m}_L^{i+1}$. Define 
	\[G_i:=\{\sigma\in G| \sigma\text{ acts trivially on
	}B/\mathfrak{m}_L^{i+1}\}.\]
	The $G_i$ form a descending chain of subgroups of $G$, and $G_i$ is
	called the \emph{$i$-th ramification subgroup of
	$G$}\index{ramification subgroup}. This filtration is called the
	\emph{ramification filtration of $G$ in the lower
	numbering}\index{ramification filtration!lower
	numbering}.
\end{definition}

We make the definition of the $G_i$ a little bit more concrete:
\[G_i=\left\{\sigma\in G |\forall b\in B: v_L(\sigma(b)-b)\geq i+1\right\}.\]

By Theorem \ref{thm:generator}, there exists $x\in B$ such that $B=A[x]$. Then
we can also write
\begin{equation}\label{eq:GiViaGenerator}G_i=\{\sigma\in G |
	v_L(\sigma(x)-x)\geq i+1\}\end{equation}
Indeed, $x$ also
generates $B/\mathfrak{m}_L^{i+1}$ as an $A$-algebra, so $\sigma$ acting
trivially on $B/\mathfrak{m}_L^{i+1}$ is equivalent to $v_L(\sigma(x)-x)\geq i+1$.

\begin{proposition}[{\cite[Ch.~IV, Prop.~1]{Serre/LocalFields}}]
	We have $G_{-1}=G$ and $G_0=I$ the inertia group. The $G_i$ are normal subgroups of $G$, and
	$G_i=\{1\}$ for $i\gg 0$
\end{proposition}
\begin{proof}
	The identifications of $G_0$ and $G_{-1}$ follow directly from the
	definition, and so does the claim that $G_i$ are normal. 
	Finally, since $G$ is a
	finite group, the description \eqref{eq:GiViaGenerator} shows that $G_i=\{1\}$ for $i$ large enough.
\end{proof}

The description \eqref{eq:GiViaGenerator} shows that if $x$ is a generator of $B$ as an
$A$-algebra, then
for $\sigma\in G\setminus\{\id\}$, the integer
\eq{iG}{i_G(\sigma):=v_L(\sigma(x)-x)}
is independent of the choice of $x$. We obtain a map $i_G:G\rightarrow
\Z_{\geq 0}\cup \{\infty\}$, which has the property that
$G_i=i_G^{-1}([i+1,\infty])$. As the $G_i$ are normal subgroups, we also see
that $i_G$ is a class function, i.e.~that it is constant on conjugacy classes.

\begin{example}\label{ex:ASlowerNumbering}
	We begin to study an important example: Let $k$ be an algebraically
	closed field of characteristic $p>0$, and $k\llbracket x\rrbracket$
	the ring of formal power series in one indeterminate. Write
	$K:=\Frac(k\llbracket x\rrbracket)=k\llparen x\rrparen$, and consider
	the polynomial $t^{p^n}-t-\frac{1}{x}\in K[t]$.  It is not difficult to see
	that it is irreducible. Indeed, if we make the change of coordinates
	$u=\frac{1}{t}$, we obtain an isomorphism of fields	
	\[L:=K[u]/(u^{p^n}+xu^{p^n-1}-x)\cong K[t]/(t^{p^n}-t-\frac{1}{x}).\]
	Note that $u^{p^n}+xu^{p^n-1}-x$ is a separable \emph{Eisenstein polynomial} (\cite[Ch.~I,
	Prop.~17]{Serre/LocalFields}).
	The extension $L/K$ is a Galois extension with group $\F_{p^n}$, where
	$a\in \F_{p^n}$ acts via $t\mapsto t+a$, i.e.~$u\mapsto
	\frac{u}{1+au}$.	This extension is called
	\emph{Artin-Schreier extension}\index{Artin-Schreier extension}.

	Moreover, according to \cite[Ch.~I, Prop.~17]{Serre/LocalFields}, the
	integral closure of $k\llbracket x\rrbracket$ in $L$ is
	$k\llbracket x\rrbracket [u]/(u^{p^n}+xu^{p^{n}-1}-x)$, which is a discrete valuation ring with
	uniformizer $u=\frac{1}{t}$. The ramification index of this extension
	is $p^n$: $u^{p^n}=x(1-u^{p^n-1})$, and $1-u^{p^n-1}$ is a unit.

	To determine the ramification groups of $\Gal(L/K)=\F_{p^n}$, we
	compute  for every $a\in \F_{p^n}\setminus\{0\}$:
	\[v_L\left(\frac{u}{1+au}-u\right)=v_L\left(u\left(\frac{1-(1+au)}{1+au}\right)\right)=2.\]
	It follows that $\F_{p^n}=G=G_0=G_1\supsetneq G_2=0$.

	Following \cite{Laumon/Swan}, one similarly proves that for the
	$\F_{p}$-Galois extension
	$L=K[t]/(t^{p}-t-x^{-m})$, one has $G_0=G_1=\ldots=G_m\supsetneq
	G_{m+1}=0$, if $(m,p)=1$.
\end{example}
\begin{proposition}[Compatibility with taking
	subgroups]\label{prop:compatibilityWithSubgroups}
	Let $H$ be a subgroup of $G$ and $L^H$ its fixed field. Then
	$L/L^H$ is Galois with group $H$, and $H_i=G_i\cap H$ for all $i$.
\end{proposition}
\begin{proof}
This is clear, since if $x$ generates $B$ over $A$, then $x$ also generates
$B$ over the valuation ring of $L^H$. 
\end{proof}

\begin{remark}\label{rem:quotientsInLowerNumbering}
	\begin{enumerate}[label=(\alph*)]
		\item If we apply the proposition to the group $G_0$, then $L^{G_0}$ is the
			maximal unramified subextension of $L/K$, and $L/L^{G_0}$ is totally
			ramified. To study the higher ramification groups
			$G_i$, $i>0$, we may 	assume that $L/K$ is totally ramified. In this case the
			degree of $L/K$ is equal to the ramification index of $L/K$.
		\item If $H$ is a normal subgroup of $G$, then $G/H$ is the
			Galois group of the extension $L^H/K$. One might
			hope that the
			ramification filtration is compatible with quotients,
			i.e.~that for $i\geq 0$, the image of $G_i$ in $G/H$
			is the $i$-th ramification subgroup of $G/H$.
			This is not the case. The image of $G_i$ in
			$G/H$ is one of the ramification subgroups of $G/H$,
			but in general not $(G/H)_i$.  The indexing has to be
			adjusted. To make this
			precise one introduces the \emph{upper numbering
			filtration}\index{ramification filtration!upper
			numbering}, see Section \ref{sec:upperNumbering}.
	\end{enumerate}
\end{remark}

Now let $\pi$ be a uniformizer
for $L$ and assume that $L/K$ is totally ramified. In this case $G=G_0$, and $\pi$
generates $B$ as  $A$-algebra, see Theorem \ref{thm:generator}. 
For $\sigma\in G_0$, we see that $v_L(\sigma(\pi))=1$, so $\sigma(\pi)/\pi\in
B^\times$. Moreover, by \eqref{eq:GiViaGenerator}, $G_i$ is characterized as
the set of $\sigma\in G_0$, such that $\sigma(\pi)-\pi\in (\pi)^{i+1}$. Consequently, 
\[G_i=\left\{\sigma\in G \left| \frac{\sigma(\pi)}{\pi}\equiv 1 \mod
	\mathfrak{m}_L^{i}\right.\right\}.\]
\begin{definition}
For $i>0$, the set $U_L^{i}:=1+\mathfrak{m}_L^i\subset B^\times$ is easily seen to be a
subgroup, the group of \emph{$i$-th units}\index{higher units}. We define
$U_L:=U_L^0:=B^\times$.
\end{definition}
This defines a descending filtration of the group of units $U_L$, and
the observation that $\sigma(\pi)/\pi\in B^\times$ can be made more precise:
\begin{proposition}\label{prop:map-lowernumbering-to-units}
	For $i\geq 0$, the assignment $\sigma\mapsto \sigma(\pi)/\pi$ induces an injective homomorphism
	of groups
	\[G_i/G_{i+1}\hookrightarrow U^{i}_L/U_{L}^{i+1},\]
	which is independent of the choice of the uniformizer $\pi$.
\end{proposition}
\begin{proof}
	We have seen that $\sigma(\pi)/\pi\in U^i_L$ if $\sigma \in G_i$. If
	$\tau$ is a second uniformizer, then $\tau=u\pi$ with $u\in
	B^\times$, and 
	\[\frac{\sigma(\tau)}{\tau}=\frac{\sigma(u)}{u}\cdot\frac{\sigma(\pi)}{\pi}.\]
	If $\sigma \in G_i$, then by definition $\sigma(u)\equiv u\mod\mathfrak{m}_L^{i+1}$, and hence
	$\sigma(u)/u\in U^{i+1}_L$, as $u\in B^\times$. We obtain a map 
	$G_i\rightarrow U^i_L/U^{i+1}_L$, independent of the choice of $\pi$. This is a group homomorphism: If
	$\sigma, \sigma'\in G_i$, then
	\[\frac{\sigma\sigma'(\pi)}{\pi}=\frac{\sigma'(\pi)}{\sigma'(\pi)}\cdot\frac{\sigma(\sigma'(\pi))}{\pi}=\frac{\sigma'(\pi)}{\pi}\cdot\frac{\sigma(\sigma'(\pi))}{\sigma'(\pi)},\]
	and $\sigma'(\pi)$ also is a uniformizer of $B$.

	Finally, $\sigma(\pi)/\pi \in U^{i+1}_L$ if and only if $\sigma\in
	G_{i+1}$, so the kernel of the map $G_i\rightarrow
	U_L^i/U_L^{i+1}$ is $G_{i+1}$.  This completes the proof.
\end{proof}

We can now use the fairly concrete structure of the groups $U_L^i$ to get some
information about the groups $G_i$:
\begin{corollary}\label{cor:structure-of-ramification-groups}
	\begin{enumerate}[label=\emph{(\alph*)}]
		\item $G_0/G_1$ is cyclic, and if $p=\Char(k(v_K))>0$, then the order of $G_0/G_1$ is prime to $p$.
		\item If $p=0$, then $G_i=0$ for $i>0$.
		\item If $p>0$, then for $i>0$, the groups $G_i$ are
			$p$-groups, and the quotients $G_i/G_{i+1}$ are
			abelian $p$-groups.
		\item $G_0$ is a semi-direct product of a cyclic group of
			order prime to $p$ with a $p$-group. In particular,
			$G_0$ is solvable and $G_1$ is its unique $p$-Sylow
			subgroup.
	\end{enumerate}
\end{corollary}
\begin{proof}
	\begin{enumerate}[label={(\alph*)}]
		\item We know that the quotients $G_i/G_{i+1}$ are isomorphic to subgroups
	of $U^i_L/U^{i+1}_L$. If $i=0$, then $U^0_L/U^1_L=k(v_L)^\times$. As
	every finite subgroup of the multiplicative group of a field is cyclic, we see
	that $G_0/G_1$ is cyclic. If $p>0$, there are no elements of order $p$
	in $k(v_L)^\times$, so $G_0/G_1$ is cyclic of order prime to $p$.
\item Note that for $i>0$, $U_L^i/U_L^{i+1}$ is isomorphic to
	$\mathfrak{m}^{i}_L/\mathfrak{m}_L^{i+1}$, which is isomorphic to the
	additive group underlying $k(v_L)$. If $p=0$, then the additive group
	$k(v_L)$ has no nontrivial elements of finite order, so
	$G_i/G_{i+1}=0$ for all $i>0$. As $G_i=0$ for $i\gg0$, it follows that
	$G_1=0$.
\item If $p>0$ and $i>0$, then again $G_i/G_{i+1}$ is isomorphic to
	a finite subgroup of the additive group of the residue field $k(v_L)$,
	hence an abelian $p$-group. This implies that the order of $G_1$ is a
	$p$-power, so $G_1$ is a $p$-group.
\item We just have to show that the sequence
	\begin{equation*}
		\begin{tikzcd}
			0\rar&G_1\rar&G_0\rar&G_0/G_1\rar&0
		\end{tikzcd}
	\end{equation*}
	admits a splitting. But this follows from the fact that the orders of
	$G_0/G_1$ and $G_1$ are coprime.
	
	\end{enumerate}
\end{proof}

 The
ramification filtration of the Galois group $G=\Gal(L/K)$ allows us to
``measure'' the quality of ramification. If $L/K$ is totally ramified, then
its ramification index equals the degree $[L:K]$, the more $G_i$ are nonzero,
the ``worse'' is the ramification. We will make this more precise later on.
For now, we content ourselves with the following definition:

\begin{definition}\label{defn:tameRamification}
	A finite separable extension $L/K$ of complete discretely valued fields is called \emph{tamely ramified}\index{tame
	ramification} if its ramification index is prime to the characteristic
	$p$
	of the residue field $k(v)$ of $K$, and if the extension of the
	residue fields is separable. Otherwise $L/K$ is called \emph{wildly
	ramified}\index{wild ramification}.

	These definitions are easily extended to the case of general separable extensions
	of discretely valued fields which are not necessarily
	complete: If $(K,v)$ is a discretely valued field,
	$L/K$ a finite separable extension and $w$ a discrete valuation on $L$
	extending $v$ (Definition \ref{defn:localRamification}), then
	$(L,w)/(K,v)$ is \emph{tamely ramified} if the extension $L_w/K_v$ of the
	completions is tamely ramified, i.e.~if the ramification index of $w$
	over $v$ is prime to $p$ and if the residue extension is separable. Otherwise, $(L,w)/(K,v)$ is
	\emph{wildly
		ramified}.
\end{definition}

\begin{exercise}
Let $L/K$ be a Galois extension of discretely valued fields, with separable
residue extension and Galois group $G$. Show that $L/K$ is tamely ramified if and
only if $G_1=0$ if and only if $(|G_0|,p)=1$.
\end{exercise}

\subsection{The ramification filtration --- upper numbering}\label{sec:upperNumbering}
We continue with the previous setup: Let $L/K$ be a finite Galois extension of
complete discretely valued fields with Galois group $G$ and separable residue
extension. Let $H\subset G$ be a
normal subgroup, such that $G/H=\Gal(L^H/K)$. As indicated in Remark
\ref{rem:quotientsInLowerNumbering}, the image
of the ramification filtration of $G$ in $G/H$ is the ramification filtration
of $\Gal(L^H/K)$, up to changing the numbering. We start this section by
making this remark precise, before we then fix this numbering defect by
introducing the \emph{upper numbering for the ramification filtration.} Among
other things, this will have the advantage that we can pass to the inverse
limit to obtain a ramification filtration by closed subgroups of the absolute
Galois group.

Recall that we defined a class function $i_G$ on $G$, by setting
$i_G(\sigma)=v_L(\sigma(x)-x)$ for $\sigma\neq \id$, $x$ a generator of
$B$ as $A$-algebra, and $i_G(\id)=\infty$.

\begin{proposition}
	Let $H$ be a normal subgroup of $G$, and $e'$ the ramification index
	of $L/L^H$. We can then compute $i_{G/H}$ in terms of $i_G$:
	\[i_{G/H}(s)=\frac{1}{e'}\sum_{\sigma\mapsto s}i_G(\sigma).\]
\end{proposition}
\begin{proof} Let $A$ and $B$ be the valuation rings of $K,L$ and $B'$ the
	valuation ring of $L^H$.  By Theorem \ref{thm:generator} we find $x\in
	B$, $y\in B'$ such that $B=A[x]$, $B'=A[y]$.  By definition we have
	for $s\in G/H$:
	\[e'i_{G/H}(s)=v_L(s(y)-y).\]
	Let $\sigma\in G$ be a fixed element of the preimage of $s$. Then
	$\sigma|_{L^H}=s$ and 
	every element of the preimage of
	$s$ is of the form $\sigma\tau$ for $\tau\in H$. We show that the
	elements
	$\sigma y-y$ and $\prod_{\tau\in H}(\sigma\tau x -x)$ of $B$ are associated,
	i.e.~generate the same ideals. 

	Let $f(X)\in B'[X]$ be the minimal polynomial of $x$ over $L^H$. Write
	$f(X)=X^n+b_{n-1}X^{n-1}+\ldots+b_0$ with $b_i\in B'$. The
	coefficients of the polynomial $(\sigma f)(X)-f(X)\in B'[X]$ are 
	$\sigma(b_i)-b_i$, $i=0,\ldots, n-1$. Let $B_i(Y)\in
	A[Y]$ be polynomials such that $B_i(y)=b_i$. Then
	$\sigma(b_i)-B_i(\sigma(y))=0$, so $Y-\sigma(y)$ divides the polynomial
	$\sigma(b_i)-B_i(Y)\in B'[Y]$, whence $y-\sigma(y)$ divides $\sigma(b_i)-b_i$. 
	This shows that $\sigma y-y$ divides $(\sigma
	f)(x)-f(x)=\prod_{\tau\in H}(\sigma\tau x -x)$.

	Conversely, let $G\in A[X]$ be a polynomial such that $G(x)=y$. Then
	$x$ is a zero of $G(X)-y\in B'[X]$, so $G(X)-y=f(X)h(X)$
	for some polynomial $h(X)$. Applying $\sigma$ gives
	\[G(X)-\sigma y = (\sigma f)(X)(\sigma h)(X)\]
	and finally 
	\[y-\sigma y= G(x)-\sigma y=(\sigma h)(x)
		\prod_{\tau\in H}(\sigma\tau x - x).\]
\end{proof}

For notational reasons, the following definition is convenient:
\begin{definition}\label{defn:reals}
	If $u\in [-1,\infty)\subset \R$, then $G_u:=G_{\lceil
	u\rceil}$, where $\lceil u\rceil$ denotes the smallest
	integer $\geq u$.
\end{definition}
We will use this notation in the proof of the following corollary.

\begin{corollary}[Herbrand\index{Herbrand's theorem}]\label{cor:Herbrand}
	In the situation of the proposition, the ramification groups of $G/H$
	are the images of the ramification groups of $G$.
\end{corollary}
\begin{proof}
	Let $i\in \Z_{\geq -1}$, $\sigma_0\in G_i$ and let $s_0$ be its image in $G/H$. The
proposition allows us to compute $i_{G/H}(s_0)$:
\[i_{G/H}(s_0)=\frac{1}{e'}\sum_{\sigma\mapsto s_0}i_{G}(\sigma).\]
Write $H_j:=G_j\cap H$ and $h_j:=|H_j|$. We can rewrite the above formula:
\[
	i_{G/H}(s_0)=\frac{1}{e'}\left(\sum_{j=-1}^{\infty}\sum_{\sigma\in H_j\setminus H_{j+1}} i_G(\sigma_0\sigma) + i_G(\sigma_0)\right).
\]
Note that $i_G(\sigma_0\sigma)\geq \min\{i_G(\sigma), i_G(\sigma_0)\}$ with
equality if $i_G(\sigma)\neq i_{G}(\sigma_0)$ (Exercise \ref{ex:nonArchimedianTriangleInequality}), and that $i_G(\sigma)=j+1$ if
$\sigma\in H_j\setminus H_{j+1}$. Using that $\sigma_0\in G_i$, we compute
\begin{align*}
	i_{G/H}(s_0)&=\frac{1}{e'}\left(\sum_{j=-1}^{i-1}(h_j-h_{j+1})(j+1)+i_G(\sigma_0)+\sum_{j\geq i} \sum_{\sigma\in H_j\setminus H_{j+1}}i_{G}(\sigma\sigma_0)\right)\\
&\geq\frac{1}{e'}\left(\sum_{j=-1}^{i-1}(h_j-h_{j+1})(j+1)+(i+1)\cdot\left(1+\sum_{j\geq i, h_j>1}(h_j-h_{j+1})\right)\right)\\
&=\frac{1}{e'}\sum_{j=0}^{i}h_j\end{align*}
This last number is important and we give it a name:
\[\varphi_H(i)+1:= \frac{1}{|G_0\cap H|}\sum_{j=0}^{i}|G_j\cap H|\in \Q.\]
Our computation shows that the image $G_iH/H\subset G/H$ of $G_i$ is a
subgroup of $(G/H)_{\phi_H(i)}$ (see Definition \ref{defn:reals}). Let us show that this is in fact an equality.
Let $s\in (G/H)_{\phi_{H}(i)}\setminus \{1\}$, and let $i'$ be maximal with the property that
there exists a preimage $\sigma_0$ of $s$ in $G_{i'}$. Let $\sigma\in
H_j\setminus H_{j+1}$. If $j<i'$, then $i_G(\sigma_0\sigma)=i_G(\sigma)=j+1$.
If $j\geq i'$, then $i_G(\sigma_0\sigma)\geq i'+1$, and hence by the maximality
of $i'$, $i_G(\sigma_0\sigma)=i'+1$. Our computation from above then shows
\begin{align*}
	i_{G/H}(s)&=\frac{1}{e'}\left(\sum_{j=-1}^{i'-1}(h_j-h_{j+1})(j+1)+(i'+1)\cdot\left(1+
	\sum_{j\geq i', h_j>1}(h_j-h_{j+1})\right)\right)\\
	&=\frac{1}{e'}\sum_{j=0}^{i'}h_j\\
	&=\phi_{H}(i')+1.\end{align*}
	We immediately see that $i_{G/H}(s)\geq \phi_{H}(i)+1$ implies that
	$i'\geq i$, so $s\in G_iH/H$. This completes the proof that
	$G_iH/H=(G/H)_{\phi_H(i)}$.
%
\end{proof}

To define the upper numbering filtration of $G$, we extend the definition of
the function $\phi$ that appeared in the proof of Corollary
\ref{cor:Herbrand}.

\begin{definition}\label{defn:HerbrandFunction}
	As before, let $L/K$ be a finite Galois extension of complete
	discretely valued fields with group $G$, and separable residue
	extension of degree $f$.
	\begin{enumerate}[label={(\alph*)}]
		\item For $u\in [-1,0)$, define $(G_0:G_u):=(G_{u}:G_0)^{-1}$,
				i.e.~$(G_0:G_{-1})=(G_{-1}:G_0)^{-1}=\frac{1}{f}$ and
			$(G_0:G_u)=1$ if $u\in (-1,0)$.
		\item Define the function $\phi_{L/K}:[-1,\infty)\rightarrow [-1,\infty)$ by
				\[\phi_{L/K}(u)=\int_{0}^u\frac{1}{(G_0:G_t)}\mathrm{d}t\]
	\end{enumerate}
\end{definition}

 If $u\geq 0$, then
	 \begin{equation}\label{eq:formulaForPhi}\phi_{L/K}(u)+1=\frac{1}{|G_0|}\left(|G_0|+|G_1|+\ldots+|G_{\lfloor
		 u\rfloor}| + (u-\lfloor u\rfloor)|G_{\lfloor u \rfloor+1}|\right),\end{equation}
		 where $\lfloor u\rfloor$ is the largest integer $\leq u$. In particular, if $u=m\in \Z_{\geq 0}$, then
\[\phi_{L/K}(m)+1=\frac{1}{|G_0|}\sum_{i=0}^m|G_i|.\]
This formula shows how the function $\phi_{L/K}$ is related to the function
$\phi_H$ from the proof of Corollary \ref{cor:Herbrand}:
$\phi_{L/L^{H}}=\phi_H$.

Using this notation, we restate the result of our calculation from the proof of Corollary \ref{cor:Herbrand} for
future reference.
\begin{corollary}\label{cor:quotientCalculation}
If $H\subset G$ is a normal subgroup, then for $u\in [-1,\infty)\subset  \R$
	we have
	\[G_{u}H/H=(G/H)_{\phi_{L/L^{H}}(u)}.\]
\end{corollary}

The basic properties of $\phi_{L/K}$ are easy to verify:
\begin{proposition}
	\begin{enumerate}[label=\emph{(\alph*)}]
		\item $\phi_{L/K}$ is continuous, piecewise linear,
			increasing, and concave.
		\item $\phi_{L/K}(0)=0$.
		\item $\phi_{L/K}$ is a homeomorphism of $[-1,\infty)$ onto itself.
	\end{enumerate}
\end{proposition}
\begin{definition}
	The inverse of $\phi_{L/K}$ is denoted $\psi_{L/K}$. 
\end{definition}
\begin{proposition}
	\begin{enumerate}[label=\emph{(\alph*)}]
		\item $\psi_{L/K}$ is continuous, piecewise linear,
			increasing, and convex.
		\item $\psi_{L/K}(0)=0$.
		\item $\psi_{L/K}([-1,\infty)\cap \Z)\subset \Z$.
	\end{enumerate}
\end{proposition}

\begin{exercise}\label{ex:psiOnIntegers}
	Prove that $\psi_{L/K}(\Z\cap [-1,\infty))\subset \Z$ (Use formula
	\eqref{eq:formulaForPhi}).
\end{exercise}

\begin{example}\label{ex:phiandpsiPrimeCase}
	Lets compute $\phi_{L/K}$ and $\psi_{L/K}$ in two simple special
	cases.
	\begin{enumerate}
		\item If $L/K$ is unramified, then $\phi_{L/K}=\psi_{L/K}=\id$.
			Indeed, $(\phi_{L/K})|_{[-1,0)}=\id_{[-1,0)}$ is
				always true. If $u\geq 0$, then 
				\[\phi_{L/K}(u)=\frac{1}{|G_0|}\left(|G_1|+\ldots+|G_{\lfloor
				u\rfloor}|+(u-\lfloor u \rfloor)|G_{\lfloor
				u\rfloor+1 }|\right)=u\]
			as $|G_i|=1$ for $i\geq 0$.
		\item Now assume that $L/K$ is Galois with Galois group
			$\Z/\ell\Z$, where $\ell$ is a prime number. Moreover,
			assume that $G=G_0=G_1=\ldots = G_t$, and that
			$G_{i}=\{1\}$ for $i>t$.
			For $0\leq u\leq t$, we compute
			\[\phi_{L/K}(u)=\frac{u\cdot\ell}{\ell}=u.\]
			For $u>t$, we get
			\[\phi_{L/K}(u)=t+\frac{u-t}{\ell}\]
			so \[\psi_{L/K}(u)=\ell(u-t)+t,\]
			for $u>t$.

			This computation applies to Example
			\ref{ex:ASlowerNumbering}.
	\end{enumerate}
\end{example}
The next proposition shows that the functions $\phi$ and $\psi$ are transitive with respect to Galois subextensions.

\begin{proposition}\label{prop:transitivityOfPhi}
	With the notations from above, let $H\subset G$ be a normal subgroup.
	Then
	\[\phi_{L/K}=\phi_{L^H/K}\circ\phi_{L/L^H}\text{ and
	}\psi_{L/K}=\psi_{L/L^H}\circ\psi_{L^H/K}.\]
\end{proposition}
\begin{proof}
	Clearly it is enough to prove the claim about the function
	$\phi_{L/K}$. Note that $\phi_{L/L^H}$ is piecewise linear and hence
	almost everywhere continuously differentiable. More precisely, it is
differentiable on $[-1,\infty)\setminus \Z$, because if $u\not\in\Z$, then the
	left and right derivatives of $\phi_{L/L^H}$ at $u$ are equal to
	$(H_0:H_u)^{-1}$. Similarly, $\psi_{L/L^H}$ is continuously
	differentiable at all $u$, for which  $\psi_{L/L^{H}}(u)$ is not an
	integer, and its derivative at such a $u$ is
	\begin{equation}\label{eq:derivativeOfPsi}\psi'_{L/L^H}(u)=(\phi'_{L/L^H}(\psi_{L/L^{H}}(u)))^{-1}=\left(H_0:H_{\psi_{L/L^H}(u)}\right).\end{equation}

	Next, recall Proposition \ref{prop:compatibilityWithSubgroups}, which
	tells us that $H_u=H\cap G_u$, for all $u\in [-1,\infty)$.

	Using Corollary \ref{cor:quotientCalculation} we compute
	\begin{align*}\phi_{L^{H}/K}\circ\phi_{L/L^{H}}(u)&=\int_0^{\phi_{L/L^{H}}(u)}\frac{\mathrm{d}t}{\left((G/H)_0:(G/H)_t\right)}\\
		&=\int_0^{\phi_{L/L^H}(u)}\frac{\left(H_0:H_{\psi_{L/L^H}(t)}\right)}{\left(G_0:G_{\psi_{L/L^{H}}(t)}\right)}\mathrm{d}t.
		\end{align*}
	Now, with a certain degree of sloppiness, we use
	\eqref{eq:derivativeOfPsi} to write
	\[\int_0^{\phi_{L/L^H}(u)}\frac{\left(H_0:H_{\psi_{L/L^H}(t)}\right)}{\left(G_0:G_{\psi_{L/L^{H}}(t)}\right)}\mathrm{d}t=
		\int_0^{\phi_{L/L^H}(u)}\frac{\psi'_{L/L^H}(t)}{\left(G_0:G_{\psi_{L/L^{H}}(t)}\right)}\mathrm{d}t.\]
		Of course $\psi_{L/L^{H}}$ is not everywhere differentiable,
		so to be precise, one should do this calculation on every
		piece where it is differentiable and then sum up.
	
	In the end, we obtain
	\begin{align*}
		\phi_{L^H/K}\circ\phi_{L/L^H}&=
		\int_0^{\phi_{L/L^H}(u)}\frac{\psi'_{L/L^H}(t)}{\left(G_0:G_{\psi_{L/L^{H}}(t)}\right)}\mathrm{d}t\\
		&=\int_0^u\frac{\mathrm{d}t}{\left(G_0:G_t\right)}\\&=\phi_{L/K}(u)
	\end{align*}
	as claimed.
\end{proof}

Now we have gathered all the facts to define the upper numbering for the
ramification filtration.
\begin{definition}\label{defn:upperNumbering}
	We continue to use the notations from above. In particular
	$G=\Gal(L/K)$ is the Galois group of a finite Galois extension of
	complete discretely valued fields with separable residue extension.
	For a real number $v\in [-1,\infty)$, define $G^v:=G_{\psi_{L/K}(v)}$ or
	equivalently $G^{\phi_{L/K}(u)}=G_u$\index{ramification
	filtration!upper numbering}.

	Note that $G=G^{-1}$, $G^0=G_0$.
\end{definition}

As promised, the upper numbering filtration is compatible with passing to
quotients:
\begin{proposition}[Upper numbering filtration and
	quotients, Herbrand's
Theorem]\label{prop:quotientsUpperNumbering}\label{prop:Herbrand}
	Let $H\subset G$ be a normal subgroup. For every real number $v\in
	[-1,\infty)$, we have 
		\[G^{v}H/H = (G/H)^v.\]
\end{proposition}
\begin{proof}
	Most of the work has already been done in the proof of Corollary
	\ref{cor:Herbrand} and Proposition \ref{prop:transitivityOfPhi}:
	\begin{equation*}
		\begin{tikzcd}[column sep=1.2cm]
			G^vH/H
			\dar[transform canvas={xshift=0.2ex},-]{\text{Def.}}
			\dar[transform canvas={xshift=-0.2ex},-]
			&(G/H)^{\phi_{L^H/K}(\phi_{L/L^H}(\psi_{L/K}(v)))}
			\rar[transform canvas={yshift=0.2ex},-]
			\rar[transform canvas={yshift=-0.2ex},-]{\text{Prop.~\ref{prop:transitivityOfPhi}}}&(G/H)^v\\
			G_{\psi_{L/K}(v)}H/H
			\rar[transform canvas={yshift=0.2ex},-]
			\rar[transform canvas={yshift=-0.2ex},-]{\text{Prop.~\ref{cor:quotientCalculation}}}&(G/H)_{\phi_{L/L^H}(\psi_{L/K}(v))}
			\uar[transform canvas={xshift=0.2ex},-,swap]{\text{Def.}}
			\uar[transform canvas={xshift=-0.2ex},-]
		\end{tikzcd}
	\end{equation*}
\end{proof}

Proposition \ref{prop:quotientsUpperNumbering} allows us to define a
ramification filtration of the absolute Galois group of $K$.

\begin{definition}\label{defn:upperNumberingInfiniteExtension}
	Let $L/K$ be a (possibly infinite) Galois extension, such that the
	extension of the residue fields is separable. For a real number $u\in
	\left[-1,\infty\right)$, we define $\Gal(L/K)^u:=\varprojlim_{L'} \Gal(L'/K)^u$,
	where $L'$ runs through the set of finite Galois subextensions of $L/K$.
\end{definition}

\begin{remark}
If the residue field of $K$ is perfect this applies to a separable closure
$K^{\sep}$ of $K$. We
obtain a descending filtration of $\Gal(K^{\sep}/K)$.
\end{remark}

\subsection{The norm}
Let $L/K$ be an extension of discretely valued fields, and recall the
definition of the norm\index{norm} map $N_{L/K}$ (Definition \ref{defn:norm}). 
Note that if $L/K$ is Galois with Galois group $G$, then \[N_{L/K}:L^\times\rightarrow K^\times,
	\;x\mapsto\prod_{\sigma\in
G}\sigma(x).\]
\begin{proposition}
	The norm $N_{L/K}$ is a homomorphism of abelian groups. If $L/K$
	is a Galois extension of complete discretely valued fields with
	separable residue extension of degree $f$, then 
	\begin{enumerate}
		\item $v_K(N_{L/K}(x))=fv_L(x)$ 
		\item $N_{L/K}(B^\times)\subset A^\times$, where $A, B$ are the
	valuation rings of $K,L$.
	\end{enumerate}
\end{proposition}
\begin{proof}
	Exercise.
\end{proof}

\begin{proposition}[{\cite[Ch.~V]{Serre/LocalFields}}]\label{prop:norm}
	Let $L/K$ be a totally ramified Galois extension of complete
	discretely valued fields and write
	$N:=N_{L/K}$.
	\begin{enumerate}
		\item\label{item:normA}  For $m\geq 0$,  $N$ induces a homomorphism of groups
			\[N:U_L^{\psi_{L/K}(m)}\rightarrow U_K^{m}\]
			which is surjective if $G_{\psi_{L/K}(m)}=0$ or if the residue
		field of $K$ is algebraically closed.
	\item\label{item:normB}  For $m\geq 0$, $N$ induces a homomorphism of groups
		\[N:U_L^{\psi_{L/K}(m)+1}\rightarrow U_K^{m+1},\]
		which is surjective if $G_{\psi_{L/K}(m+1)}=0$ or if the residue
		field of $K$ is algebraically closed.
	\item\label{item:normC} From the above we get homomorphisms
		\[N_m: U_L^{\psi_{L/K}(m)}/U_L^{\psi_{L/K}(m)+1}\rightarrow
			U_K^{m}/U_{K}^{m+1},\]
		such that the sequence
	\[0\rightarrow
		G_{\psi_{L/K}(m)}/G_{\psi_{L/K}(m)+1}\xrightarrow{\theta_m}
		U^{\psi_{L/K}(m)}_L/U^{\psi_{L/K}(m)+1}_L\xrightarrow{N_m}
		U_K^{m}/U_{K}^{m+1}\]
		is exact, 
		where the  map $\theta_m$ is the map defined in Proposition
		\ref{prop:map-lowernumbering-to-units}.
	\item\label{item:normD} $N_m$ is surjective if one of the following holds:
		\begin{itemize}
			\item The residue field of $K$ is algebraically
				closed.
			\item The residue field of $K$ is perfect, and
				$G_{\psi_{L/K}(m)}=G_{\psi_{L/K}(m)+1}$.
			\item $G_{\psi_{L/K}(m)}=0$.
		\end{itemize}
\end{enumerate}
\end{proposition}
\begin{proof}
	This is rather lengthy. 	
	We will prove a more
	explicit version of \ref{item:normD} by induction:
	\begin{claim}Let $k$ be the residue field of $L$ and $K$. Recall from
		the proof of Corollary
		\ref{cor:structure-of-ramification-groups} that we have
		group isomorphisms $U^0_L/U_L^1\cong U^0_K/U^1_K\cong k^\times$, and
		that after choosing uniformizers $\tau\in \mathcal{O}_K$,
		$\pi\in \mathcal{O}_L$, we have group isomorphisms
		$U^i_L/U^{i+1}_L\cong U^i_{K}/U^{i+1}_K\cong (k,+)$ for $i>0$. We claim:
		\begin{itemize}
	\item Identifying $U_L/U_L^1$ and $U_K/U_K^1$ with
	$k^\times$,  $N_0:k^\times\rightarrow k^\times$ is given by the nonconstant monomial $X^{|G_0|}$.
	\item  For $m\geq 1$, if we identify $N_m$ with a map
	$k\rightarrow k$, then  $N_m$ is given by a
	nonconstant polynomial of the form
	$P_m(X)$ which is a linear combination of monomials of the form
	$X^{p^h}$, where $\Char(k)=p\geq 0$. Moreover $\deg
	P_m=|G_{\psi_{L/K}(m)}|$.
	\item Keeping the notations, if $k$ is perfect, then there exists a polynomial $Q_m$ with 
		$P_m=(Q_m)^{p^h}$, where $p^h=|G_{\psi_{L/K}(m)+1}|$.
\end{itemize}
\end{claim}
Part \ref{item:normD} follows from the claim. Indeed, if the residue field $k$
is algebraically closed, this is obvious. If $k$ is perfect and $m\geq 0$ such
that $G_{\psi_{L/K}(m)}=G_{\psi_{L/K}(m)+1}$, we distinguish the cases $m=0$
and $m>0$. If $m=0$, then $G_0=G_1$ is a nontrivial $p$-group and the claim
states that 
$N_0:k^\times \rightarrow k^\times$ can be identified with the polynomial map
$X^{|G_0|}$, which is surjective, as $k$ is assumed to be perfect. If $m>0$,
then the claim shows that $N_m$ can be identified with
$P_m=(Q_m)^{|G_{\psi_{L/K}(m)+1}|}$ and that $\deg P_m=|G_{\psi_{L/K}(m)}|$ so
$\deg Q_m=1$. As $k$ is assumed to be perfect, it follows that $N_m$ is
surjective. Finally, if $G_{\psi_{L/K}(m)}=0$ then $m>0$, as otherwise $L/K$
would be trivial, and the claim implies that $N_m$ can be identified with a
nonconstant polynomial of degree $1$, so $N_m$ is surjective. This shows that
the claim implies \ref{item:normD}.

The surjectivity statements from \ref{item:normA} and
\ref{item:normB} are implied by \ref{item:normD} using a completeness
argument
(\cite[Ch.~V, Lemma 2]{Serre/LocalFields}).

It remains to prove \ref{item:normC} and the claim. Since the Galois group of $L/K$
is solvable, it has a quotient which is cyclic of prime order. Using
induction and the transitivity of the norm and $\psi$ (Proposition
\ref{prop:transitivityOfPhi}), one reduces the
proof of \ref{item:normC} and the claim to
the case where the Galois group of $L/K$ is $G=G_0=\Z/\ell\Z$ with $\ell$ a prime (see
\cite[Ch.~V, \S6]{Serre/LocalFields}). We
assume that $G_0=G_1=\ldots=G_t$, and $G_i=\{1\}$ for $i>t$. 

In this case we already computed the
functions $\phi_{L/K}$ and $\psi_{L/K}$ in Example
\ref{ex:phiandpsiPrimeCase}. We will use the following lemma:
\begin{lemma}\label{lemma:norm}Let $B/A$ be the extension of discrete valuation rings
	attached to $L/K$. For brevity, we write $N:=N_{L/K}$ and $\Tr:=\Tr_{L/K}$.
	\begin{enumerate}[label=\emph{(\roman*)}]
		\item $\mathfrak{D}_{B/A}=\mathfrak{m}_L^{(t+1)(\ell-1)}$. 
		\item If $n\geq 0$, then 
			\[\Tr(\mathfrak{m}_L^n)=\mathfrak{m}_K^{\left\lfloor\frac{(t+1)(\ell-1)+n}{\ell}\right\rfloor}.\]
		\item If $x\in \mathfrak{m}_L^n$, then 
			\[N(1+x)\equiv1+\Tr(x)+N(x)\mod
				\Tr(\mathfrak{m}^{2n}_L).\]
		\end{enumerate}
	\end{lemma}
	\begin{proof}
\begin{enumerate}[label={(\roman*)}]
	\item Proposition \ref{valuation-of-discriminant},
		\ref{item:valuation-of-discriminant:b}.
	\item Recall that if $\mathfrak{b}\subset L$ and $\mathfrak{a}\subset
		K$ are fractional ideals,
		then $\Tr(\mathfrak{b})\subset \mathfrak{a}$ if and only if
		$\mathfrak{b}\subset \mathfrak{a}\mathfrak{D}_{B/A}^{-1}$
		(Lemma \ref{lemma:maxCharOfDifferent}). It
		follows that
		$\Tr(\mathfrak{m}^n_L)\subset\mathfrak{m}^r_K=\mathfrak{m}^{\ell
		r}_L$ if and only if $r\leq \frac{(t+1)(\ell-1)+n}{\ell}$.
	\item Exercise (see \cite[Ch.~V, \S3., Lemma 5]{Serre/LocalFields}).
\end{enumerate}
\end{proof}
We can proceed with the proof of \ref{item:normC} and the claim.
\begin{itemize}
	\item[$m=0$:] Clearly $N(U_L)\subset U_K$ and
		$N(U^1_L)\subset U^1_K$. The map $N_0:k^\times
		\rightarrow k^\times$ can be identified
		with the polynomial $X^\ell$. Let us determine the kernel of
		$N_0$: Note that $\ell=p:=\Char(k)$ if and only if $t>0$, as $G_1$ is
		the unique $p$-Sylow subgroup of $G_0$ (Corollary
		\ref{cor:structure-of-ramification-groups}). Thus, if $t=0$, then the kernel
		of $N_0$ is of order $\ell$ and contains $\theta_0(G_0)$ which
		is also of order $\ell$. If $t>0$, then $\ell=p$ and hence
		$N_0$ is injective. Thus
		for $m=0$ we have checked that \ref{item:normC} holds. 
	\item[$1\leq m < t$]: In this case $\psi_{L/K}(m)=m$ (Example
		\ref{ex:phiandpsiPrimeCase}), so we first show that
		$N(U^m_L)\subset U^m_K$. If $x\in \mathfrak{m}^m_L$, then
		$\Tr(x)\in \mathfrak{m}^{m+1}_K$ according to Lemma \ref{lemma:norm}, as
		\[\left\lfloor \frac{(t+1)(\ell-1)+m}{\ell}\right\rfloor>
			\left\lfloor\frac{\ell
			m+(\ell-1)}{\ell}\right\rfloor\geq m.\]
			This also shows that $\Tr(\mathfrak{m}^{2m}_L)\subset
			\mathfrak{m}^{m+1}_K$, so
			Lemma \ref{lemma:norm} yields
			\[N(1+x)\equiv 1+N(x)\mod \mathfrak{m}^{m+1}_K.\]
			But $N(x)\in \mathfrak{m}^{m}_K$, so we see that
		$N(U_L^{m})\subset U_K^m$, and also that $N(U_{L}^{m+1})\subset
		U_{K}^{m+1}$.

		To make $N_m$ explicit, note that since $t>0$ we have
		$\ell=p$, and if $x=u \pi^m$ with $u\in U_L$ and $\pi$ a
		uniformizer, then $N(x)\equiv u^pN(\pi)^m\mod \mathfrak{m}_K^{m+1}$
		(recall that $\sigma u -u \in \mathfrak{m}_L^{t+1}$ for all
		$u$). If
		$\tau$ is a uniformizer for $K$, then $N(\pi^m)=a_m\tau^m$,
		with $a_m\in U_K$.
		Then $N(1+u\pi^m)\equiv 1+a_mu^p\tau^m\mod\mathfrak{m}^{m+1}_K$. Thus $N_m$ can be identified
		with the polynomial $\alpha_m X^p$, where $\alpha_m$ is the
		image of $a_m$ in $k^\times$.
		From this we also see that $N_m$ is injective, and that it is
		surjective if $k$ is perfect.
	\item[$m=t$:] Let $x\in \mathfrak{m}^t_L$. Lemma \ref{lemma:norm} shows that
		\[N(1+x)\equiv 1+\Tr(x)+N(x)\mod \mathfrak{m}^{t+1}_K.\]
		Moreover,
		\[\left\lfloor \frac{(t+1)(\ell-1)+t}{\ell}\right\rfloor=t\]
		and
		\[\left\lfloor \frac{(t+1)(\ell-1)+t}{\ell}\right\rfloor=t+1\]
		so $N(U_L^{t})\subset U_K^t$ and $N(U_L^{t+1})\subset
		U_K^{t+1}$. As $\psi_{L/K}(t)=t$, this is what we wanted. Let
		us
		make it more explicit:
		Again let $\tau$ be a uniformizer of $K$, and write
		$\Tr(\pi^t)=b\tau^t$, $N(\pi^t)=a\tau^t$, with $a,b\in A$.
		Every element $x\in U_L^{t}/U_L^{t+1}$ is the image of an
		element of the form $1+u\pi^t$ with $u\in A$. We compute:
		\begin{align*}
			N(1+u\pi^t)&\equiv 1+ \Tr(u\pi^t)+N(u\pi^t)\mod
			\mathfrak{m}_K^{t+1}\\
			&\equiv 1+(bu+au^p)\tau^t\mod
			\mathfrak{m}_K^{t+1}.
		\end{align*}
		Writing $\alpha,\beta$ for the images of  $a,b$ in $k$, we see
		that $N_t$ can be written as the polynomial $\beta X+\alpha
		X^p$. As $\alpha\neq 0$,  we see that the kernel of $N_m$ has
		order at most $p=\ell$, but it contains $\theta_t(G_t)$ (Proposition \ref{prop:map-lowernumbering-to-units}), so it
		has to be equal to $\theta_t(G_t)$. Hence $\beta\neq0$, and 
		\ref{item:normC} is proved.
	\item[$m>t$:] We computed in Example \ref{ex:phiandpsiPrimeCase}
		that $\psi_{L/K}(m)=t+\ell(m-t)$. We have
		\begin{equation}\label{floorEquation}\frac{(t+1)(\ell-1)+\psi_{L/K}(m)}{\ell}=\frac{\ell-1}{\ell}+m,\end{equation}
			so $\Tr(\mathfrak{m}_L^{\psi_{L/K}(m)})\subset
			\mathfrak{m}^{m}_K$. Moreover, $\psi_{L/K}(m)\geq m+1$ and
			$N(\mathfrak{m}^{\psi_{L/K}(m)}_L)\subset
			\mathfrak{m}^{\psi_{L/K}(m)}_K$, so  if
			$x\in \mathfrak{m}_L^{\psi_{L/K}(m)}$, then
			\[N(1+x)\equiv 1+ \Tr(x)\mod \mathfrak{m}_K^{m+1}.\]
		This shows that  $N(U_L^{\psi_{L/K}(m)})\subset U_K^{m}$ and
		similarly also $N(U_{L}^{\psi_{L/K}(m)+1})\subset
		U_K^{m+1}$.
		Again, if $x=u\pi^{\psi_{L/K}(m)}$ with $u\in A$, and
		$\Tr(\pi^{\psi_{L/K}(m)})=a_m\tau^m$, then
		$N(1+u\pi^{\psi_{L/K}(m)})\equiv1+a_mu\tau^{m}\mod\mathfrak{m}_K^{m+1}$, so $N_m$
		can be written as the polynomial $\alpha_mX$, if $\alpha_m$ is the image of
		$a_m$ in $k$. Finally, $\alpha_m\neq0$, as otherwise
		$\Tr(\mathfrak{m}^{\psi_{L/K}(m)}_L)\subset
		\mathfrak{m}^{m+1}_K$, which is impossible according to
		\eqref{floorEquation} and Lemma \ref{lemma:norm}.
\end{itemize}
\end{proof}

For future reference, we prove a fact about the norm of a cyclic
extension.

\begin{theorem}[Hilbert's Theorem 90]\label{thm:hilbert90}\index{Hilbert's
	Theorem 90}
	Let $L/K$ be a Galois extension of fields with cyclic Galois group $G$
	and generator $g$
	(no other assumptions are necessary). If $v\in L^\times$ has norm $1$,
	then there exists $x\in L^\times$, such that $v=\frac{g(x)}{x}$. In
	other words, the quotient group
	\[\ker(N_{L/K})\left/\left\{\left.\frac{g(x)}{x}\right| x\in
	L^\times\right\}\right.\]
	is trivial.
\end{theorem}
\begin{proof}
	Write $V=\ker(N_{L/K})$ and
	$\overline{V}:=V/\left\{\left.\frac{g(x)}{x}\right|x\in
L^\times\right\}$.
The group $\overline{V}$ can be identified with the group cohomology
	$H^1(G,L^\times)$. A brief reminder: A crossed map $\phi:G\rightarrow
	L^\times$, is a map $\phi$, such that $\phi(g_1g_2)=g_1(\phi(g_2))\cdot\phi(g_1)$, $g_1,g_2\in G$. A
	particular example of a crossed map is $g\mapsto
	\frac{gx}{x}$ for $x\in L^\times$; such a crossed map is called
	principal. The set of crossed maps inherits a group
	structure, and 
	\[H^1(G,L^\times)=\{\text{crossed maps}\}/\{\text{principal
	crossed maps}\}.\]
	If $G$ is a cyclic group with generator $g$ and if
	$\phi:G\rightarrow L^\times$ is a crossed map, then $\phi(g)\in
	L^\times$ has norm $1$. It is not difficult to check that this defines 
	a homomorphism $u:\{\text{crossed maps}\}\rightarrow \ker(N_{L/K})$. Moreover, the map $u$ is surjective: If $x\in
	L^\times$ has norm $1$, then $\psi(1):=1$,
	$\psi(g^r):=\prod_{i=0}^{r-1}g^ix$, $1\leq r<|G|$, defines a crossed
	map with $\psi(g)=x$. Finally, one checks that	\[u^{-1}\left(\left\{\left. \frac{g(x)}{x}\right|x\in L^\times\right\}
		\right)=\left\{\text{principal crossed maps}\right\},\]
	from which it follows that $u$ induces an isomorphism
	$H^1(G,L^\times)\xrightarrow{\cong}
	\overline{V}$.

	To prove the theorem, it hence remains to show that every crossed map
	$G\rightarrow L^\times$ is principal, and this is true even if
	$G=\Gal(L/K)$ is not cyclic.

	Let $\phi$ be a crossed map and for $y\in L^\times$ write
	\[x:=\sum_{\sigma\in G}\phi(\sigma)\cdot \sigma(y).\]
	If there exists $y$ such that $x\neq 0$, then $\phi$ is principal.
	Indeed, in this case, for any $\tau\in G$ we can write
	\[\tau(x)=\sum_{\sigma\in G}\tau(\phi(\sigma))\cdot \tau\sigma(y).\]
	Since $\phi(\tau\sigma)=\tau\phi(\sigma)\cdot\phi(\tau)$, we see that
	\[\tau(x)=\phi(\tau)^{-1}\sum_{\sigma\in
	G}\phi(\tau\sigma)\tau\sigma(y)=\phi(\tau)^{-1}\cdot x\]
	which proves that $\phi$ is principal, if $x\neq 0$.

	To see that we can find $y\in L^\times$ such that $x\neq 0$, one
	observes (this is not entirely trivial) that the
	family of maps $\sigma:L^\times \rightarrow L^\times$, $\sigma \in G$
	is linearly
	independent over $L$. In particular: $\sum_{\sigma\in G}
	\phi(\sigma)\cdot \sigma:L\rightarrow L$ is not the zero-map. 
\end{proof}

\begin{exercise}
	\begin{enumerate}[label={(\alph*)}]
		\item Complete the proof that $H^1(G,L^\times)\cong \overline{V}$.
		\item If $T$ is a group and $f_1,\ldots, f_r:T\rightarrow
			L^\times$ pairwise distinct homomorphisms, prove that
			the $f_i$ are linearly independent over $K$, i.e.~if
			there are $a_1,\ldots, a_r\in L$, such that
			$\sum_{i=1}^{r}a_if_i=0$ as map $T\rightarrow L$, then
			$a_1=a_2=\ldots=a_r=0$.
	\end{enumerate}
\end{exercise}
\subsection{Jumps and the theorem of Hasse-Arf}

We denote by $L/K$ a finite Galois extension of complete
discretely valued fields with separable residue extension.
Let $G$ be its Galois group.
We know that there are only finitely many subgroups
which appear in the ramification filtration, regardless of whether we use the
upper or lower numbering. The
real numbers $u\in [-1,\infty)$, where the upper numbering filtration
	$\{G^u\}$ transitions from one group to the next, are of special
	interest.

\begin{definition}
		The numbers $u\in [-1,\infty)$, such that $G^u\supsetneqq
			G^{u+\epsilon}$ for all $\epsilon>0$ are called
			\emph{jumps}\index{jump} or \emph{breaks}\index{break}
			of this filtration.
\end{definition}
The jumps are not necessarily integers, if $L/K$ is not abelian.
\begin{example}
	In \cite[Sec.~4]{Serre/Rationality}, an example of a totally ramified, finite Galois
	extension $L/\Q_2$ is constructed, for which the jumps of the upper
	numbering filtration of its Galois group are not integers. More
	precisely, the Galois group $G:=\Gal(L/\Q_2)$ is the group of quaternions
	$\{\pm 1,\pm i,\pm j,\pm k\}$ with $(-1)^2=1$, $i^2=j^2=k^2=ijk=-1$.
	Its center is $Z(G)=\{\pm 1\}$, and $L/\Q_2$ has the ramification
	filtration
	$G_0=G_1=G$, $G_2=G_3=Z(G)$, $G_n=\{1\}$, $n>3$. It follows that the
	two jumps are $1$ and $3/2$.
\end{example}
%

\begin{theorem}[Hasse-Arf]\label{thm:HasseArf}
	If the group $G$ is \emph{abelian}, then the jumps are integers.	
\end{theorem}

\begin{remark}\label{rem:jumpsleq0}
	We already know a few jumps of the ramification filtration: Since
	$\phi_{L/K}(u)=u$ for $u\in [-1,0]$, we see that the jumps in the
	filtration which happen in $[-1,0]$ can only happen at $-1$ and $0$.
	Thus, in the proof of the Hasse-Arf theorem, it suffices to treat the jumps $>0$.
\end{remark}

This deep theorem is fundamental for everything which we want to cover later on, so
we will sketch a proof. For those who know local class field theory, the
theorem of Hasse-Arf is easy to derive from it, at least in the case that the
residue field of $K$ is a finite field.
\begin{example}\index{class field theory}
	Fix a separable closure $K^{\sep}$ of $K$, and let $K^{\ab}$ be the compositum
	of all finite abelian extensions of $K$ contained in $K^{\sep}$.
	If the residue field of $K$ is finite, local class field theory states
	among other things  that there exists a continuous morphism
	$\theta:K^\times\rightarrow \Gal(K^\ab/K)$, the \emph{local Artin
	map}\index{local Artin map} or \emph{local reciprocity
	map}\index{local reciprocity map}, which becomes an isomorphism after
	profinite completion.  Here, the profinite completion of the
	(non-discrete) topological group
	$K^\times$ is to be understood as $\varprojlim K^\times/U$, where $U$
	runs through the finite index, normal, open subgroups.

	More relevantly for us, if $L/K$ is a finite abelian extension, then
	the composition 
	\[K^\times\xrightarrow{\theta}\Gal(K^\ab/K)\twoheadrightarrow \Gal(L/K)\]
	maps $U_K^v$ onto $\Gal(L/K)^{v}$ for all real $v\in [0,\infty)$, see
		\cite[Thm.~1, Ch.~VI.4]{CasselsFroehlich}. Here
		we use the definition $U_{K}^v:=U_{K}^{\lceil v\rceil}$.
	The jumps of the filtration $\{U_K^v\}$ are obviously integers, hence
	the same is true for the upper numbering filtration of $\Gal(L/K)$,
	whence the Hasse-Arf theorem follows from local class field theory, if
	the residue field of $K$ is finite. 

	Of course proving the existence of the local reciprocity map and
	establishing its properties is difficult. We will follow
	\cite{Serre/LocalFields} and give a more elementary proof of the
	Hasse-Arf theorem for arbitrary residue fields.
\end{example}

The proof of the Hasse-Arf Theorem \ref{thm:HasseArf} is fairly long and
intricate, so we first make our life easier by simplifying the situation.

\begin{proposition}\label{prop:simpleHasseArf}
	Recall that $L/K$ is a finite Galois extension of complete discretely
	valued fields with separable residue extension. Additionally, assume
	that $L/K$ is totally ramified and 
	that the group
$G=\Gal(L/K)$ is \emph{cyclic}. 

If $\mu$ is the largest integer such that
$G_{\mu}\neq 1$, then $\phi_{L/K}(\mu)$ is an integer.
\end{proposition}

This special case of Theorem \ref{thm:HasseArf} actually implies the theorem.

\begin{lemma} If Proposition \ref{prop:simpleHasseArf} is true, then so is
	the Theorem \ref{thm:HasseArf} of Hasse-Arf.
\end{lemma}
\begin{proof}
	Let $L/K$ be a finite abelian extension with Galois group $G$ and let $v\in
	[0,\infty)$ be a jump, i.e.~$G^v\supsetneqq G^{v+\epsilon}$ for all
		$\epsilon>0$. By Remark \ref{rem:jumpsleq0} we may assume that
		$v>0$, hence $G^v\subsetneqq G^0=G_0$. Replacing $K$ by $L^{G_0}$, we may assume that $L/K$ is totally
		ramified.
		
		Let $\epsilon_0>0$ be sufficiently small, so
		that $v+\epsilon_{0}$ is smaller than the next jump in the
		filtration.
		Since $G$ is a finite abelian group, so is $G/G^{v+\epsilon_0}$, and the structure theorem for
		such groups tells us that there exists a cyclic group $H$, and
		a surjective map $\gamma:G\twoheadrightarrow H$, such that
		$\gamma(G^{v+\epsilon_0})=1$, and $\gamma(G^v)\neq 1$. 
		
		The cyclic group $H$ is the Galois group of a subextension $L'/K$  of
		$L/K$, and by Proposition \ref{prop:quotientsUpperNumbering},
		we know that $H^u=\gamma(G^u)$ for all $u\in 
		\left[-1,\infty\right)$. It
		follows that $H^{v+\epsilon_0}=1$ and $H^v\neq 1$, and that
		there are no jumps between $v$ and $v+\epsilon_0$.
		Finally, write $v=\phi_{L'/K}(\mu)$. Since
		$H_{\mu}=H^{\phi_{L'/K}(\mu)}$, by continuity of
		$\phi$, it follows that $H_{\mu+\epsilon}=1$ for all
		$\epsilon>0$. This means that $\lceil\mu\rceil<\lceil
		\mu + \epsilon\rceil$ for all $\epsilon>0$, and hence
		$\mu$ is an integer; in fact it is the largest integer such that
		$H_\mu\neq 1$. Now we are in the situation of
		Proposition \ref{prop:simpleHasseArf}, which implies, if true,
		that $v=\phi_{L'/K}(\mu)$ is an integer.
\end{proof}

To prove the Hasse-Arf Theorem \ref{thm:HasseArf}, we can thus concentrate on
totally ramified, cyclic extensions $L/K$, which we will do in the following
section.

Before we proceed with the proof of the Hasse-Arf Theorem, we continue Example
\ref{ex:ASlowerNumbering}.
\begin{example}\label{ex:ASupperNumbering}
	Let $K=k\llparen x\rrparen$ with $k$ an algebraically closed
	field of characteristic $p>0$. Let $L$ be the Galois extension given
	by $L:=K[t]/(t^p-t-x^{-m})$, where $(m,p)=1$; its Galois group is
	$G=\F_p$. We saw in Example \ref{ex:ASlowerNumbering} that 
	$G_0=G_1=\ldots=G_m\supsetneq G_{m+1}=0$. Let us find the jumps
	of $L/K$. For this we just have to compute $\phi_{L/K}(m)$, which is
	easy because of \eqref{eq:formulaForPhi}:
	\[\phi_{L/K}(m)=\sum_{i=0}^m\frac{|G_i|}{|G_0|}-1=m.\]
	We see that in this simple example, the conclusion of the Hasse-Arf theorem holds true.
\end{example}

\subsection{The ramification filtration of a cyclic extension}

In this section, let $L/K$ be a totally ramified, finite cyclic Galois extension of complete discretely
valued fields, $G$ its Galois group and $g$ a generator. Let $V$ be the kernel of the norm $N:L^\times\rightarrow
K^\times$. By Hilbert's Theorem 90 (Theorem \ref{thm:hilbert90})
$V=\{\frac{gy}{y}|y\in L^\times\}$. Define 
\[W:=\left\{x\in V| x=\frac{gy}{y}\text{ for some }y\in U_L\right\}\subset V,\]
where $B$ is the valuation ring of $L$ and $U_L=B^\times$. Clearly, $W$ is a subgroup of $V$.

For $i\geq 0$, define $V_i:=V\cap U_L^i$ and $W_i:=W\cap U_L^i$. The
quotients $V_i/W_i$ define a filtration of the finite cyclic group $V/W$.

\begin{proposition}\label{prop:filteredHom}
Let $\pi\in B$ be a uniformizer and consider $G$ to be filtered by the
ramification filtration $\{G_i\}$ in the lower numbering.
\begin{enumerate}
	\item The assignment $\theta(\sigma):=\frac{\sigma\pi}{\pi}$ defines an
		isomorphism of groups $\theta:G\rightarrow V/W$. 
	\item $\theta$ respects the filtrations,
		i.e.~$\theta|_{G_i}:G_i\hookrightarrow
		V_i/W_i$ for all $i\geq 0$.
\end{enumerate}
\end{proposition}
\begin{proof}
	The following argument partly follows lecture notes by I.~Fesenko.
	First, note that for any $\sigma \in G$,
	$\theta(\sigma)=\frac{\sigma(\pi)}{\pi}\in V$. Indeed,
	$v_L(\theta(\sigma))=v_L(\sigma(\pi))-v_L(\pi)=0$, so
	$\theta(\sigma)\in U_L$. Also, 
	\[N_{L/K}(\theta(\sigma))=\prod_{\tau\in
	G}\frac{\tau\sigma(\pi)}{\tau(\pi)}=\frac{\prod_{\tau\in
	G}\tau\sigma(\pi)}{\prod_{\tau\in G}\tau(\pi)}=1.\]

	Let us see that $\theta$ is a homomorphism of groups: For any
	$\tau,\sigma \in G$ we have
	\[\frac{\tau\sigma(\pi)}{\pi}\cdot
		\frac{\pi^2}{\sigma(\pi)\tau(\pi)}=\frac{\sigma\left(\frac{\tau(\pi)}{\pi}\right)}{\frac{\tau(\pi)}{\pi}},\]
		and $\frac{\tau(\pi)}{\pi}\in U_L$. Hence
		$\theta(\sigma\tau)\equiv \theta(\sigma)\theta(\tau)\mod W$.
	
		To see that $\theta$ is injective, assume that
		$\frac{g^r(\pi)}{\pi}\in W$. As we just saw that $\theta$ is a homomorphism, this means that there exists $u\in U_L$ such that
	$\theta(g^r)=\theta(g)^r=\left(\frac{g(\pi)}{\pi}\right)^r=\frac{g(u)}{u}$. But this means
	$g\left(\frac{\pi^r}{u}\right)=\frac{\pi^r}{u}$, so $\pi^r/u\in K$.
	Hence the degree $[L:K]$ divides $r$, so $g^r=1$, and $\theta$ is
	injective.

	For the surjectivity, note that every element in $V$ has the shape
	$\frac{g(u\pi^i)}{u\pi^i}$, for $u\in U_L$ and $i\in \Z$. But in $V/W$
	we have $\frac{g(u\pi^i)}{u\pi^i}=\frac{g(\pi^i)}{\pi^i}=\theta(g^i)$.

	Finally, if $\sigma\in G_i$, then $\frac{\sigma(\pi)}{\pi}\in U_L^i$ by
	definition, so $\theta$ indeed induces injective morphisms
	$\theta_i:G_i\rightarrow V_i/W_i$. 
%
\end{proof}

\begin{proposition}\label{prop:almostHasseArf}
	In addition to the assumptions and notations of Proposition
	\ref{prop:filteredHom}, assume that the residue field of $K$ is not
	the prime field $\F_p$. Then $G_m=0$ if and only if $V_m/W_m=0$. 
\end{proposition}
\begin{proof}
	We have seen that $\theta$ restricts to an injective map
	$G_m\hookrightarrow V_m/W_m$, so it just remains to show that
	$V_m/W_m=0$ if $G_m=0$. We proceed by decending induction on $m$. We
	know that $G_m=0$ for $m\gg 0$, and the same is true for $V_m/W_m$.
	Indeed, as $L/K$ is separable, there exists $t\in L$ such that
	$\Tr_{L/K}(t)=1$ (Proposition \ref{prop:separability}). Write $M:=\max\{0,-v_L(t)\}$. We claim that $V_m=W_m$
	for all $m>M$. Let $x\in V_m$, and write
	\[y:=t+\sum_{i=1}^{|G|-1}xg(x)g^2(x)\cdot\ldots\cdot
		g^{i-1}(x)g^{i}(t).\]
		Subtracting $\Tr_{L/K}(t)=1$ from both sides we obtain 
	\[y-1=\sum_{i=1}^{|G|-1}(xg(x)g^2(x)\cdot\ldots\cdot
		g^{i-1}(x)-1)g^{i}(t),\]
	and hence $v_L(y-1)>0$. It follows that $y\in U_L^1$, and that 
	$\frac{y}{g(y)}=x,$
	since $N_{L/K}(x)=1$:
	\begin{align*}
		\frac{y}{g(y)}&=\frac{xy}{xg(y)}\\
		&=\frac{xy}{\sum_{i=1}^{|G|-1}xg(x)\cdot\ldots\cdot
		g^{i-1}(x)g^{i}(t)+N_{L/K}(x)t}\\
		&=x.
	\end{align*}
	Thus $x=\frac{y}{g(y)}\in W_m$, so $W_m=V_m$ for $m>M$.

	Now let $G_m=0$, and assume that we know that $V_{m+1}/W_{m+1}=0$. Our
	goal is to show that $V_m/W_m=0$.

	We will actually prove something stronger: 
\begin{lemma}\label{lemma:almostHasseArf}
	We continue to use the notations and assumptions from Proposition
	\ref{prop:almostHasseArf}. In particular, the residue field of $K$ is
	not the prime field. Write $n+1:=\lceil
	\phi_{L/K}(m)\rceil$. If $V_{m+1}/W_{m+1}=0$ and
	$G_{\psi_{L/K}(n+1)}=0$, then $V_m/W_m=0$.
\end{lemma}
\begin{proof}
	Note that $\psi_{L/K}(n+1)$ is an integer $\geq m$ (Exercise
	\ref{ex:psiOnIntegers}), with equality if
	and only if $\phi_{L/K}(m)$ is an integer.

	We treat these two cases separately: $\phi_{L/K}(m)\in \Z$ and
	$\phi_{L/K}(m)\not\in\Z$.

	First, let us assume that $\phi_{L/K}(m)\in \Z$,
	i.e.~$\phi_{L/K}(m)=n+1$. Proposition
	\ref{prop:norm} then produces a commutative diagram
	\begin{equation*}
		\begin{tikzcd}
			~ &0\dar&0\dar&0\dar\\
			0\rar&V_{m+1}\dar\rar&V_{m}\dar\rar
			&G_{m}/G_{m+1}\dar{\theta_m}\\
			0\rar&U_L^{m+1}\rar\dar{N}&U_{L}^m\rar\dar{N}&
			U^{m}_L/U^{m+1}_L\rar\dar{N_m}&0\\
			0\rar&U_K^{n+2}\rar&U_{K}^{n+1}\rar&
			U^{n+1}_K/U^{n+2}_K\rar&0.\\
		\end{tikzcd}
	 \end{equation*}
	 As $V_{m+1}/W_{m+1}=0$ by assumption, we
	 see that the composition 
	 \[V_m\rightarrow G_m/G_{m+1}\rightarrow
		 (V_m/W_m)/(V_{m+1}/W_{m+1})=V_m/W_m\]
		 is the canonical quotient map. But
		 $G_m=G_{\psi_{L/K}(n+1)}=0$, so $V_m/W_m=0$.


	 Now assume that $\phi_{L/K}(m)\not\in\Z$,
	 i.e.\[n<\phi_{L/K}(m)<n+1\text{ and }
		 \psi_{L/K}(n)+1\leq m<\psi_{L/K}(n+1).\]
	 We have $V_m/V_{m+1}\subset U^m_L/U^{m+1}_L$, and by assumption
	 $V_{m+1}=W_{m+1}$. Thus, to show that $V_m/W_m=0$, it suffices to
	 show that $V_m$ and $W_m$ have the same image in $U^m_L/U^{m+1}_L$.

	 Let $x\in U^m_L$ be arbitrary. We have the inclusion
	 $U^{\psi_{L/K}(n+1)}_L\subset U_L^m$, and the norm induces a
	 surjective morphism $N:U^{\psi_{L/K}(n+1)}_L \twoheadrightarrow
	 U_K^{n+1}$ by Proposition \ref{prop:norm}. It follows that there
	 exists $y\in U^{\psi_{L/K}(n+1)}_L$ with $N(y)=N(x)$. Then
	 $N(xy^{-1})=1$ and hence $xy^{-1}\in V_m$. As $\psi_{L/K}(n+1)\geq
	 m+1$, we see that $y\in U_{L}^{m+1}$, so $x\equiv xy^{-1}\mod
	 U_L^{m+1}$. It follows that $V_m$ maps surjectively to
	 $U_L^m/U_L^{m+1}$.

	 As for the image of $W_m$, from what we showed in the previous
	 paragraph, it follows that $V_m/W_m$ maps surjectively
	 to  $(U_L^{m}/U_L^{m+1})/W_m$, so this group is also cyclic. But for
	 $m>0$, the group $U_L^{m}/U_{L}^{m+1}$ is isomorphic to the additive
	 group underlying the residue field of $K$, which is not
	 cyclic, as the residue field of $K$ is assumed not to be the prime
	 field.  Thus the image of $W_m$ in $U_L^{m}/U_L^{m+1}$ is
	 nontrivial. 
	 
	 Let $x\in W_m\setminus U^{m+1}_L$. Then there exists $y\in U_L$ such
	 that $x=\frac{g(y)}{y}$. By modifying $y$ with an element from $U_K$,
	 we may assume that $y\in U_L^1$, i.e.~$y=1+z$, $v_L(z)>0$. For $a\in
	 U_K$ define $y_a:=1+az$ and $x_a:=\frac{g(y_a)}{y_a}$. 
	 
	 Denote by
	 $\gamma:U_L^{m+1}/U_L^m\rightarrow
	 \mathfrak{m}_L^{m}/\mathfrak{m}_L^{m+1}$ the isomorphism	 given by 
	 $\gamma(1+u\pi^m)=u\pi^m\mod \mathfrak{m}_L^{m+1}$, where $\pi$ is a
	 uniformizer of $L$ and $\mathfrak{m}_L$ the maximal ideal 
	 of the discrete valuation ring of $L$.  We claim that
	 $x_a\in W_m$, and that $\gamma(x_a)= a\gamma(x)\mod \mathfrak{m}_L^{m+1}$. Since
	 $\mathfrak{m}_L^{m}/\mathfrak{m}_L^{m+1}$ is a $1$-dimensional vector space over the residue
	 field of $K$, this would show that $W_m$ maps surjectively to
	 $U_L^m/U_L^{m+1}$, which is what we wanted to show.

	 As for the claim, we compute
	 \begin{align*}
		 x_a-1&=\frac{g(y_a)-y_a}{y_a}\\
		 &=\frac{1+ag(z)-(1+az)}{y_a}\\
		 &=\frac{a(gz-z)}{y_a}\\
		 &=a\frac{g(y)-y}{y_a}\\
		 &=a\frac{y}{y_a}(x-1).\end{align*}
	 It follows that $v_L(x_a-1)=v_L(a)+m\geq m$, so $x_a\in W_m$.
	 Finally, to see that $\gamma(x_a)=a\gamma(x)$, it suffices to remark that $\frac{y}{y_a}\in U^1_L$. 
	 \end{proof}
\end{proof}
\begin{remark}
	If the residue field of $K$ is algebraically closed, almost the same
	proof shows that the isomorphism $\theta:G\xrightarrow{\cong} V/W$ induces
	isomorphisms $G_i\xrightarrow{\cong}V_i/W_i$ for all $i$.
\end{remark}

Finally we can prove Proposition \ref{prop:simpleHasseArf} and hence the
Hasse-Arf theorem \ref{thm:HasseArf}.

\begin{proof}[Proof of Proposition \ref{prop:simpleHasseArf}]
	Recall that $L/K$ is a totally ramified, cyclic extension of complete
	discretely valued fields with separable residue extension. If the
	residue field of $K$ is the prime field, we can replace $K$ by an
	unramified extension with larger residue field (\cite[V, \S 4]{Serre/LocalFields}). This does not change
	$G_i$ for $i\geq 0$. If
	$\mu$ is the largest positive integer such that $G_\mu\neq 0$, then we
	want to show that $\phi_{L/K}(\mu)$ is an integer. If it is not an
	integer, then there exists an integer $\nu$, such that
	$\nu<\phi_{L/K}(\mu)<\nu+1$, and hence $\mu<\psi_{L/K}(\nu+1)$. It
	follows that $G_{\psi_{L/K}(\nu+1)}=0$. Now Lemma
	\ref{lemma:almostHasseArf} applies, and thus $V_{\mu}/W_\mu=0$. But we
	have seen in Proposition \ref{prop:filteredHom} that $\theta:G\xrightarrow{\cong} V/W$ respects the
	filtrations, so $G_\mu\hookrightarrow V_{\mu}/W_{\mu}=0$, which is a
	contradiction. Thus $\phi_{L/K}(\mu)\in \Z$, and the proof is complete.
\end{proof}


\section{The Swan representation}\label{sec:swan}
Let $K$ be a complete discretely valued field with perfect residue field and
$K^{\sep}$ a separable closure. If $\ell$ is a prime number different from the
residue characteristic of $K$ and $E$ a finite extension of $\Q_{\ell}$, then
we want to study continuous morphisms $\rho:\Gal(K^{\sep}/K)\rightarrow
\GL_r(E)$, where $\GL_r(E)$ is considered as a topological group with the
topology induced by $E$ (see Definition \ref{defn:topologyOnGL}).

In Section \ref{sec:swanconductor} will define the \emph{Swan conductor} $\sw(\rho)\in \Z$ of such a
representation, which is
additive in short exact sequences and zero if and only if $\rho$ is
\emph{tame}, i.e.~if $\rho$ restricts to the trivial representation on the wild ramification group. We will see in Section \ref{sec:GOS}
that this invariant has a cohomological interpretation.
Before we come to the Swan conductor, we will use the theorem of Hasse-Arf to  attach to a finite Galois extension
$L/K$ with group $G$, first a complex representation, then a projective
$\Z_\ell[G]$-module $\Sw_G$, the \emph{Swan representation}\index{Swan
representation}. 

\subsection{A minimal amount of representation theory of finite groups}
We roughly follow \cite{Serre/Representations}, but \cite{Lam} and \cite{Webb}
are also helpful.

\subsubsection{Basics}
Let $G$ be a finite group, $E$ a field and $V$ a finite dimensional $E$-vector
space. A morphism $\rho:G\rightarrow
\GL(V)$ is called a \emph{representation of $E$}. If $\rho':G\rightarrow
\GL(V')$ is a second representation, then a \emph{morphism of representations}
$\alpha:\rho\rightarrow \rho'$ is an $E$-linear map $\phi:V\rightarrow V'$, such
that for every $g\in G$, $v\in V$, we have $\phi(\rho(g)(v))=\rho'(g)(\phi(v))$. We
write $\Repf_EG$ for the category of representations of $G$ on finite
dimensional $E$-vector spaces. This category is equivalent to the category of
finitely generated left modules over the group ring $E[G]$. Recall that $E[G]=E\otimes_{\Z}\Z[G]$,
where $\Z[G]$ is the free $\Z$-module with basis $G$, and multiplication
induced by the rule $ag\cdot bg'=(ab)gg'\in \Z[G]$, for $a,b\in\Z$. From this perspective it is
clear that we have the usual notions of subrepresentations, quotients and constructions like direct sums, products,
etc. In the sequel, we will use the notions of ``representation'' and
``finitely generated $E[G]$-module'' interchangeably. Unless
explicitly stated otherwise,  the word ``$E[G]$-module'' means ``finitely generated left-$E[G]$-module''.

Let $\phi:H\rightarrow G$ be a homomorphism of finite groups, and
$\psi:E[H]\rightarrow E[G]$ the induced morphism of $E$-algebras. If $V$ is a
left-$E[G]$-module, then using $\psi$ we can also consider it as an
left-$E[H]$-module. The corresponding representation of $H$ is called the \emph{restriction
	of $V$ to $H$, denoted $\RES_{\phi}
V$}\index{representation!restriction}. Similarly, if $W$ is a
left-$E[H]$-module, and if we consider $E[G]$ as a right-$E[H]$-module via $\phi$,  then the left-action of $E[G]$ on itself makes $E[G]\otimes_{E[H]} W$ into a left-$E[G]$-module. The
corresponding representation of $G$ is said to be \emph{the representation induced by
	$\phi$ and $W$, denoted $\Ind_{\phi}
W$}\index{induced!representation}\index{representation!induced}. These constructions define
	functors $\Ind_{\phi}:\Repf_E H\rightarrow \Repf_E G$, and
	$\RES_{\phi}:\Repf_E G\rightarrow \Repf_E H$. Considering the module
	theoretic definitions of these constructions, we see that these
	functors are actually adjoint:
	\eq{Ind-Res}{\Hom_{\Repf_E H}(-,\RES_\phi(-))=\Hom_{\Repf_E G}(\Ind_\phi(-),-).}

Here are some important examples of representations.
\begin{example}\label{ex:representationsFirstExamples}\leavevmode
\begin{enumerate}[label=(\alph*)]\label{ex:reps}
	\item The $1$-dimensional vector space $E$ with $G$-action given by $ge=e$ for all
		$g\in G$, $e\in E$. This is the \emph{trivial representation
		of rank $1$}, which we denote by $\mathbf{1}_G$ or just $E$. If
		$\rho:G\rightarrow \GL(V)$ is a finite dimensional
		$E$-representation of $G$, such that $\rho(g)=\id$ for all $g\in G$,
		then $\rho$ is isomorphic to a direct sum of finitely many
		copies of $\mathbf{1}_G$. Such a representation $\rho$ is called
		\emph{trivial}. (The trivial $E[G]$-module is the trivial $G$-representation of rank $0$.)

		If $V$ is an $E[G]$-module, we write $V^G:=\{v\in V| \forall g\in
G: gv=v\}$. This is the maximal trivial subrepresentation of $V$.
	\item The ring $E[G]$ is a left-$E[G]$-module. The corresponding
		representation is called the \emph{regular
		representation}\index{representation!regular} of $G$ over $E$.
	\item\label{item:representationsFirstExamples:augmentation} We have a surjective morphism of $E[G]$-modules $E[G]\rightarrow
		E$, $g\mapsto 1$. In other words: The trivial representation
		of rank $1$
		is a quotient of the regular representation. Its kernel is the
		ideal $\Aug_E(G):=( g-1|g\in G)$, and it is
		called the \emph{augmentation
		representation}.\index{representation!augmentation}
	\item If the order of $G$ is invertible in $E$, then the short exact
		sequence
		\[0\rightarrow \Aug_E(G)\rightarrow E[G]\rightarrow
		E\rightarrow 0\]
		is split by the map $\theta:E\rightarrow E[G]$,
		$\theta(e)=\frac{e}{|G|}\sum_{g\in G}g$.
	\item\label{item:internalHom} If $V,W$ are two finite dimensional representations of $G$ over
		$E$, then $\Hom_E(V,W)$ has a canonical structure of a
		representation of $G$: If $\phi\in \Hom_E(V,W)$ and $g\in G$,
		then define $g\bullet \phi$ by
		$(g\bullet\phi)(v)=g\phi(g^{-1}v)$. Clearly
		$\Hom_E(V,W)^G=\Hom_{E[G]}(V,W)$. 
		
		We write $\IntHom(V,W)$ for
		the representation on $\Hom_E(V,W)$ just defined.
	\item If $V$ is a finite dimensional representation of $G$, then
		$\IntHom(V,\mathbf{1}_G)$ is a finite dimensional representation of
		$G$, the \emph{dual representation $V^\vee$ of
		$\rho$}\index{representation!dual}. The underlying vector space is the
		dual vector space $\Hom_E(V,E)$ of $V$ and $G$ acts via
		$(g\bullet \phi)(v)=\phi(g^{-1}v)$.
	\item If $V$ and $W$ are finite dimensional representations of $G$, then the
		tensor product $V\otimes_E W$ carries a canonical $G$-action,  defined via
		$g(v\otimes w)=gv\otimes gw$ for $g\in G$, $v\in V$, $w\in W$. Using this definition we see
		that the natural isomorphism of $E$-vector spaces
		$V^\vee\otimes_E W\rightarrow \Hom_E(V,W)$ is $G$-invariant, and
		hence induces a natural isomorphism of representations
		\[V^{\vee}\otimes W\cong\IntHom(V,W).\]
		Similarly we see that if $U$ is a third representation of $G$,
		then there is a natural isomorphism
		\[\IntHom(U\otimes V, W)\cong \IntHom(U,\IntHom(V,W)).\]
\end{enumerate}
\end{example}

\begin{definition}
	A representation $V$ of $G$ over $E$ is called 
	\begin{itemize}
		\item \emph{irreducible}\index{irreducible!representation} or
	\emph{simple}\index{simple representation} if it is irreducible as  $E[G]$-module, i.e., if it is nonzero and does
	not have nontrivial sub-$E[G]$-modules.
\item \emph{semi-simple}\index{semi-simple representation} if $V$ is a direct sum of irreducible $E[G]$-modules.
\item \emph{decomposable} if $V\cong V_1\oplus V_2$ with $V_1\not\cong V$ and $V_2\not\cong V$. 
\item \emph{indecomposable}\index{indecomposable representation} if $V$ is not decomposable.
	\end{itemize}
\end{definition}

\begin{theorem}[Maschke]\index{Maschke's Theorem}\label{thm:maschke}
	If $G$ is a finite group with order coprime to $\Char(E)$, then every
	representation of $G$ on a finite dimensional $E$-vector space is
	semi-simple.
\end{theorem}
\begin{proof}
	Let $V$ be a finite dimensional $E$-representation of
	$G$, and $W\subset V$ a subrepresentation. Pick a projection $P:V\twoheadrightarrow W$  of
	$E$-vector spaces, which splits the inclusion $W\subset V$, i.e.~such
	that $P|_W=\id_W$, and define 
	\[P_0:=\frac{1}{|G|}\sum_{g\in G}gPg^{-1}.\]
	It is easy to check that $P_0$ is a morphism of $E[G]$-modules, and we
	still have $P_0|_W=\id_W$. It follows that $\ker(P_0)$ is a
	subrepresentation of $V$, and $V=W\oplus \ker(P_0)$ as
	$E[G]$-modules.
\end{proof}

To see the contrast to the situation where $|G|$ and $\Char(E)$ are not
coprime, do the following exercise.
\begin{exercise}\label{ex:repsInPosChar}
	Let $p=\Char(E)$. If $G$ is a finite group, then the two cases
	$(|G|,p)=1$ and $(p,|G|)>1$ are very different.
	
	As a ``worst case'' scenario, prove that if $p>0$ and
	if $G$ is a $p$-group
	acting on a finite dimensional $E$-vector space $V\neq 0$, then $V^G\neq 0$. Here $V^G=\{v\in V| \forall g\in G:
	gv=v\}$.

	Deduce that $V$ is a successive extension of the trivial
	rank $1$ representation of $G$ by itself.
\end{exercise}

\begin{proposition}\label{prop:RepresentationsOfAbelianGroups}Let $E$ be a field containing all roots of unity of order
	dividing $|G|$. If $G$ is abelian, then every irreducible
	representation of $G$ over $E$ has rank $1$. If $|G|$ is invertible in
	$E$, then the converse
	is also true.
\end{proposition}
\begin{proof}
	Let $\rho:G\rightarrow \GL(V)$ be a nonzero irreducible representation
	of $G$. For $g\in G$, the order of $\rho(g)$ divides $|G|$. Let $v\in
	V$ be an eigenvector of $\rho(g)$ and $\lambda$
	its eigenvalue. Then $\lambda$ is a root of unity of order dividing
	$|G|$, thus $\lambda\in E$ by assumption. For any other $h\in G$ we
	have $\rho(gh)(v)=\rho(hg)(v)=\lambda \rho(h)(v)$. Thus the
	$\lambda$-eigenspace of $\rho(g)$ is a nonzero subrepresentation of $\rho$,
	hence equal to $\rho$ as $\rho$ is irreducible.
	This shows that every element $g\in G$ acts via scalar multiplication
	on $V$. As $\rho$ is irreducible, $\dim V=1$.

	If $|G|$ is invertible in $E$, let
	$\rho$ be a faithful representation of $G$, i.e.~an injective map
	$\rho:G\rightarrow \GL(V)$ (such a representation always exists).  If
	the irreducible representations of $G$ have rank $1$,  Maschke's
	Theorem shows that the matrices $\{\rho(g)|g\in G\}$ are
	simultaneously diagonizable. It follows that $G$ is abelian.
\end{proof}

It is convenient to introduce the following notation.
\begin{definition}\label{def:GrothendieckRing}
	If $G$ is a finite group and $E$ a field, we write $R_E(G)$ for the
	\emph{Grothendieck ring}\index{Grothendieck ring} of
	the category $\Repf_E G$. Recall that $R_E(G)$ is the abelian
	group generated by isomorphism classes $[V]$ of finite dimensional 
	$E$-representations of $G$, subject to  the relations $[V]=[V_1]+[V_2]$ if
	there exists a short exact sequence
	\[0\rightarrow V_1\rightarrow V\rightarrow V_2\rightarrow 0.\]
	The tensor product over $E$ makes $R_E(G)$ into a commutative ring (\emph{Exercise}: Check this!). It is easy to
	see that $R_E(G)$ is the free $\Z$-module generated by the
	isomorphism classes of
	irreducible representations of $G$. We will see in Proposition
	\ref{prop:compositionSeries} that there are only finitely many of
	those.

	There is a natural map  of monoids
	\begin{equation}\label{eqn:isoclassMap}\left(\left\{\text{isomorphism
		classes of objects in }\Repf_E
	G\right\},\oplus\right)\rightarrow R_E(G).\end{equation} Its image is the
	subset of all linear combinations $\sum_{V_i\text {irreducible}}a_i [V_i]$, $a_i\geq 0$, and is denoted by $R_E^+(G)$.

	Maschke's Theorem \ref{thm:maschke} can be phrased as follows: If the
	order $G$
	is prime to the characteristic of $E$, then the map
	\eqref{eqn:isoclassMap} is injective.
\end{definition}

Next, we see that there are only finitely many irreducible representations of a
finite group.

\begin{proposition}\label{prop:compositionSeries}
	If $G$ is a finite group and $E$ a field, then any irreducible finite dimensional
	representation of $G$ on $E$ is a composition factor of the regular
	representation $E[G]$. In particular, up to isomorphism there are only
	finitely many irreducible representations of $G$ on $E$.
\end{proposition}
\begin{proof}
 As $E[G]$ and all its subrepresentations are finite dimensional $E$-vector spaces, $E[G]$ admits a finite composition series and thus, up to isomorphism, $E[G]$ has only finitely many composition factors due to the Jordan-H\"older Theorem (see e.g. \cite[I, \S 4.7, Thm.~6]{Bourbaki/AlgebraI}).

If $V$ is an irreducible $E[G]$-module, pick any $v\in V\setminus
	\{0\}$. The map $x\mapsto xv$, is an $E[G]$-linear morphism
	$E[G]\rightarrow V$, which is surjective as $V$ is irreducible. Thus $V$ is a composition factor of $E[G]$.
\end{proof}

We close this section with a fundamental but easy lemma.
\begin{lemma}[Schur's Lemma]\index{Schur's Lemma}\label{lemma:schur}
	Let $V_1, V_2$ be two finite dimensional \emph{irreducible}
	$E[G]$-modules. If $f:V_1\rightarrow V_2$ is a morphism of
	representations, then precisely one of the following statements is
	true:
	\begin{enumerate}[label=\emph{(\alph*)}]
		\item $f=0$.
		\item $f$ is an isomorphism and if $E$ is algebraically closed, then for any other isomorphism
			$g:V_1\xrightarrow{\cong} V_2$, there exists $a\in E$,
			such that $f=ag$.
	\end{enumerate}
	In particular, if $V$ is an irreducible $E[G]$-module, then $\End_{E[G]}(V)$ is a division algebra, and
	$\End_{E[G]}(V)=E$ if $E$ is algebraically closed.
\end{lemma}
\begin{proof}
	If $f$ is not an isomorphism, then $f=0$, due to the irreducibility of
	$V_1, V_2$.
	Now assume that $f$ and $g$ are isomorphisms of $E[G]$-modules and
	that $E$ is algebraically closed. Let
	$a$ be an eigenvalue of $g^{-1}f:V_1\rightarrow V_1$.  Then
	$g^{-1}f-a\id_{V_1}$ has a
	nonzero kernel, so $g^{-1}f=a\cdot \id_{V_1}$ as claimed.
\end{proof}
We give an exemplary application:
\begin{proposition}\label{prop:irreducibilityCriterionNeu}
	Let $G$ be a group and $E$ a field such that $|G|$ is invertible in
	$E$. If $V$ is a representation of $G$ such that
	$\End_{E[G]}(V)=E$, then $V$ is irreducible. If $E$ is algebraically
	closed, the converse is true as well.
\end{proposition}
\begin{proof}
	Since $|G|$ is invertible $V$ can be written as $V=V_1\oplus V_2$ with
	$V_1$ irreducible. We compute
	\[\dim_E \End_{E[G]}(V)\geq \dim_{E}
	\End_{E[G]}(V_1)+\dim_E\End_{E[G]}(V_2)\geq 2\]
	if $V_2\neq 0$. It follows that $V=V_1$ if $\End_{E[G]}(V)=E$.
	The converse is just Schur's Lemma.
\end{proof}
\subsubsection{Extension of scalars}
If $E\subset F$ is an extension of fields, then $-\otimes_E F$ gives rise to a
functor $\Repf_E(G)\rightarrow \Repf_F(G)$. In this section, we analyze this
functor.

\begin{lemma}\label{lemma:homAndBaseChange}
	Let $E\subset F$ be an extension of fields and $G$ a finite group. If $V$ is
	an $E[G]$-module, then $V\otimes_E F=F[G]\otimes_{E[G]} V$ is an
	$F[G]$-module. The natural map \[V^G\otimes_E F\rightarrow (V\otimes_E
		F)^G\] is an isomorphism. Consequently, if $W$ is a second
	$E[G]$-module, the natural map \[\hom_{E[G]}(V,W)\otimes_E F\cong
	\hom_{F[G]}(V\otimes_E F, W\otimes_E F)\]
	is an isomorphism.
\end{lemma}
\begin{proof}
	Since $F$ is flat over $E$, the natural map 
	of $F$-vector spaces
	\[\hom_E(V,W)\otimes_E F\rightarrow \hom_F(V\otimes_E F, W\otimes_E
		F),\]
	is an isomorphism.
	It is easily seen to be $G$-equivariant, so
	$\IntHom(V,W)\otimes_EF\cong \IntHom(V\otimes_E F,W\otimes_E F)$ as
	$F[G]$-modules, and
	the second statement of the lemma follows from the first as
	$\Hom_{E[G]}(V,W)=\IntHom(V,W)^{G}$.

	For the first statement, consider the $E$-linear map
	$\phi:V\rightarrow \bigoplus_{g\in G} V$, $v\mapsto (gv-v)_{g\in G}$. Its kernel
	is the $E$-vector space $V^G$. Tensoring with $F$ gives the map
	\[\phi\otimes \id_F:V\otimes_E F\rightarrow (V\otimes_E F)^{|G|},\quad
		v\otimes \lambda\mapsto (g(v\otimes\lambda)-v\otimes
		\lambda)_{g\in G}\]
		Its kernel is
		$(V\otimes_E F)^G$, but it is also $V^G\otimes_E F$, as $F$ is
		flat over $E$.
\end{proof}
\begin{proposition}\label{prop:injectivityOnK0}
	Let $E\subset F$ be an arbitrary field and $G$ a finite group. The
	ring homomorphism $\phi:R_E(G)\rightarrow R_F(G)$ induced by $-\otimes_E F$
	is injective.
\end{proposition}
\begin{proof}
Recall that $R_E(G)$ (resp.~$R_F(G)$) is the free abelian group generated by
the classes of irreducible $E[G]$-modules (resp.~classes of irreducible
$F[G]$-modules). If $c\in R_E(G)$ we can write $c=c_+-c_-$ with $c_+,c_-\in
R^+_E(G)$, where $R^+_E(G)$ is the submonoid of $\Z_{\geq 0}$-linear combinations of classes of
irreducible $E[G]$-modules. We have $\phi(c)=0$ if and only
if $\phi(c_+)=\phi(c_-)$. Thus it suffices to show that the map
$R^+_E(G)\rightarrow R^+_F(G)$ is injective. 

Let $V_1,V_2$ be irreducible finitely generated $E[G]$-modules.
If $\Char(E)=0$, then by Maschke's Theorem \ref{thm:maschke} the
$F[G]$-modules $V_1\otimes_E F$ and $V_2\otimes_E
F$ are semi-simple, i.e.~direct sums of irreducible $F[G]$-modules.

If $V_1\not\cong V_2$, then the decompositions over $F$ have no common factors.
Indeed, otherwise Lemma \ref{lemma:homAndBaseChange} would imply that \[0=\dim_E\Hom_{E[G]}(V_1,V_2)=\dim_F\Hom_{F[G]}(V_1\otimes_E
F,V_2\otimes_E F)>0.\]

This shows that $R^+_E(G)\rightarrow R^+_F(G)$ is injective, if $\Char(E)=0$.

If $\Char(E)>0$, we have to generalize the argument a little
(\cite[(7.13)]{Lam}): If $V_1\not\cong V_2$ are irreducible, then
$V_1\otimes_E F$ and $V_2\otimes_E F$ are perhaps not semi-simple, but we show
that their composition series have no common factor. To this end, we show that
there is an element $e\in E[G]$ acting as the identity on $V_1$ and as $0$ on
$V_2$. Then $e\otimes 1$ acts as identity on all composition factors of $V_1\otimes_E F$ and as $0$ on all composition factors of $V_2\otimes_E F$, which would complete the proof of the proposition.

Every simple $E[G]$-module is of the shape $E[G]/M$, where $M\subset E[G]$ is a
maximal left-ideal. Indeed, if $V$ is a simple $E[G]$-module, pick $v\in
V\setminus \{0\}$. The map $\phi_v:E[G]\rightarrow V$, $x\mapsto xv$ is
surjective and its kernel is maximal, as $V$ is simple. Let $\rad(E[G])$ denote
the intersection of all maximal left-ideals. If $x\in \rad(E[G])$, and if $V$
is a simple $E[G]$-module, then for any $v\in V$, $x\in \ker(\phi_v)$, so $x$
acts trivially on every simple $E[G]$-module.  Conversely, if $x\in E[G]$ acts
trivially on every simple $E[G]$-module, then $x\in \rad(E[G])$, as for every
maximal left-ideal $M$, $x$ must map to $0$ in the simple module $E[G]/M$.

Thus $\rad(E[G])$ is the set of elements acting trivially on every simple $E[G]$-module. This shows that $\rad(E[G])$ is a two-sided ideal, so that $E[G]/\rad(E[G])$ is a ring.
As a left-$E[G]$-module, $E[G]/\rad(E[G])$ is semi-simple; in fact it is the direct sum of all simple $E[G]$-modules, with multiplicities (see Lemma \ref{lemma:maxSSQuotient}).
If, say,
\[E[G]/\rad(E[G])=V_1\oplus \ldots \oplus V_n\]
with $V_1,\ldots, V_n$ simple $E[G]$-modules (hence left-ideals of
$E[G]/\rad(E[G])$), then there are orthogonal idempotents $e_1,\ldots, e_n\in E[G]/\rad(E[G])$,
such that $1=e_1+\ldots+e_n$, and such that for every $j=1,\ldots, n$, $e_j$
acts as the identity on $V_j$, but as $0$ on all $V_i$ with $i\neq j$. If
$f_j\in E[G]$ is a lift of $e_j$, then $f_j$ acts as identity on $V_j$ and as $0$ on $V_i$, $i\neq
j$. The proof is complete.
%
\end{proof}

\begin{definition}
	Let $G$ be a finite group and $E$ a field. We say that a
	representation $V$ of $G$ on $E$ is \emph{absolutely
	irreducible}\index{absolutely irreducible} if $V\otimes_E F$ is
	irreducible for all fields $F\supset E$.
\end{definition}
If $|G|$ is invertible in $E$, then we can easily give a criterion for
absolute irreducibility:
\begin{proposition}\label{prop:absolutelyIrreducible}
	Let $G$ be a group and $E$ a field such that $|G|$ is invertible in
	$E$. The following statements for a representation $V$ of $G$ on $E$ are equivalent: 
	\begin{enumerate}
		\item\label{absIr1} $\End_{E[G]}(V)=E$
		\item\label{absIr2} $V$ is absolutely irreducible.
		\item\label{absIr3} There exists an algebraically closed field $F\supset E$
			such that $V\otimes_E F$ is irreducible.
	\end{enumerate}
\end{proposition}
\begin{proof}
	Assume \ref{absIr1} and let $F$ be an extension of $E$. Then
	$\End_{F[G]}(V\otimes_E F)=F$ by Lemma \ref{lemma:homAndBaseChange},
	so $V\otimes_E F$ is irreducible according to Proposition
	\ref{prop:irreducibilityCriterionNeu}, so $V$ is absolutely
	irreducible. 

	\ref{absIr2} trivially implies \ref{absIr3}.
	Assume \ref{absIr3} holds. Then $\dim_F \End_{F[G]}(V\otimes_E F)=1$
	according to Proposition \ref{prop:irreducibilityCriterionNeu}, so again we see
	that $\End_{E[G]}(V)=E$. This completes the proof.
\end{proof}
\begin{remark}
	An analogous criterion is true even if $|G|$ is not invertible in $E$, but
	the proof is more complicated, see e.g.~\cite[Thm.~7.5]{Lam}.
\end{remark}
\begin{definition}
	We say that $E$ is a \emph{splitting field}\index{splitting field} for
	$G$, if every irreducible representation of $G$ over $E$ is absolutely
	irreducible.
\end{definition}
\begin{proposition}\label{cor:repsAlgClosedFields}
	If $G$ is a finite group and $E$ a splitting field, then any field
	$F\supset E$ is a splitting field. In this case the map
	$R_E(G)\rightarrow R_F(G)$ is an isomorphism.
\end{proposition}
\begin{proof}
	Assume that $E$ is a splitting field. If $0\subsetneqq C_1\subsetneqq\ldots \subsetneqq
	C_n=E[G]$ is a composition series, then 
	\[0\subsetneqq C_1\otimes_E F\subsetneqq\ldots \subsetneqq
		C_n\otimes_E F=F[G]\]
		is a composition series of $F[G]$, as
		$(C_{i}\otimes_EF)/(C_{i-1}\otimes_E
		F)=(C_{i}/C_{i-1})\otimes_E F$ is (absolutely) irreducible for
		$i=1,\ldots, n$ by assumption. If $W$ is an irreducible
		$F[G]$-module, then  $W\cong
		(C_i/C_{i-1})\otimes_E F$ for some $i$ according to Proposition \ref{prop:compositionSeries}, so $W$ is absolutely
	irreducible and $F$ a splitting field. Moreover, the class of $W$ in $R_F(G)$ lies in the image of the
	map $R_E(G)\hookrightarrow R_F(G)$.
\end{proof}

Proposition \ref{prop:absolutelyIrreducible} translates to:
\begin{corollary}
	Let $E$ be a field and $G$ a group such that $|G|$ is invertible in
	$E$. If $E$ is algebraically closed, then $E$ is a splitting field.
\end{corollary}
\begin{remark}
	The same statement is true if $|G|$ is not necessarily invertible, but
	the proof is a little bit more complicated, see
	\cite[Thm.~8.3]{Lam}.
\end{remark}
\begin{proposition}\label{prop:everythingComesFromBelow}
	Let $G$ be a  finite group and $E$ a field. If $E$ is a splitting field
	(e.g.~algebraically closed), then
	every irreducible representation of $G$ on $E$ comes via base change
	from a finite extension $E_0$ of the prime field.
\end{proposition}
\begin{proof}
	Write $\F$ for the prime field of $E$.
	Let $\overline{E}$ be an algebraic closure of $E$ and $V$ an
	irreducible representation of $G$ over $E$. If $E$ is a splitting
	field, then $V\otimes\overline{E}$ is also irreducible, so it comes
	via base change from $\overline{\F}\subset \overline{E}$ according to
	Proposition \ref{cor:repsAlgClosedFields}, and hence
	from a finite extension $\F\subset E_V$ contained in $\overline{\F}$. Let $E_0$
	be the compositum of the fields $E_V$ in $\overline{\F}$. As there are
	only finitely many isomorphism classes of irreducible representations
	of $G$ (Proposition \ref{prop:compositionSeries}), $E_0$ is finitely
	generated over $\F$.
\end{proof}
\begin{remark}
	If $\Char(E)=0$, we will see
	in Corollary \ref{cor:everythingComesFromQZeta} that one can always take $E_0=\Q(\zeta_m)$ as a splitting
	field, where $m$ is the exponent of $G$ and $\zeta_m$ an $m$-th root
	of unity. Moreover, the functor $\Repf_{\Q(\zeta_m)}(G)\rightarrow
	\Repf_{\overline{E}}(G)$ is an equivalence because of Maschke's
	Theorem and Lemma \ref{lemma:homAndBaseChange}.

	If $\Char(E)=p>0$, then one can show that the field
	$\Z[\zeta_m]/\mathfrak{p}$ is a splitting field for $G$, where
	$\mathfrak{p}$ is any prime ideal of $\Z[\zeta_m]$ lying over $p$.
\end{remark}

\subsubsection{Characters}We continue to denote by $E$ a field and by $G$ a
finite group. 

\begin{definition}
	Let $G$ be a finite group. A map $\phi:G\rightarrow E$ satisfying
	the condition
	\[\phi(hgh^{-1})=\phi(g)\]
	for all $g,h\in G$, is called \emph{class function}\index{class
	function}.

	If $\rho:G\rightarrow \GL(V)$ is a representation of $G$ on a finite
	dimensional $E$-vector space, then $\chi_\rho:G\rightarrow E$,
	$\chi_\rho(g):=\chi_V(g):=\Tr(\rho(g))$ is a class function; it is called the
	\emph{character of $\rho$}\index{character}.

	If $\rho$ is irreducible, then its character is also called
	\emph{irreducible}\index{irreducible!character}.
\end{definition}

\begin{lemma}
	If $\rho:G\rightarrow \GL(V)$ and $\rho':G\rightarrow \GL(V')$ are
	isomorphic representations over $E$ of the finite
	group $G$, then $\chi_{\rho}=\chi_{\rho'}$.
\end{lemma}
\begin{proof}
	If $\gamma:V\rightarrow V'$ is a $G$-invariant $E$-isomorphism, then
	$\rho(g)=\gamma\rho'(g)\gamma^{-1}$, and
	$\Tr(\rho'(g))=\Tr(\gamma\rho'(g)\gamma^{-1})$.
\end{proof}

\begin{example}\label{ex:characters}
	\begin{enumerate}[label={(\alph*)}]
		\item Let $V_1$, $V_2$ be two $E[G]$-modules with
			characters $\chi_1, \chi_2$.
			\begin{itemize}
				\item If  $V$ is an $E[G]$-module with character $\chi$, which is an extension of $V_1$ by $V_2$, then 
                                  $\chi=\chi_1+\chi_2$.
				\item The character of $V_1\otimes V_2$ is $\chi_1\cdot
					\chi_2$.
				\item  If $\chi_1^{\vee}$ is the character of
					the dual representation $V_1^{\vee}$, then we see directly from the definition
					that $\chi^{\vee}_1(g)=\chi_1(g^{-1})$ for all $g\in G$.

					If $E$ is a subfield of $\C$, then
					$\chi_1^{\vee}(g)=\overline{\chi_1(g)}$,
					where $\overline{(-)}$ denotes complex
					conjugation. Indeed,
					the eigenvalues of each $g$ acting on
					$V_1$ are roots of unity, and if
					$\xi\in \C$ is a root of unity, then
					$\xi^{-1}=\overline{\xi}$.
       %

				\item The representation $\IntHom(V_1,V_2)$ has
					the character
					$\chi_2\cdot\chi_1^{\vee}$, so if
					$E\subset \C$, then we can write
					$\chi_2\cdot\overline{\chi_1}$.
			\end{itemize}
		\item  Let $E$ be the trivial rank $1$ representation. Its
			character is the constant map $g\mapsto 1\in E$. We denote it by
			$1_G$.
		\item Let $r_G$ denote the character of the regular
			representation $E[G]$ of $G$. Then $r_G(1)=|G|$, and
			$r_G(g)=0$ for $g\neq 1$. Indeed, if $g\neq 1$, then
			$gh\neq h$ for all $h$, so the $|G|\times |G|$-matrix
			of $g$ acting on $E[G]$ (with respect to standard
			basis)  has only zeroes on the diagonal.
		\item From the previous two examples we can compute the
			character of the augmentation representation (Example
			\ref{ex:reps}), which we denote by $u_G$:
			$u_G=r_G-1_G$, so $u_G(1)=|G|-1$ and $u_G(g)=-1$ for
			$g\neq 1$.
	\end{enumerate}
\end{example}

\begin{definition}
	Assume that $|G|$ is invertible in $E$.
	Let $G$ be a finite group and write $\mathbf{C}_{E,G}$ for the 
	$E$-vector space of all class functions
	$G\rightarrow E$.  For arbitrary maps $\phi,\psi:G\rightarrow E$ we define
		\[\left<\phi,\psi\right>_G:=\frac{1}{|G|}\sum_{g\in
		G}\phi(g)\psi(g^{-1}).\]
		This is a symmetric bilinear form on the space of maps
		$G\rightarrow E$, and also on $\mathbf{C}_{E,G}$.
	If confusion is unlikely, we drop the subscript $G$ from
	 $\left<-|-\right>_G$.
\end{definition}

\begin{proposition}[Orthogonality of characters]\label{prop:Orthogonality}
	As before, assume that $|G|$ is invertible in $E$.
	\begin{enumerate}
	\item\label{item:hom} If $V_1, V_2$ are arbitrary representations with characters $\chi_1,\chi_2$, then
		\[\left<\chi_1,\chi_2\right>=\dim_E(\Hom_{E[G]}(V_1,V_2)).\]
		\item Let $\chi_1,\chi_2$ characters of
	irreducible representations $V_1, V_2$ of $G$. Then
	\[\left<\chi_1,\chi_2\right>=\begin{cases}
			a\in \Z_{\geq 1} & \text{if }V_1\cong V_2\\
			0&\text{otherwise}.
		\end{cases}
		\]
	\item If $E$ is algebraically closed then $a=1$.
\end{enumerate}
\end{proposition}
\begin{proof}
	Taking Schur's Lemma \ref{lemma:schur} into account, it follows that we just have to prove
\ref{item:hom}. 
First let $V$ be any representation and write $V=V/V^G\oplus V^G$. The element $\frac{1}{|G|}\sum_{g\in G}g\in E[G]$
induces an $E[G]$-linear map $V\rightarrow V$ whose image is $V^G$, and which is the identity restricted to $V^G$. It follows that
\[\Tr\left(\frac{1}{|G|}\sum_{g\in G}g\right)=\dim_E V^G.\]
We apply this observation to $\IntHom(V_1,V_2)$. 
The character of $\IntHom(V_1,V_2)$ is $\chi_{V_1}^\vee\cdot\chi^{~}_{V_2}$, so 
\begin{align*}
	\dim_E \Hom_{E[G]}(V_1,V_2)&=\dim_E \left(\IntHom(V_1,V_2)\right)^G\\
	&=\frac{1}{|G|}\sum_{g\in G}\chi_{V_1}(g^{-1})\chi_{V_2}(g)\\
	&=\left<\chi_{V_1},\chi_{V_2}\right>.
\end{align*}
\end{proof}

\begin{corollary}\label{cor-irred-rep-in-rep}
	Assume that $|G|$ is invertible in $E$. The isomorphism class of a representation over $E$ is determined by its
	character.

	More precisely, if $V$ is an $E$-representation of $G$, then 
	\[V\cong\bigoplus_{W\text{ irreducible}}W^{\frac{\left<\chi_{V},\chi_{W}\right>}{\left<\chi_{W},\chi_{W}\right>}}\]
	where the direct sum is over the distinct irreducible representations of
	$G$ on $E$.
\end{corollary}
Proposition \ref{prop:absolutelyIrreducible} translates to:
\begin{corollary}\label{cor:irreducibilityCriterion}
	As before, assume that $|G|$ is invertible in $E$.
	If $V$ is a representation of $G$ over $E$, then $V$ is absolutely
	irreducible if and only if 
	$\left<\chi_V,\chi_V\right>=1$.       %
       %

\end{corollary}
%
%

\begin{definition}
	We say that an irreducible representation $W$ is \emph{contained in
	$V$ with multiplicity $m$}, if
	$\frac{\left<\chi_V,\chi_W\right>}{\left<\chi_{W},\chi_W\right>}=m$,
	i.e.~if $m$ is maximal such that $W^{\oplus m}\subset V$.
\end{definition}

\begin{corollary}\label{cor:finitelyManyIrreducibleReps}
	If $G$ is a finite group and $|G|$ invertible in $E$, then  an irreducible representation $W$ of $G$
	with $\dim_E W = d$
	is contained in the regular representation $E[G]$ with multiplicity
	$d/\left<\chi_W,\chi_W\right>$.

	If $W_1,\ldots, W_r$ are the
	distinct irreducible representations of $G$, then
	\[|G|=\sum_{i=1}^r\frac{(\dim
	W_i)^2}{\left<\chi_{W_i},\chi_{W_i}\right>}.\]
\end{corollary}
\begin{proof}
	If $W$ is an irreducible representation of $G$, then by Example
	\ref{ex:characters} we can compute
	\[\left<r_G,\chi_W\right>=\frac{1}{|G|}\sum_{g\in
	G}r_G(g){\chi_W\left(g^{-1}\right)}=\chi_W(1)=\dim W.\]
	It follows that
	\[E[G]=\sum_{i=1}^r W_i^{\frac{\dim W_i}{\left<\chi_{W_i},\chi_{W_i}\right>}},\]
	so
	\[|G|=\dim_{E}E[G]=\sum_{i=1}^r \frac{(\dim
	W_i)^2}{\left<\chi_{W_i},\chi_{W_i}\right>}\]
	as claimed. 
\end{proof}
\begin{remark}\label{rem:GrothendieckVsCharacters}
	\begin{enumerate}[label={(\alph*)}]
		\item If $G$ is a finite group and $|G|$ invertible in $E$, then we will
	identify the group $R_E(G)$ (Definition \ref{def:GrothendieckRing})
	with the free abelian group on the finite set of
	irreducible characters of $G$ on $E$.
		\item If $E\subset F$ is an extension of fields and $\chi$ a
	character of a representation $V$ of $G$ over $E$, then the character
	of $V\otimes_E F$ is the composition
	$G\xrightarrow{\chi}E\hookrightarrow F$, and abusing notation we will also
	write $\chi$ for the character of $V\otimes_E F$.
	\end{enumerate}
\end{remark}
\begin{proposition}\label{prop:comesFromBelow}
	We continue to assume that $|G|$ is invertible in $E$.
	Let $E\subset F$ be an extension of fields. A representation of $G$
	over $F$ comes from a representation of $G$ over $E$ if and only if its class lies in
$R_E(G)\subset R_F(G)$ (Proposition \ref{prop:injectivityOnK0}).\end{proposition}
\begin{proof}
	Clearly, if a representation comes from $E$, then its class lies in
	$R_E(G)$.
	Conversely, 
	let $\chi$ be the character of a representation of $G$ over $F$, and
	assume that $\chi\in R_E(G)$. Let
	let $\psi_1,\ldots, \psi_r$ be the irreducible characters of $G$ over
	$E$. By assumption we can write
	$\chi=\sum_{i=1}^ra_i\psi_i$, with $a_i\in \Z$. Since $\chi$ and
	$\psi_i$ are characters of representations over $F$ (Remark
	\ref{rem:GrothendieckVsCharacters}), we have
	$\left<\psi_i,\psi_i\right>>0$, and 
	\[a_i\left<\psi_i,\psi_i\right>=\left<\chi,\psi_i\right>\geq 0,\]
	so $a_i\geq 0$.
\end{proof}


With the preceding corollary in mind, we make a detour to complex
representations.
\subsubsection{Complex representations}
In this section let $E=\C$.
\begin{definition}
	Recall that $\mathbf{C}_{\C,G}$ denotes the space of class functions
	$G\rightarrow \C$. For $\phi,\psi\in \mathbf{C}_{\C,G}$ we define \[(\phi|\psi)_G=\frac{1}{|G|}\sum_{g\in G}\phi(g)\overline{\psi(g)}.\]
			If no confusion is possible we write $(-,-)=(-,-)_G$. This is a scalar product on $\mathbf{C}_{E,G}$:
		$(\phi|\psi)=\overline{(\psi|\phi)}$, $(\phi|\phi)>0$ for $\phi\neq
		0$, and $(\phi|\psi)$ is linear (resp.~semi-linear) in $\phi$
		(resp.~$\psi$).
\end{definition}

\begin{lemma}
	If $\chi$ is the character of a complex representation  $\rho$ of $G$, and $\phi\in
	\mathbf{C}_{\C,G}$, then
	\[(\phi|\chi)=\left<\phi,\chi\right>.\]
\end{lemma}
\begin{proof}
	This is clear as $\chi(g)$ is a sum of roots of unity, and if $\xi\in
	\C$ is a root of unity then $\xi^{-1}=\overline{\xi}$.
\end{proof}

We have seen that the irreducible characters of $G$ form an orthonormal system
in $\mathbf{C}_{\C,G}$ with respect to $\left<-,-\right>_G$. In fact, they also span $\mathbf{C}_{\C,G}$.
\begin{theorem}\label{thm:orthonormalbasis}
	Let $G$ be a finite group. Its irreducible characters $\chi_1,\ldots,
	\chi_r$ form an orthonormal basis of the space of class functions
	$\mathbf{C}_{\C,G}$.
\end{theorem}
\begin{proof}
	Let $\phi\in \mathbf{C}_{\C,G}$ be a class function  such that
	$\left<\phi,\chi\right>=0$ for all irreducible characters $\chi$ of
	$G$. We want to show that $\phi=0$.

	If $\rho:G\rightarrow \GL(V)$ is any representation of $G$,
	define 
	$\rho_{\phi}:=\sum_{g\in G}\phi(g)\rho(g^{-1})$. Then $\rho_{\phi}$ is an
	endomorphism of $V$, and it is easily checked that $\rho_{\phi}$ is
	actually an endomorphism of the representation $\rho$. If $\rho$ is
	irreducible, then Schur's Lemma \ref{lemma:schur} shows 
	that $\rho_\phi=\lambda\cdot \id_V$ for some $\lambda\in \C$. We have
	\[\lambda\dim
		V=\Tr(\rho_{\phi})=
		|G|\left<\phi,\chi_\rho\right>=0.\]
		It follows that $\rho_\phi=0$ even if $\rho$ is not
		irreducible, since $\rho$ is the direct sum of irreducible subrepresentations.
		Now let $\rho$ be the regular
	representation of $G$, i.e.~$G$ acts on the $|G|$-dimensional
	$E$-vector space with basis $\{e_g|g\in G\}$ via
	$\rho(g)(e_h)=e_{gh}$. Then 
	\[0=\rho_{\phi}(e_1)=\sum_{g\in G}\phi(g)\rho(g^{-1})(e_1)=\sum_{g\in
	G}\phi(g)e_{g^{-1}}.\]
	It follows that $\phi=0$, which is what we wanted to show.
\end{proof}
\begin{corollary}
	If $G$ is a finite group, then the number of nonisomorphic irreducible
	complex representations is equal to the number of conjugacy classes in
	$G$. This is also true over any splitting field of
	characteristic $0$.
\end{corollary}
\begin{proof}
	It is clear that the dimension of $\mathbf{C}_{\C,G}$ is the number of
	conjugacy classes in $G$. By the theorem, this dimension is precisely
	the number of irreducible characters. If $E$ is a splitting 
	field of characteristic $0$, then there exists a splitting 
	field $E'\subset E$ which is finitely generated over $\Q$, such that
	every representation of $G$ over $E$ comes from $E'$. The field $E'$
	can be embedded in $\C$,  so the claim follows
	from Proposition \ref{cor:repsAlgClosedFields}.
\end{proof}

The following corollary will be crucial for Section \ref{sec:complexSwan}.
\begin{corollary}\label{cor:classFunctionCharacter}
	A class function $\phi$ is the character of a complex representation of $G$ if
	and only if it is a linear combination
	\[\phi=a_1\chi_1+\ldots+a_r\chi_r\]
	with $a_i\in \Z_{\geq 0}$ and $\chi_i$ characters of irreducible
	representations of $G$.
\end{corollary}

\subsubsection{Brauer's Theorem}
We will need one more tool in our bag before we can return to extensions of
local fields. We continue to work over the field of complex numbers.
\begin{definition}\label{defn:inducedClassFunction}
	Let $\alpha: H\rightarrow G$ be a morphism of groups. If $\phi\in
	\mathbf{C}_{\C,G}$, then the composition $\phi\circ \alpha$ is a class
	function on $H$, which we call \emph{the restriction of
	$\phi$}\index{restriction!class function}, and denote by $\alpha^*\phi$.
	
	Clearly, if $\phi$ is the character of a representation $\rho$ of $G$,
	i.e.~if $\phi\in R_{\C}^+(G)$, then $\alpha^*\phi$ is the character of
	the representation $\RES_{\phi}\rho$.

	If $\varphi\in \mathbf{C}_{\C,H}$, there is an \emph{induced class function} on $G$ \index{induced!class
	function} denoted by
	$\alpha_*\phi$, which can be computed using the following two special
	cases:
	\begin{enumerate}[label={(\alph*)}]
		\item\label{item:induction1} If $\alpha$ is injective, then
			\[\alpha_*\phi(g):=\frac{1}{|H|}\sum_{\begin{subarray}{c}
					x\in G\\
					xgx^{-1}\in H\end{subarray}}
					\phi(xgx^{-1}).\]
				\item\label{item:induction2} If $\alpha$ is surjective, then 
			\[\alpha_*\phi(g):=\frac{1}{|\ker(\alpha)|}\sum_{h\mapsto
			g}\phi(h).\]
		\end{enumerate}
	In general we define $\alpha_*\phi$ by factoring $\alpha$ into a
	surjective map followed by an injective map.
\end{definition}
\begin{lemma}\label{lem:ind-char}If $\alpha:H\rightarrow G$ is a morphism of finite groups and
	$\rho:H\rightarrow \GL(V)$ a complex representation of $H$ with character $\chi_\rho$, then $\alpha_*\chi_\rho$ is
	the character of $\Ind_{\alpha}\rho$.
\end{lemma}
\begin{proof}
	Since $\alpha$ is the composition of an injection and a surjection, it
	suffices to check that the character of $\Ind_{\alpha}\rho$ satisfies the two formulas
	\ref{item:induction1} and \ref{item:induction2} if $\alpha$ is either
	injective, or surjective.

	In the first case this is a simple calculation, which we leave as an
	exercise. In the second case, $\alpha$ induces an isomomorphism
	$H/\ker(\alpha)\rightarrow G$. Since $\ker(\alpha)$ is a normal subgroup of $H$, the
	subspace $V':=\{v-hv \in V| h\in \ker(\alpha), v\in V\}$ is a subrepresentation,
	and $\ker(\alpha)$ acts trivially on $V/V'$. It is not difficult to
	show that the map $\Ind_\alpha
	\rho=E[H/\ker(\alpha)]\otimes_{E[H]}V\rightarrow V/V'$, defined by
	$h\ker(\alpha)\otimes
	v\mapsto hv+V'$ is an isomorphism of $H$ (or
	$H/\ker(\alpha)$)-representations.
	Without loss of generality we may assume that $\rho$ is irreducible.
	Hence $V'=V$ or $V'=0$. If $V'=0$, then $\ker(\alpha)$ acts trivially
	on $V$, so $\chi_{\rho}(hh')=\chi_{\rho}(h)$ if $h'\in \ker(\alpha)$.
	Hence the formula. 

	Otherwise, $V'=V$, $\Ind_{\alpha}\rho=0$, and we check that
	$\alpha_*\chi_\rho(g)=0$. Indeed, if $h_0\in \ker(\alpha)$
	then for all $v\in V'$ we see
	\[\sum_{h\in \ker(\alpha)}h(v-h_0v)=0.\]
	As $V=V'$,
	$\sum_{h\in
	\ker(\alpha)}hv=0$ for all $v\in V$. This means that
	\[\sum_{h\mapsto	g}\Tr_V(h)=\Tr_V\left(\tilde{g}\sum_{h\in\ker(\alpha)}h\right)=0\]
	for every $g\in H/\ker(\alpha)$ with lift $\tilde{g}\in H$. Hence $\alpha_*\chi_{\rho}=0$.
\end{proof}

The adjunction relation between induction and restriction of representations
induces a relation for induction and restriction of class functions:
\begin{proposition}[Frobenius reciprocity]\index{Frobenius
	receprocity}\label{prop:FrobeniusReciprocity}
	If $\alpha:H\rightarrow G$ is a homomorphism of groups, $\psi$ a class
	function on $H$, $\phi$ a class function on $G$, then
	\[(\psi|\alpha^*\phi)=(\alpha_*\psi|\phi).\]
\end{proposition}
\begin{proof}
	By Theorem \ref{thm:orthonormalbasis}, we know that $\psi$ is a $\C$-linear
	combination of characters of $H$, and similarly for $\phi$. We may
	thus assume that $\psi=\chi_W$, $\phi=\chi_V$. In this case we have
	\begin{align*}(\chi_W|\alpha^*\chi_V)&=\dim \Hom_{E[H]}(W,\RES_{\alpha}
	V)\\&=\dim \Hom_{E[G]}(\Ind_{\alpha} W,
	V)\\&=(\alpha_*\chi_W|\chi_V).\end{align*}
\end{proof}

In practice, a question about a general representation can often be reduced to
simpler representations on a different group.
\begin{theorem}[Brauer]\label{thm:Brauer}\index{Brauer's Theorem}
	If $G$ is a finite group and $\chi$ a character of a finite
	dimensional complex representation of $G$, then $\chi$ is a
	$\Z$-linear combination of characters of the form
	$\alpha_{i,*}\chi_i$, where $\alpha_i:H_i\hookrightarrow G$ is the
	inclusion of a subgroup, and $\chi_i$ is the character of a
	$1$-dimensional representation of $H$.
\end{theorem}

Theorem \ref{thm:Brauer} follows from a slight variant, also due to Brauer.
Recall that if $p$ is a prime number, a group $G$ is called
\emph{$p$-elementary}\index{$p$-elementary group} if $G$ is the direct product
of a cyclic group of order prime to $p$ with a $p$-group.
\begin{theorem}\label{thm:Brauer2}
	Let $G$ be a finite group and $V_p\subset R(G)$ the subgroup generated
	by those characters, which are induced by $p$-elementary subgroups of
	$G$. The index of $V_p$ in $R(G)$ is finite and prime-to-$p$.
\end{theorem}

We do not give a proof of Theorem \ref{thm:Brauer2} (see \cite[Ch.~10, Thm.~18]{Serre/Representations}),
but we show how to deduce Theorem \ref{thm:Brauer} from this.
\begin{proof}[Proof of Theorem \ref{thm:Brauer}]
	If $G$ is abelian then we know from Proposition
	\ref{prop:RepresentationsOfAbelianGroups} that every irreducible representation
	has dimension $1$, so we assume that $G$ is not abelian. 

	By Theorem \ref{thm:Brauer2}, $R(G)=\sum_{p\text{ prime}}V_p$, and
	thus every character of $G$ is a $\Z$-linear combination  of characters
	induced by a character of a $p$-elementary subgroup, for some $p$. Hence
	we may assume that $G$ is a $p$-elementary group. Let $\rho:G\rightarrow
	\GL(V)$ be an irreducible representation and $\chi$ its character. By
	induction on $|G|$ we may assume that $\rho$ is injective. Write
	$G=C\times P$, where $C$ is cyclic of order prime to $p$, and $P$ a
	$p$-group. Clearly $C$ is contained in the center of $G$, and by
	assumption $G/Z(G)$ is a nonzero $p$-group. Thus $Z(G/Z(G))\neq 0$,
	and $G/Z(G)$ contains a cyclic normal $p$-group; let $H\subset G$ be
	its preimage in $G$. Then $H$ is abelian and is not contained in
	$Z(G)$. Since $\rho$ is injective by assumption, $\rho(H)$ is abelian
	and not contained in the center of $\rho(G)$, i.e.~there exists $h\in
	H$, such that $\rho(h)$ is not given by multiplication with a scalar.
	This means that the irreducible decomposition of $\rho|_H:H\rightarrow
	\GL(V)$ contains nonisomorphic representations, since $H$ is abelian.
	Let $V=V_1\oplus\ldots \oplus V_n$ be the isotypical decomposition of $V$ as
	an $E[H]$-module, i.e.~the irreducible representations contained in
	$V_i$ are all isomorphic, and $V_i,V_j$ contain no common isomorphic
	irreducible subrepresentations, unless $i=j$. The elements of $G$
	permute the $V_i$ transitively, as $V$ is an irreducible
	$E[G]$-module.

	Define $H':=\{g\in G| gV_1=V_1\}$. This is a subgroup of $G$,
	containing $H$, and as $V$ is irreducible $H'\neq G$. We claim that
	$\rho$ is induced by a representation of $H'$. Indeed, $H'$ acts
	on $V_1$, and $E[G]\otimes_{E[H']}V_1\xrightarrow{\simeq} V$, as
	$V=\sum_{g\in G}gV_1$. 

	We have seen that $V$ is induced by a representation of $H'\subsetneqq
	G$, so the claim follows by induction.
\end{proof}
\begin{corollary}
	Let $G$ be a finite group of exponent $m$ and $E$ an algebraically
	closed field of characteristic $0$. Let $\zeta_m\in E$ be a primitive
	$m$-th root of unity. The injection $R_{\Q(\zeta)}(G)\rightarrow
	R_{E}(G)$ is an isomorphism. In other words, $\Q(\zeta)$ is a
	splitting field for $G$.
\end{corollary}
\begin{proof}
	Clearly $E$ contains an algebraic closure of $\Q$, which we can embed
	into $\C$. Proposition \ref{cor:repsAlgClosedFields} hence reduces the
	proof to
	the case $E=\C$. Let $\chi$ be the character of a complex
	representation of $G$. By Brauer's theorem, $\chi$ can be written
	\[\chi=\sum_{i=1}^r a_i \alpha_{i,*}\phi_i\]
	with $a_i\in \Z$, $\alpha_i:H_i\hookrightarrow G$ a subgroup and
	$\phi_i$ a $1$-dimensional character of $H_i$. In particular:
	$\phi_i:H_i\rightarrow \C^\times$ is a homomorphism, so $\phi_i$ has
	image contained in $\Q(\zeta_m)$ (even in its ring of
	integers), so $\phi_i$ comes from $\Q(\zeta_m)$, and
	hence so does $\alpha_{i,*}\phi_i$. 
\end{proof}
\begin{corollary}\label{cor:everythingComesFromQZeta}With the notation of the
	previous corollary, scalar extension induces an equivalence
	$\Repf_{\Q(\zeta)}G\rightarrow \Repf_{E}G$.
\end{corollary}
\begin{proof}
	This is just Proposition \ref{prop:comesFromBelow} together with Lemma
	\ref{lemma:homAndBaseChange}.
\end{proof}
	



\subsection{Existence of the complex Artin and Swan
representations}\label{sec:complexSwan}
We now return to the study of a finite Galois extension $L/K$ of complete discretely
valued fields $L/K$ with separable residue extension of degree $f$. We fix $G:=\Gal(L/K)$. Also
recall the notations $v,v_L$ for the valuations of $K$ and $L$, and $A,
B$ for the discrete valuation rings of $K$, $L$.

Recall that we defined the function $i_G:G\rightarrow \Z_{\geq 0}\cup
\{\infty\}$ as $i_G(g)=v_L(gx-x)$. Here $x\in B$ is an element generating $B$ as an $A$-algebra.
Such an element exists by Theorem \ref{thm:generator}.
We also saw that $g\in G_n$ if and only if $i_G(g)\geq n+1$, which implies
that $i_G(hgh^{-1})=i_G(g)$, for any $h\in G$, as $G_n$ is normal. Thus $i_G$ is a class
function.

\begin{theorem}[Artin]\label{thm:ExistenceOfArtinComplex}
	In the situation above, the function $a_G:G\rightarrow \Z$,
	\[a_G(g):=\begin{cases}
			-fi_G(g) & \text{if }g\neq 1\\
			f\sum_{g'\neq 1}i_G(g')&\text{if }g=1
		\end{cases}
		\]
	is the character of a representation of $G$ over the complex numbers.
\end{theorem}
The proof of this theorem is the object of this section, and it justifies the
following definition.
\begin{definition}\label{def:swan-char}
	The function $a_G$ attached to the extension $L/K$ is called the \emph{Artin character of $G$} (or \emph{of $L/K$})\index{Artin
	character}. Once we know that $a_G$ is the character of a
	representation, we will see in Corollary
	\ref{cor:artinCharImpliesSwanChar} that the function $\sw_G:=a_G - (r_G - r_{G/G_0})$
	is also a character of $G$. It is called the 
	\emph{Swan character of $L/K$}\index{Swan character}. Here we write $r_{G/G_0}$
	for the
	composition $G\rightarrow G/G_0\xrightarrow{r_{G/G_0}} \C$. If $L/K$ is
	totally ramified, i.e.~if $f=1$, then $\sw_G=a_G-u_G$, where
	$u_G=r_G-1_G$ is the augmentation character of $G$ (see Example
	\ref{ex:characters}). 
\end{definition}

Since $i_G$ is a class function, so is $a_G$, and by Corollary
\ref{cor:classFunctionCharacter}, Theorem \ref{thm:ExistenceOfArtinComplex} is
equivalent to the following theorem.

\begin{theorem}\label{thm:ArtinIntegers}
	For any character $\chi$ of a complex representation of $G$, we have
	$f(\chi):=(a_G|\chi)\in \Z_{\geq 0}$.
\end{theorem}
Note that $(a_G|\chi)=(\chi|a_G)$, as $a_G$ takes values in $\Z$.

The proof of the theorem consists of two steps. First, we show that
$(a_G|\chi)$ is always a non-negative \emph{rational} number. Second, to show
that $(a_G|\chi)$ is an integer, we use Brauer's theorem \ref{thm:Brauer} to
reduce to the case where $\chi$ is induced by a degree $1$ character of a
subgroup. If $\chi$ is itself a character of a representation $\rho$
of dimension $1$, then $G/\ker(\rho)$ is an abelian group corresponding to a
Galois subextension $K'/K$ of $L/K$. In this case, we can relate the number
$(a_G|\chi)$ to the jumps in the ramification filtration of $G/\ker(\rho)$,
which we know to be integers by the Theorem of Hasse-Arf \ref{thm:HasseArf}.

Let us first see that without loss of generality we may assume that $L/K$ is
totally ramified, i.e.~that the residue extension is trivial.
\begin{lemma}\label{lemma:totallyRamifiedArtin}
	Write $\alpha:G_0=G^0\hookrightarrow G$ for the inclusion of the inertia group
	into $G$. Then $\alpha_*a_{G_0}=a_G$ and $\sw_G=\alpha_* \sw_{G_0}$.
\end{lemma}
\begin{proof}
	Recall that for $g\in G$, we have
	\[\alpha_*a_{G_0}(g)=\frac{1}{|G_0|}\sum_{\begin{subarray}{c}x\in
			G\\xgx^{-1}\in G_0\end{subarray}}a_{G_0}(xgx^{-1}),\]
			according to Definition
			\ref{defn:inducedClassFunction}.
			As $G_0$ is normal in $G$, $xgx^{-1}\in G_0$ if and
			only if $g\in G_0$. It follows that for $g\in
			G\setminus G_0$,
			\[\alpha_* a_{G_0}(g)=0=-fi_G(g)=a_G(g).\]
			Note that if $g\in G_0$, then 
			$i_{G_0}(xgx^{-1})=i_G(xgx^{-1})=i_G(g)$. Thus, if $g\neq 1$,
			we get
			\[\alpha_{*}a_{G_0}(g)=-\frac{|G|}{|G_0|}i_G(g)=a_G(g).\]
			Finally, if $g=1$,
			\[\alpha_*a_{G_0}(1)=\frac{|G|}{|G_0|}\sum_{g\in
				G_0\setminus \{1\}}i_{G_0}(g)=f\sum_{g\in
					G\setminus \{1\}}i_G(g)=a_G(1),\]
			as $i_G(g)=0$ for $g\not\in G_0$. This proves the
			statement about $a_G$.

			For $\sw_G$, we just have to note that
			$\alpha_*r_{G_0}=r_G$ and
			$\alpha_*1_{G_0}=r_{G/G_0}$.
			\end{proof}

Hence, for the rest of the section, we assume that $L/K$ is totally ramified, i.e.~that
$G=G_0$ and $f=1$.

\begin{lemma}\label{lemma:ArtinRational}
	For every complex character $\chi$ of $G$, we have $f(\chi)=(a_G|\chi)\in
	\Q_{\geq 0}$. More precisely, if $\chi$ is the character of a
	representation $V$, then
	\begin{equation}\label{eq:codimFormula}f(\chi)=\sum_{i\geq
		0}\frac{1}{|G:G_i|}\left(\dim V -\dim
		V^{G_i}\right).\end{equation}
\end{lemma}
\begin{proof}
	First recall that $u_{G_i}=r_{G_i}-1_{G_i}$ is the augmentation
	character of the $i$-th ramification group in the lower numbering.
	If $\alpha_i:G_i\hookrightarrow G=G_0$ is the inclusion map, then
	$\alpha_{i,*}u_{G_i}(g)=0$ if $g\not\in G_i$, and if $g\in
	G_{i}\setminus\{1\}$ we get
	\[\alpha_{i,*}u_{G_i}(g)=-\frac{|G|}{|G_i|}=-|G:G_i|.\]
	Hence, if $g\in G_r\setminus G_{r+1}$, then
	\[a_G(g)=-(r+1)=\sum_{i=0}^{\infty}\frac{1}{|G:G_i|}\alpha_{i,*}u_{G_i}(g).\]
	For $1\in G$, we compute
	\[\sum_{i=0}^{\infty}\frac{\alpha_{i,*}u_{G_i}(1)}{|G:G_i|}=\sum_{i=0}^{\infty}(|G_i|-1)=\sum_{g\neq
	1}i_G(g)=a_G(1),\]
	so we have proved that
	\[a_G=\sum_{i=0}^\infty\frac{1}{|G:G_i|}\alpha_{i,*}u_{G_i}.\]

	If $\phi$ is a class function on $G$, then we use
	Frobenius reciprocity (Proposition \ref{prop:FrobeniusReciprocity}) to
	compute
	\begin{align*}
		f(\phi)&=(a_G|\phi)\\
		&=\sum_{i=0}^{\infty}\frac{1}{|G:G_i|}(\alpha_{i,*}u_{G_i}|\phi)\\
		&=\sum_{i=0}^{\infty}\frac{1}{|G:G_i|}(u_{G_i}|\alpha_i^*\phi)
	\end{align*}
	Recall that $u_{G_i}=r_{G_i}-1_{G_i}$, and
	$(r_{G_i}|\alpha_i^*\phi)=\overline{\phi(1)}$. 
	We obtain
	\begin{equation*}\label{eq:codimFormulaForClassFunction}
		f(\phi):=\sum_{i=0}^{\infty}\frac{1}{|G:G_i|}\left(\overline{\phi(1)}-(1_{G_i}|\alpha^*\phi)\right)
	\end{equation*}
	Finally, if $\phi=\chi_V$ for $V$ a $E[G]$-module, then $\overline{\phi(1)}=\phi(1)=\dim
	V$ and \[(1_{G_i}|\alpha_i^*\phi)=\dim_{E[G]}(\mathbf{1}_{G_i},V)=\dim
	V^{G_i},\] so we obtain \eqref{eq:codimFormula}.
\end{proof}
\begin{corollary}\label{cor:artinCharImpliesSwanChar}
	If $a_G$ is the character of a complex representation of $G$, then so is $\sw_G$.
\end{corollary}
\begin{proof}
	We again assume that $L/K$ is totally ramified, see Lemma \ref{lemma:totallyRamifiedArtin}. If $\chi$ is the
	character of the representation $V$, then 
	$\left<r_G-1_G,\chi\right>=\dim V/V^G$. As $G=G_0$, the preceeding lemma shows
	that
	\[\left<\sw_G,\chi\right>=f(\chi)-\dim V/V^G=\sum_{i\geq 1}\frac{1}{|G:G_i|}\dim
		V/V^{G_i}\geq 0.\]
	Thus, if $a_G$ is a character of $G$, then $f(\chi)$ is an integer and
	hence $\left<\sw_G,\chi\right>\in
	\Z_{\geq 0}$, which implies that  $\sw_G$ is a character of $G$.
\end{proof}

Now we concentrate on characters of $1$-dimensional representations.
\begin{lemma}\label{lemma:1dimensionalCharacter}
	Let $\chi$ be the character of a $1$-dimensional complex
	representation $V$ of $G$. Then $\chi(G)\subset \C^\times$, and
	$\chi:G\rightarrow \C^\times$ is a homomorphism of groups. Let
	$H:=\ker(\chi)$ and $K'/K$ the associated subextension of $L/K$.
	Write $c'_{\chi}$ for the largest integer for which
	$(G/H)_{c'_{\chi}}\neq \{1\}$. Then
	\[f(\chi)=\phi_{K'/K}(c'_{\chi})+1\in \Z,\]
	where $\phi_{K'/K}$ is the function from Defintion
	\ref{defn:HerbrandFunction}.
\end{lemma}
\begin{proof}
	Write $c_{\chi}:=\psi_{L/K'}(c'_{\chi})$.
	From Corollary \ref{cor:quotientCalculation} we know that
	$(G/H)_{c'_{\chi}}$ is the image of $G_{c_\chi}$ in $G/H$. Then $c_{\chi}$ is
	the largest integer, such that $G_{c_\chi}$ does not act trivially on
	$V$. Since $V$ is $1$-dimensional, this means that
	\[V^{G_i}=\begin{cases}V & \text{ if }i> c_{\chi}\\0 &\text{ else,
		}\end{cases}\]
	so by \eqref{eq:codimFormula}, we obtain
	\[f(\chi)=\sum_{i=0}^{c_{\chi}}\frac{1}{|G:G_i|}=\phi_{L/K}(c_{\chi})+1,\]
	according to \eqref{eq:formulaForPhi} below Definition
	\ref{defn:HerbrandFunction}.
	Finally, we have seen in Proposition \ref{prop:transitivityOfPhi} that
	\[\phi_{L/K}(c_{\chi})=\phi_{K'/K}\phi_{L/K'}(\psi_{L/K'}(c'_{\chi}))=\phi_{K'/K}(c'_{\chi}),\]
	so we get $f(\chi)=\phi_{K'/K}(c'_{\chi})+1$ as claimed. Moreover, as
	$G/H\subset \C^\times$, the extension $K'/K$ is abelian, so
	$\phi_{K'/K}(c'_{\chi})\in \Z$ according to the Hasse-Arf Theorem
	\ref{thm:HasseArf}.
\end{proof}

Now we can complete the proof of Theorem \ref{thm:ArtinIntegers} and hence of
Theorem \ref{thm:ExistenceOfArtinComplex}.
\begin{proof}[Proof of Theorem \ref{thm:ArtinIntegers}]
	Let $\chi$ be the character of a $\C[G]$-module $V$.  By Lemma
	\ref{lemma:ArtinRational} we know that $f(\chi)\in \Q_{\geq 0}$. It
	thus remains to show that $f(\chi)$ is an integer. Brauer's Theorem
	\ref{thm:Brauer} tells us that $\chi$ can be written as 
	\[\chi=\sum_{i=1}^ra_i \chi'_i\]
	with $a_i\in \Z$ and $\chi'_i=\gamma_{i,*}\chi_i$, where
	$\gamma_i:H_i\hookrightarrow G$ is a subgroup and $\chi_i$ the
	character of a \emph{$1$-dimensional} representation of $H_i$.
	Thus it remains to show that $f(\chi'_i)\in \Z$.

	Note that $f(\chi'_i)=(a_G|\chi'_i)=(\gamma_i^*a_G|\chi_i)$ by
	Frobenius reciprocity, so we
	have to understand the class function $\gamma_i^*a_G$.
	\begin{lemma}
		If $\gamma:H\hookrightarrow G$ is a subgroup, then there
		exists a non-negative integer $\lambda$, such that
		\[\gamma^*a_G=\lambda r_H + a_H.\]
	\end{lemma}
	\begin{proof}
		First let $h\in H\setminus \{1\}$. Then $r_H(h)=0$, and
		$i_G(h)=i_H(h)$, so $\gamma^*a_G(h)=\lambda r_H(h)+a_H(h)$ for
		any $\lambda$.

		The interesting case is $h=1$.  Recall that $A,B$ denote the
		valuation rings of $K,L$. Let $x$ be a generator of
		$B$ as an $A$-algebra. Then Theorem \ref{thm:ableitung}
		shows that $\mathfrak{D}_{B/A}=(f'(x))$, if $f$ is the monic
		minimal polynomial of $x$. Note that $f'(x)=\prod_{g\neq
		1}(x-gx)$, so $v_L(\mathfrak{D}_{B/A})=a_G(1)$, see
		Proposition \ref{valuation-of-discriminant}. Similary, if
		$K'=L^H$ and if $B'$ denotes the valuation ring of $K'$, then 
		$a_H(1)=v_L(\mathfrak{D}_{B/B'})$. Finally, Proposition
		\ref{prop:transitivityOfDifferent} shows that
		\[\mathfrak{D}_{B/A}=\mathfrak{D}_{B/B'}\mathfrak{D}_{B'/A},\]
		and applying $v_L$ gives
		\[a_G(1)=a_H(1)+v_L(\mathfrak{D}_{B'/A})=a_H(1)+|H|v_{K'}(\mathfrak{D}_{B'/A}).\]
		Since $r_H(1)=|H|$, the proof is complete, taking
		$\lambda=v_{K'}(\mathfrak{D}_{B'/A})$.
	       %

	       %
	       %
	       %
		%
	       %
	\end{proof}
	Returning to the proof of Theorem \ref{thm:ArtinIntegers}, we use the
	lemma to compute
	\begin{align*}
		f(\chi'_i)&=(\gamma_i^*a_G|\chi_i)\\
		&=\lambda(r_{H_i}|\chi_i)+(a_{H_i}|\chi_i)\\
		&=\lambda\chi_i(1)+(a_{H_i}|\chi_i)
	\end{align*}
	Since $\chi_i$ is the character of a $1$-dimensional representation of
	$H$, we have $\chi_i(1)=1$, and we know that $(a_H|\chi_i)\in \Z$, by
	Lemma \ref{lemma:1dimensionalCharacter}.
\end{proof}

\begin{remark}
	Note that the existence of the Artin representation implies the
	Hasse-Arf theorem: Let $L/K$ be an abelian Galois extension with group
	$G$. As before, it suffices to prove Proposition
	\ref{prop:simpleHasseArf}. We may hence assume that $G$ is cyclic of
	order $n$. Fixing a primitive $n$-th root of unit in $E$, we get an
	injective homomorphism $\chi:G\rightarrow E^\times$. Let $\mu$ be the
	largest integer such that $G_{\mu}\neq 0$. The existence of the Artin
	representation shows that $f(\chi)=(a_G|\chi)\in \Z$. But as in Lemma
	\ref{lemma:1dimensionalCharacter} we see that
	$f(\chi)=\phi_{L/K}(\mu)+1$, so $\phi_{L/K}(\mu)\in \Z$, which  implies 
	Proposition \ref{prop:simpleHasseArf}.
\end{remark}
\subsection{Integrality over $\Z_{\ell}$}\label{sec:rationality}
As before, let $L/K$ be a Galois extension of complete discretely valued fields with group
$G$, and assume that the extension of the residue fields is separable of
degree $f$.
We have seen in the previous section that the Artin function $a_G$ and the
Swan function $\sw_G:=a_G-(r_G-r_{G/G_0})$ are 
characters of  complex representations. For this reason, we call these
$\Z$-valued functions the \emph{Artin character} and the \emph{Swan
character}. By Proposition
\ref{prop:everythingComesFromBelow}, we know that the attached representations are
realizable over $\overline{\Q}$, and hence also over $\overline{\Q_{\ell}}$,
for any prime $\ell$. It is therefore meaningful to ask, whether the Artin and
Swan representations are defined over a ``nice'' subfield (or ring) of
$\overline{\Q_{\ell}}$. Answering this question is the main purpose of this
section.
\begin{theorem}\label{thm:integralityOfSwan}
	Let $\ell$ be a prime different from the residue characteristic of
	$K$. Then 
		\begin{enumerate}[label=\emph{(\alph*)},ref={(\alph*)}]
		\item\label{item:rationality} The Artin and Swan representations are realizable over
			$\Q_{\ell}$.
		\item\label{item:integrality} There exists a finitely
			generated projective left-$\Z_{\ell}[G]$-module $\Sw_G$,
			which is unique up to isomorphism, such that
			$\Sw_G\otimes_{\Z_\ell}\Q_{\ell}$ is isomorphic to the
			Swan representation, i.e.~has character $\sw_G$.
	\end{enumerate}
\end{theorem}
\begin{remark}
	\begin{enumerate}[label={(\alph*)}]
		\item The analogous question about the realizability of
			$\sw_G$ over $\Q$ or even $\Z$ has a negative answer.
			Indeed, there are even examples of extensions where
			$\sw_G$ is not realizable over $\R$,
		\cite{Serre/Rationality}.
		\item There is no direct construction of the
			$\Z_\ell[G]$-module $\Sw_G$ known, but there is 
			a cohomological description of $\Sw_G$, see
			section \ref{sec:GOS-CohSwan}. Note however, that for this
			construction, one needs (amongst other
			ingredients), Theorem \ref{thm:integralityOfSwan}.
	\end{enumerate}
\end{remark}

Clearly \ref{item:integrality} of Theorem \ref{thm:integralityOfSwan} implies
\ref{item:rationality}.  In the proof of Theorem
\ref{thm:integralityOfSwan}, \ref{item:integrality}, one needs a good
understanding of the representation theory of $G$ over three different bases:
$\Q_{\ell}$, $\Z_{\ell}$ and $\F_{\ell}$. Until now, we have mainly discussed representations over a field of
characteristic prime to the order of the group (recall that even if $L/K$ is totally
ramified,  $G$ is an extension of a cyclic
group of order prime to $p$ and the  $p$-group $G_1$, so $|G|$ will not
necessarily be
prime to $\ell$). A complete presentation of the necessary
representation theory is beyond the scope of these notes, but we will
nonetheless try to give an overview.

\subsubsection{Projective representations}
If $E$ is a field and $G$ a group such that $|G|$ is invertible in $E$, then
it is a consequence of Maschke's Theorem \ref{thm:maschke} that two finitely generated
$E[G]$-modules are isomorphic if and only if their classes in $R_E(G)$ agree.
This is no longer the case if $|G|$ is not invertible in $E$ (Example
\ref{ex:repsInPosChar}).  We can recover this nice property, once we
concentrate on \emph{projective} $E[G]$-modules. We write $P_E(G)$ for the
Grothendieck group of projective finitely generated $E[G]$-modules.

Recall that if $R$ is a (not necessarily commutative) ring, and $M$ a
left-$R$-module, then $M$ is called \emph{projective}, if there exists a free
$R$-module $F$, such that $M$ is a direct summand of $F$.  

In the proof of Proposition \ref{prop:compositionSeries} we saw that an irreducible representation of a finite group $G$ on a field $E$ is a quotient of $E[G]$. Thus an irreducible representation is projective if and only if it is a direct summand of $E[G]$.

\begin{definition}Let $R$ be a ring.
	\begin{itemize}
		\item A surjective morphism $f:V\rightarrow W$ of left $R$-modules is
			an \emph{essential epimorphism} if for every
			proper submodule $V'\subsetneqq V$ we have
			$f(V')\subsetneqq W$.
		\item An essential epimorphism $f:P\rightarrow V$ with $P$
			projective is called \emph{projective cover} or
			\emph{projective envelope} of $V$\index{projective
			cover}.
	\end{itemize}
\end{definition}

\begin{lemma}\label{lemma:maxSSQuotient}
	Let $E$ be a field, $G$ a finite group and $V$ a finitely generated
	$E[G]$-module. Up to unique isomorphism
	there exists a unique quotient $q_{\max}:V\twoheadrightarrow S_{\max}$ with $S_{\max}$
	semi-simple  and $\dim_E S_{\max}$ maximal among all semi-simple quotients of
	$V$. 

	If $f:V\twoheadrightarrow S'$ is any quotient with $S'$
	semi-simple, then $S'$ is isomorphic to $S_{\max}$ if and only if $f$ is
	an essential epimorphism.
\end{lemma}
\begin{proof}
	Simple quotients of $V$ correspond to maximal submodules of $V$. Thus,
	as $\dim_E V<\infty$, there exists a quotient $q_{\max}:V\twoheadrightarrow S_{\max}$  such
	that $S_{\max}$ is semi-simple and such that its dimension is maximal among
	the semi-simple quotients of $V$.
	If $q:V\twoheadrightarrow S$ is a simple quotient, then $q_{\max}\oplus
	q:V\rightarrow S_{\max}\oplus S$ is not surjective, and as $S$ is
	simple, this implies that the image of $q_{\max}\oplus q$ does not
	contain $0\oplus S\subset S_{\max}\oplus S$, so $\image(q_{\max})\cap (0\oplus S)=0$.
	It follows that $\ker(q_{\max})\subset\ker(q)$. 
	
	Let $\rad(V)$ denote the intersection of all maximal submodules of $V$. we have proved that $\ker(q_{\max})\subset \rad(V)$. On the other hand, an element $v\in \rad(V)$ maps to $0$ in every simple quotient, and hence also in every semi-simple quotient. Thus
	$\ker(q_{\max})=\rad(V)$ and up to unique isomorphism, $q_{\max}$ is the quotient morphism $V\twoheadrightarrow V/\rad(V)$.
%
%
%
%
%
%
%

	Note that the maximal semi-simple quotient  is an essential
	epimorphism: Otherwise there would exist a maximal submodule
	$M\subsetneq V$  such that $q_{\max}(M)=S_{\max}$. This would mean that
	$M$ maps surjectively onto every simple quotient of $V$, which is
	impossible, as $V/M\neq 0$ is simple.

	If $q':V\rightarrow S'$ is a semi-simple quotient, then
	$S'$ is isomorphic to a direct summand of $S_{\max}$.
	In particular $q'(q_{\max}^{-1}(S'))=S'$, so if $q'$ is an essential
	epimorphism, then $q_{\max}^{-1}(S')=V$ and hence $S_{\max}=S'$.
\end{proof}
\begin{definition}
	Thanks to the lemma, we can talk about the \emph{maximal semi-simple
	quotient of $V$}.\index{maximal semi-simple quotient}
\end{definition}

\begin{proposition}\label{prop:projectiveEnvelope}
	Let $G$ be a group and $E$ a field.
	\begin{enumerate}[label=\emph{(\alph*)},ref=(\alph*)]
		\item\label{prop:projectiveEnvelope-K-projective} If $|G|$ is
			invertible in $E$, then every finitely generated $E[G]$-module is projective.
		\item\label{item:projectiveEnvelope} For every finitely
			generated $E[G]$-module $V$, there exists a projective
			envelope $f_V:P_V\rightarrow V$. The pair $(P_V,f_V)$
			is unique up to isomorphism.
		\item\label{item:projectiveEnvelopeOfSS}Every finitely
			generated projective
			$E[G]$-module $P$ is the projective envelope of its
			maximal semi-simple quotient.
		\item An indecomposable finitely generated projective $E[G]$-module is
			the projective envelope of a simple
			$E[G]$-module.
		\item\label{it:projectiveGrothendieckRing} If $P_E(G)$ denotes the Grothendieck ring of finitely
			generated projective
			$E[G]$-modules, then $P_E(G)$ is freely generated by
			the projective envelopes of simple $E[G]$-modules. In
			particular, two finitely generated projective $E[G]$-modules
			$P, P'$ are isomorphic if and only if their classes in
			$P_E(G)$ agree.
	\end{enumerate}
\end{proposition}
\begin{proof}
	\begin{enumerate}[label={(\alph*)}]
		\item We have seen that if $|G|$ is invertible, then every finitely generated $E[G]$-module
			is the direct sum of irreducible $E[G]$-modules and
			that an irreducible $E[G]$-module is a direct summand
			of the regular representation, which is a free
			$E[G]$-module.
		\item We just give a sketch, for full details see
			\cite[Prop.~41]{Serre/Representations}. If $V$ is a
			finitely generated $E[G]$-module, there exists a
			finitely generated projective $E[G]$-module $Q$
			and a surjection $Q\twoheadrightarrow V$. Let
			$P\subset Q$ be a submodule of minimal length amongst
			the submodules $M$
			for which the composition $M\rightarrow Q\rightarrow
			V$ is surjective. Clearly the map $P\rightarrow V$ is
			an essential epimorphism.
			One proceeds to show that $P$ is a direct summand of
			$Q$ and hence projective.

			The uniqueness of projective envelopes is not
			difficult to prove.
		\item	This follows directly from the lemma.
	\item One easily sees that the projective envelope of a direct sum is the
	direct sum of the envelopes of the summands. We also know from
	\ref{item:projectiveEnvelopeOfSS} that every
	finitely generated projective $E[G]$-module is the projective envelope
	of a semi-simple $E[G]$-module. Hence the claim. 	
\item If $P$ is a finitely generated projective $E[G]$-module, $P$ can be
	written as a finite direct sum of indecomposable finitely generated
	projective $E[G]$-modules, each of which is the
	projective envelope of a simple $E[G]$-module $V_i$. Thus $P_E(G)$ is
	generated by the projective envelopes of simple $E[G]$-modules.
	On the other hand, every
	extension $0\rightarrow P_1\rightarrow P\rightarrow P_2\rightarrow
	0$ of projective $E[G]$-modules splits.\end{enumerate}
\end{proof}
\begin{example}
Assume $\Char(E)=p>0$ and let $G$ be a finite
$p$-group. 
Then every nonzero $E[G]$-module has nontrivial invariants (Exercise \ref{ex:repsInPosChar}). This means that every simple
$E[G]$-module is the trivial representation $E$. Given a $G$-equivariant
map $\phi:E[G]\rightarrow E$, we see that $\phi(g)=g\phi(1)=\phi(1)$
for every $g\in G$. This means that every simple quotient of $E[G]$ is
isomorphic to $E[G]\twoheadrightarrow E$, $g\mapsto 1$. Consequently, the
maximal semi-simple quotient of $E[G]$ is $E$, as a map $E[G]\rightarrow E\oplus E$ cannot be surjective.

Next, we show that $E[G]$ is indecomposable. Indeed, if $E[G]=U\oplus V$ with
$U,V\subsetneq E[G]$, then $\dim_E E[G]^G\geq 2$. But $E[G]^G=E\cdot \sum_{g\in
	G}g$ is $1$-dimensional. Proposition \ref{prop:projectiveEnvelope},  \ref{item:projectiveEnvelopeOfSS}, shows that the
morphism $E[G]\rightarrow E$ is the projective envelope of $E$.
\end{example}

Now let $E'/E$ be an extension of fields; we have
seen in Proposition \ref{prop:injectivityOnK0} that $R_E(G)\rightarrow
R_{E'}(G)$ is injective. The same is true for $P_E(G)$:

\begin{proposition}
The map $P_E(G)\rightarrow P_{E'}(G)$ is
injective.
\end{proposition}
\begin{proof}
	We have seen in Proposition \ref{prop:injectivityOnK0} that
	$R_{E}(G)\rightarrow R_{E'}(G)$ is injective.  
	More precisely, we have
	seen that if $V_1,V_2$ are nonisomorphic irreducible $E[G]$-modules,
	then $V_1\otimes_E E'$ and
	$V_2\otimes_EE'$ have no common composition factors.
	
	If $f:W\rightarrow V$ is an essential
	epimorphism of finitely generated $E[G]$-modules, then so is
	$f':W\otimes_E E'\rightarrow V\otimes_E E'$. To see this, fix a basis
	$\{e_i\}_{i\in I}$ of $E'$ over $E$. Then $W\otimes_E E'$ considered
	as a (not necessarily finitely generated) $E[G]$-module is just
	$\bigoplus_iW\otimes e_i\cong W^{\oplus|I|}$ and similarly for $V\otimes_E E'$. If $M\subset
	W\otimes_E E'$ is a maximal submodule such that $f'(M)=V\otimes_E E'$,
	then $M\cap (W\otimes e_i)\twoheadrightarrow V\otimes e_i$. As
	$f$ is an essential epimorphism it follows that $M=W\otimes_E E'$, and
	the claim is proved.

	Thus, if $V_1, V_2$ are nonisomorphic irreducible $E[G]$-modules with
	projective envelopes $P_1, P_2$, then $P_i\otimes E'\rightarrow
	V_i\otimes E'$, $i=1,2$, are projective envelopes. The maximal
	semi-simple quotient of $P_i\otimes E'$ is the maximal semi-simple
	quotient of $V_i\otimes E'$ according to Lemma
	\ref{lemma:maxSSQuotient}. It follows that the maximal semi-simple
	quotients of $P_1\otimes E'$ and $P_2\otimes E'$ have no common
	factors, and thus $P_E(G)\rightarrow P_{E'}(G)$ is injective.
\end{proof}
\subsubsection{Representation theory in mixed characteristic}\label{subsubsec:RepMixedChar}
Now let $E$ be a complete discretely valued field of characteristic $0$ with residue field $k$ of
characteristic $\ell>0$ and let $G$  be a finite group.  If $A:=\mathcal{O}_E$ is the valuation ring of $E$, we write $P_A(G)$ for the
Grothendieck ring of finitely generated projective $A[G]$-modules. Reduction
modulo the maximal ideal $\mathfrak{m}$ of $A$ gives us a homomorphism of groups
$P_A(G)\rightarrow P_k(G)$. 
\begin{proposition}
	$P_A(G)\rightarrow P_k(G)$ is an isomorphism.
\end{proposition}
\begin{proof}
	We only give a rough sketch of the proof. For details, see
	\cite[\S14.4]{Serre/Representations}. If $P_1, P_2$
	are projective $A[G]$-modules, and $\bar{f}:\overline{P}_1\rightarrow
	\overline{P}_2$ is an isomorphism of $k[G]$-modules, then projectivity
	implies that $\bar{f}$ lifts to a morphism of $A[G]$-modules
	$f:P_1\rightarrow P_2$, which is an isomorphism by Nakayama's Lemma.
	This shows that the map $P_A(G)\rightarrow P_k(G)$ is injective.

	To see that it is surjective, one uses that $A$ is complete. Let
	$\mathfrak{m}$ be the maximal ideal of $A$. Write
	$A_n:=A/\mathfrak{m}^n$. If $\overline{P}$ is a projective
	$k[G]$-module, let $P_n\rightarrow \overline{P}$ denote the projective
	envelope of $\overline{P}$ considered as an $A_n[G]$-module. It is not
	difficult to verify that $P_n/\mathfrak{m}P_n\xrightarrow{\cong}
	\overline{P}$ as
	$k[G]$-modules. Passing to the limit over $n$ produces an
	$A[G]$-module $\hat{P}$, which is free as an $A$-module and whose
	reduction is isomorphic to the projective $k[G]$-module $P$. A calculation is
	required to show that this implies that $\hat{P}$ is actually
	projective as an $A[G]$-module.
\end{proof}

\begin{proposition}\label{prop:triangle}
	There is a commutative diagram of rings: 
	\begin{equation*}\label{eq:triangle}
		\begin{tikzcd}
			P_A(G)\cong
			P_k(G)\ar[swap]{dr}{e}\ar{rr}{c}&&R_k(G)\\&R_E(G)\ar[swap]{ur}{d}
		\end{tikzcd}
	\end{equation*}
	compatible with base
	change along finite extensions $E\subset E'$. 
\end{proposition}
\begin{proof}
	The map $c$ is just given by ``forgetting'' the projectivity of a
	module, while $e$ is given by $-\otimes_A E$. Only $d$
	requires work: Given a $E[G]$-module $V$, pick a finitely
	generated free $A$-submodule $W$ of $V$ with $\rank_{A} W=\dim_E V$.
	Such a submodule is called \emph{lattice}.  Define $W':=	\sum_{g\in G}
	gW$. This is also a lattice, and an $A[G]$-module. We claim that the
	class of $W'/\mathfrak{m}W'$ in $R_k(G)$ is independent of the choice
	of the lattice $W$ (the isomorphism class of the representation
	$W'/\mathfrak{m}W'$ is not necessarily independent of the choice of
	$W$!). Let $W_1, W_2$ be $G$-stable lattices in	$V$, without loss of
	generality we may assume $W_2\subset W_1$: There exists some
	$r\geq 0$, such that $\mathfrak{m}^{r}W_2\subset W_1$.

	There exists a minimal $n$
	such that $\mathfrak{m}^nW_1\subset W_2$. We proceed by induction on
	$n$. If $n=0$, there is nothing to show. If $n=1$, the module
	$W_1/W_2$ is killed by multiplication with a uniformizer $\pi$ of $A$.
	This means we have an short exact sequence of $k[G]$-modules
	\[0\rightarrow W_1/W_2\xrightarrow{\pi}W_2/\mathfrak{m}W_2\rightarrow
		W_1/\mathfrak{m}W_1\rightarrow W_1/W_2\rightarrow 0\]
	This means that the classes of $W_1/\mathfrak{m}W_1$ and
	$W_2/\mathfrak{m}W_2$ agree in $R_k(G)$, as claimed.
	
	If $n>1$, take $W_3:=\mathfrak{m}^{n-1}W_1+W_2$. Then
	$\mathfrak{m}W_3\subset W_2\subset W_3$, so by the previous
	calculation the classes of $W_2/\mathfrak{m}W_2$ and
	$W_3/\mathfrak{m}W_3$ agree in
	$R_k(G)$. On the other hand, $\mathfrak{m}^{n-1}W_1\subset W_3\subset
	W_1$, so by induction the classes of $W_3/\mathfrak{m}W_3$ and
	$W_1/\mathfrak{m}W_1$ agree. This concludes the construction of the map
	$d:R_E(G)\rightarrow R_k(G)$.

	It is easy to see that this map is a ring homomorphism, and the compatibility with base change is also
	clear.
\end{proof}

\begin{exercise}\label{ex:projectivity-prime-to-p}
	The triangle \eqref{eq:triangle} is not so interesting if the order
	$|G|$ is prime to the residue characteristic $\Char(k)=\ell$. Indeed,
	under that assumption, prove that
	\begin{enumerate}[label={(\alph*)}]
		\item Every finitely generated $k[G]$-module is projective.
		\item Every finitely generated $A[G]$-module which is free
			over $A$ is projective.
		\item A projective $A[G]$-module $V$ is simple if and only if
			$V\otimes_A k$ is a simple $k[G]$-module.
		\item The maps $c,d,e$ in the triangle \eqref{eq:triangle} are
			isomorphisms.
	\end{enumerate}
\end{exercise}
We will give -- without proof -- some important properties of the maps
\eqref{eq:triangle}.

\begin{theorem}\label{thm:mainRepThmIntegrality}\leavevmode
	\begin{enumerate}[label=\emph{(\alph*)},ref=(\alph*)]
		\item\label{thm:mainRepThmIntegrality-d-surj} 
                     The map $d:R_E(G)\rightarrow R_k(G)$ is surjective.
		\item\label{thm:mainRepThmIntegrality-e-inj}
                       The map $e:P_k(G)\rightarrow R_E(G)$ is injective.
	       \item\label{thm:mainRepThmIntegrality-comes-from-below} Let $E'$ be a finite extension of the complete,
			discretely valued field $E$, and write $k'$ for its
			residue field.  We then have a commutative diagram
			\begin{equation*}
				\begin{tikzcd}
					P_{k'}(G)\ar{r}{e'}&R_{E'}(G)\\
					P_k(G)\ar[hook]{u}\ar{r}{e}&R_E(G),\ar[hook]{u}
				\end{tikzcd}
			\end{equation*}
			and a class $x\in R_{E'}(G)$ lies in $\image(e)$ if
			and only if its character $\chi_x$ has values in $E$ and
			if $\chi_x(g)=0$ whenever $\ell|\ord(g)$.
	\end{enumerate}
\end{theorem}
Here, the character $\chi_x$ is to be understood as follows: $E'$ has
characteristic $0$, and therefore $x$ can be uniquely written a $\Z$-linear
combination of classes $[V_i]$ of irreducible $E'$-representations;
$x=\sum_{i=1}^na_i[V_i]$, $a_i\in \Z$. We then define $\chi_x:=\sum_{i=1}^n
a_i\chi_{V_i}:G\rightarrow E'$.

The proof of Theorem \ref{thm:mainRepThmIntegrality} is long and rather complicated,
see \cite[Ch.~17]{Serre/Representations}.
We omit it. Note that statement
\ref{thm:mainRepThmIntegrality-comes-from-below} can be understood as a
generalization of Proposition \ref{prop:comesFromBelow}.

Admitting the theorem, we can prove Proposition \ref{prop:integralityAbstract}, which was the reason for our
discussion.

\begin{proposition}\label{prop:integralityAbstract}
	We continue to denote by $E$ a complete discretely valued field of
	characteristic $0$ with
	valuation ring $A$ and residue field $k$ of characterisitc $\ell>0$
	
	Let $E'$ be a finite extension of $E$ and $A'$ the ring of integers of
	$E'$. A class $x\in R_{E'}(G)$ arises from a projective $A[G]$-module
	if and only if
	\begin{enumerate}[label=\emph{(\alph*)},ref=(\alph*)]
		\item\label{item:prop44a} The character $\chi_x$ takes values in $E$.
		\item\label{item:prop44b} There exists $n\in \N$, such that $nx$ is the class of a
			$E'$-representation, which comes from a projective
			$A'[G]$-module.
	\end{enumerate}
	Moreover, the projective $A[G]$-module realizing $x$ is unique up to
	isomorphism.
\end{proposition}
\begin{proof}
	Clearly, if the class $x$ comes from a projective $A[G]$-module, then
	\ref{item:prop44a} and \ref{item:prop44b} are satisfied.

	Conversely, assume that $x\in R_{E'}(G)$ satisfies \ref{item:prop44a}
	and \ref{item:prop44b}. Let $V'$ be the projective $A'[G]$-module
	corresponding to $nx$. Write $V$ for the $A[G]$-module obtained from
	$V'$ via the inclusion $A\subset A'$. Then $V$ is a projective
	$A[G]$-module as $A'$ is free over $A$.
	
	Since $E'/E$ is finite, $V\otimes_A E$ is an
	$E$-vector space of dimension $[E':E]\cdot \dim_{E'}(V'\otimes_{A'} E')$.  We compute the
	character of $V\otimes_A E$. Let $b_1,\ldots, b_r\in E'$ be a basis of $E'$ as a
	$E$-vector space such that $b_1=1$, and fix $g\in G$. If $v'\in V'$ is an eigenvector
	of $g$ acting on $V'$, then $v,
	b_2v,b_3v,\ldots$ are linearly independent eigenvectors of $g$ acting
	on $V$. Thus $\Tr(g|V\otimes_A E)=[E':E]\Tr(g|V'\otimes_{A'} E')$.
	Using the notation from Proposition \ref{eq:triangle}, we see that
	$e([V])=[E':E]nx\in R_{E'}(G)$.
	But by construction $e([V])\in R_E(G)\subset R_{E'}(G)$, so by Theorem
	\ref{thm:mainRepThmIntegrality} we see that $\chi_{e([V])}(g)=0$
	whenever $\ell|\ord(g)$. This means that $\chi_x(g)=0$
	whenever $\ell|\ord(g)$. Again by Theorem
	\ref{thm:mainRepThmIntegrality}, it follows that $x\in
	\image(P_A(G)\xrightarrow{e}R_E(G))$. Let $y\in P_A(G)$ be an element
	such that $e(y)=x$. Since $e$ is injective, $[V]=[E':E]ny$. As
	$V\in P^+_A(G)$, the same is true for $y$, i.e.~there exists a
	projective $A[G]$-module $Y$, such that $Y\otimes_A E'\cong V'$,
	i.e.~$[Y\otimes_A E]=x\in R_E(G)$.

	The projective $A[G]$-module $Y$ is unique up to isomorphism by
	Proposition
	\ref{prop:projectiveEnvelope}, \ref{it:projectiveGrothendieckRing}:
	The isomorphism class of $Y$ is determined by its class in $P_A(G)$,
	and by the injectivity of $e$, this class is uniquely determined.
\end{proof}

\subsubsection{The Swan character comes from $\Z_{\ell}$}
Using Proposition \ref{prop:integralityAbstract}, we can now finally prove the
$\ell$-adic integrality of the Swan conductor.

\begin{proof}[Proof of Theorem \ref{thm:integralityOfSwan}]
	Recall that we defined $\sw_G:=a_G-(r_G-r_{G/G_0})$, and that we saw that $a_G$,
	hence $\sw_G$, can be realized over a finite extension $E/\Q_{\ell}$.
	We want to apply Proposition
	\ref{prop:integralityAbstract}. We know that $\sw_G$ takes values in
	$\Z$, so to prove that $\sw_G$ can be realized by a projective
	$\Z_{\ell}[G]$-module, we just have to check that some multiple
	$n\sw_G$ can be realized by a projective $\mathcal{O}_E[G]$-module,
	where $\mathcal{O}_E$
	is the integral closure of $\Z_{\ell}$ in $E$.
	To see this,  note that by Lemma \ref{lemma:totallyRamifiedArtin} we
	know that $\sw_G$ is induced by $\sw_{G_0}$, so we
	may assume that $L/K$ is totally ramified, i.e.~that $G_0=G$. We computed in 
	Lemma \ref{lemma:ArtinRational} that
	\[a_{G}=\sum_{i=0}^{\infty}\frac{1}{|G:G_i|}\alpha_{i,*}u_{G_i},\]
	where $\alpha_i:G_i\hookrightarrow G$ are the inclusions of the
	ramification subgroups. It follows that
	\[|G|\sw_{G}=\sum_{i=1}^{\infty}|G_i|\alpha_{i,*}u_{G_i}.\]
	Since $u_{G_i}=r_{G_i}-1$, each $\alpha_{i,*}u_{G_i}$ can be realized
	over $\mathcal{O}_E[G]$ (even over $\Z_{\ell}[G]$). It remains to see
	hat the augmentation representations $U_{G_i}$, $i>0$, are projective
	$\Z_{\ell}[G_i]$-modules.
	
	As the groups $G_i$ are
	$p$-groups, for $i>0$, and since $p\neq \ell$, every finitely generated
	$\Z_{\ell}[G_i]$-module $V$ which is free (=torsion-free) over
	$\Z_{\ell}$ is projective. Indeed, one finds a surjection
	$F\twoheadrightarrow V$, with $F$ a free $\Z_{\ell}[G_i]$-module, and as
	$(|G_i|,\ell)=1$ for $i>0$, this surjection splits, and $V$ is a
	direct summand of $F$.  In particular the augmentation representation
	of $G_i$ is projective over $\Z_{\ell}[G_i]$ for $i>0$.

	It follows that $|G|\sw_{G}$ can be realized by a projective
	$\mathcal{O}_E[G]$-module, and Proposition \ref{prop:integralityAbstract} shows
	that $\sw_{G}$ can be realized by a finitely generated projective
	$\Z_\ell[G]$-module, unique
	up to isomorphism.
\end{proof}

\subsection{Swan conductor of an $\ell$-adic Galois representation}\label{sec:swanconductor}
In this section let $K$ be a complete discretely valued field with perfect residue field $k$ of characteristic $p>0$. Let
$K^{\sep}$ be a separable closure of $K$, and write
$G_K:=\Gal(K^{\sep}/K)$. Fix a second
prime number $\ell\neq p$. We are now interested in \emph{continuous
representations} $\rho:G_K\rightarrow \GL(V)$, where $V$ is
a finite dimensional vector space over a finite extension $E$ of $\Q_{\ell}$, equipped with its $\ell$-adic topology.
Such a representation is called \emph{$\ell$-adic Galois
representation}\index{$\ell$-adic Galois representation}\index{Galois
representation!$\ell$-adic}. 
\begin{definition}\label{defn:topologyOnGL}Let us recall the definition of the topology on
$\GL(V)$: The topology on $V$ makes $\End(V)$ into a topological ring, and we
can identify the group $\GL(V)$ with the subspace of $\End(V)\times \End(V)$
defined by $XY=1$, i.e.~with the set of pairs $(A,A^{-1})$, $A\in \GL(V)$. Using this
identification we equip $\GL(V)$ with the subspace topology to make it a
topological group. Analogously, if  $\mathcal{V}$ is a free
$\mathcal{O}_E$-module of finite rank, we can make
$\Aut_{\mathcal{O}_E}(\mathcal{V})$ into a topological group.
\end{definition}
\begin{lemma}\label{lemma:lattice}
	Let $E$ be a finite extension of $\Q_{\ell}$, $V$ a finite dimensional
	$E$-vector space, $G$ a profinite group and
	$\rho:G\rightarrow \GL(V)$ continuous representation.  Writing $\mathcal{O}_E$ for the ring of integers of
	$E$, there exists a free $\mathcal{O}_E$-submodule $\mathcal{V}\subset V$, such
	that $V=\mathcal{V}\otimes_{\mathcal{O}_E} E$ and such that $\rho$ factors
	\[\rho:G\rightarrow \GL(\mathcal{V})\rightarrow \GL(V),\]
	where $\GL(\mathcal{V}):=\Aut_{\mathcal{O}_E}(\mathcal{V})$.
\end{lemma}
\begin{proof}
	Let
	$e_1,\ldots, e_r$ be a basis of $V$ and use it to identify $\GL(V)$ with
	$\GL_r(E)$. The 
	inclusion $\GL_r(\mathcal{O}_E)\subset \GL_r(E)$ given by the choice
	of the basis makes $\GL_r(\mathcal{O}_E)$ into an \emph{open}
	subgroup of $\GL_r(E)$. Then $\{\rho^{-1}(M\GL_r(\mathcal{O}_E))|M\in \GL_r(E)\}$
	is an open covering of $G$, and as $G$ is compact, it follows that
	there exists a minimal $n\geq 1$ and matrices $M_1,\ldots, M_n\in
	\GL_r(E)$, such that
	$\image(\rho)\subset M_1\GL_r(\mathcal{O}_E)\cup\ldots\cup
	M_n\GL_r(\mathcal{O}_E)$. 
	
	If
	$\mathcal{V}':=\sum_{i=1}^re_i\mathcal{O}_E \subset V$, then for every $g\in G$
	there exists $i$ such that $\rho(g)(
	\mathcal{V}')=M_i\mathcal{V}'$, and conversely, for every $i$ there
	exists $g\in G$ such that
	$M_i\mathcal{V}'=\rho(g)\mathcal{V}'$.

	Defining $\mathcal{V}:=\sum_{i=1}^n M_i\mathcal{V}'\subset V$ hence gives a
	$G$-stable, free sub-$\mathcal{O}_E$-module of $V$, such that
	$\mathcal{V}\otimes_{\mathcal{O}_E} E=V$.
\end{proof}

We will be particularly interested in how the \emph{wild
ramification subgroup} acts in an $\ell$-adic representation. Recall that the
upper numbering filtration on finite Galois groups is compatible with  taking
quotients (Herbrand's Theorem, Proposition \ref{prop:Herbrand}), which allowed
us to introduce the \emph{upper numbering filtration} on the infinite,
absolute
Galois group $G_K$ (Definition
\ref{defn:upperNumberingInfiniteExtension}). We saw that $G_K^0\unlhd G_K$ is
the kernel of the projection $G_K\twoheadrightarrow \Gal(\bar{k}/k)$.  If
$L/K$ is a finite  Galois extension, we also showed that
$\Gal(L/K)_1=\Gal(L/K)^{\phi_{L/K}(1)}$ is the unique $p$-Sylow subgroup of
$\Gal(L/K)$, and that the quotient $\Gal(L/K)/\Gal(L/K)_1$ is cyclic of order
prime-to-$p$ (Corollary \ref{cor:structure-of-ramification-groups}).  Passing to the inverse limit over all finite Galois extensions
we obtain a closed normal subgroup $P_K\unlhd G_K$, such that $P_K$ is a
pro-$p$-group, and such that $G_K/P_K$ is pro-cyclic with every finite
quotient of order prime to $p$. In particular, if $H\unlhd G_K$ is an open (hence finite index) normal subgroup,
then the image of $P_K$ in the finite Galois group $G_K/H$ is $(G_K/H)_1$.

\begin{definition}\label{defn:tameRepresentation}
	The pro-$p$-group $P_K$ is called the \emph{wild inertia
	group}\index{wild inertia group} or \emph{wild ramification group}.
	Let $R$ be a commutative ring and $\rho:G_K\rightarrow \GL_r(R)$ a
	homomorphism.
	\begin{enumerate}[label={(\alph*)}]
		\item $\rho$ is called \emph{unramified} if
			$G_{K}^0\subset \ker(\rho)$.
		\item $\rho$ is called \emph{tame} or \emph{tamely ramified} if
			$P_K\subset \ker(\rho)$. Otherwise $\rho$ is called
			\emph{wild} or \emph{wildly ramified}.
	\end{enumerate}
	Our main interest will be in the cases $R=\mathcal{O}_E, E$ or $\F_{\lambda}$, and
	$\rho$ continuous.
\end{definition}

We now reduce a Galois representation modulo $\ell$. 
\begin{definition}\label{defn:reductionModOfGaloisRep}
	Let $E$ be a finite extension of $\Q_{\ell}$ with valuation ring
	$\mathcal{O}_E$, uniformizer $\lambda$, and finite residue field
	$\F_\lambda$.  
	If $\mathcal{V}$ is a free finitely generated
	$\mathcal{O}_E$-module, and 
	$\rho:G_K\rightarrow \GL(\mathcal{V})$ a continuous
	representation, then the composition $\bar{\rho}:G_K\rightarrow
	\GL(\mathcal{V})\rightarrow
	\GL(\overline{\mathcal{V}})$, with
	$\overline{\mathcal{V}}=\mathcal{V}/\lambda\mathcal{V}$,  is called the
	\emph{reduction modulo $\lambda$}\index{reduction of Galois
	representation}\index{Galois representation!reduction} of $\rho$.
\end{definition}

The next lemma shows that in the situation of the definition, the
pro-$p$-group $P_K\subset G_K$ acts on $\mathcal{V}$ and $\overline{\mathcal{V}}$ through
the same \emph{finite} group.

\begin{lemma}\label{lemma:proPfactorsThroughFiniteQuotient}
	Let $\ell\neq p$ be two primes,
	$E$ a finite extension of $\Q_{\ell}$, $\mathcal{O}_E$ its
	valuation ring and $\F_{\lambda}$ its residue field. 
	If $P$ is a pro-$p$-group and $\rho:P\rightarrow
	\GL_r(\mathcal{O}_E)$ a continuous representation, then the image of
	$\rho$ is finite, and $\rho(P)\cap
	\ker(\GL_r(\mathcal{O}_E)\twoheadrightarrow \GL_r(\F_{\lambda}))=\{1\}$.	

	If 
	 $\rho:P\rightarrow \GL_r(E)$ is a continuous
	representation, then $\rho$ factors through a finite quotient of $P$.
\end{lemma}
\begin{proof}
	In Lemma \ref{lemma:lattice},  we saw  that any continuous map
	$P\rightarrow \GL_r(E)$ factors through (a conjugate of) the inclusion
	$\GL_r(\mathcal{O}_E)\subset \GL_r(E)$, so the final statement follows from the first.

	The kernel of the projection $\GL_r(\mathcal{O}_E)\rightarrow
	\GL_r(\F_{\lambda})$ is an open subgroup which can be written as
	$H:=\id+\lambda M_r(\mathcal{O}_E)$, where $\lambda$ is a uniformizer of
	$\mathcal{O}_E$ and $M_r(\mathcal{O}_E)$ the ring of $r\times
	r$-matrices with entries in $\mathcal{O}_E$. By continuity
	$\rho^{-1}(H)$ is an open subgroup of $P$ containing the kernel of
	$\rho$. The group $H$ is a pro-$\ell$-group, and it is not
	difficult to see that there are no non-trivial maps between pro-$p$ and
	pro-$\ell$-groups, if $\ell\neq p$. It
	follows that $\rho^{-1}(H)\subset\ker(\rho)$, and hence $\rho(P)\cap
	H=\{1\}$. As $H$ has finite index in $\GL_r(\mathcal{O}_E)$,
	$\ker(\rho)$ has finite index in $P$, which completes the proof.
\end{proof}

\begin{corollary}\label{cor:tameIFFreductionTame}
	If $\mathcal{V}$ is a free $\mathcal{O}_E$-module and
	$\rho:G_K\rightarrow \GL(\mathcal{V})$ a continuous representation,
	then $\rho$ is tame if and only if the reduction $\bar{\rho}:G_K\rightarrow
	\GL(\mathcal{V})\rightarrow \GL(\overline{\mathcal{V}})$ is tame. 
\end{corollary}
\begin{proof}
	The lemma shows that $\rho(P_K)\cong\bar{\rho}(P_K)$.
\end{proof}

We now use the Swan
representation to define a measure for ``how wild'' a representation is.
\begin{definition}\label{defn:SerresB}
	Let $E$ be a finite extension of $\Q_{\ell}$ with valuation ring
	$\mathcal{O}_E$, uniformizer $\lambda$, and residue field
	$\F_\lambda$. Let $\mathcal{V}$ be a free finitely generated
	$\mathcal{O}_E$-module, and 
	$\rho:G_K\rightarrow \GL(\mathcal{V})$ a continuous
	representation. 
	\begin{enumerate}[label={(\alph*)}]
\item Since $\F_{\lambda}$ is a finite field, $G:=G_K/\ker(\bar{\rho})$ is the
	Galois group of a finite extension $L/K$. Following Serre, we define
	\begin{align*}b(\rho):=b(\bar{\rho})&:=\dim_{\F_{\lambda}}\hom_{\F_{\lambda}[G]}({\Sw_G}\otimes_{\Z_{\ell}}\F_{\lambda} ,\bar{\rho}).
\end{align*}
	We also write $b(\mathcal{V})$, if the representation $\rho$ is
	implicitly given.
	Note that since $\Sw_G$ is a projective $\Z_{\ell}[G]$-module,
	$\overline{\Sw_G}$ is a projective $\F_{\lambda}[G]$-module, so the number $b(\mathcal{V})$ only depends on the class of
	$\overline{\mathcal{V}}$ in $R_{\F_{\lambda}}(G)$.
\item   Let $V$ be a finite dimensional $E$-vector space and
	$\rho:G_{K}\rightarrow \GL(V)$ a continuous representation. If
	$\mathcal{V}\subset V$ is a $G_K$-stable $\mathcal{O}_E$-lattice for
	$V$,  then the number $b(\mathcal{V})$ only depends on
	$\rho:G_K\rightarrow \GL(V)$, not on
	the choice of
	the lattice. Indeed, since $\Sw_G$ is projective, $b(\mathcal{V})$
	only depends on the class of $\overline{\mathcal{V}}$ in $R_{\F_{\lambda}}(G)$,
	which only depends on $\rho$ (same argument as in proof as Proposition
	\ref{prop:triangle}).  We
	denote this number by $b(V)$ or $b(\rho)$.
	\end{enumerate}
\end{definition}
\begin{remark}\label{rem:independenceOfFiniteQuotient}~\begin{enumerate}
	\item In the definition of $b(\rho)$ we used the group
	$G=G_K/\ker(\bar{\rho})$. Instead, we could have used $G_K/N$ for any open
	normal subgroup $N$ of $G_K$ contained in $\ker(\bar{\rho})$ without changing the
	result. See the final comments in the proof of Theorem
	\ref{thm:integralityOfSwanConductor}.
\item Taking this into account, we see that in the situation of Definition
	\ref{defn:SerresB}, if $\rho$ itself factors through a finite quotient
	$G$ of $G_K$, then 
	\begin{align*}b(\rho)&=\dim_{\F_\lambda[G]}\Hom_{\F_\lambda[G]}(\Sw_G\otimes_{\Z_\ell}\F_{\lambda},\bar{\rho})\\
		&=\rank_{\mathcal{O}_E}\Hom_{\mathcal{O}_E[G]}(\Sw_G\otimes_{\Z_{\ell}}\mathcal{O}_E,\rho)\\
		&=\dim_{E}\Hom_{E[G]}(\Sw_G\otimes_{\Q_{\ell}}E,\rho\otimes E).
	\end{align*}
	To see this, apply the exact functor
	$\Hom_{\mathcal{O}_{E}[G]}(\Sw_G\otimes_{\Z_{\ell}}\mathcal{O}_E,-)$
	to the short exact sequence of $\mathcal{O}_E[G]$-modules
	\[0\rightarrow \mathcal{V}\xrightarrow{\cdot
	\lambda}\mathcal{V}\rightarrow \overline{\mathcal{V}}\rightarrow 0.\]
\end{enumerate}
\end{remark}

\begin{proposition}\label{prop:formulaForSerresB}
	In the situation of the definition, let $\mathcal{V}$ be a finitely generated, free $\mathcal{O}_E$-module
	equipped with a continuous representation $\rho:G_K\rightarrow
	\GL(\mathcal{V})$.  
		If $G:=G_K/\ker(\bar{\rho})$, then
			\begin{equation}\label{eq:formulaForSerresb}b(\mathcal{V})=\sum_{i=1}^{\infty}\frac{|G_i|}{|G_0|}\dim_{\F_\lambda}(\overline{\mathcal{V}}/\overline{\mathcal{V}}^{G_i})\end{equation}
			where
			$\overline{\mathcal{V}}=\mathcal{V}/\lambda\mathcal{V}$,
			with $\lambda$ a uniformizer of
			$\mathcal{O}_E$.
\end{proposition}
\begin{proof}
		We have seen that $|G_0|\cdot \sw_G=\sum_{i\geq 1}|G_i|u^*_{G_i}$, where $u^*_{G_i}$ denotes the
			character of the augmentation representation (Example
			\ref{ex:representationsFirstExamples},
			\ref{item:representationsFirstExamples:augmentation}) 
			$U_{G_i,\mathcal{O}_E}$ of $G_i$,
			induced along the inclusion $G_i\subset G$. It follows
			that 
			\begin{align*}
				|G_0|b(\mathcal{V})&=\dim_{\F_{\lambda}}(\hom_{\F_{\lambda}[G]}(\overline{\Sw_G^{\oplus
				|G_0|}},\overline{\mathcal{V}}))\\
				&=\sum_{i\geq
				1}\dim_{\F_{\lambda}}(\hom_{\F_{\lambda}[G]}(\overline{\Ind_{G_i}^GU_{G_i,\mathcal{O}_E}}^{\oplus
			|G_i|},\overline{\mathcal{V}}))\\
			&=\sum_{i\geq
			1}|G_i|\dim_{\F_{\lambda}}\hom_{\F_{\lambda}[G]}(\overline{\Ind^G_{G_i}U_{G_i,\mathcal{O}_E}},\overline{\mathcal{V}})
		\end{align*}
		 As the $G_i$, $i>0$, are $p$-groups, the augmentation
		representation is a projective $\mathcal{O}_E[G_i]$-module
		(Exercise \ref{ex:projectivity-prime-to-p}),
		its reduction is the augmentation representation of $G_i$ over
		$\F_{\lambda}$, because we already saw that
		$P_A(G_i)\rightarrow P_k(G_i)$ is an isomorphism. It is easily
		checked that induction and reduction modulo $\lambda$ commute,
		so
		$\overline{\Ind_{G_i}^GU_{G_i,\mathcal{O}_E}}=\Ind_{G_i}^GU_{G_i,\F_{\lambda}}$.

		Finally, note that by the properties of induction
		\[\hom_{\F_{\lambda}[G]}(\Ind_{G_i}^GU_{G_i,\F_{\lambda}},\overline{\mathcal{V}})=\hom_{\F_{\lambda}[G_i]}(U_{G_i,\F_{\lambda}},\RES_{G_i}^G\overline{\mathcal{V}})=\overline{\mathcal{V}}/\overline{\mathcal{V}}^{G_i}\]
		The claim follows.
\end{proof}

By definition, a representation $\mathcal{V}$ of $G_K$ is tame, if $P_K$ acts trivially. On the other
hand, the invariant $b(\mathcal{V})$ only depends on the reduction
of $\mathcal{V}$ modulo the maximal ideal. Nonetheless, as promised,
$b(\mathcal{V})$ does measure whether $\mathcal{V}$ is tame or not. 
\begin{proposition}
	Let $\rho:G_K\rightarrow \GL_r(\mathcal{O}_E)$ be a continuous
	representation. Then the following are equivalent:
	\begin{enumerate}
		\item\label{item:b0} $\rho\otimes E: G_K\xrightarrow{\rho}
			\GL_r(\mathcal{O}_E)\hookrightarrow\GL_r(E)$ is tame.
		\item\label{item:b1} $\rho$ is tame.
		\item\label{item:b2} $\bar{\rho}:G_K\xrightarrow{\rho}
			\GL_r(\mathcal{O}_E)\twoheadrightarrow
			\GL_r(\F_{\lambda})$ is tame.
		\item\label{item:b3} $b(\rho)=0$.
	\end{enumerate}
\end{proposition}
\begin{proof}
	The equivalence of \ref{item:b0} and \ref{item:b1} is clear, while the equivalence of 
	\ref{item:b1} and \ref{item:b2}  is Corollary
	\ref{cor:tameIFFreductionTame}. The equivalence of
	\ref{item:b2} and \ref{item:b3} follows from the formula of Proposition
	\ref{prop:formulaForSerresB}.
\end{proof}

\subsubsection{The break decomposition}
We continue to use the notations from the previous section.  The simple
observation made in Lemma \ref{lemma:proPfactorsThroughFiniteQuotient} tells us that  in any $\ell$-adic representation
$\rho:G_K\rightarrow \GL_r(E)$, where $E$ is a finite extension of
$\Q_{\ell}$, $\ell\neq p$, the wild inertia group $P_K$\index{wild inertia
group} always acts through
a finite quotient.

We follow Katz \cite[Ch.~1]{Katz/Kloosterman} in studying  fairly general
representations of the pro-$p$-group $P_K$ (Definition
\ref{defn:tameRepresentation}), which factor through a finite
quotient. See also \cite[\S2.]{Laumon/Fourier}.

Let us first establish some more facts about the ramification filtration on
$G_K$.
\begin{lemma}\label{lemma:absoluteUpperNumberingFiltration} For $\lambda \in \R_{\geq 0}$, write
	\[G_K^{\lambda+}:=\overline{\bigcup_{\lambda'>\lambda}G_K^{\lambda'}}\]
	where $\overline{(-)}$ denotes the closure in the topological group
	$G_K$. 
	
	The decreasing filtration
$G_K^{\lambda}$, $\lambda \in \R_{\geq 0}$, satisfies the following
properties.
	\begin{enumerate}[label=\emph{(\alph*)},ref=(\alph*)]
		\item\label{item:ramificationFiltration1} $\bigcap_{\lambda>0}G_K^{\lambda}=\{1\}$
		\item For $\lambda>0$ we have 
			\[G_K^\lambda=\bigcap_{0<\lambda'<\lambda}G_K^{\lambda'}.\]
		\item $P_K=G_K^{0+}$
	\end{enumerate}
\end{lemma}
\begin{proof}
	\begin{enumerate}[label={(\alph*)}]
		\item If $L/K$ is a finite Galois
			extension, 
			the quotient map $G_K\twoheadrightarrow \Gal(L/K)$
			induces surjections $G_K^{\lambda}\rightarrow
			\Gal(L/K)^{\lambda}$. Thus, if $g\in G_K$ lies
			$G^{\lambda}_K$ for all $\lambda >0$, then $g$ maps to
			$1$ in $\Gal(L/K)$ for every finite Galois extension
			$L/K$. The claim follows.
		\item The profinite group $G_K/G_K^{\lambda}$ corresponds to
			an algebraic extension $L/K$ (not necessarily finite)
			and the ramification filtration on its Galois group $G_K/G_K^{\lambda}$
			is given by the image of the filtration of $G_K$. Thus
			the claim follows from 
			\ref{item:ramificationFiltration1}.
		\item Let $L/K$ be a finite Galois extension. If $\epsilon>0$,
			then $\psi_{L/K}(\epsilon)>0$, so
			$\Gal(L/K)^{\epsilon}=\Gal(L/K)_{\lceil\psi_{L/K}(\epsilon)\rceil}\subset
			\Gal(L/K)_1$. 			Since the image of $P_K$ in $\Gal(L/K)$ is precisely
			$\Gal(L/K)_1$ it follows that $G_K^{0+}\subset P_K$.
			But we see more: there exists an $\epsilon_0>0$, such
			that $\Gal(L/K)^{\epsilon_0}=\Gal(L/K)_1$, so the closed
			subgroups $G_K^{0+}\subset P_K$ have the same images
			in every finite quotient, which shows that they are
			equal. (Recall that a closed subgroup $H$ of a profinite
			group $\varprojlim_{N} G/N$ is also profinite and
			isomorphic to $\varprojlim_{N} H/(N\cap H)$)
	\end{enumerate}
\end{proof}

\begin{definition}
	For convenience, we say that a \emph{$P_K$-module} is a $\Z[1/p]$-module $M$,
	together with a morphism $\rho:P_K\rightarrow \Aut_\Z(M)$, which
	factors through a finite discrete quotient. A morphism of $P_K$-modules
	is a morphism of $\Z[1/p]$-modules compatible with the action of
	$P_K$.
\end{definition}

\begin{proposition}[{\cite[Ch.~I, Prop.~1.1]{Katz/Kloosterman}}]\label{prop:breakDecomposition}
	For readability, write $G:=G_K$ and $P:=P_K$.
	Let $M$ be a $P$-module.
	\begin{enumerate}[label=\emph{(\alph*)},ref=(\alph*)]
		\item\label{item:breakDecomposition1} There exists a unique decomposition
			$M=\bigoplus_{x\in \R_{\geq 0}}M(x)$ of $P$-modules,
			such that
			\begin{enumerate}[label=\emph{(\roman*)},ref=(\roman*)]
				\item\label{property:decomposition1} $M(0)=M^P$
				\item\label{property:decomposition2} $M(x)^{G^x}=0$ for
					$x>0$.
				\item\label{property:decomposition3} $M(x)^{G^{y}}=M(x)$
					for $y>x$.
			\end{enumerate}
		\item $M(x)=0$ for all but finitely many $x\in \R_{\geq 0}$.
		\item\label{item:breakDecomposition3} For every $x\in \R_{\geq 0}$, the assignment $M\mapsto
			M(x)$ is an exact endofunctor on the category of
			$P$-modules.
		\item\label{item:breakDecomposition4} $\Hom_P(M(x),M(y))=0$ unless $x=y$.
	\end{enumerate}
\end{proposition}
\begin{proof}
	Let $\rho:P\rightarrow \Aut_{\Z}(M)$ be the representation implicit in
	saying that $M$ is a $P$-module and let $H=\image(\rho)$. By
	assumption, $H$ is a finite discrete group. For $x\in \R_{\geq 0}$,
	let $H(x):=\rho(G_K^x)$, and $H(x+):=\rho(G_K^{x+})=\bigcup_{y>x}H(y)$. In particular,
	$H=H(0+)$. Note that for every $x\in \R_{\geq 0}$, $H(x)$ and $H(x+)$
	are normal subgroups of $H$.

	For every $x\in \R_{\geq 0}$, we define elements of $\Z[1/p][H]$
	\[\pi(x):=\frac{1}{|H(x)|}\sum_{h\in H(x)}h\quad\text{  and
	}\quad\pi(x+):=\frac{1}{|H(x+)|}\sum_{h\in H(x+)}h\]
	We have seen such elements before: They define splittings of the
	projections
	$\Z[1/p][H(x)]\twoheadrightarrow \Z[1/p]$, $h\mapsto 1$, and similarly
	for $H(x+)$.
	\begin{lemma}
		Almost all the elements $\pi(x+)(1-\pi(x))$ are $0$, and the
		nonzero elements of the
		set $\{\pi(x+)(1-\pi(x))|x>0\}\cup\{\pi(0+)\}$ form a set of
		orthogonal, central idempotents in $\Z[1/p][H]$, whose sum is $1$.
	\end{lemma}
	\begin{proof}
		An easy calculation shows that for all $x\geq 0$, $\pi(x)$ and $\pi(x+)$ are
		idempotents. They are central as the groups $H(x)$,
		$H(x+)$ are normal subgroups of $H$. It follows that the same
		is true for $\pi(x+)(1-\pi(x))$ and $\pi(0+)$.
		For the orthogonality, note that $\pi(x+)\pi(x)=\pi(x)$, so
		$\pi(x+)(1-\pi(x))=\pi(x+)-\pi(x)$. For $y> x$ we see that $\pi(x)\pi(y)=\pi(x)$, and we compute
		\begin{align*}(\pi(x+)-\pi(x))(\pi(y+)-\pi(y))&=\pi(x+)\pi(y+)-\pi(x+)\pi(y)\\&=\pi(x+)-\pi(x+)\pi(y),\end{align*}
		which is $0$ if $\pi(y)\neq \pi(x)$. If $\pi(x)=\pi(y)$, then
		$H(x)=H(y)$, so $H(x)=H(x+)$ and hence $\pi(x+)-\pi(x)=0$.

       %

		Since $H$ is a finite group, we have $H(x)=H(x+)$ for almost
		all $x\in \R_{\geq 0}$, thus only finitely many
		$\pi(x+)(1-\pi(x))$ are nonzero. 

		Finally, let us compute the sum of the idempotents:
		\begin{equation}\label{eq:sumOfIdempotents}
			\sum_{x>0}\pi(x+)(1-\pi(x))+\pi(0+)=\sum^r_{n=1}(\pi(x_n+)-\pi(x_n))+\pi(0+)
		\end{equation}
		where $x_1<x_2<\ldots<x_r$ are the finitely many elements of
		$\R_{> 0}$ where $H(x_n+)\neq H(x_n)$. Note that in this case
		\[H(0+)=H(x_1),\quad H(x_{n-1}+)=H(x_n), n\in [2,r],\quad H(x_r+)=\{1\}.\]
        It follows that the coefficient of $h\in H(x_n)\setminus H(x_n+)$ in
		\eqref{eq:sumOfIdempotents} is
		\[\underbrace{-\frac{1}{|H(x_n)|}+\frac{1}{|H(x_{n-1}+)|}}_{=0}\underbrace{-\frac{1}{|H(x_{n-1})|}+\frac{1}{|H(x_{n-2}+)|}}_{=0}+\ldots=0
			\]
		A similar computation shows that the coefficient of $1$ in
		\eqref{eq:sumOfIdempotents} is $1$, so the proof is complete.
	\end{proof}
	From the lemma we obtain the decomposition of $M$ that we were
	looking for: For $x>0$, $M(x):=\{m\in M|\pi(x+)(1-\pi(x))m=m\}$, and
	$M(0):=\{m\in M| \pi(0+)m=m\}$. By construction, this is a
	decomposition into finitely many subrepresentations.

	We verify that it has the desired properties. For property
	\ref{property:decomposition1}, we just have to observe that
	$H(0+)=\image(P)$, as $P=G_K^{0+}$. For \ref{property:decomposition3},
	let $x_n$ be defined as in the previous paragraph,
	i.e.~such that $H(x_n)\supsetneqq H(x_n+)$. For $m\in M(x_n)$, we know by
	definition that $m=\pi(x_n+)m-\pi(x_n)m$. Thus, if $y>x_n$, then
	$H(y)\subset H(x_n+)$, and hence $m$ is invariant under $H(y)$, as
	$\pi(y)\pi(x)=\pi(x)$ if $y\geq x$. But this
	also means that $m=\pi(x_n+)m-\pi(x_n)m=m-\pi(x_n)m$, so
	$\pi(x_n)m=0$. This shows that $M(x_n)^{G^{x_n}}=0$, so
	\ref{property:decomposition2} also is verified.
	
	To prove that the decomposition of $M$ is unique, let
	$M=\bigoplus_{x\geq 0}M'(x)$ be a second decomposition with the
	properties \ref{property:decomposition1} -
	\ref{property:decomposition3}. Clearly, $M'(x)\neq 0$ if and only if
	$x$ is one of the $x_1,\ldots, x_r$ (or $x=0$) from above, and the same is true
	for $M(x)$.  We know that $M'(0)=M^P=M(0)$. For $i\in \{1,\ldots,r\}$, let $m\in M'(x_i)$. Then
	$\pi(x_i)m\in M'(x_i)^{G^{x_i}}=0$ and $\pi(y)m=m$ for $y>x_i$. We can use the first decomposition to uniquely write
	$m=m_0+m_1+\ldots+m_r$ with $m_0\in M^P$ and $m_j\in M(x_j)$. Note that $\pi(x_i)m_j=0$ for $i\leq j$ and $\pi(x_i)m_j=m_j$ for $i>j$. Thus, 
	\[0=\pi(x_i)m=m_0+\ldots+m_{i-1}.\]
	If $i=r$ is the final jump, we immediately see that $m=m_r$.
	If $i<r$, then 
	\[m=\pi(x_{i+1})m=m_0+\ldots+m_i.\]
	Again we see $m=m_i$.  We conclude that $M'(x_i)\subset M(x_i)$. The converse inclusion follows by symmetry.

	
	We already remarked that the decomposition of $M$ only has finitely many
	nonzero
	components, so it remains to check \ref{item:breakDecomposition3} and
	\ref{item:breakDecomposition4}. If $M,N$ are two $P$-modules and
	$\phi:M\rightarrow N$ an $H$-equivariant map, then $\phi$
	restricts to $\phi(x):M(x)\rightarrow N(x)$. This is obvious for
	$x=0$. If $x_i>0$ is one of the finitely many elements with the
	property that 
	$H(x_i)\supsetneqq H(x_i+)$, then for $m\in M(x_i)$, we immediately
	see that $\pi(x_i)\phi(m)=0$, and $\pi(x_i+)\phi(m)=\phi(m)$. Hence
	$\phi(m)\in N(x_i)$.  
	
	This also implies
	\ref{item:breakDecomposition4}: if $y\neq x$, then in the decomposition 
	$M(x)=\bigoplus_{y\in \R_{\geq 0}}(M(x))(y)$ of the $P$-module $M(x)$,
	we have $(M(x))(y)=0$ if $y\neq x$. Thus there are no nonzero
	$P$-morphisms $M(y)\rightarrow M(x)$, unless $y=x$.

	To see that the functor $M\mapsto M(x)$ is exact, let $0\rightarrow M'\rightarrow
	M\rightarrow M''\rightarrow 0$ be a short exact sequence of
	$P$-modules. Clearly $M'(x)\subset M(x)$, and $M(x)\twoheadrightarrow
	M''(x)$, as the map $M\rightarrow M''$ is surjective, and
	$M(y)\rightarrow M''(x)$ is the zero map if $y\neq x$. To see that the
	sequence $0\rightarrow M'(x)\rightarrow M(x)\rightarrow
	M''(x)\rightarrow 0$ is exact in the middle, we can use the same
	argument: $M'$ maps surjectively onto the kernel of $M\rightarrow
	M''$.
\end{proof}

\begin{definition}\label{def:breakDecomposition}
	If $M$ is a $P_K$-module, then the decomposition from Proposition
	\ref{prop:breakDecomposition} is called the \emph{break
	decomposition}\index{break decomposition} of $M$.
\end{definition}

\begin{corollary}\label{cor:preparationForSwan}Let $A$ be a $\Z[1/p]$-algebra, and $M$ an $A$-module, on
	which $P=P_K$ acts $A$-linearly through a finite quotient, i.e.~the representation $P\rightarrow
	\Aut_\Z(M)$ factors through a finite quotient and through $\Aut_A(M)\subset \Aut_\Z(M)$.
	\begin{enumerate}[label=\emph{(\alph*)},ref=(\alph*)]
		\item In the break decomposition $M=\bigoplus_{x\in
				\R_{\geq 0}}M(x)$, every $M(x)$ is an
				$A$-submodule of $M$.
		\item If $B$ is an $A$-algebra, then the break decomposition
			of $B\otimes_A M$ is $\bigoplus_{x\in \R_{\geq
			0}}B\otimes_AM(x)$.
		\item If the $\Z[1/p]$-algebra $A$ is local and noetherian,
			and $M$ a free $A$-module of finite rank, then every
			$M(x)$ is free of finite rank.
	\end{enumerate}
\end{corollary}
\begin{proof}
	\begin{enumerate}[label={(\alph*)}]
		\item If $a\in A$, then multiplication by $a$ induces a
			$P$-equivariant endomorphism of $M$, hence by
			Proposition \ref{prop:breakDecomposition},
			\ref{item:breakDecomposition3}, multiplication by $a$
			maps $M(x)$ to $M(x)$ for every $x\in \R_{\geq 0}$.
		\item This is clear by construction: the idempotents
			$\pi(x),\pi(x+)$ used in the construction of the
			decomposition also exist in $A[P]$ and $B[P]$.
		\item If $M$ is a free $A$-module of finite rank, then if $A$
			is noetherian, its
			$A$-submodule $M(x)$ is projective of finite rank.
			Hence if $A$ is local, $M(x)$ is free of
			finite rank.
	\end{enumerate}
\end{proof}

\begin{definition}\label{def:Swan-conductor}
	If $A$ is a local noetherian $\Z[1/p]$-algebra and $M$ a free $A$-module
	of finite rank on which $P_K$ acts $A$-linearly through a finite
	quotient, then the number $\rank_AM(x)$ is called
	\emph{multiplicity\index{multiplicity of a
	break}\index{break!multiplicity}
	of $x$}. If $\rank_AM(x)>0$, then $x$ is called \emph{break of
	$M$}\index{break}.

	The real number 
	\[\Swan(M):=\sum_{x\geq 0}x\rank_A(M(x))\]
	is called the \emph{Swan conductor of $M$}.
\end{definition}

\begin{remark}\label{rem:SwanAndBasechange}
	\begin{enumerate}[label={(\alph*)}]
		\item With this definition, it is obvious that $\Swan(M)=0$ if and only if
	the action of $P_K$ on  $M$ is trivial, i.e.~if and only if $M^{P_K}=M$.
\item If $B$ is an $A$-algebra, then $\Swan(M)=\Swan(M\otimes_A B)$ by
	Corollary \ref{cor:preparationForSwan}.
\item $\Swan(M)$ is additive in short exact sequences by Proposition
	\ref{prop:breakDecomposition}, \ref{item:breakDecomposition3}.
	\end{enumerate}
\end{remark}

Now we return to  representations of the larger groups $G_K^0$ or $G_K$.
\begin{lemma}\label{lemma:breakDecompositionForGReps} Let $M$ be a $\Z[1/p]$-module on which $G_K^0$
	(resp.~$G_K$) acts such that the restricted action of $P_K$ on $M$
	factors through a finite quotient. Then the break decomposition
	$M=\bigoplus_{x\geq 0}M(x)$ is a decomposition of $G_K^0$-modules
	(resp.~$G_K$-modules).
\end{lemma}
\begin{proof}
	We only do the proof for $G_K$; the argument for $G_K^0$ is identical.
	Let $H\subset \Aut_{\Z}(M)$ denote the finite image of $P_K$, and as
	in
	the proof of Proposition \ref{prop:breakDecomposition}, for $x>0$ let
	$H(x)$ be the image of $G_K^x$, and for $x\geq 0$ let $H(x+)$ be the
	image of $G_K^{x+}$. We again  get idempotents $\pi(x)$ and $\pi(x+)$, and
	$\pi(0+)$. As $P_K$
	is a normal subgroup of $G_K$, the elements in the image of $G_K$ in
	$\Aut_\Z(M)$ commute with the $\pi(x)$, $\pi(x+)$:
	\[g\pi(x+)=\left(\frac{1}{|H(x+)|}\sum_{h\in H(x+)
	}ghg^{-1}\right)g=\pi(x+)g,\]
	and analogously for $\pi(x)$. As $M(x)=\{m\in M|
	\pi(x+)(1-\pi(x))m=m\}$, the claim follows.
\end{proof}

\subsubsection{Integrality}
Now let us return to $\ell$-adic representations. We denote by $K$
a complete discretely valued field with perfect residue field $k$ of
characteristic $p>0$. We write $G_K=\Gal(K^{\sep}/K)$ with respect to a
fixed separable closure $K^{\sep}$ of $K$, and $P_K=G_K^{0+}$ for its wild
inertia group, see the discussion preceding Definition
\ref{defn:tameRepresentation} and Lemma
\ref{lemma:absoluteUpperNumberingFiltration}.

Let $E$ be a finite extension of $\Q_{\ell}$ with $\ell$ a prime different from $p$. 
\begin{definition}
	If $V$ is a finite dimensional $E$-vector space and
	$\rho:G_K\rightarrow \GL(V)$ a continuous representation, then we know
	by Lemma \ref{lemma:proPfactorsThroughFiniteQuotient} that the
	restriction of $\rho$ to $P_K$ factors through a finite group. By
	Lemma \ref{lemma:breakDecompositionForGReps} and Corollary
	\ref{cor:preparationForSwan} the break decomposition
	\[V=\bigoplus_{x\in \R_{\geq 0}}V(x)\]
	is a decomposition of continuous $E$-representations of $G_K$.
	(For the continuity observe that $\rho$ factors through
	$\bigoplus_{x\geq 0}\GL(V(x))$, which carries the subspace topology in
	$\GL(V)$).

	The real number $\Swan(V)$ is called the \emph{Swan conductor of
	$V$}\index{Swan conductor}. By definition, $\Swan(V)=0$ if and only if
	$V$ is tame.
\end{definition}

If $V$ is a finite dimensional $E$-vector space and  $\rho:G_K\rightarrow
\GL(V)$ a continuous representation, then according to Lemma
\ref{lemma:lattice} there exists a free finite rank $\mathcal{O}_E$-submodule
$\mathcal{V}\subset V$ with $\rank_{\mathcal{O}_E} \mathcal{V}=\dim_E V$, such that $\rho$ factors through
$\GL(\mathcal{V})\subset \GL(V)$. We then get an $\mathcal{O}_E$-linear,
$G_K$-stable break decomposition 
\[\mathcal{V}=\bigoplus_{x\geq 0}\mathcal{V}(x)\]
and Corollary \ref{cor:preparationForSwan} shows that
$V(x)=\mathcal{V}(x)\otimes_{\mathcal{O}_E} E$. It follows that
$\Swan(V)=\Swan(\mathcal{V})$.

Similarly, if $\F_{\lambda}$ is the (finite) residue field of $E$ and
$\overline{\mathcal{V}}:=\mathcal{V}\otimes_{\mathcal{O}_E} \F_{\lambda}$ denotes the reduction of
$\mathcal{V}$, then we obtain an $\F_{\lambda}$-linear,
$G_K$-stable break decomposition
\[\overline{\mathcal{V}}=\bigoplus_{x\geq 0}\overline{\mathcal{V}}(x).\]
Again it follows from Corollary \ref{cor:preparationForSwan} that 
$\overline{\mathcal{V}}(x)=\mathcal{V}(x)\otimes_{\mathcal{O}_E}\F_{\lambda}$,
and hence 
that $\Swan(V)=\Swan(\mathcal{V})=\Swan(\overline{\mathcal{V}})$.

The final goal of this section is to prove that $\Swan(V)$ is an
integer.

\begin{theorem}\label{thm:integralityOfSwanConductor}
	If $V$ is a finite dimensional $E$-vector space, and
	$\rho:G_K\rightarrow \GL(V)$ a continuous representation, then
	$\Swan(V)=b(V)$  (Definition \ref{defn:SerresB}). 
\end{theorem}
\begin{remark}
	For $\Swan(V)$ to be an integer, it is necessary that
	$V$ really is a representation not just of $P_K$ but (at least) of
	$G_K^0$. For representations of $P_K$ it can still be shown that their
	Swan conductors are rational numbers, see \cite[p.~214]{Katz/IndependenceOfL}.
\end{remark}
\begin{proof}[Proof of Theorem \ref{thm:integralityOfSwanConductor}]
	We compute the number $\Swan(\overline{\mathcal{V}})$, where
	$\overline{\mathcal{V}}=\mathcal{V}\otimes_{\mathcal{O}_E}\F_\lambda$
	with $\mathcal{V}$ a free $\mathcal{O}_E$-module with continuous
	action $\rho_{\mathcal{O}_E}:G_K\rightarrow \GL(\mathcal{V})$ such
	that $\rho$ factors through $\rho_{\mathcal{O}_E}$. As the residue field
	$\F_{\lambda}$ of $E$ is a finite field, $\GL(\overline{\mathcal{V}})$ is a finite group,
	and consequently $G_K$ acts on $\overline{\mathcal{V}}$ through a
	finite quotient $G$, which corresponds to a finite Galois extension $L/K$.
	If $\overline{\mathcal{V}}=0$ then there is nothing to do, so assume
	$\overline{\mathcal{V}}\neq 0$.

	Let $x\in \R_{\geq 0}$ be an element such that	$\overline{\mathcal{V}}(x)\neq 0$. Then $G^x\neq G^{x+\epsilon}$ for all $\epsilon >0$ and
	hence $G_{\lceil\psi_{L/K}(x)\rceil}\neq
	G_{\lceil\psi_{L/K}(x)+\epsilon\rceil}$ for all $\epsilon>0$,
	so $\psi_{L/K}(x)\in \Z_{\geq 0}$. By the computation following Definition
\ref{defn:HerbrandFunction}, this means that
\[x=\phi_{L/K}(\psi_{L/K}(x))=\sum_{i=1}^{\psi_{L/K}(x)}\frac{|G_i|}{|G_0|}.\]
On the other hand, for an integer $i\in [1,\psi_{L/K}(x)]$ we have
$G^x=G_{\psi_{L/K}(x)}\subset G_i$, so 
	\[\overline{\mathcal{V}}(x)^{G_i}\subset
		\overline{\mathcal{V}}(x)^{G^x}=0.\]
		If $i>\psi_{L/K}(x)$, then $\phi_{L/K}(i)>x$, and hence
		$\overline{\mathcal{V}}(x)^{G_i}=\overline{\mathcal{V}}(x)$.
	Consequently we compute
	\[\dim_{\F_{\lambda}}(\overline{\mathcal{V}}(x)/\overline{\mathcal{V}}(x)^{G_i})=\begin{cases}\dim_{\F_{\lambda}}\overline{\mathcal{V}}(x)&i\leq
			\psi_{L/K}(x)\\0&i>\psi_{L/K}(x),\end{cases}
			\]
			and thus
			\begin{align*}
				\Swan(\overline{\mathcal{V}}(x))&=x\dim_{\F_{\lambda}}(\overline{\mathcal{V}}(x))\\
				&=\sum_{i=1}^{\infty}\frac{|G_i|}{|G_0|}\dim_{\F_{\lambda}}(\overline{\mathcal{V}}(x)/\overline{\mathcal{V}}(x)^{G_i})
			\end{align*}
	Both sides of this equation are additive with respect to direct sums,
	so we finally get
	\[\Swan(\overline{\mathcal{V}})=\sum_{i=1}^{\infty}\frac{|G_i|}{|G_0|}\dim_{\F_{\lambda}}(\overline{\mathcal{V}}/\overline{\mathcal{V}}^{G_i}).\]
	This shows that
	\[\Swan(\overline{\mathcal{V}})=b(V),\]
	see Proposition \ref{prop:formulaForSerresB}. In particular, note that
	the left-hand side and hence the right-hand side is independent of the
	choice of $G$.
\end{proof}

\begin{example}\label{ex:ASSwan}
	We continue Example \ref{ex:ASlowerNumbering}. Consider the field
	$K:=k\llparen x\rrparen$, where $k$ is an algebraically closed field of
	characteristic $p>0$. Let $L:=K[t]/(t^p-t-x^{-m})$. The Galois
	group of this extension is $G=\F_p$. 

	Assume that $(m,p)=1$, pick a prime $\ell\neq
	p$, and an arbitrary nontrivial $1$-dimensional representation $\chi:\F_p\rightarrow
	\Q_{\ell}^\times$. The reduction $\overline{\chi}:\F_p\rightarrow
	\F_\ell^{\times}$ is nontrivial, and we compute $\Swan(\chi)$: In
	Example \ref{ex:ASlowerNumbering} we saw that
	$G_0=G_1=\ldots=G_m\supsetneq G_{m+1}=0$.
	It follows that
	\[\Swan(\chi)=\Swan(\overline{\chi})=\sum_{i=1}^{\infty}\frac{|G_i|}{|G_0|}\dim_{\F_{\ell}}(\overline{\chi}/\overline{\chi}^{G_i})=\sum_{i=1}^m1=m\]
\end{example}

\part{Generalities about \'etale topology and $\ell$-adic sheaves}

\section{Recollection on \'etale morphisms}
We recall the definition of an \'etale morphism and state some of its properties without proof. 
The general references for this section are \cite{EGA4} and \cite{SGA1}, see also \cite[I, \S 3]{Milne/EtaleCohomologyBook}.

\begin{definition}\label{def-etale-algebra}
Let $A$ be a ring (commutative with $1$ as always). We say that an $A$-algebra $B$ is {\em \'etale}\index{etale@\'etale algebra} (or that $B$ is \'etale over $A$)
if $B$ is finitely presented as an $A$-algebra and one of the following equivalent conditions is satisfied:
\begin{enumerate} 
\item $B$ is formally \'etale over $A$, i.e for all $A$-algebras $C$ and all nilpotent ideals $I\subset C$ the map
                \[\Hom_{A-{\rm alg}}(B, C)\to \Hom_{A-{\rm alg}}(B, C/I),\quad  \varphi\mapsto \varphi \text{ mod }I\]
                 is bijective.
\item For all $A$-algebras $C$ and all ideals $I\subset C$ with $I^2=0$ the map
                \[\Hom_{A-{\rm alg}}(B, C)\to \Hom_{A-{\rm alg}}(B, C/I),\quad  \varphi\mapsto \varphi \text{ mod }I\]
                 is bijective.
\item $B$ is flat as an $A$-module and is unramified, i.e. $\Omega^1_{B/A}=0$.
\item $B$ is flat as an $A$-module and for each prime ideal $\fq\subset B$ 
        the natural map $k(\fp)=A_{\fp}/\fp A_{\fp}\to B_{\fq}/\fp B_{\fq}$ is
	a separable extension of fields,
       where $\fp:=A\cap \fq$. 
\item\label{def-etale-algebra-presentation} 
    If $B=A[x_1,\ldots, x_n]/I$ is a presentation of $B$, then for all prime ideals $\fp\subset A[x_1,\ldots, x_n]$
        with $\fp \supset I$, there exist polynomials $f_1,\ldots, f_n\in I$ such that 
             \[I_{\fp}= (f_1,\ldots, f_n)\subset A[x_1,\ldots, x_n]_{\fp}\quad\text{and}\quad
                     {\rm det}\left(\frac{\partial f_i}{\partial x_j}\right)_{i,j}\not\in \fp. \]
\end{enumerate}
\end{definition}

\begin{definition}\label{etale-morphism}
A morphism of schemes $f:X\to Y$ is {\em \'etale} \index{etale@\'etale morphism} if for any point $x\in X$ with image 
$y=f(x)\in Y$ there exist open neighborhoods $V=\Spec B$ of $x$ and $U=\Spec A$ of $y$ such that
$f$ restricts to a morphism $f_{|V}:V\to U$ and the induced map $A\to B$ makes
$B$ into an \'etale $A$-algebra.
\end{definition}

\begin{example}\label{ex-etale}\leavevmode
\begin{enumerate}
\item Let $k$ be a field and $A$ a $k$-algebra. Then $A$ is \'etale over $k$ iff 
          $A\cong \prod_{i\in I} L_i$, where $I$ is some finite index set and $L_i/k$ is a finite separable extension for every $i\in I$.
\item Open immersions are \'etale.
 \item\label{ex-etale:unramified}  In the situation of Definition \ref{def:unramified} the $A$-algebra $B$ is \'etale iff $L/K$ is unramified.
\end{enumerate}
\end{example}

\begin{proposition}[{\cite[(17.3.3), (17.3.4),
	(17.7.3)]{EGA4}}]\label{prop-etale}\leavevmode
\begin{enumerate}
\item The composition of \'etale morphisms is \'etale.
\item Let $Y'\to Y$ be a morphism of schemes.
        If $f: X\to Y$ is \'etale, then the base change map $X\times_{Y}Y'\to Y'$ is \'etale.
        The converse is true if $Y'\to Y$ is faithfully flat and quasi-compact (the latter condition is satisfied if e.g. 
         $Y'$ is noetherian.)
\item If $f:X\to S$ and $g:Y\to S$ are \'etale, so is any $S$-morphism $X\to Y$.
\end{enumerate}
\end{proposition}

\begin{proposition}[{\cite[Exp. I, Cor 5.4]{SGA1}}]\label{prop-determination-etale-maps}
Consider a diagram
\[\begin{tikzcd}
			X\ar[swap]{dr}{a}\ar{rr}{f,\,g}& &Y\ar{dl}{b}\\
                       &S,&
		\end{tikzcd}\]
where $X, Y, S$ are locally noetherian and connected, $b$ is separated and \'etale and $a=b\circ f= b\circ g$.
Assume there exists a point $x\in X$ with $f(x)=g(x)=:y$ and that $f$ and $g$ induce the same morphism
 $k(y)\to k(x)$ on the residue fields. Then $f=g$.
  \end{proposition}


\section{The \'etale fundamental group}\label{sec:fundamentalGroup}
In this section let $X$ be a connected noetherian scheme. We partly follow
\cite{SGA1} and \cite{Szamuely}.
\subsection{Definition and basic
properties}\label{sec:fundamentalGroup:basics}
\begin{definition}
	\begin{enumerate}
		\item An \emph{\'etale
			covering}\index{covering!\'etale}\index{\'etale
			covering} of $X$ is a surjective finite \'etale
			morphism $f:Y\rightarrow X$. The category of such finite
			\'etale morphisms is denoted $\FEt_X$.
		\item A \emph{geometric point of $X$}\index{geometric point} is a morphism $\bar{x}:\Spec
			\Omega\rightarrow X$, with $\Omega$ an algebraically closed field.
			Note that giving $\bar{x}$ is equivalent to giving a point $x$ of $X$
			and an embedding of the residue field $k(x)$ into $\Omega$.
		\item If $f:Y\rightarrow X$ is a finite \'etale cover and
			$\bar{x}:\Spec \Omega_1\rightarrow X$, $\bar{y}:\Spec
			\Omega_2\rightarrow Y$ geometric points, we say that
			\emph{$\bar{y}$ lies over $\bar{x}$} if there exists
			an algebraically closed field $\Omega$ containing both
			$\Omega_1$ and $\Omega_2$, such that the diagram
			\begin{equation*}
				\begin{tikzcd}
					\Spec \Omega\ar{rd}\rar&\Spec
					\Omega_2\rar{\bar{y}}&Y\dar{f}\\
					&\Spec \Omega_1\rar{\bar{x}}&X
				\end{tikzcd}
			\end{equation*}
			commutes. For the geometric point $\Spec
			\Omega_2\xrightarrow{\bar{y}} Y \xrightarrow{f} X$ we
			also write $f(\bar{y})$.
		\item If $\bar{x}:\Spec \Omega\rightarrow X$ is a geometric point of $X$, we denote by
			$\Fib_{\bar{x}}$ the functor
			\[\Fib_{\bar{x}}:\FEt_X\rightarrow \FinSet,\;
				(Y\rightarrow X)\mapsto \Hom_X(\Spec \Omega,
				Y).\]
				Note that $\Fib_{\bar{x}}(Y\rightarrow X)$ can
				be naturally 
				identified with the finite set
				underlying the geometric fiber
				$Y_{\bar{x}}:=Y\times_X \Spec \Omega$.
		\item If $\bar{x}$ is a geometric point of $X$, then
			the group $\Aut(\Fib_{\bar{x}})$ is called the
			\emph{\'etale fundamental group of $X$ with
				respect to the base point $\bar{x}$},
				and it is denoted
				$\pi_1^{\et}(X,\bar{x})$.\index{\'etale
				fundamental group} When confusion is unlikely,
				we will also write $\pi_1(X,\bar{x})$ for the
				\'etale fundamental group.
	\end{enumerate}
\end{definition}

Grothendieck defined the \'etale fundamental group in \cite{SGA1}, and we will
summarize its main properties in this section.

\begin{theorem}[{Grothendieck (\cite[V.7]{SGA1})}]
	Let $X$ be a connected noetherian scheme and $\bar{x}$ a geometric
	point. The \'etale fundamental group $\pi_1^{\et}(X,x)$ is a profinite group and the
	functor $\Fib_{\bar{x}}$ induces an equivalence
	\[\FEt_X\xrightarrow{\cong}\left\{\text{finite sets with continuous
	}\pi_1^{\et}(X,\bar{x})\text{-action}\right\}.\]
	In particular, open subgroups of $\pi^{\et}_1(X,\bar{x})$ correspond
	to (pointed) finite \'etale coverings, and open normal subgroups $U$ correspond to
	(pointed) Galois coverings with Galois group
	$\pi_1^{\et}(X,\bar{x})/U$ (see below).
\end{theorem}
This theorem has been extended to the non-noetherian case in \cite{Szamuely}.
\begin{definition}
	A finite \'etale covering $f:Y\rightarrow X$ is called \emph{Galois
	covering}\index{covering!Galois} if $Y$ is connected and if $\Aut_X(Y)$ acts transitively on
	$\Fib_{\bar x}(Y)$. If $f$ is Galois, then $\Aut_{X}(Y)$ is called the
	\emph{Galois group
	of $f$}.
\end{definition}
The property of being Galois is independent of the choice of the base point by
Proposition
\ref{prop-determination-etale-maps}. By the same proposition $\Aut_X(Y)$ is a finite
group, and if $Y$ is connected then $\Aut_X{(Y)}$ acts transitively on
$\Fib_{\bar{x}}(Y)$ if and only if $|\Aut_X{(Y)}|=\deg(Y/X)$.

Alternatively, let $G$ be a finite group, and consider the scheme $G_X:=\coprod_{g\in
G}X_{(g)}$, where $X_{(g)}=X$.
This is the constant $X$-group scheme associated with $G$. If $f:Y\rightarrow
X$ is a finite \'etale covering, then a homomorphism $G\rightarrow \Aut_X(Y)$
gives rise to an action $\phi:G_X\times_X Y\rightarrow Y$. If $Y$ is connected,
then $f$ is a Galois covering with group $G$ if and only if the morphism
\[(\id, \phi):G_X\times_X Y\rightarrow Y\times_X Y \]
is an isomorphism, i.e.~if $X$ is a \emph{$G$-torsor}.
\begin{proposition}[{\cite[\S 1.5]{Serre/Galois}, see also \cite[Prop.~5.3.9]{Szamuely}}]\label{prop:galoisclosure}
	Let $X$ be a connected noetherian scheme and $f:Y\rightarrow X$ a
	finite \'etale covering. There exists a Galois
	covering $g:P\rightarrow X$, factoring through $f$.
	
	If $Y$ is connected then one finds $g$, such that any
	other Galois covering $g':P'\rightarrow X$ which factors through $f$
	also factors through $g$. More succinctly, we have the following
	diagram:
	\begin{equation*}
		\begin{tikzcd}[column sep=huge]
			P'\ar[swap]{ddr}{g'}\ar{dr}\ar[dashed]{r}{\exists}& P\ar{d}\ar[bend
			left=50]{dd}{g}\\
			& Y\dar{f}\\
			& X
		\end{tikzcd}
	\end{equation*}
\end{proposition}

\begin{example}\label{ex:fundamentalGroupOfAField}
	Let $K$ be a field and $X=\Spec K$. Giving a finite \'etale covering
	$f:Y\rightarrow X$ is equivalent to giving a finite \'etale
	$K$-algebra $A$, and setting $Y=\Spec A$. The $K$-algebra $A$ is a
	finite product of separable extensions of $K$, see Example
	\ref{ex-etale}. Hence $Y$ is connected if and only if $A$ itself is a
	separable field extension of $K$, and $f:Y\rightarrow X$ is a Galois
	covering if and only if $A/K$ is a Galois extension.
\end{example}

To see that $\pi_1^{\et}(X,\bar{x})$ is a profinite group in a natural way, Grothendieck
proceeds as follows. Let $\mathbf{\Gal}_X$ be the set of Galois coverings of $X$.
Let $I$ be an index set, so that we can write
$\mathbf{\Gal}_X=\{f_{\alpha}:P_{\alpha}\rightarrow X|\alpha\in I\}$. We
define an ordering on $I$ by writing $\alpha\leq \beta$ if and only if there
exists an $X$-morphism $P_{\beta}\rightarrow P_{\alpha}$. By Proposition
\ref{prop:galoisclosure} this system is directed: For $\alpha,\beta\in I$
every connected component of the
fiber product $P_{\alpha}\times_XP_{\beta}$ is an \'etale covering of $X$. By
Proposition \ref{prop:galoisclosure} we find $f_{\gamma}:P_{\gamma}\rightarrow
X$ factoring through $P_{\alpha}\times_XP_{\beta}$, so 
$\alpha,\beta\leq \gamma$. Now we fix arbitrary geometric points $p_{\alpha}\in
\Fib_{\bar{x}}(P_{\alpha})$. By Proposition
\ref{prop-determination-etale-maps}, for $\alpha\leq \beta\in I$, there exists a
\emph{unique} map $\phi_{\beta,\alpha}:P_{\beta}\rightarrow P_{\alpha}$,
mapping $p_{\beta}$ to $p_{\alpha}$.  The system
$(P_{\alpha},\phi_{\beta,\alpha})$ is a directed projective system. Note that
its inverse limit exists in the category of schemes because the maps
$\phi_{\beta,\alpha}$ are finite, and hence affine
(\cite[Prop.~8.2.3]{EGA4}), but in general it is not \'etale
over $X$, as it is usually not locally of finite type. Now let
$f:Y\rightarrow X$ be an arbitrary finite \'etale covering. For every
$\alpha\in I$ we obtain a map $\hom_X(P_{\alpha},Y)\rightarrow
\Fib_{\bar{x}}(Y)$, $g\mapsto g(p_{\alpha})$.  These maps give rise to a
map
\[\varinjlim_{I}\hom(P_{\alpha},Y)\rightarrow
	\Fib_{\bar{x}}(Y),\]
which is easily seen to be a bijection. In Grothendieck's words,
$\Fib_{\bar{x}}$ is \emph{strictly pro-representable}.

Moreover, it is not difficult to see that every automorphism of
$\Fib_{\bar{x}}$ comes from a unique automorphism of the projective system
$(P_\alpha,\phi_{\alpha,\beta})$: If $\psi\in \Aut(\Fib_{\bar{x}})$ then
for every $\alpha\in I$, there exists a unique automorphism of $P_{\alpha}$
mapping $p_{\alpha}$ to $\psi(p_{\alpha})$, as the $P_{\alpha}$ are Galois.
It follows that
\begin{equation}\label{eqn:proRepresentability}\pi_1^{\et}(X,\bar{x})=\Aut(\Fib_{\bar{x}})=\Aut\left(\varinjlim_I
	\hom_X(P_{\alpha},-)\right)=\varprojlim
	\Aut_X(P_\alpha)^{\opp},\end{equation}
where $(-)^{\opp}$ denotes the opposite group.

\begin{example}\label{ex:fundamentalGroupBasicExample}
	\begin{enumerate}
		\item Let us bring the above construction down to earth: Let $K$ be a field
	and write $X=\Spec K$. Choosing a
	geometric point $\bar{x}$ of $\Spec K$  amounts to choosing an embedding of $K$ into an
	algebraically closed field $\Omega$. Without loss of generality we
	take $\Omega=K^{\alg}$, an algebraic closure of $K$.  The Galois coverings
	$P_{\alpha}$ of $X$ in the above construction correspond to finite
	Galois extensions $L_{\alpha}$. Fixing a geometric point $p_{\alpha}$
	lying over $\bar{x}$ is equivalent to fixing a $K$-embedding of
	$L_{\alpha}$ into ${K}^{\alg}$. The transition maps
	$\phi_{\alpha,\beta}$ then correspond to embeddings $L_{\alpha}\subset
	L_{\beta}\subset K^{\alg}$, so the projective system
	$(P_{\alpha},\phi_{\alpha,\beta})$ corresponds to the inductive system
	of finite Galois extension of $K$ \emph{contained in $K^{\alg}$}.
	
	Finally, \eqref{eqn:proRepresentability} shows that
	\[\pi_1^{\et}(X,\bar{x})=\varprojlim
		\Aut_X(\Spec L_{\alpha})^{\opp}=\varprojlim_{\begin{subarray}{c}L\subset
			K^{\alg}\\L/K\text{ finite Galois}\end{subarray}}
		\Aut_K(L)=\Gal(K^{\alg}/K).\]
	\item If $X$ is a connected normal noetherian scheme, write $K(X)$ for
		its function field, i.e.~for the stalk at its generic point.
		Fix an algebraic closure $K(X)^{\alg}$, and let
		$\bar{\eta}:\Spec K(X)^{\alg}\rightarrow X$ denote the
		associated geometric point of $X$.
		If $K(X)^{\unr}$ is the composite of all finite Galois extensions
		$L/K(X)$, such that the normalization of $X$ in $L$ is \'etale
		over $X$, then $K(X)^{\unr}$ is an (usually infinite) Galois extension
		of $K(X)$ and there is a natural isomorphism 
		\[\Gal(K(X)^{\unr}/K(X))\cong \pi^{\et}_1(X,\bar{\eta}).\]
\end{enumerate}
\end{example}
\begin{proposition}[Dependence on base
	point]\label{prop:basePoint}
	Again $X$ is noetherian and connected.
	If $\bar{x}$ and $\bar{x}'$ are two geometric points of $X$, then
	there is a continuous isomorphism
	\[\pi_1^{\et}(X,\bar{x})\cong\pi_1^{\et}(X,\bar{x}'),\]
	which is canonical up to inner automorphism.
\end{proposition}
Proving the proposition amounts to constructing an isomorphism of the inverse
system $(P_{\alpha},\phi_{\alpha,\beta})$ constructed using $\bar{x}$ and the
analogous inverse system constructed using $\bar{x}'$, which is not too
difficult.

From the Definition of the fundamental group we easily see that it is
functorial for schemes equipped with a geometric point.
\begin{proposition}[Functoriality]\label{prop:functorialityOfFundamentalgroup}
	If $f:X'\rightarrow X$ is a morphism of connected noetherian schemes,
	and $\bar{x}':\Spec \Omega\rightarrow X'$ a geometric point, then $f$
	induces a continuous homomorphism of groups
	\[\pi^{\et}(X',\bar{x}')\rightarrow \pi_1^{\et}(X,f(\bar{x}')).\]
\end{proposition}
\begin{proof}
Let $Z\rightarrow X$ be a finite \'etale covering. Then $Z'=Z\times_X
X'\rightarrow X'$ is an \'etale covering, and
$\Fib_{\bar{x}'}(Z')=\Fib_{f(\bar{x}')}(Z)$. It follows that any automorphism of
$\Fib_{\bar{x}'}$ induces an automorphism of $\Fib_{f(\bar{x}')}$.
\end{proof}
\begin{example}
	Let $X$ be connected and normal with function field $K(X)$. Fix an
	algebraic closure $K(X)^{\alg}$ of $K(X)$, and write $\bar{\eta}$ for the
	associated geometric point. With this setup Example
	\ref{ex:fundamentalGroupBasicExample} shows that
\[\Gal(K(X)^{\alg}/K(X))=\pi^{\et}_1(\Spec(K(X)),\bar{\eta})\rightarrow
	\pi^{\et}_1(X,\bar{\eta})\] is surjective.
\end{example}
\subsection{Complex varieties}
If $X$ is a scheme of finite type over the field of complex numbers $\C$, then
to $X$ one associates a complex analytic space $X^{\an}$. Its underlying
set is $X(\C)$. If $X$ is a smooth complex variety, then
$X^{\an}$ is a complex manifold. In this section we summarize how the \'etale
fundamental group of $X$ and the fundamental group of the topological space
$X^{\an}$ are related. For details and proofs we refer the reader to
\cite[Exp.~XII]{SGA1}.

If $f:Y\rightarrow X$ is a finite \'etale covering, then
$f^{\an}:Y^{\an}\rightarrow X^{\an}$ is a covering space of $X^{\an}$ with
finite fibers. This defines a functor from the category $\FEt_X$ to the
category of covering spaces of $X^{\an}$ with finite fibers. The Riemann
Existence Theorem (\cite[Exp.~XII, Thm.~5.1]{SGA1}) states that this functor is
in fact an equivalence. We will use the following reformulation in terms of
fundamental groups.

\begin{theorem}\label{thm:complexFundamentalGroup}
	Let $X$ be a connected finite type $\C$-scheme and $x\in X(\C)$.
	Then there is a natural continuous isomorphism
	\[{\pi_1(X^{\an},x)}^{\wedge}\xrightarrow{\cong}\pi_1^{\et}(X,x),\]
	where the left-hand side is the \emph{profinite completion} of the
	fundamental group of the topological space $X^{\an}$, based at $x$.
\end{theorem}
One uses this theorem to prove:
\begin{corollary}\label{cor:fundamentalGroupOfComplexCurve}
	Let ${X}$ be a smooth complex curve of genus $g$. Define
	$U:={X}\setminus\{x_1,\ldots x_n\}$ as complement of finitely many
	closed points. If $\bar{u}\in U$ is a geometric point, then
	the \'etale fundamental group $\pi_1^{\et}(U,\bar{u})$ is isomorphic
	to the profinite completion of the group
	\[\left<a_1,\ldots, a_g,b_1,\ldots,
		b_g,c_1,\ldots, c_n\left|
		[a_1,b_1]\cdot\ldots\cdot[a_g,b_g]c_1\cdot\ldots\cdot
		c_n=1\right.\right>.\]
More generally the corollary is true over any algebraically closed field of
characteristic $0$, \cite[Exp.~XIII, Cor.~2.12]{SGA1}. 
\end{corollary}

Similarly, since the fundamental group of a smooth quasi-projective complex variety
is a finitely generated group (this can be deduced from the curve case via
Lefschetz type theorems), we obtain:
\begin{corollary}
	If $X$ is a smooth complex quasi-projective variety, and $x\in X(\C)$, then
	$\pi^{\et}_1(X,x)$ is topologically finitely generated.
\end{corollary}

With considerable effort, 
the theory of specialization maps for fundamental groups allows to generalize
the corollary:
\begin{theorem}
	Let $k$ be an algebraically closed field and $X$ a smooth
	quasi-projective $k$-variety
	with geometric point $\bar{x}$. Then $\pi_1^{\et}(X,\bar{x})$ is
	topologically finitely generated in each of the following two cases:
	\begin{itemize}
		\item $\Char(k)=0$ (implicit in \cite[p.291]{SGA1}, see also
			\cite[Thm.~1]{Kollar}\emph{)}.
		\item $\Char(k)>0$ and $X$ is proper \emph{(}\cite[Exp.~X,
			Thm.~2.9]{SGA1}\emph{)}.
	\end{itemize}
\end{theorem}

The picture is different for non-proper varieties in positive characteristic:
\begin{example}[\cite{Gille}]\label{ex:wildRamification}
	If $k$ is an algebraically closed field of characteristic
	$p>0$, $U$ a smooth affine curve over $k$ and $\bar{u}$ a geometric
	point of $U$, then the maximal
	pro-$p$-quotient of $\pi^{\et}_1(U,\bar{u})$ is a free pro-$p$-group on
	$\#k$ generators.
This can be seen by cohomological methods:
One checks that $H^i_{\et}(U,\F_p)=H^i(\pi^{\et}_1(U,\bar{u}),\F_p)$, for all
$i\geq 0$, where the
right-hand side is continuous group cohomology. Since $U$ is affine, we have
$H^2_{\et}(U,\F_p)=0$. The numbers of generators and relations of the maximal
pro-$p$-quotient of $\pi_1^{\et}(U,\bar{u})$  can be  described in terms
of the $\F_p$-dimension of the groups $H^1$ and $H^2$ above (\cite[\S4.]{Serre/GaloisCohomology}).
\end{example}

\subsection{Tame coverings}
The pathologies that appear in characteristic $p>0$, exemplified by Example
\ref{ex:wildRamification},  can partly be remedied by
restricting the attention to \emph{tame} coverings.

\begin{definition}\label{defn:tameness}
	\begin{enumerate}
		\item\label{def:tameness-wrt-emb} (\cite[Def.~2.2.2]{Grothendieck/Tame}) Let $X$ be a normal connected scheme and $U\subset X$ an open subset.
	If $f:V\rightarrow U$ is a finite \'etale covering, then $f$ is said
	to be \emph{tamely ramified along $X\setminus U$}, if and only if for
	every codimension $1$ point $\eta\in X\setminus U$, the extension
	$K(X)\subset K(V)$ is tamely ramified with respect to the discrete
	valuation defined by $\eta$.
\item\label{defn:tamenessInGeneral} (\cite[Sec.~4]{Kerz/tameness}) If $U$ is a normal connected separated scheme of finite type over a
	field, then a finite \'etale covering $f:V\rightarrow U$ is said to be
	\emph{tame}\index{covering!tame} if for all open immersions $U\hookrightarrow X$ with $X$ a
	proper connected normal $k$-scheme, $f$ is tamely ramified with
	respect to $X\setminus U$.

	If $U$ is not connected, then $f:V\rightarrow U$ is tame if the
	induced coverings of the components of $U$ are tame.
\item The category of tame coverings of $U$ is denoted $\FEt^{\tame}_U$. It is
	a full subcategory of $\FEt_U$.
	\end{enumerate}
\end{definition}
\begin{remark}
	\begin{enumerate}
	\item The normality assumption guarantees that the local ring
			$\mathcal{O}_{X,\eta}$ at a codimension $1$ point is a
			discrete valuation ring.
	\item If $U$ is a scheme over a field of characteristic $0$, then
		every finite \'etale covering is tame.
	\item  In Definition \ref{defn:tameness}, \ref{defn:tamenessInGeneral}, note that there always exists
		a normal connected \emph{proper} $k$-scheme containing $U$: By Nagata's
		compactification theorem (\cite{Luetkebohmert/Nagata}, \cite{Conrad/Nagata}) there exists a proper $X$
		containing $U$, which we can then normalize.
	\item In Definition \ref{defn:tameness}, \ref{defn:tamenessInGeneral}, if $\dim U=1$, then there
		exists a \emph{unique} normal proper curve $X$ containing $U$.
		So in this case tameness of coverings of $U$ only has to be
		checked on this single compactification.
	\end{enumerate}
\end{remark}
In \cite[Exp.~XIII, \S.2]{SGA1} Grothendieck defined the notion of a tame
covering, but only
with respect to a normal crossings divisor (see Definition \ref{def:ncd} and Theorem
\ref{thm:differentNotionsOfTameness}, \ref{item:grothendieck} below). In
characteristic $0$, the famous theorem of Hironaka on
resolution of singularities states that if $U$ is smooth then one can always find a
\emph{smooth} proper variety $X$ containing $U$, such that
$X\setminus U$ is a strict normal crossings divisor. Proving
the analogous statement in positive characteristic is an open
problem.
Theorem \ref{thm:differentNotionsOfTameness} below shows that Definition
\ref{defn:tameness}, \ref{defn:tamenessInGeneral}, agrees with Grothendieck's notion, whenever a comparison
is possible.  

\begin{definition}\label{def:ncd}
	Let $X$ be a locally noetherian scheme and $D$ an effective 
	Cartier divisor on $X$.
	\begin{enumerate}
	\item $D$ is said to have \emph{strict normal
		crossings}\index{strictly normal crossings} if for every $x\in
		\supp D$ the following hold:
		\begin{itemize}
			\item $\mathcal{O}_{X,x}$ is a regular local ring.
			\item There exists a regular sequence $t_1,\ldots,
				t_n$ and $m\in\{1,\ldots,n\}$,  such that the ideal
				$\mathcal{O}_X(-D)_x\subset
				\mathcal{O}_{X,x}$ is
				generated by $t_1\cdot\ldots\cdot t_m$. In other words: locally, $D$ 
				looks like the intersection of coordinate
				hyperplanes.
		\end{itemize}
	\item $D$ is said to have \emph{normal crossings}, if for every $x\in
		\supp D$ there exists an \'etale morphism $f:V\rightarrow X$,
		such that $x\in \image(f)$ and such that $f^*D$ has strict
		normal crossings. In other words: \'etale locally, $D$ looks
		like the intersection of coordinate hyperplanes.
	\end{enumerate}
\end{definition}

\begin{theorem}[{Kerz-Schmidt,
		\cite{Kerz/tameness}}]\label{thm:differentNotionsOfTameness}
	Let $f:V\rightarrow U$ be a finite \'etale covering of regular
	separated schemes over a field $k$. The following statements are
	equivalent:
	\begin{enumerate}
		\item\label{thm:tamenessA} $f$ is tame.
		\item\label{thm:tamenessB} For any regular $k$-curve $C$ and any morphism
			$\phi:C\rightarrow U$, the induced covering $V\times_U
			C\rightarrow C$ is tamely ramified along $\overline{C}\setminus
			C$, where $\overline{C}$ is the unique normal proper
			$k$-curve containing $C$ as an open subscheme.
	\end{enumerate}
	If there exists an open immersion $U\hookrightarrow X$ with $X$ a
	smooth, proper, separated, finite type $k$-scheme, such that
	$X\setminus U$ is a normal crossings divisor, then
	\ref{thm:tamenessA} and \ref{thm:tamenessB} are equivalent to
	\begin{enumerate}[resume]
		\item\label{item:grothendieck} $f$ is tamely ramified along $X\setminus U$.
	\end{enumerate}
\end{theorem}

\subsection{The tame fundamental group}
With essentially the same method as used in Section \ref{sec:fundamentalGroup:basics}
one proves:

\begin{theorem}[{\cite{Kerz/tameness}}]\label{thm:tameFundamentalGroup}
	Let $U$ be a finite type, connected, separated, normal $k$-scheme, and $\bar{u}$ a
	geometric point of $U$. Denote by $\Fib^{\tame}_{\bar{u}}$ the
	restriction of the functor $\Fib_{\bar{u}}$ to $\FEt^{\tame}_U$. Then
	the group $\pi_1^{\tame}(U,\bar{u}):=\Aut(\Fib^{\tame}_{\bar{u}})$ is profinite, and
	$\Fib_{\bar{u}}^{\tame}$ induces an equivalence
	\[\FEt^{\tame}_{U}\xrightarrow{\cong} \left\{\text{finite sets with continuous
	}\pi_1^{\tame}(U,\bar{u})\text{-action}\right\}.\]
\end{theorem}
\begin{definition}\label{def:tameFundamentalGroup}
	The profinite group $\pi_1^{\tame}(U,\bar{u})$ from the theorem is called the
	\emph{tame fundamental group of $U$}\index{tame fundamental group}
	with respect to the base point $\bar{u}$. 
\end{definition}
It is not difficult to check the following basic properties from the
definitions.
\begin{proposition}\label{prop:basic-properties-tame-fund-group}
	Let $k$ be a field and $U$, $\bar{u}$ as in Theorem
	\ref{thm:tameFundamentalGroup}. Then 
	\begin{enumerate}
		\item\label{prop:basic-properties-tame-fund-group-surj} There is a canonical continuous surjective map
			\begin{equation}\label{eq:tameQuotient}\pi_1^{\et}(U,\bar{u})\twoheadrightarrow
				\pi_1^{\tame}(U,\bar{u}).
			\end{equation}
			corresponding to the inclusion $\FEt^{\tame}_U\subset
			\FEt_{U}$.
		\item The kernel of \eqref{eq:tameQuotient} can be described
			as follows: If $X$ is a normal compactification of $U$
			and $\eta\in X\setminus U$ a codimension $1$ point,
			denote by $K^{\sh}_\eta$ the fraction field of the strict
			henselisation of the discrete valuation ring $\mathcal{O}_{X,\eta}$.
			Fix an algebraic closure $\overline{K^{\sh}_\eta}$. Then there is a
			morphism $G:=\Gal(\overline{K^{\sh}_\eta}/K^{\sh}_\eta)\rightarrow
			\pi^{\et}_1(U,\bar{u})$, canonical up to inner automorphism of
			$\pi^{\et}_1(U,\bar{u})$ by Proposition \ref{prop:basePoint} and
			Proposition \ref{prop:functorialityOfFundamentalgroup}. Denote by
			$P_\eta$
			the image of the wild ramification group of $G$.
			The kernel of \eqref{eq:tameQuotient} is the smallest closed normal
			subgroup $\mathbf{P}$ of $\pi^{\et}_1(U,\bar{u})$ containing $P_\eta$, where $\eta$ runs
			through all codimension $1$ points on all normal compactifications $X$
			of $U$.
		\item If $\Char{k}=p\geq 0$, then the map
			\eqref{eq:tameQuotient} induces an isomorphism
			\[\pi_1^{\et}(U,\bar{u})^{(p')}\xrightarrow{\cong}\pi_1^{\tame}(U,\bar{u})^{(p')}\]
			where $( - )^{(p')}$ denotes the maximal
			pro-prime-to-$p$-quotient.
	\end{enumerate}
\end{proposition}
\begin{proof}
	We sketch the argument.
	\begin{enumerate}
		\item As before let $\{P_{\alpha}|\alpha\in I\}$ be the set of
			Galois coverings of $U$.  We saw above Example
			\ref{ex:fundamentalGroupOfAField} that fixing points $p_{\alpha}\in
			\Fib_{\bar{u}}(P_{\alpha})$ defines unique morphisms
			$\phi_{\beta,\alpha}:P_{\beta}\rightarrow P_{\alpha}$
			for $\beta\geq \alpha$, and that the projective system
			$(P_{\alpha},\phi_{\beta,\alpha})$ induces an
			isomorphism $\pi_1^{\et}(U,\bar{u})\cong
			\varprojlim_{\alpha\in I}\Aut_U(P_{\alpha})^{\opp}$.
			If $I^{\tame}$ denotes the subset of $I$ such that
			$\alpha\in I^{\tame}$ if and only if
			$P_{\alpha}\rightarrow U$ is tame, then $I^{\tame}$ is
			also a directed set (exercise!), and
			$\pi_1^{\tame}(U,\bar{u})=\varprojlim_{\alpha\in
				I^{\tame}}\Aut_U(P_{\alpha})^{\opp}$. The
				claim follows.
			\item Just for this proof, if $G$ is a profinite
				group, we write $\Set(G)$ for the category of
				finite sets with continuous $G$-action. The
				categories
				$\Set(\pi_1^{\et}(U,\bar{u})/\mathbf{P})$ and
				$\Set(\pi_1^{\tame}(U,\bar{u}))$ are 
				strictly full subcategories of
				$\Set(\pi_1^{\et}(U,\bar{u}))$, and our goal
				is to show that they are they are identical. 

				In other words, we need to show that
				if $f:V\rightarrow U$  is a finite connected \'etale
				covering, then $f$ is tame if and only if
				$\mathbf{P}$ acts trivially on
				$\Fib_{\bar{u}}(V)$.  In fact we may assume
				$f$ to be Galois \'etale.

				Let $K(U)$, $K(V)$ denote the function fields of
				$U$ and $V$. Note that
				$\Gal(K(V)/K(U))=\Aut_U(V)^{\opp}$ is
				(noncanonically) isomorphic to
				$\Fib_{\bar{u}}(V)$ in
				$\Set(\pi_1^{\et}(U,\bar{u}))$.
				If $\mathbf{P}$ acts trivially on
				$\Fib_{\bar{u}}(V)$, then the image of the
				group $P_\eta$
				 in $\Gal(K(V)/K(U))$ is trivial for all
				codimension $1$ points $\eta\in X\setminus U$ in all normal
				compactifications $X$ of $U$. But this means
				that $K(V)/K(U)$ is tamely ramified at $\eta$.
				
				Conversely, if $f$ is tame, then $K(V)/K(U)$
				is tame with respect to all codimension $1$
				points $\eta$ in $X\setminus U$ in some normal
				compactification $X$ of $U$. This means that
				the image of $P_\eta$ is trivial in
				$\Gal(K(V)/K(U))$. But
				$\Gal(K(V)/K(U))=\Aut_U(V)^{\opp}$ is a
				quotient of $\pi_1^{\et}(U,\bar{u})$, so the
				image of 
				$\mathbf{P}$ in $\Gal(K(V)/K(U))$ is trivial,
				and thus $\mathbf{P}$ acts trivially on
				$\Fib_{\bar{u}}(V)$.
			\item This follows from the previous part, as the
				groups $P_\eta$ are pro-$p$-groups.
		\end{enumerate}
\end{proof}

Using the fact that every curve can be lifted to characteristic $0$ together
with Corollary \ref{cor:fundamentalGroupOfComplexCurve},
Grothendieck proves the following structure theorem for curves.
\begin{theorem}[{\cite[Exp.~XIII, Cor.~2.12]{SGA1}}]
	If $k$ is an algebraically closed field of characteristic $p\geq 0$,
	$X$ a smooth projective curve of genus $g$ over $k$ and $U=X\setminus \{x_1,\ldots,
	x_n\}$, with $x_1,\ldots, x_n$ closed points, then for any geometric point $\bar{u}$ of $U$, the
	tame fundamental group
	$\pi_1^{\tame}(U,\bar{u})$
	is a quotient of the profinite completion of the group
	\begin{equation}\label{eq:surfaceGroup}\left<a_1,\ldots, a_g,b_1,\ldots,
		b_g,c_1,\ldots, c_n\left|
		[a_1,b_1]\cdot\ldots\cdot[a_g,b_g]c_1\cdot\ldots\cdot
		c_n=1\right.\right>.
	\end{equation}
		Moreover, the maximal prime-to-$p$-quotient of
		$\pi_1^{\tame}(U,\bar{u})$ is isomorphic to the
		pro-prime-to-$p$-completion of \eqref{eq:surfaceGroup}.
\end{theorem}

%


\section{\'Etale sheaves}
\label{sec:l-adic-sheaves}
\setcounter{subsection}{1}\setcounter{subsubsection}{0}
In this section we recall the definition of an \'etale sheaf on a scheme and give some basic constructions. 
\begin{convention}\label{conv-noetherian}
Throughout this section schemes are assumed to be {\em separated} and {\em noetherian}.
\end{convention}

\subsubsection{{\'E}tale neighborhoods}\label{etale-neighborhoods}
Let $X$ be a scheme. Recall that a {\em geometric point}\index{geometric point} of $X$ is a morphism 
$\bar{x}=\Spec k\to X$, where $k$ is an algebraically closed field. If $x$ is the image of $\bar{x}\to X$,
we say that $\bar{x}$ is a geometric point over $x$ or that $x$ is the center of $\bar{x}\to X$.
By abuse of notation we also denote by $\bar{x}$ the morphism $\bar{x}\to X$. 
An {\em \'etale neighborhood}\index{etale@\'etale neighborhood} of $\bar{x}$ is a diagram
\[\begin{tikzcd}
      ~ & U\ar{d}{u}\\
      \bar{x}\ar{r}\ar{ur} & X,
    \end{tikzcd}\]
where $u$ is \'etale. A morphism between two \'etale neighborhoods of $\bar{x}$ is given by the obvious commutative 
diagram. Note that the opposite category of \'etale neighborhoods of $\bar{x}$ is filtered.
(Indeed: If $U_*=(\bar{x}\to U\to X)$ and $V_*=(\bar{x}\to V\to X)$ are two \'etale neighborhoods, their cartesian product 
$\bar{x}\to U\times_X V\to X$ is another \'etale neighborhood mapping to $U_*$ and $V_*$. If 
$f,g: U_*\to V_*$ are two morphisms of \'etale neighborhoods,
then the connected component $U_0$ of $U$ containing the center of $\bar{x}$
naturally defines an \'etale neighborhood $U_{0*}$ of $\bar{x}$ with a morphism $j: U_{0,*}\to U_*$, 
and by Proposition \ref{prop-determination-etale-maps} we have $f\circ j=g\circ j$.)

\subsubsection{{\'E}tale sheaves}\label{etale-sheaves}
Let $X$ be a scheme. We denote by $(\et/X)$ the category with objects the \'etale $X$-schemes and morphisms 
the $X$-morphisms (these are automatically \'etale by Proposition \ref{prop-etale}). Let $A$ be a ring.
Then an {\em \'etale presheaf} \index{etale@\'etale presheaf} of $A$-modules on $X$ is a functor
\[\sF: (\et/X)^{\rm op}\to (A\text{-mod}).\]
In case $A=\Z$, we will simply say that $\sF$ is an \'etale presheaf on $X$ or is a presheaf on $X_{\et}$.
Let $\bar{x}\to X$ be a geometric point. The stalk of a presheaf of $A$-modules $\sF$ on $X_{\et}$ at $\bar{x}$ 
is given by 
\[\sF_{\bar{x}}:=\varinjlim_{\bar{x}\to U} \sF(U),\]
where the limit is over the \'etale neighborhoods of $\bar{x}$. (One can take just one neighborhood 
for each isomorphism class to avoid set theoretical problems.)

We say $\sF$ is an {\em \'etale sheaf}\index{etale@\'etale sheaf} of
$A$-modules on $X$ or a \emph{sheaf of $A$-modules} on $X_{\et}$,
if it is a presheaf of $A$-modules on $X_{\et}$ and for  any \'etale map $U\to X$ and all 
families  $\{u_i: U_i\to U\,|\,i\in I\}$ of \'etale maps with $\bigcup_i u_i(U_i)=U$ the sequence
\eq{sheaf-cond}{0\to \sF(U)\to \prod_i \sF(U_i)\rightrightarrows \prod_{i,j}\sF(U_i\times_U U_j)}
is exact.
Recall that the exactness of this sequence means that (1) an element $s\in \sF(U)$ is zero iff $s_{|U_i}=0$ for all $i$
and (2) for a collection of elements $s_i\in \sF(U_i)$, $i\in I$, there exists an element $s\in \sF(U)$ with $s_{|U_i}=s_i$ iff
    ${s_i}_{|U_{i}\times_U U_j}={s_j}_{|U_{i}\times_U U_j}$ for all $i,j\in I$. 
(In particular, if  $\{U_i\to U\}$ is a Zariski open cover this is just the usual sheaf condition.)
A morphism between (pre-)sheaves of $A$-modules on $X_{\et}$ is a natural transformation of functors. 

If $P$ is a presheaf of $A$-modules on $X_{\et}$ then there exits a sheaf of $A$-modules $aP$ on $X_{\et}$ and
a morphism $P\to aP$ such that any morphism $P\to \sF$ with $\sF$ a sheaf on $X_{\rm et}$ factors uniquely as
$P\to aP\to \sF$. The sheaf $aP$ is called the \emph{sheaf associated to $P$}
or \emph{the sheafification of $P$} and it satisfies
$(aP)_{\bar{x}}=P_{\bar{x}}$ for any geometric point $\bar{x}\to X$.
The category of presheaves of $A$-modules on $X_{\et}$ is abelian. Hence the category of sheaves of $A$-modules on 
$X_{\et}$ is also abelian. (The cokernel of a morphism of sheaves is given by the sheafification of the cokernel 
in the bigger category of presheaves.) A sequence of sheaves of $A$-modules 
$0\to \sF'\to \sF\to \sF''\to 0$ is exact if and only if the sequence of $A$-modules 
$0\to \sF'_{\bar{x}}\to \sF_{\bar{x}}\to \sF''_{\bar{x}}\to 0$
is exact for all geometric points $\bar{x}\to X$. 

\begin{remark}\label{rmk:sheaf-criterion}
Let $\sF$ be a presheaf on $X_{\et}$. Then $\sF$ is a sheaf iff
the sequence \eqref{sheaf-cond} is exact in the following two situations:
(1) $U\to X$ is an \'etale map and  $U=\bigcup_i U_i$ is a Zariski open cover 
(i.e. the family $\{u_i\}$ is given by the open immersions $\{U_i\inj U\}$)
  and (2) $U\to X$ is an \'etale map with $U$ affine and the family $\{u_i\}$ consists just of one
   affine, surjective and \'etale map  $U'\to U$, see e.g. \cite[II, Prop. 1.5.]{Milne/EtaleCohomologyBook}. 
\end{remark}

\begin{example}\label{ex-etale-sheaf}
Let $X$ be a scheme.
\begin{enumerate}
\item\label{constant-sheaf} 
    Let $M$ be an $A$-module. For $U\to X$ an \'etale map, denote by $\pi_0(U)$ the set of connected components
of $U$ (it is finite by Convention \ref{conv-noetherian}). 
Then $U\mapsto M^{\pi_0(U)}$ defines a sheaf of $A$-modules on $X_{\et}$.
It is called the {\em constant \'etale sheaf associated to $M$} \index{constant sheaf} and denoted by
$M_X$ or simply by $M$ again.
\item\label{ex-etale-sheaf-b} The assignment $U\mapsto \Gamma(U, \sO_U)$ defines a sheaf of $\Gamma(X, \sO_X)$-modules on $X_{\et}$.
        (Indeed, by Remark \ref{rmk:sheaf-criterion}  it suffices to show that for a faithfully flat morphism of rings
        $\varphi :A\to B$, the sequence $0\to A\xr{\varphi} B\xr{d} B\otimes_A B$ is exact, 
       where $d(b)=1\otimes b- b\otimes 1$. Since $\varphi$ is faithfully flat
       this is equivalent to proving that the sequence
         is exact when tensored with $B$. Thus we have to show that 
                  $0\to B\xr{d_0} B^{\otimes_A 2}\xr{d_1} B^{\otimes_A 3}$ is exact, where
         $d_0(b)=1\otimes b$ and $d_1(b_1\otimes b_2)=b_1\otimes 1\otimes b_2- 1\otimes b_1\otimes b_2$.
          Define $s_0: B^{\otimes 2}\to B$ by $s_0(b_1\otimes b_2)=b_1b_2$ and 
          $s_1: B^{\otimes 3}\to B^{\otimes 2}$, $s_1(b_1\otimes b_2\otimes b_3)=b_1\otimes b_2b_3$.
          Then $\id_B=s_0\circ d_0$ and $\id_{B^{\otimes 2}}= d_0\circ s_0 + s_1\circ d_1$,
          hence the exactness.)
          The underlying sheaf of abelian groups on $X_{\et}$ is denoted by $\G_{a,X}$ \index{Ga@$\G_a$}
           or simply $\G_a$ if it is clear that we view it as a sheaf on $X_{\et}$.
\item\label{ex-etale-sheaf-c} The assignment $U\mapsto \Gamma(U,\sO_U)^\times$ defines a sheaf on $X_{\et}$.
         This follows directly from b) above. We denote this sheaf by
	 $\G_{m,X}$ \index{Gm@$\G_m$} or simply $\G_m$. 
\item\label{ex-etale-sheaf-d}
          Let $n\ge 1$ be a natural number. We define the sheaf $\mu_n$ on $X_{\et}$ as the kernel of 
        multiplication by $n$ on $\G_m$, i.e. we have an exact sequence
       \[0\to \mu_n\to \G_m\xrightarrow{\cdot n}\G_m.\]
         Explicitly $\mu_n(U)=\{a\in \Gamma(U,\sO_U)\,|\, a^n=1\}$. 
         If $R$ is a ring we also write by abuse of notation $\mu_n(R):=\{a\in R\,|\, a^n=1\}$. We elaborate a bit.
           \begin{enumerate}[label=(\roman*)]
              \item\label{Kummer-i} Assume $n$ is invertible in $\sO_X$ then
                      \[0\to \mu_n\to \G_m\xrightarrow{\cdot n}\G_m\to 0\]
                      is an exact sequence of sheaves on $X_{\et}$. 
                      Only the surjectivity has to be proved and this follows from the fact that for all $U$ and all $a\in \G_m(U)$
               the map $\Spec \sO_U[t]/(t^n-a)\to U$ is \'etale.
             \item\label{Kummer-ii} 
                    Let $R\neq 0$ be a ring in which $n$ is invertible and in which the polynomial $t^n-1$ decomposes, i.e. 
                           \eq{Kummer-ii1}{t^n-1=(t-\zeta_1)\cdots (t-\zeta_n)\quad \text{in }R[t].}
                 Then $\{\zeta_1,\ldots, \zeta_n\}$ is a cyclic subgroup of order $n$ of $R^\times$. 
                  Indeed, let $K$ be a field and $\varphi: R\to K$ a morphism. Then $\varphi$ induces
                  a morphism $\{\zeta_1,\ldots, \zeta_n\}\to \mu_n(K)$. This morphism is injective.
            (Else differentiating \eqref{Kummer-ii1} would give $n\varphi(\zeta_i)^{n-1}=0$ for some $i$,
             which is absurd since $n, \varphi(\zeta_i)\in K^\times$.) Thus this morphism is bijective and we know that $\mu_n(K)$
              is cyclic of order $n$.
          
                Furthermore, any morphism  $R\to A$ into a {\em local} ring $A$ induces a bijection 
                       $\{\zeta_1,\ldots, \zeta_n\}=\mu_n(A)$.
              Indeed, if $A$ is a local $R$-algebra any $a\in A$ with $a^n=1$ is of the form $a=\zeta_i+x$ for some $i$ and $x\in \fm_A$
              the maximal ideal of $A$. We obtain 
                $1= (\zeta_i+x)^n= 1+ xu$ with $u\in A^\times$, i.e. $x=0$. This proves the bijection. 

            \item\label{Kummer-iii} Assume $R$ is as in \ref{Kummer-ii} above and $X$ is an $R$-scheme. 
                      Then the choice of a cyclic generator $\zeta\in R$ of $\{\zeta_1,\ldots, \zeta_n\}$
                     defines an isomorphism 
                    \[\mu_n\cong (\Z/n\Z)_X \quad \text{on }X_{\et}.\]
                          Explicitly for all \'etale maps $U\to X$ we have an isomorphism
                   \[(\Z/n\Z)^{\pi_0(U)}\xrightarrow{\simeq}\mu_n(U),\quad 
                                     (\bar{m_i})_{i\in \pi_0(U)}\mapsto (\zeta^{m_i})_{i\in \pi_0(U)}.\]
                    It suffices to prove this Zariski (or even \'etale) locally and hence it follows from \ref{Kummer-ii} above.
           \end{enumerate}
              Notice that the statements   \ref{Kummer-i}- \ref{Kummer-iii} are wrong if $n$ is not invertible on $X$.
  \item Let $G$ be a commutative group scheme, i.e. a commutative group object in the category of schemes. 
           Then the assignment $U\mapsto G(U):=\Hom(U, G)$ defines a sheaf of abelian groups on $X_{\et}$, 
           see e.g. \cite[VIII, Thm.~5.2]{SGA1}. In fact, it is not hard to see that 
             the Examples \ref{ex-etale-sheaf-b}-\ref{ex-etale-sheaf-d} are of this form.
\end{enumerate}
\end{example}

\subsubsection{Direct and inverse image}\label{direct-inverse-image}
Let $\pi:X\to Y$ be a morphism of schemes and $A$ a ring.  If $\sF$ is a sheaf of $A$-modules on $X_{\et}$, 
then we obtain
 a sheaf of $A$-modules $\pi_*\sF$ on $Y_{\et}$ via
\[\pi_*\sF(V)= \sF(V\times_Y X),\quad V\to Y\text{ \'etale}.\]
We call this sheaf the {\em direct image of $\sF$ under $\pi$}\index{direct image}. We get a 
left exact functor from the category of $A$-modules on $X_{\et}$ to the category of $A$-modules on $Y_{\et}$, 
$\sF\mapsto \pi_*(\sF)$. This functor has a left adjoint denoted by $\pi^*$, i.e. for each $\sF$ on $X_{\et}$ and 
$\sG$ on $Y_{\et}$ there is an isomorphism of abelian groups
\[\Hom(\sG,\pi_*\sF)\cong \Hom(\pi^*\sG, \sF),\]
which is functorial in $\sF$ and $\sG$. We call $\pi^*\sG$ the 
{\em inverse image of $\sG$ under $\pi$}\index{inverse image}.
If $\bar{x}$ is a geometric point on $X$ and we denote $\pi(\bar{x})$
the geometric point of $Y$ given by the composition $\bar{x}\to X\to Y$, then we have
\eq{inverse-image1}{(\pi^*\sG)_{\bar{x}}=\sG_{\pi(\bar{x})}.}
It follows that $\pi^*$ is an exact functor.
By abuse of notation we also write $\sG_{|X}:=\pi^* \sG$. In case $u:U\to Y$ is \'etale $u^*\sG$ coincides
with the restriction of the functor $\sG$ to the category $(\et/U)$.  
If $\pi: X\to Y$ is any morphism and $u: U\to Y$ is \'etale and we denote by $\pi_1: X\times_Y U\to U$ and 
$u_1:X\times_Y U\to X$ the projection maps then we clearly have
\eq{etale-base-change}{u^*\pi_*\sF=\pi_{1*}u_1^*\sF.}

\begin{example}\label{ex-pb-pf-constant}
\begin{enumerate}
\item Let $\pi: X\to Y$ be a morphism of schemes and $M$ an $A$-module. Then 
       (with the notation from Example \ref{ex-etale-sheaf}, \ref{constant-sheaf})
             \[\pi^*(M_Y)=M_X.\]
          Indeed, for an \'etale map $V\to Y$ there is a natural morphism
	  \[M_Y(V)=M^{\pi_0(V)}\to M^{\pi_0(V\times_Y X)}=\pi_*(M_X)(V).\] 
           (If $V$ is connected it is the diagonal.) This induces a morphism of sheaves 
             $M_Y\to \pi_*(M_X)$ on $Y_{\et}$, by adjunction also a morphism
             $\pi^*(M_Y)\to M_X$. It suffices to check that this latter map is an isomorphism on the stalks at the geometric
 points of $X$. This follows directly from \eqref{inverse-image1}.
\item In general $\pi_*(M_X)$ is not isomorphic to $M_Y$. For example, if $i: X\inj Y$ is a closed immersion 
         then $i_*(M_X)$ is zero in the stalks at geometric points whose center lies in $Y\setminus X$.
\end{enumerate}
\end{example}

\subsubsection{Action of a finite group}\label{finite-group-action}
Let $G$ be a finite group acting on a scheme $X$, i.e. we are given a group homomorphism 
$G\to \Aut(X)$. Assume $G$ acts admissibly\index{admissible action} on $X$, i.e. $X$ is a union of open affines $U=\Spec A$ such that
the action of $G$ restricts to an action on $U$. (This is e.g. the case if $X$ is quasi-projective over an affine scheme).
Then we can form the quotient $\pi: X\to X/G$, where $X/G$ is defined by glueing the schemes $U/G:=\Spec A^G$, for
$U$ as above. We have $\Hom(X,Y)^G=\Hom(X/G,Y)$ for all schemes $Y$, where $G$ acts on $\Hom(X,Y)$ via 
precomposition. In general, $X/G$ might not be noetherian, which violates our 
Convention \ref{conv-noetherian}.  Therefore, by convention we will only consider finite group actions on $X$ for which the
quotient $X/G$ is actually noetherian. 
This is e.g. the case if $X$ is of finite type over some (noetherian) base $S$ and  $G$ acts via $S$-automorphisms.
See \cite[Exp.V, 1]{SGA1}.

An \'etale sheaf of $A$-modules on $(X,G)$ \index{sheaf with $G$-action}
is an \'etale sheaf of $A$-modules $\sF$ on $X$ together with morphisms 
\[\sF(\sigma): \sF\to \sigma^*\sF,\quad \sigma\in G,\]
such that $\sF(1_{G})=\id_{\sF}$ and $\sF(\tau\sigma)=\tau^*(\sF(\sigma))\circ\sF(\tau)$.
In particular, the $\sF(\sigma)$'s are isomorphisms.
Sometimes we just say $\sF$ is a sheaf with $G$-action on $X$ and it is understood that
this action is compatible with the given action on $X$.
For any geometric point $\bar{x}\to X$ we obtain isomorphisms 
$\sF_{\bar{x}}\xr{\cong} \sF_{\sigma(\bar{x})}$, where $\sigma(\bar{x})$ denotes the composition 
$\bar{x}\to X\xr{\sigma}X$. Also notice that we have natural isomorphisms
 $\sigma_*\sigma^*\sF\cong \sF\cong \sigma^*\sigma_*\sF$ induced by adjunction;
hence to give a map $\sF(\sigma)$ as above is equivalent to give a map $\sigma_*\sF\to \sF$.

Let $\pi:X\to X/G$ be as above. The action of $G$ on $\sF$ induces maps $\sigma_*\sF\to \sF$ and  applying
$\pi_*$ we obtain a map $\pi_*\sF=\pi_*\sigma_*\sF\to\pi_*\sF$. Thus $G$ acts on $\pi_*\sF$
(where we equip $X/G$ with the trivial $G$-action). We can therefore form the sheaf
of $G$-invariant elements $(\pi_*\sF)^G$ whose section on $U\to X/G$ (\'etale map) are given by
\[(\pi_*\sF)^G(U)=\{a\in \sF(U\times_{X/G} X)\,|\, \sigma\cdot a=a \text{ for all }\sigma\in G\}, \]
where we abbreviate the notation $\sigma\cdot a:= \sF(\sigma)(a)$.

\subsubsection{Extension by zero}\label{extension-by-zero}
Let $j:U\inj X$ be an open immersion and denote by $i:Z\inj X$ the closed
immersion of the complement of $U$ (equipped with some scheme structure).
Let $\sF$ be a sheaf of $A$-modules on $U$. Then we define the {\em extension by zero}\index{extension by zero}
of $\sF$ as
\[j_!\sF:={\rm Ker}(j_*\sF\to i_*i^*j_*\sF),\]
where $j_*\sF\to i_*i^*j_*\sF$ is the adjunction map for $(i_*, i^*)$.
For a point $x\in X$ and $\bar{x}$ a geometric point over $x$ we have
\[(j_!\sF)_{\bar{x}}=\begin{cases} \sF_{\bar{x}}, &\text{if } x\in U, \\ 0,& \text{if }x\in Z.\end{cases}\]
It follows that $j_!$ is an exact functor from the category $A$-modules on $U_{\et}$ to the category of
$A$-modules on $X_{\et}$. It is left adjoint to $j^*$ i.e. there is a functorial isomorphism
\[\Hom(j_!\sG, \sF)\cong \Hom(\sG, j^*\sF),\]
where $\sG$ is a sheaf on $U_{\et}$ and $\sF$ a sheaf on $X_{\et}$.

\begin{definition}\label{defn:local-constant}
Let $X$ be a scheme, $A$ a noetherian ring and $M$ an $A$-module. 
Then we call an \'etale sheaf $\sF$ of $A$-modules on $X$ 
{\em locally constant with stalk}\index{locally constant sheaf} $M$ 
if there exists a family $\{u_i:U_i\to X\}$ of \'etale maps  with $\bigcup_i u_i(U_i)=X$ such that $\sF_{|U_i}\cong M_{U_i}$
for all $i$.
\end{definition}

\begin{proposition}\label{prop:loc-constant-sheaf-rep}
In the situation of Definition \ref{defn:local-constant} assume $M$ is finite (as a set). Then 
$\sF$ is represented by a finite \'etale group scheme $X_{\sF}$ over $X$, i.e. there is a finite \'etale morphism $X_{\sF}\to X$ such that
\[\sF\cong\Hom_X(-, X_{\sF}) \quad\text{on } (\et/X).\]
Assume $X$ is connected, then it follows that there exists a connected finite \'etale Galois cover $P\to X$ which trivializes
$\sF$, i.e. $\sF_{|P}\cong M_P$.
\end{proposition}

Before we prove the proposition we recall the following statement from descent theory:

\begin{theorem}[{\cite[VIII, Thm.~2.1,  Cor.~5.7]{SGA1}} and {\cite[Cor.~(17.7.3)]{EGA4}}]\label{thm:descent}
Let $f:U\to X$ be faithfully flat and quasi-compact. Let $V\to U$ be an affine morphism such that
there exists an isomorphism $\gamma: p_1^*V\xr{\simeq} p_2^* V$ satisfying
\[p_{13}^*(\gamma)\circ p_{12}^*(\gamma)= p_{23}^*(\gamma),\]
where $p_i: U\times_X U\to U$ and $p_{ij}: U\times_X U \times_X U\to U\times_X U$ denote 
 the projection on the respective factors, $p_i^*V$ denotes the pullback of  $V$ along $p_i$ and 
$p_{ij}^*(\gamma)$ denotes the pullback of $\gamma$ along $p_{ij}$.

Then there exists an affine morphism $Y\to X$ such that $V=Y\times_X U$. 
(One says  $V/U$ {\em descends} to $Y/X$.) Furthermore if $V\to U$ is finite or \'etale, then so is $Y\to X$.
\end{theorem}

\begin{proof}[Proof of Proposition \ref{prop:loc-constant-sheaf-rep}]
We find a finite family $\{u_i: U_i\to X\}$ of \'etale maps 
such that $u_i^*\sF\cong M_{U_i}$. The constant sheaf $M_{U_i}$ is
represented by the $U_i$-group scheme $U_i\times M$ ($=$ disjoint union over $|M|$ copies of $U_i$ with group action
induced by the one on $M$). Then $U_i\times M\to U_i$ is a finite \'etale
$U_i$-group scheme representing
$u_i^*\sF$.  Set $U:=\bigsqcup_i U_i$. Then we have an \'etale and surjective (in particular
quasi-compact by Convention \ref{conv-noetherian})  morphism $U\to X$ and a finite \'etale $U$-group scheme $U\times M$ representing $\sF_{|U}$.
Denote by $p_1, p_2: U\times_X U\to U$ the two projections.
The natural isomorphism $p_1^*(\sF_{|U})\cong p_2^*(\sF_{|U})$  induces a gluing data
$p_1^*(U\times M)\cong p_2^* (U\times M)$ and hence the finite \'etale $U$-group scheme $U\times M$
descends to a finite \'etale map $X_\sF\to X$, see Theorem \ref{thm:descent}. 
By construction and the sheaf property $X_{\sF}$ represents $\sF$ over $X$.

If $X$ is connected, then any connected component of $X_{\sF}$ is still a finite \'etale covering of $X$.
We take a finite \'etale Galois covering $P\to X$ which factors over all the
connected components of $X_{\sF}$; this can be achieved as in the discussion below Example
\ref{ex:fundamentalGroupOfAField}.
Then $X_\sF\times_X P\cong P\times M$ represents $\sF_{|P}$, i.e.
$\sF_{|P}\cong M_P$.
\end{proof}

\subsubsection{Constructible sheaves}\label{constructible-sheaves}
Let $X$ be a scheme. A subset $Z\subset X$ is called \emph{locally closed} if it is the intersection of an open and a closed subset
in $X$. We equip such a $Z$ with the reduced scheme structure and obtain an immersion $i: Z\inj X$.

Let $A$ be a noetherian ring which is torsion (i.e. $mA=0$ for some natural number $m$)  and 
$\sF$ a sheaf of $A$-modules on $X_{\et}$. Then we say that $\sF$ is {\em constructible}\index{constructable sheaf}
if there exist finite type $A$-modules $M_1,\ldots, M_n$ and locally closed subsets $X_1,\ldots, X_n\subset X$ such that
\[X=\bigsqcup_i X_i\quad \text{and}\quad \sF_{|X_i} \text{ is locally constant with stalks }M_i.\]
We have (see \cite[IX, Prop. 2.6, Prop. 2.9]{SGA4}):
\begin{enumerate}
\item\label{constr-sh-a} The category of constructible sheaves of $A$-modules is abelian.
\item\label{constr-sh-b} If $0\to \sF'\to\sF\to \sF''\to 0$ is a short exact sequence of \'etale sheaves of $A$-modules, and 
$\sF',\sF''$ are constructible, then so is $\sF$.
\item\label{constr-sh-c} An \'etale sheaf of $A$-modules is constructible iff it is a noetherian object in the category
of \'etale sheaves of $A$-module, i.e., any ascending sequence of submodules becomes stationary.
\end{enumerate}

\begin{remark}\label{rmk-pf-pb-constructible}
Let $\pi:X\to Y$ be a morphism and $\sG$ a constructible sheaf of $A$-modules on $Y_{\et}$, then 
$\pi^*\sG$ is a constructible $A$-module on $X_{\et}$. (This follows directly from Example \ref{ex-pb-pf-constant}.)
If $\pi$ is proper and $\sF$ is a constructible $A$-module on $X_{\et}$, then $\pi_*\sF$ is a constructible 
$A$-module on $Y_{\et}$, see \cite[Exp. XIV, Thm 1.1]{SGA4}. 
\end{remark}

\section{{$\ell$}-adic sheaves}
In this section we recall the notions
of constructible and lisse sheaves with $\ell$-adic coefficients. We show that lisse sheaves correspond
to continuous representations of the fundamental group. 
The references for this section are \cite[Arcata, II and  Rapport, 1.2]{SGA4.5}, \cite[(1.1)]{Deligne/WeilII} and 
\cite[Exp. V, VI]{SGA5}. {\em We keep Convention \ref{conv-noetherian}.}

\begin{definition}\label{def-lisse-sheaf}
Let $R$ be a complete local DVR with maximal ideal $\fm$. Assume $R/\fm$ has characteristic $\ell>0$.
Let $X$ be a scheme.
\begin{enumerate}
\item A {\em constructible $R$-sheaf}\index{constructible $R$-sheaf} on $X$ is a projective system of 
        $R$-modules  $\sF=(\sF_n)_{n\ge 1}$ on $X_{\et}$ satisfying the following two properties:
 \begin{enumerate}[label=(\roman*)]
   \item $\fm^n \cdot \sF_n=0$ and $\sF_n$ is a constructible $R/\fm^n$-module on $X_{\et}$, see \ref{constructible-sheaves}.
   \item\label{def-lisse-sheaf-b} $\sF_n= \sF_{n+1}\otimes_{R/\fm^{n+1}} R/\fm^n$, for all $n\ge 1$.
\end{enumerate}
\item A {\em lisse} $R$-sheaf\index{lisse sheaf} on $X$ is a constructible $R$-sheaf $\sF=(\sF_n)$ such that
 each $\sF_n$ is a locally constant sheaf of $R/\fm^n$-modules.
\end{enumerate}
\end{definition}

\begin{example}\label{ex-lisse}
\begin{enumerate}
\item\label{ex-lisse-finite-constant} 
       Let $R$ be as  above and $\sF$ a locally constant sheaf of finitely generated $R/\fm^{n_0}$-modules on $X_{\et}$.
         Then we can view $\sF$ also as a lisse $R$-sheaf $(\sF_n)$, via $\sF_n:= \sF$ for $n> n_0$ and
$\sF_n:= \sF\otimes_{R/\fm^{n_0}}R/\fm^n R$, for $n\le n_0$.  
  \item Let $\pi:X\to Y$ be a proper morphism and $\sF=(\sF_n)$ a constructible $R$-sheaf on $X$.
          We want to define the constructible sheaf $\pi_*\sF$. The first try is to consider
          the system $(\pi_*(\sF_n))$. But this might not satisfy \ref{def-lisse-sheaf-b}. 
          To repair this we have to use the following two facts 
         (see \cite[Exp. VI, Lem 2.2.2]{SGA5} and \cite[Exp. V, Prop. 3.2.3]{SGA5}):
\begin{enumerate}[label=(\roman*)]
\item The system $(\pi_*(\sF_n))$ satisfies
            the Mittag-Leffler-Artin-Rees condition, i.e., there exists a natural number $r_0$ such that for all
             $n$ and all $r\ge r_0$ we have  
               \[{\rm image}(\pi_*\sF_{n+r}\to \pi_*\sF_n)={\rm image}(\pi_*\sF_{n+r_0}\to \pi_*\sF_n)=:I_n.\]
\item There exists a natural number $s_0$ such that for all $n$ and all $s\ge s_0$ we have the natural isomorphism
              \[I_{n+s}\otimes_{R} R/\fm^n\xrightarrow{\simeq} I_{n+s_0}\otimes_{R} R/\fm^n.\]
\end{enumerate}
         Then we define $\pi_*\sF=(I_{n+s_0}\otimes_{R} R/\fm^n)_n$. There is a natural morphism of projective systems
         $\pi_*\sF\to (\pi_*(\sF_n))$ which is an {\em AR-isomorphism}, i.e., if we denote by $(K_n)$ and $(C_n)$ the kernel and 
       cokernel, respectively, then there exists a natural number $N$ such that transition maps 
         $K_{n+N}\to K_{n}$ and $C_{n+N}\to C_n$ are zero, for all $n\ge 1$. (Here we can take $N=r_0+s_0$.)
               
\item Let $\ell$ be an invertible prime on $X$. We have surjections of $\Z_\ell$-sheaves on $X_{\et}$,
     $\mu_{\ell^{n+1}}\to \mu_{\ell^n}$, $a\mapsto a^\ell$. The corresponding projective system of 
        $\Z_\ell$-modules on $X_{\et}$ is denoted by $\Z_\ell(1)$,
                 \[\Z_\ell(1):=(\mu_{\ell^n}).\]
             It follows from Example \ref{ex-etale-sheaf}, \ref{ex-etale-sheaf-d}, \ref{Kummer-iii}, that $\Z_\ell(1)$
      is a lisse $\Z_\ell$-sheaf on $X$. But notice that in general there is no \'etale covering $\{U_i\to X\}$
            such that ${\mu_{\ell^n}}_{|U_i}= (\Z/\ell^n)_{U_i}$ for all $n$ (e.g. $X=\Spec \Q$).
          This shows that lisse sheaves are in general not locally constant projective systems of sheaves.
\end{enumerate}
       \end{example}

\subsubsection{The category of lisse $R$-sheaves}\label{cat-lisse}
Let $R,\fm, \ell$ be as in Definition \ref{def-lisse-sheaf} and $X$ a scheme. 
Let $\sF=(\sF_n), \sG=(\sG_n)$ be two constructible $R$-sheaves on $X$. 
Then we have morphisms of $R$-modules
\[\Hom(\sF_{n+1}, \sG_{n+1})\to \Hom(\sF_n, \sG_n), \quad \varphi\mapsto \varphi\otimes R/\fm^n.\]
The category of constructible (resp.~lisse) $R$-sheaves is the category with morphisms given by 
\[\Hom(\sF,\sG)=\varprojlim_n \Hom(\sF_n, \sG_n).\]
It is an abelian category, see \cite[V, Thm 5.2.3]{SGA5}.

\subsubsection{The category of $\bar{\Q}_{\ell}$-sheaves}
Let $X$ be a scheme and $\ell$ an invertible prime on $X$. Fix an algebraic closure $\bar{\Q}_\ell$ of $\Q_\ell$.
\begin{enumerate}
\item\label{E-lisse} 
      Let $E$ be a {\em finite} field extension of $\Q_\ell$ inside $\bar{\Q}_\ell$ and $R$ the integral closure of
        $\Z_\ell$ in $E$. The category of constructible $E$-sheaves on $X$ is by definition the localization 
           of the category of constructible $R$-sheaves localized with respect to the full-subcategory of 
            torsion sheaves.  This means the following:
             \begin{description}
         \item[\bf Objects] There is an essentially surjective functor 
            \[(\text{constr.~} R\text{-sheaves})\to (\text{constr.~} E\text{-sheaves}), \quad \sF\mapsto \sF\otimes_R E\]
         \item[\bf Morphisms] $\Hom(\sF\otimes_R E, \sG\otimes_R E)= \Hom(\sF,\sG)\otimes_R E$.
               \end{description}

	       We say that a constructible $E$-sheaf $\sF\otimes_R E$ is
	       \emph{lisse} if there exists an \'etale covering $\{U_i\to X\}$ and
     lisse $R$-sheaves $\sF_i$ on $U_i$ such that $\sF_{|U_i}\otimes_R E\cong \sF_i\otimes_R E$.

\item If $E'/E/\Q_\ell$ are finite extensions inside $\bar{\Q}_\ell$, then there is a natural functor
      $(\text{constr.~} E\text{-sheaves})\to (\text{constr.~} E'\text{-sheaves})$ on objects given by 
       \[   \sF\otimes_R E\mapsto (\sF\otimes_R E)\otimes_E E'= \sF\otimes_{R} E'. \]      
 The category of constructible $\bar{\Q}_\ell$-sheaves\index{constructible $\bar{\Q}_\ell$-sheaves} is the 
inductive 2-limit over the categories of constructible $E$-sheaves, $E\subset \bar{\Q}_\ell$. This means the following:
\begin{description}
\item[\bf Objects] For all finite extensions $E/\Q_\ell$ inside $\bar{\Q}_\ell$ there are functors 
                    \[(\text{constr.~} E\text{-sheaves})\to (\text{constr.~}
			    \bar{\Q}_\ell\text{-sheaves}), \quad 
                              \sF\mapsto \sF\otimes_E \bar{\Q}_\ell.\]
                       These satisfy:
             \begin{enumerate}
                   \item Each object in $(\text{constr.~} \bar{\Q}_\ell\text{-sheaves})$ is isomorphic to an object of the
                           form $\sF\otimes_E \bar{\Q}_\ell$, for some constructible $E$-sheaf $\sF$.
                   \item If $E'/E$ is a finite extension and $\sF\in
			   (\text{constr.~} E\text{-sheaves})$, 
                           then there is a canonical isomorphism
                                            \[\sF\otimes_E \bar{\Q}_\ell\cong (\sF\otimes_E E')\otimes_{E'}\bar{\Q}_\ell.\]
              \end{enumerate}
 \item[\bf Morphisms] For $\sF, \sG\in (\text{constr. } E\text{-sheaves})$ we have
                               \[\Hom(\sF\otimes_E \bar{\Q}_\ell, \sG\otimes_E \bar{\Q}_\ell)=  
                                \Hom(\sF, \sG)\otimes_E \bar{\Q}_{\ell}.\]
\end{description}

A {\em lisse $\bar{\Q}_\ell$-sheaf}\index{lisse $\bar{\Q}_\ell$-sheaf} on $X$  is a constructible
$\bar{\Q}_\ell$-sheaf which is locally of the form $\sF\otimes_E\bar{\Q}_\ell$, with $\sF$ a lisse $E$-sheaf.
\end{enumerate}

\begin{definition}\label{defn-constant-lisse-sheaf}
Let $X$ be a scheme and $\ell$ an invertible prime on $X$. 
\begin{enumerate}
\item Let $R$ be a complete local DVR with maximal ideal $\fm$ and residue characteristic $\ell>0$.
        We say that a lisse $R$-sheaf $\sF=(\sF_n)$ 
        is {\em constant}\index{constant lisse sheaf} or {\em trivial}\index{trivial lisse sheaf}
        if there exists a finitely generated $R$-module $M$ such that we have an isomorphism 
        of projective systems  $(\sF_n)\cong ((M\otimes_{R} R/\fm^n)_X)$, where the index $X$
        denotes the constant sheaf associated to the module, see Example \ref{ex-etale-sheaf}, \ref{constant-sheaf}.
        In this case we also write $\sF=M_X$ or just $\sF=M$.
\item Let $E$ be a finite extension of $\Q_\ell$ with ring of integers $R$. We say that a lisse 
       $E$-sheaf  $\sF$ is {\em constant} or {\em trivial} 
        if there exists a finite dimensional $E$-vector space  $V$ and an
        $R$-lattice $M\subset V$ such that $\sF= M_X\otimes_R E$ in the above notations.
        As lisse $E$-sheaf $M_X\otimes_R E$ depends (up to isomorphism) only on $V$.
        We write $\sF=V_X$ or just $\sF=V$.
\item We say that a lisse $\Qlb$-sheaf is {\em constant} or {\em trivial} if there exists a finite dimensional 
        $\Qlb$-vector space $V$, a finite extension $E/\Q_\ell$ and a finite dimensional $E$-vector space $V_E$
         such that $V=V_E\otimes_E \Qlb$ and $\sF= V_{E,X}\otimes_E \Qlb$. 
         This is independent of the choice of $E$ and $V_E$. We write
	 $\sF=V_X$ or just $\sF=V$.
\end{enumerate}
\end{definition}

\begin{convention}\label{conv-l-adic-coefficients}
Let $X$ be a scheme and assume $\ell$ is a prime number which is invertible on $X$.
We say that a ring $A$ is an {\em $\ell$-adic coefficient ring}\index{$\ell$-adic coefficient ring}
if it is either $\Qlb$, a finite extension of $\Q_\ell$ or equal to $R$ or $R/\fm^n$, $n\ge 1$, where
$R$ is a complete local DVR finite over $\Z_{\ell}$ with maximal ideal $\fm$.
We can therefore speak about constructible (resp.~lisse, resp.~constant) $A$-sheaves on $X$.
(In case $A=R/\fm^n$  a lisse $A$-sheaf $\sF$ is just a locally constant and constructible sheaf of
$A$-modules on $X_{\et}$; by Example \ref{ex-lisse}, \ref{ex-lisse-finite-constant}, we can view $\sF$
also as a lisse $R$-sheaf.) We can write any constructible $A$-sheaf in the form $(\sF_n)\otimes_{R} A$
with $R$ as above and $(\sF_n)$ a constructible $R$-sheaf.
\end{convention}

\subsubsection{Tate twist}\label{Tate twist}\index{Tate twist}
Let $X$ be a scheme. For $n\ge 1$ and $i\ge 0$ denote by 
\[\Z/\ell^n(i)\]
the sheaf on $X_{\et}$ associated to 
\[U\mapsto 
          \underbrace{\mu_{\ell^n}(U)\otimes_{\Z/\ell^n}\ldots\otimes_{\Z/\ell^n}\mu_{\ell^n}(U)}_{i\text{-times}}\]
and by
\[\Z/\ell^n(-i)\]
the sheaf on $X_{\et}$ associated to
\[U\mapsto \Hom_U(\Z/\ell^n(i)_{|U}, \Z/\ell^n_{|U}).\]
Let $A$ be an $\ell$-adic coefficient ring and $\sF$ a constructible $A$-sheaf.
Then we define the $i$-th Tate twist of $\sF$
\[\sF(i),\quad i\in \Z,\]
as follows: Take a DVR $R$ finite over $\Z_\ell$ and
a constructible $R$-sheaf $(\sF_n)_n$ such that $\sF=(\sF_n)\otimes_R A$. Notice
that the $\sF_n$ are in particular sheaves of $\Z/\ell^n$-modules on $X_{\et}$. We define
$\sF_n(i)$ as the sheaf on $X_{\et}$ associated to $U\mapsto \sF_n(U)\otimes_{\Z/\ell^n} \Z/\ell^n(i)(U)$
and  define
\[\sF(i):= (\sF_n(i))\otimes_R A.\]
This does not depend on the choice of $R$ and $(\sF_n)$. 
It is a constructible $A$-sheaf and it is lisse if $\sF$ is. We have $\sF(i)(j)=\sF(i+j)$, for $i,j\in\Z$.

\subsubsection{Stalk of a lisse $A$-sheaf}\label{stalk-lisse-sheaf}
Let $\bar{x}\to X$ be a geometric point of $X$ and $\ell$ a prime number which is invertible on $X$.
\begin{enumerate}
\item Let $R$ be a complete local DVR with finite residue field of characteristic $\ell$ and 
          $\sF=(\sF_n)$ a lisse $R$-sheaf on $X$.
    We define the stalk of $\sF$ at $\bar{x}$ to be
       \[ \sF_{\bar{x}}:= \varprojlim_n \sF_{n,\bar{x}}.\]
\item Let $E$ be a finite extension of $\Q_\ell$ and $\sF$ a lisse $E$-sheaf.
        We define the stalk of $\sF$ at $\bar{x}$ to be
\[\sF_{\bar{x}}:= (\varprojlim_n \sF'_{n,\bar{x}})\otimes_R E,\]
where $R$ is the ring of integers of $E$ over $\Q_\ell$ and 
$\sF'=(\sF'_n)$ is a lisse $R$-sheaf defined on some \'etale neighborhood of $\bar{x}$ such that
locally around $\bar{x}$ we have $\sF=\sF'\otimes_R E$. This definition is independent of the choice of
    $\sF'$ (up to isomorphism). 
\item Let $\sF$ be a lisse $\Qlb$-sheaf. We define the stalk of $\sF$ at $\bar{x}$ to be 
            \[\sF_{\bar{x}}:= (\varprojlim_n \sF'_{n,\bar{x}})\otimes_R \Qlb,\]
where $R$ is the ring of integers of a finite extension $E$ over $\Q_\ell$ and 
$\sF'=(\sF'_n)$ is a lisse $R$-sheaf defined on some \'etale neighborhood of $\bar{x}$ such that
locally around $\bar{x}$ we have $\sF=(\sF'\otimes_R E)\otimes_{E}\Qlb$. This definition is independent of the choice of
$E$ and $\sF'$ (up to isomorphism). 
\end{enumerate}
Thus, all together, if $A$ is an $\ell$-adic coefficient ring the stalk $\sF_{\bar{x}}$ of a lisse $A$-sheaf $\sF$ 
at a geometric point $\bar{x}$ is defined.
Furthermore (with the above notation)  $\varprojlim_n \sF'_{n,\bar{x}}$ is a finite type $R$-module.
(In fact it is separated for the $\fm$-adic topology and $R$ is an
$\fm$-adically complete ring, thus
 it is generated by any lift of a system of generators of $(\varprojlim_n \sF'_{n,\bar{x}})/\fm=\sF'_{1,\bar{x}}$,
see e.g. \cite[Thm 8.4]{Matsumura}. ) Hence  $\sF_{\bar{x}}$ is a finite type $A$-module for general $A$.

We say that a lisse $A$-sheaf $\sF$ is {\em free}\index{free lisse $A$-sheaf} if its stalks are free $A$-modules.
In this case and if $X$ is connected the rank of $\sF$\index{rank of a free
lisse sheaf} -- denoted by $\rk(\sF)$ -- 
is by definition the rank of $\sF_{\bar{x}}$ at some (and hence any) geometric point of $X$. 

\subsubsection{$\ell$-adic representations}\label{rep-pi1}
Let $A$ be an $\ell$-adic coefficient ring (Convention
\ref{conv-l-adic-coefficients}), which is {\em not} $\Qlb$.  Let $M$ be a finitely generated $A$-module,
then we can equip it with the $\ell$-adic topology as follows: There is a DVR
$R$, finite over $\Z_\ell$,  such that $A$ is either
the fraction field  or  a quotient of $R$. We can find a finitely generated $R$-module $M'$ such that $M=M'\otimes_R A$ and we
define the topology on $M$ in such a way that the image of $\{m+\ell^n\cdot M'\}_{n\ge 0}$ in $M$ 
is a system of open neighborhoods of $m\in M$. 
This topology is independent of the choice of $M'$. (Indeed, if $A$ is a finite ring this is just the discrete topology and otherwise
$M'$ is an $R$-lattice in $M$ and for any other $R$-lattice $M''$ we find integers $a,b\ge 1$ such that $\ell^a \cdot M'\subset M''$ and 
$\ell^b\cdot M''\subset M'$; hence $M'$ and $M''$ define the same topology.) In particular the $A$-module of $A$-linear 
endomorphisms $\End_A(M)$ also carries the $\ell$-adic topology and the composition of endomorphisms is 
a continuous operation. Now let $\Aut_A(M)$ be the group of $A$-linear automorphisms $\Aut_A(M)$. 
We have an inclusion $\Aut_A(M)\to \End_A(M)\times \End_A(M)$, $\sigma\mapsto (\sigma, \sigma^{-1})$.
We put the product topology on $ \End_A(M)\times \End_A(M)$ and the subspace topology on $\Aut_A(M)$.
In this way   $\Aut_A(M)$ becomes a topological group (compare to Definition
\ref{defn:topologyOnGL}).

Let $X$ be a connected scheme and $\bar{x}\to X$ a geometric point. 
We define an {\em $A$-representation of $\pi_1(X,\bar{x})$}\index{$A$-representation of $\pi_1(X,\bar{x})$} 
to be a continuous group homomorphism
\[\pi_1(X,\bar{x})\to \Aut_A(M),\]
where $M$ is a finitely generated $A$-module.

By definition, a {\em $\Qlb$-representation of $\pi_1(X,\bar{x})$} is a group
homomorphism $\rho_{V}:\pi_1(X,\bar{x})\to \Aut_\Qlb(V)$, where $V$ is a
finite dimensional $\Qlb$-vector space, 
such that there exists a finite field extension $E/\Q_\ell$, a finite dimensional 
$E$-vector space $V_0$, an isomorphism $V\cong V_0\otimes_E\Qlb$ and an $E$-representation 
$\rho_{V_0}: \pi_1(X,\bar{x})\to \Aut_E(V_0)$ such that $\rho_V$ is equal to the composition
\[\rho_V: \pi_1(X,\bar{x})\xr{\rho_{V_0}} \Aut_E(V_0)\xr{\otimes_E
\Qlb}\Aut_{\Qlb}(V).\]

\begin{theorem}\label{thm-sheaf-vs-rep}
Let $X$ be a connected scheme and $\bar{x}$ a geometric point. Let $A$ be an $\ell$-adic coefficient ring. 
Then there is a natural equivalence of categories
\[(\text{lisse } A\text{-sheaves on }X)\xr{\simeq} (A\text{-representations of }\pi_1(X,\bar{x})),\quad \sF\mapsto \sF_{\bar{x}}.\]
Furthermore, let $\pi: X'\to X$ be a finite \'etale Galois cover with Galois group $G$. Then the above
equivalence induces an equivalence between the following subcategories
\[(\text{lisse } A\text{-sheaves on }X, \text{ constant on }X')\xr{\simeq} (\text{finitely gen.~} A[G]\text{-modules}).\]
If $A$ is finite then under this equivalence  a finitely generated $A[G]$-module $M$ corresponds to the 
lisse $A$-sheaf given by  $(\pi_*M_{X'})^G$ 
(see \ref{finite-group-action} and \eqref{thm-sheaf-vs-rep1} below for how we view $M_{X'}$ as a sheaf on $(X',G)$).

If $A$ is infinite and $M$ is a finitely generated $A[G]$-module,
then there exists a complete local DVR $R$, finite over $\Z_\ell$ and contained in $A$, and a finitely generated $R[G]$-module
$N$ such that $M=N\otimes_R A$. In this case $M$ corresponds to the lisse $A$-sheaf
$((\pi_*N_{X'})^G\otimes_RR/\fm^n)_n\otimes_R A$.
\end{theorem}

\begin{proof}
 Let $(P_\alpha)$ be an inverse system of finite \'etale Galois coverings of $X$ with maps $\bar{x}\to P_\alpha$
making the diagram 
\[\begin{tikzcd}
      ~ & P_\alpha\ar{d}\\
      \bar{x}\ar{r}\ar{ur} & X,
    \end{tikzcd}\]
commutative and such that an element of $\pi_1(X,\bar{x})$ is the same as a compatible system of $X$-automorphisms
of $(P_\alpha)$  fixing $\bar{x}$, i.e.~such that
$\pi_1^{\et}(X,\bar{x})=\varprojlim \Aut_X(P_{\alpha})^{\opp}$ (see
\eqref{eqn:proRepresentability}). 

{\em First case: $A$ is a finite ring.} Let $\sF$ be a lisse $A$-sheaf.
Given an \'etale neighborhood $U\to X$ of $\bar{x}$ we set $U_\alpha:=U\times_X P_\alpha$ and
obtain an inverse system $(U_\alpha)$ of \'etale neighborhoods of $\bar{x}$. 
An element $\sigma\in \pi_1(X,\bar{x})$ induces a compatible system of $X$-automorphisms of $(U_\alpha)$. 
Therefore such an element induces a map 
\[\sF(U)\to \sF(U_\alpha)\xr{\sigma^*} \sF(U_\alpha)\to \sF_{\bar{x}}.\]
Taking the limit over all \'etale neighborhoods $U$ of $\bar{x}$ we get a morphism 
$\sigma^*:\sF_{\bar{x}}\to \sF_{\bar{x}}$. We obtain in this way the structure of a 
 $\pi_1(X,\bar{x})$-representation on $\sF_{\bar{x}}$. Clearly this
 construction is functorial.

For the other direction let $M$ be an $A$-representation of $\pi_1(X,\bar{x})$. The group
$\Aut_A(M)$ is finite hence there is a finite \'etale Galois cover $\pi:P=P_\alpha\to X$ with Galois group $G$ such that
$\pi_1(X,\bar{x})\to \Aut_A(M)$ factors over $G$.  We consider the constant sheaf $M_{P}$ as a sheaf on
$(P,G)$ (see \ref{finite-group-action}) by defining the map $M_P(\sigma): \sigma_*M_P\to M_{P}$ ($\sigma\in G$)
via
\eq{thm-sheaf-vs-rep1}{M_P(V\times_{P,\sigma} P)=M\xr{\sigma^{-1}\cdot} M= M_P(V)}
where $V\to P$ is \'etale with $V$ connected.
Then set 
\[\sF_M:=\pi_*(M_P)^G.\]
We claim that $\sF_M$ is \'etale locally isomorphic to the constant sheaf $M$ and that this construction is 
independent of the choice of $\pi:P\to X$. Then $\sF_M$ is a lisse $A$-sheaf and the assignment 
$M\mapsto\sF_M$ is functorial.
To prove the claim let $V\to P$ be an \'etale map and assume $V$ is connected. We have 
$V\times_X P\cong V\times G$.  For $\sigma\in G$, the automorphism $\id_V\times \sigma$ ($\sigma\in G$) on $V\times_X P$ 
translates via this isomorphism into the automorphism $\id\times {\sigma\cdot}$ of $V\times G$. 
Hence
\mlnl{\sF_M(V)=\{a\in \pi_*(M_P)(V)\,|\, M_P(\sigma)(a)=a\text{ for all }\sigma\in G\}\\ = 
 \{(a_\tau)\in M\times G\,|\,  \sigma^{-1}\cdot a_{\sigma\cdot \tau}=a_{\tau} \text{ for all } \sigma,\tau\in G\}
              \cong M,}
where the last isomorphism is given by $ M\ni m\mapsto (\sigma\cdot m)_\sigma \in M\times G$.
Thus $\sF_{M|P}=M_P$. Moreover if we choose another $P'\to X$ we obtain a sheaf
 $\sF'_{M}$ in the same way. Then we find $P''\to X$ dominating $P'$ and $P$ and we construct $\sF''_M$.
There is a natural map $\sF_M\to \sF''_M$ induced by $\pi_*M_{P}\to \pi''_*(M_{P|P''})=\pi''_*M_{P''}$
and by the above we obtain 
\[\sF_{M|P''}= M_{P''}=\sF''_{M|P''}.\]
Thus the natural map $\sF_M\to \sF''_M$ is an isomorphism \'etale locally, hence is an isomorphism globally.  
Same with $\sF'_M$. This shows that $\sF_M$ is independent of the choice of $P\to X$.

We defined two functors $\sF\mapsto \sF_{\bar{x}}$ and $M\mapsto \sF_M$ and we have to show that they are inverse to each 
other. It is straightforward to check $\sF_{M,\bar{x}}=M$ (as $\pi_1(X,\bar{x})$-representations).
Now let $\sF$ be a lisse $A$-sheaf and $\sF_{\bar{x}}=M$ its associated representation.
By Proposition \ref{prop:loc-constant-sheaf-rep} there exists a finite \'etale  Galois cover 
$\pi: P\to X$ with Galois group $G$ such that $\pi^*\sF\cong M_P$. 
(In particular the representation $M$ of $\pi_1(X,\bar{x})$ factors over $G$.)
 Then 
\[\sF_M= \pi_*(\pi^*\sF)^G=\sF.\]
Here the first equality holds by definition; for the second, first observe that there is a
canonical map $\sF\to \pi_*(\pi^*\sF)^G$ (induced by adjunction); it suffices to check that it is an isomorphism
when restricted to $P$ which follows from $\sF_{M|P}\cong M_P$. 
This finishes the proof in the case $A$ is finite.

{\em Second case: $A=R$ is a DVR finite over $\Z_{\ell}$.} Let $\fm$ be the maximal ideal of $R$.
A lisse $R$-sheaf is a projective system $\sF=(\sF_n)$ as in Definition \ref{def-lisse-sheaf}.
By the first case each $\sF_n$ gives rise to a $\pi_1(X,\bar{x})$-representation $\sF_{n,\bar{x}}$
which factors over some finite quotient. Clearly they fit together to give a continuous
$\pi_1(X,\bar{x})$-representation  on $\sF_{\bar{x}}=\varprojlim_n \sF_{n,\bar{x}}$. We already remarked in 
\ref{stalk-lisse-sheaf} that $\sF_{\bar{x}}$ is actually a finitely generated $R$-module.
This gives the functor $\sF\mapsto \sF_{\bar{x}}$ from the statement.
To construct a functor in the other direction, let $M$ be a $R$-representation of $\pi_1(X,\bar{x})$.
We obtain induced representations on $M_n:=M\otimes_R R/\fm^n$ which by the first case
correspond to lisse $R/\fm^n$-modules $\sF_{M_n}$. It is straightforward to check
that we obtain in this way a projective system $\sF_M=(\sF_{M_n})$ which defines a lisse $R$-sheaf
and that the functor $M\mapsto \sF_M$ thus defined is inverse to $\sF\mapsto \sF_{\bar{x}}$.

{\em Third case: $A$ is a finite field extension of $\Q_\ell$ or equal to $\Qlb$.}
We get the functor $\sF\mapsto \sF_{\bar{x}}$ in an analogous way as above.
For the functor in the other direction  let $V$ be an $A$-representation of $\pi_1(X,\bar{x})$.
Then by Lemma \ref{lemma:lattice} we find a DVR $R$ finite over $\Z_\ell$ and
contained in $A$ and a finitely generated $R$-submodule 
$M$  of $V$ which has a continuous $\pi_1(X,\bar{x})$-action 
and satisfies $M\otimes_R A\cong V$ as $A[\pi_1(X,\bar{x})]$-modules. 
By the second case above $M$ gives rise to a lisse $R$-sheaf $\sF_M$ and we define
$\sF_V:=\sF_M\otimes_R A$. One checks that this construction is independent of the choice of $M$
and defines a functor $M\mapsto \sF_M$ which is inverse to the functor from the statement of the theorem.
This finishes the proof.
\end{proof}

\begin{example}\label{ex:lisse-sheaf-finite-field}
Let $k$ be a finite field with $q=p^n$ elements and $\ell$ a prime different from $p$. 
Then the lisse $\Qlb$-sheaves of rank $1$ on
$\Spec k$ correspond to homomorphisms $\Gal(\bar{k}/k)\to \Qlb^\times$ which factor over
a continuous homomorphism $\chi:\Gal(\bar{k}/k)\to E^\times$, with $E$ a finite extension of $\Q_\ell$.
Let us see what this is. We have an isomorphism of topological groups
$\widehat{\Z}\xr{\cong} \Gal(\bar{k}/k)$ under which $1$ is mapped to the
$q$-power Frobenius $x\mapsto x^q$. Let $R$ be the ring of integers of $E$
with maximal ideal $\fm=(\pi)$. Then a homomorphism $\chi$ as above factors over
a lattice $\frac{1}{\pi^n}R$, i.e.~is induced by a continuous group
homomorphism
$\chi_R: \widehat{\Z}\to \Aut_R(\frac{1}{\pi^n}R)=R^\times$, see Lemma
\ref{lemma:lattice}. 
The topological group $\widehat{\Z}$ is the free profinite group on $1$
generator, so giving a continuous
map from $\widehat{\Z}$ to a profinite group $G$ is the same as giving an
element of $G$.
Here $R^\times$ is profinite and hence a homomorphism $\chi_R$ as above
corresponds to the choice of an  element of $R^\times$.

All together we see that (up to isomorphism) a lisse $\Qlb$-sheaf on $\Spec k$
of rank $1$
corresponds uniquely to an element of $\bar{\Z}_\ell^\times$, the integral
closure of $\Z_{\ell}$ in $\bar{\Q}_{\ell}$.
\end{example}

\begin{example}\label{ex:cyclotomic-char}
Let $X$ be a connected scheme, $\bar{x}\to X$ a geometric point and $\ell$ a prime number invertible on $X$.
We have the lisse sheaf $\Z_\ell(1)$ on $X$ (see \ref{Tate twist}) at our disposal. The corresponding representation
of $\pi_1(X,\bar{x})$ is the following: In the algebraically closed field $k(\bar{x})$ take
a family of elements $\zeta_n\in k(\bar{x})$, $n\ge 1$, such that $\zeta_n$ is an $\ell^n$-th primitive root
of unity and $\zeta_{n+1}^\ell=\zeta_n$ for all $n$. Let $P\to X$ be a finite \'etale Galois cover with Galois group $G_P$
and a map $\bar{x}\to P$ inducing  $\bar{x}\to X$.
Denote by $x_P$ the image of $\bar{x}\to P$.   If $H^0(P,\sO_P)$ contains an $\ell^n$-th primitive root of unity,
then it also contains a unique $\ell^n$-th primitive root of unity which maps to $\zeta_n$ under the composition
$H^0(P,\sO_P)\to \sO_{P,x_P}\to k(x_P)\inj k(\bar{x})$ (see Example \ref{ex-etale-sheaf}, \ref{ex-etale-sheaf-d});
we denote it again by $\zeta_n\in H^0(P,\sO_P)$.
If $\zeta_n\in H^0(P,\sO_P)$, then for $\sigma\in G_P$ the element $\sigma(\zeta_n)$ is again an 
$\ell^n$-th primitive root of unity. Hence
\[\sigma(\zeta_n)= \zeta_{n}^{r_{n,\sigma}}, \quad \text{for some } r_{n,\sigma}\in (\Z/\ell^n\Z)^\times.\]
Let $n_P$ be the maximal $n$ such that $\zeta_n\in H^0(P,\mathcal{O}_P)$. We obtain a group homomorphism
\[G_P\to (\Z/\ell^{n_P}\Z)^\times, \quad \sigma\mapsto r_{n_P,\sigma}.\]
This map is independent of the choice of the $\zeta_n$. If $P'\to X$ is another such Galois cover which factors via
$P'\to P$ then we obtain a commutative diagram
\[\begin{tikzcd}
     G_{P'}\ar{d}\ar{r} & (\Z/\ell^{n_{P'}})^\times\ar{d}\\
      G_P \ar{r} & (\Z/\ell^{n_P})^\times,
    \end{tikzcd}\]
with vertical maps the natural surjections.
Therefore we can take the  inverse limit over an inverse system of finite \'etale Galois covers $(P_\alpha)$
pro-representing the fiber functor  $F_{\bar{x}}$ and obtain a $\Z_\ell$-representation
\[\pi_1(X,\bar{x})\to \Z_{\ell}^\times.\]
It is the representation corresponding to the lisse sheaf $\Z_\ell(1)$.
\end{example}

\begin{example}\label{ex:ArtinSchreier}
We can use Theorem \ref{thm-sheaf-vs-rep} to construct new lisse sheaves which
are related to Example \ref{ex:ASlowerNumbering}. For example
let $k$ be a perfect field of characteristic $p>1$ containing the field with $q=p^n$ elements $\F_q$ and 
$m\ge 1$ a natural number. Then the natural inclusion 
$k[x]\subset k[x][t]/(t^q-t-x^m):=B$ defines an \'etale Galois covering
$\pi: X:=\Spec B \to \A^1_k=\Spec k[x]$. (This map is \'etale by
 Definition \ref{def-etale-algebra}, \ref{def-etale-algebra-presentation}, since $\partial(t^q-t-x^m)/\partial t=-1$.)
The Galois group is isomorphic to $\F_q$, where $a\in \F_q$ acts via $\bar{t}\mapsto \bar{t}+a$.
Let $\bar{\eta}\to X$ be a geometric point over the generic point of $X$,
we obtain 
\[\F_q=\coker(\pi_1(X,\bar{\eta})\hookrightarrow \pi_1(\A^1_k,\bar{\eta})).\]
Now let $\ell$ be a prime different from $p$ and $A$ an $\ell$-adic coefficient ring and take
a group homomorphism 
\[\psi: (\F_q,+)\to A^\times.\]
Then the composition 
\[\pi_1(\A_k^1,\bar{\eta})\twoheadrightarrow \F_q\xr{\psi}A^\times=\Aut_A(A) \]
is an $A$-representation  of $\pi_1(\A_k^1,\bar{\eta})$ of rank $1$. Hence it defines  a free
lisse $A$-sheaf $\sL_{m,\psi}$ of rank $1$ on $\A^1_k$.
Its sheaf theoretic description is as follows: If $A$ is finite, then denote by $L_{\psi}$ the 
sheaf on $(X,\F_q)$ (see \ref{finite-group-action}) whose underlying sheaf on $X$ is the
constant sheaf $A$ and where the action of $a\in \F_q$ is given by
multiplication with $\psi(a)$. Concretely,
if $\bar{x}\to X$ is a geometric point we get an isomorphism
\[L_{\psi,\bar{x}}=A\to A=L_{\psi,\bar{x}+a}, \quad b\to \psi(a)b, \]
where $\bar{x}+a$ is the composition $\bar{x}\to X\xr{+a} X$. 
Then $\sL_\psi=\pi_*(L_{\psi})^{\F_q}$. One obtains the description for
general (infinite) $A$ by the usual 
limit procedure.  
\end{example}

\begin{example}\label{ex:Lang-torsor}
One can generalize the above example for the case $m=1$ as follows (see e.g. \cite[Sommes trig.~1.]{SGA4.5}):
Let $\F_q$ be the finite field with $q=p^n$ elements and $G$ a commutative algebraic group scheme over $\F_q$.
Let $F:G\to G$ be the Frobenius morphism (it is given by $\sO_G\to\sO_G$, $x\mapsto x^p$).
Then $F^n=(F\circ\ldots\circ F)$ ($n$-times) is an $\F_q$-endomorphism of $G$.
We obtain a finite \'etale  Galois covering (the \emph{Lang
isogeny}\index{Lang isogeny})
\[F-\id_G:G\to G\]
whose kernel and Galois group is the finite group $G(\F_q)$. Let $A$ be an $\ell$-adic coefficient ring and 
$\chi: G(\F_q)\to A^\times$ a homomorphism. Then  the composition
\[\pi_1(G,\bar{\eta})\twoheadrightarrow\coker(\pi_1(G,\bar{\eta})\xr{F-\id}\pi_1(G,\bar{\eta}))=G(\F_q)
\xr{\chi} A^\times \]
is an $A$-representation of $\pi_1(G,\bar{\eta})$.  Hence we obtain a lisse
rank $1$ sheaf
$\sL_\chi$ on $G$. In case $G=\A^1_{\F_q}$ we obtain the example above for $m=1$.
\end{example}


\section{$\ell$-adic cohomology}
\setcounter{subsection}{1}\setcounter{subsubsection}{0}
In this section we recall without proofs the main properties of \'etale cohomology
on a scheme over an algebraically closed field with coefficients in a lisse $\ell$-adic sheaf.
The references for this section are \cite{SGA4}, \cite{SGA4.5} and \cite{SGA5}.
See also \cite{LeiFu}.

\begin{convention}\label{conv-ell-adic-coh}
Let $k$ be a perfect field of characteristic $p\ge 0$ and $\ell$ a prime number which is invertible in $k$. 
Throughout this section a $k$-scheme is a scheme which is {\em separated} and of finite type over $k$.
\end{convention}

\subsubsection{\'Etale cohomology}\label{etale-cohomology}\index{etale@\'etale cohomology} 
Let $f:X\to Y$ be a morphism of $k$-schemes.
\begin{enumerate}
\item Let $A$ be a torsion ring. The category of $A$-modules on $X_{\et}$ has enough injectives.
Therefore  we can define the $i$-th higher direct image of an $A$-module $\sF$ on $X_{\et}$ under $f$
as the $i$-th right derived functor of $f_*$, i.e.
\[R^if_*\sF:=H^i(f_*\sI^\bullet),\]
where $\sF\to \sI^\bullet$ is an injective resolution of $A$-modules on $X_{\et}$. 

Choose a compactification of $f$ (using Nagata's theorem \cite{Conrad/Nagata},
\cite{Luetkebohmert/Nagata}), i.e. a proper morphism $\bar{f}:\bar{X}\to Y$
together with a dominant open immersion $j: X\inj \bar{X}$ such that $f=\bar{f}\circ j$.
Then we define
\[R^if_!\sF:= R^i\bar{f}_*j_!\sF.\]
This definition is independent of the choice of the compactification, see \cite[Arcata, IV, (5.3)]{SGA4.5}.

In case $k$ is {\it algebraically closed} and $\pi: X\to \Spec k$ is the structure map we also write
\[H^i(X,\sF):=\Gamma(\Spec k, R^i\pi_*\sF)\quad \text{and}\quad H^i_c(X,\sF):= H^i(\bar{X}, j_!\sF)\]
and call it the $i$-th cohomology and 
the $i$-th {\em cohomology with compact support}\index{cohomology with compact support} of $\sF$,
respectively. (The $A$-module $H^i(X,\sF)$ is in this case also equal to the $i$-th right derived functor 
of the global section functor $\Gamma(X,-)$.)

\item Assume $k$ is algebraically closed. Let $A$ be an $\ell$-adic coefficient ring
          (in the sense of Convention \ref{conv-l-adic-coefficients}) and let
	  $\sF$ be a constructible $A$-sheaf on $X_{\et}$. 
        Then we find a DVR $R$  which is finite over $\Z_\ell$  such that $A$
	is an $R$-algebra,        and a constructible $R$-sheaf $(\sF_n)_n$ such that $\sF=(\sF_n)\otimes_R A$.
        We define the $i$-th \'etale cohomology group of $\sF$, respectively 
the $i$-th \'etale cohomology group with compact support of $\sF$ as
\[H^i(X,\sF):=(\varprojlim_n H^i(X,\sF_n))\otimes_R A,\quad H^i_c(X, \sF)=(\varprojlim_n H^i_c(X,\sF_n))\otimes_R A.\]
This definition is independent of the choice of $R$ and the constructible $R$-sheaf $(\sF_n)$, 
see \cite[VI, 2.2]{SGA5}.
\end{enumerate}

\begin{remark}\label{rmk-coh-non-alg-closed}
Let $\bar{k}$ be an algebraic closure of $k$, $A$ a finite $\ell$-adic  coefficient ring, 
$\pi: X\to \Spec k$  a $k$-scheme and $\sF$ a constructible $A$-sheaf on $X$. Then $R^i\pi_*\sF$
and $R^i\pi_!\sF$ are lisse $A$-sheaves on $\Spec k$ and the corresponding
${\rm Gal}(\bar{k}/k)$-representations are equal to $H^i(X\otimes_k \bar{k},\sF)$ and 
$H^i_c(X\otimes_k \bar{k},\sF)$, respectively.
(This follows from the definition of the stalk at $\Spec \bar{k}\to \Spec k$ and \'etale base change, cf.~\eqref{etale-base-change}.)
\end{remark}

\subsubsection{}\label{coh-properties} In the following we assume that $k$ is {\em algebraically closed}.
  Let $f:X\to Y$ be a morphism of $k$-schemes, $A$ an $\ell$-adic coefficient ring and 
   $\sF$ a constructible $A$-sheaf on $X_{\et}$. We list some properties: 
\begin{enumerate}
\item\label{coh-finite} If $A$ is finite (i.e. $A=R/\fm^n$ for some DVR $R$ finite over $\Z_\ell$),
          then $R^if_!\sF$  and $R^if_*\sF$ are constructible on $Y$, see 
\cite[Arcata, IV, Thm (6.2)]{SGA4.5} and \cite[Th.~finitude, Thm 1.1]{SGA4.5}. 
If $A$ is general
       then $H^i(X,\sF)$ and $H^i_c(X,\sF)$ are finitely generated $A$-modules, cf. \cite[VI, 2.2]{SGA5}.
 \item\label{long-exact-seq} Let $0\to \sF'\to \sF\to \sF''\to 0$ be a short exact sequence of constructible $A$-sheaves.
      Then there is a long exact sequence 
      \[\cdots\to H^i(X, \sF)\to H^i(X,\sF'')\to H^{i+1}(X, \sF')\to\cdots,\]
      and similarly with $H^i$ replaced by $H^i_c$.
    (If $A$ is finite this follows immediately from the general theory of right derived functors and the fact that $j_!$ is exact.
     For general $A$ notice that the groups $H^i(X,\sF_n)$ are finite by \ref{coh-finite}; hence  
      any projective system of subquotients of  $(H^i(X,\sF_n))$ satisfies the Mittag-Leffler condition; 
       hence the long exact sequence of  pro-groups gives a long exact sequence in the limit;
         further if $A$ is not finite, then it is flat over $R$ and hence
         tensoring with $A$ over $R$ is exact.)
\item\label{finite-map} Assume $f:X\to Y$ is finite. If $A$ is finite then $R^if_*\sF=0$ for all $i\ge 1$,
            see \cite[Arcata, II, Prop. (3.6)]{SGA4.5}. 
          This implies (via a Leray spectral sequence argument) that for general $A$ we have 
         $H^i(X, \sF)=H^i(Y,f_*\sF)$. 
 \item\label{compact-localization} Assume $X$ is a proper $k$-scheme and 
                  $j:U\inj X$ is an open subscheme with  complement $i:Z\inj X$. 
        Then there is a long exact sequence  
 \[\cdots \rightarrow H^i_c(U, j^*\sF)\to H^i(X, \sF)\to H^{i}(Z, i^*\sF)\to H^{i+1}_c(U, j^*\sF)\to\cdots.\]
       If we write $\sF=(\sF_n)\otimes_R A$ then this sequence is induced as above by the exact sequences 
              $0\to j_!j^*\sF_n\to \sF_n \to i_*i^*\sF_n\to 0$, for $n\ge 1$.
\item\label{affine-vanishing} Assume $X$ is affine. Then 
                                  $H^i(X, \sF)=0$ for all $i>\dim X$, see \cite[XIV, Cor. 3.2]{SGA4}.
\item\label{general-vanishing} We have  $H^i(X,\sF)=0$ for all $i>2\dim X$, see
          \cite[X, Cor.~4.3]{SGA4}.
\item\label{comparison-analytic} Assume $k=\C$ and $X$ is smooth over $\C$. 
          Then 
  \[H^i(X, \Z/n\Z)\cong H^i(X(\C),\Z/n\Z)\quad \text{and} \quad H^i_c(X,\Z/n\Z)\cong H^i_c(X(\C),\Z/n\Z),\]
   where $n\ge 1$ and the right-hand side of the two isomorphisms is the singular cohomology 
   (resp.~with compact supports) of the  complex manifold $X(\C)$. 
     See \cite[XI, Thm.~4.4]{SGA4} and \cite[Arcata, IV, Thm (6.3)]{SGA4.5} 
           \item {\em Cohomology with support.}\label{coh-support}\index{cohomology with support}
       Let $j: U\inj X$ be an open subset of $X$ with complement $i: Z\inj X$. 
       If $A$ is a finite $\ell$-adic coefficient ring and $\sF$ a constructible $A$-sheaf on $X$
       we define 
     \[i^!\sF:={\rm Ker}(i^*\sF\to i^*j_*j^*\sF).\]
   We obtain a left exact functor $\sF\mapsto \Gamma_Z(X,\sF):=\Gamma(Z, i^!\sF)$. Its
right derived functor is denoted by $H^i_Z(X,\sF)$.

If $A$ is a general $\ell$-adic coefficient ring, we write $\sF=(\sF_n)\otimes_R A$ and set
     \[H^i_Z(X, \sF):=(\varprojlim_n H^i_Z(X,\sF_n))\otimes_R A. \]

We have an exact sequence
\[\ldots\to H^i(X,\sF)\to H^i(U,\sF_{|U})\to H^{i+1}_Z(X, \sF)\to\cdots.\]
(In case $A$ is finite this is the usual localization sequence. 
 For general  $A$ it follows from this and \ref{coh-finite} that $H^i_Z(X,\sF_n)$ is finite
and hence as above taking the inverse limit over the sequences for $\sF_n$ and tensoring with $A$ gives the
exact sequence in general. )
    \end{enumerate}

\begin{remark}\label{rmk-etale-sing-coh}
Let $X$ be a smooth $\C$-scheme. Notice that \ref{coh-properties}, \ref{comparison-analytic} also implies 
\[H^i(X,\Z_\ell)\cong H^i(X(\C),\Z)\otimes_\Z \Z_{\ell},\quad \text{for all primes }\ell,\]
where $\Z_\ell$ on the left-hand side denotes the constant lisse $\Z_{\ell}$-sheaf $(\Z/\ell^n)_n$.
Indeed, $H^i(X(\C),\Z)$ is a finitely generated $\Z$-module, hence 
\[H^i(X(\C),\Z)\otimes_\Z \Z_{\ell}=\varprojlim_n (H^i(X(\C),\Z)\otimes_\Z \Z/\ell^n).\]
On the other hand
the exact sequence $0\to \Z\xr{\cdot \ell^n}\Z\to \Z/\ell^n\to 0$ yields an exact sequence
\[0\to H^i(X(\C),\Z)\otimes_\Z \Z/\ell^n\to H^i(X(\C),\Z/\ell^n)\to 
      {_{\ell^n}H^{i+1}(X(\C),\Z)}\to 0,\]
where we denote by $_{\ell^n}(-)$ the kernel of the multiplication by $\ell^n$.
Using \ref{comparison-analytic} and taking $\varprojlim_n$ we arrive at a short exact sequence
\[0\to H^i(X(\C),\Z)\otimes_\Z \Z_\ell \to H^i(X,\Z_\ell)\to N\to 0,\]
where $N$ consists of the sequences $(a_n)$ with $a_n\in H^{i+1}(X(\C),\Z)$, $\ell^n\cdot a_n=0$, and such that
$a_{n-1}=\ell\cdot a_n$. In particular all the $a_n$ are $\ell$-power-torsion. Thus there exists an $n_0\ge 1$ such that
$\ell^{n_0}\cdot a_n=0$ for all $n$ (since $H^{i+1}(X(\C),\Z)$ is finitely generated over $\Z$).
Hence $a_n=\ell^{n_0}\cdot a_{n+n_0}=0$, i.e. $N=0$. 
\end{remark}

\begin{theorem}[Poincar\'e duality]\label{thm-PD}\index{Poincare@Poincar\'e duality}
Let $X$ be a smooth $k$-scheme of pure dimension $d$ and $\bar{k}$ an algebraic closure of $k$. 
Let $A$ be an $\ell$-adic coefficient ring and $\sF$ a  free lisse $A$-sheaf. Then for all $i\in \Z$ 
there is a natural (i.e. functorial in $\sF$) and ${\rm Gal}(\bar{k}/k)$-equivariant isomorphism
\[H^{2d-i}(X\otimes_k \bar{k}, \sF^\vee(d))\xr{\simeq} H^i_c(X\otimes_k \bar{k} ,\sF)^\vee.\]
Here $(d)$ is the Tate twist \ref{Tate twist}, 
$H^i_c(X\otimes_k \bar{k} ,\sF)^\vee:= \Hom_{A}(H^i_c(X\otimes_k \bar{k} ,\sF), A)$ and 
$\sF^\vee$ is defined as follows: write $\sF=(\sF_n)\otimes_R A$ with a free lisse $R$-sheaf $(\sF_n)$ 
(this is possible by Theorem \ref{thm-sheaf-vs-rep} and Lemma \ref{lemma:lattice} ), then
$\sF^\vee:=(\HHom(\sF_n, R/\fm^n))\otimes_R A$ is a free lisse $A$-sheaf.
\end{theorem}

\begin{proof}
If $A$ is finite this is \cite[XVIII, (3.2.6.2)]{SGA4}. (In {\it loc.~cit.}~it is done for $A=\Z/\ell^n$, but since 
for any DVR $R$ with maximal ideal $\fm$, the ring $R/\fm^n$ is self injective, the same argument works.)
In general we can assume that $A$ is flat over $R$ and write $\sF=(\sF_n)\otimes_R A$ with $(\sF_n)$ a free lisse $R$-sheaf.
Since $\sF_n=\sF_{n+1}\otimes R/\fm^n$ we have  
$\varinjlim_m\HHom(\sF_m,R/\fm^n)= \HHom(\sF_n, R/\fm^n)$. Therefore 
the Poincar\'e  isomorphism for the $\sF_n$ yields a ${\rm Gal}(\bar{k}/k)$-equivariant isomorphism
\[H^{2d-i}(X\otimes_k \bar{k}, \sF^\vee(d))\xr{\simeq} 
   (\varprojlim_n\varinjlim_m \Hom_{R}(H^i_c(X\otimes_k \bar{k} ,\sF_m), R/\fm^n))\otimes_R A,\]
where we use that taking cohomology commutes with direct limits.
Set $H_m:=  H^i_c(X\otimes_k \bar{k} ,\sF_m)$ and $H:=\varprojlim_m H^i_m$. 
By \cite[VI, Lem. 2.2.2]{SGA5} and \cite[V, Prop. 3.2.3]{SGA5} there exists an $r_0\ge 1$ such that for $r\ge r_0$
  and $m\ge 1$ we have 
\[\image(H_{m+r}\to H_m)=\image(H_{m+r_0}\to H_m)=:H'_m\]
and there exists an $s_0\ge 1$ such that for $s\ge s_0$  and $n\ge 1$ we have an isomorphism
$H\otimes_R R/\fm^n\xrightarrow{\simeq} H'_{n+s}\otimes_R R/\fm^n$. 
Hence
\begin{align}
\varinjlim_m \Hom_{R}(H_m, R/\fm^n) &\cong \varinjlim_m \Hom_R(H'_m, R/\fm^n)\notag\\
                                                         & \cong \varinjlim_m \Hom_{R/\fm^n}(H'_m\otimes_R R/\fm^n, R/\fm^n)\notag\\
                                                         &\cong \Hom_{R/\fm^n}(H\otimes_R R/\fm^n, R/\fm^n)\notag\\
                                                         &\cong  \Hom_{R}(H, R/\fm^n).\notag
\end{align}
We obtain
\[\varprojlim_n\varinjlim_m \Hom_{R}(H_m, R/\fm^n)  = \varprojlim_n \Hom_R(H, R/\fm^n)  = \Hom_R(H, R).\]
Since $R$ is noetherian, $H$ is actually a finitely presented $R$-module and since $A$ is flat over $R$
we obtain $\Hom_R(H, R)\otimes_R A\cong \Hom_A(H\otimes_R A, A)$. Putting all the isomorphisms together
we obtain the statement.
\end{proof}

\begin{corollary}\label{cor-affine-vanishing-copact-support}
Let $X$ be a smooth affine $k$-scheme of pure dimension $d$, 
$A$ an $\ell$-adic coefficient ring and $\sF$ a free  lisse $A$-sheaf on $X$. Then
\[H^i_c(X\otimes_k \bar{k}, \sF)=0,\quad \text{for all }i<d.\]
\end{corollary}
\begin{proof}
This follows from Poincar\'e duality and \ref{coh-properties}, \ref{affine-vanishing}.
\end{proof}

\begin{theorem}[Purity]\label{thm-purity}\index{purity}
Let $i: Y\inj X$ be a closed immersion of pure codimension $c$ between {\em smooth} $k$-schemes,
$A$ an $\ell$-adic coefficient ring and $\sF$ a lisse $A$-sheaf on $X$. Then there is a natural
${\rm Gal}(\bar{k}/k)$-equivariant isomorphism
\[H^j(Y\otimes_k \bar{k}, i^*\sF)\xr{\simeq} 
            H^{j+2c}_{Y\otimes_k \bar{k}}(X\otimes_k\bar{k},\sF(c)),\quad \text{for all }j,\]
where $(c)$ on the right-hand side denotes the Tate twist.
In particular $H^{j}_{Y\otimes_k \bar{k}}(X\otimes_k\bar{k},\sF)=0$ for all $j<2c$.
\end{theorem}
\begin{proof}
If $A$ is finite we have $R^j i^!\sF=0$, $j\neq 2c$, and $R^{2c} i^!\sF=i^*\sF(-c)$, by \cite[XVI, Cor 3.8, Rem. 3.10, a)]{SGA4}.
Therefore the statement follows in this case from the local-global spectral sequence
\[E_2^{a,b}=H^a(X\otimes_k\bar{k}, i_*R^b i^!\sF(c))\Rightarrow 
                                                     H^*_{Y\otimes_k\bar{k}}(X\otimes_k\bar{k}, \sF(c)).\]
For general $A$ applying $\varprojlim_n$ and $\otimes_R A$ gives the assertion.
\end{proof}

\begin{theorem}[Lefschetz trace formula, {\cite[Cycle, Cor 3.7]{SGA4.5}}]\label{thm-LTF}
                \index{Lefschetz trace formula}
Let $X$ be a smooth proper scheme over an algebraically closed field $k$ of characteristic $p>0$ and $\ell\neq p$ 
a prime number. Let $f: X\to X$ be a $k$-morphism and assume that the graph $\Gamma_f$ of $f$ and the diagonal
$\Delta_X$ intersect properly in $X\times X$ (i.e.~the intersection scheme $\Gamma_f\cap \Delta_X$ is 
either empty or zero dimensional). Then
\[(\Gamma_f\cdot \Delta_X)=\sum_i (-1)^i\Tr(f^*|H^i(X,\Q_\ell)).\]
Here the left-hand side is the degree of the intersection product $\Gamma_f\cdot \Delta_X$, concretely
\eq{thm-LFT1}{(\Gamma_f\cdot \Delta_X):= \sum_{x\in\Gamma_f\cap \Delta_X} {\rm length}(\sO_{X,x}/I_f),}
where $I_f\subset\sO_{X,x}$ is the ideal generated by the elements $a-f^*(a)$, $a\in \sO_{X,x}$.
\end{theorem}

\begin{lemma}\label{lem-group-action-coh}
Let $k$ be an algebraically closed field, $X$ an integral $k$-scheme with generic point $\eta$, 
$E$ a finite extension of $\Q_\ell$ with ring of integers $R$ and $\sF$ a constructible $E$-sheaf on $X$.
Suppose there is a finite group $G$ acting  on $\sF$, i.e.~there is a constructible $R$-sheaf $\sF'=(\sF_n)$
such that the $\sF_n$ form a projective system of $R[G]$-modules and $\sF=\sF'\otimes_R E$.
Then the $G$-invariant sections  $\sF_n^{G}$ form a projective system which induces a 
constructible $R$-sheaf denoted by $(\sF')^G$.
We define the constructible $E$-sheaf  
\[\sF^G:= (\sF')^G\otimes_R E.\]
Then 
\[H^i(X, \sF^G)=H^i(X,\sF)^G.\]
\end{lemma}
\begin{proof}
Set $N:=|G|$.
For all $n_1\ge n$ we have $\sF_{n_1}\otimes_{R} R/\fm^n= \sF_n$ by definition. 
We obtain natural maps 
\eq{lem-group-action-coh1}{(\sF_{n_1})^G\otimes_R R/\fm^n\to \sF_n^G,}
the cokernel of which is killed by $N$. (Indeed, it suffices to check this on the stalks. If we take a local section
$s_n\in\sF_n^G$  we can lift it to a local section $s\in\sF_{n_1}$. Then $s':=\sum_{\sigma\in G}\sigma(s)$
is a $G$-invariant section of $\sF^G_{n_1}$ which modulo $\fm^n$ equals $N\cdot s_n$.) 
Now the stalks of $\sF_n^G$ are finite (as sets). It follows that the images of the maps \eqref{lem-group-action-coh1}
become stationary for $n_1\to\infty$. We define $(\sF')^G_n\subset\sF^G_n$ to be the intersection over the images of all
these maps. Then $(\sF')^G_n$ is a constructible $R/\fm^n$-subsheaf of $\sF_n^G$  and by construction 
$(\sF')^G_{n+1}\otimes_R R/\fm^n=(\sF')^G_n$. We obtain the constructible $R$-sheaf
$(\sF')^G:=((\sF')^G_n)$ from the first part of the Lemma. Since the cokernel of the inclusion
$(\sF')^G_n\inj \sF^G_n$ is killed by $N$ we obtain:
\[H^i(X, \sF^G) \overset{\text{defn}}{=} (\varprojlim_n H^i(X, (\sF')^G_n))\otimes_R E
= (\varprojlim_n H^i(X, \sF_n^G))\otimes_R E.\]
 Let $\sF_n\to \sI^\bullet$ be a resolution of $\sF_n$ by injective $R/\fm^n[G]$-modules on $X_{\et}$.
The sheaf of $G$-invariants $(\sI^j)^G$ is an  injective sheaf of $R/\fm^n$-modules for all $j$.
(Since the  functor of taking $G$-invariants from the category of  $R/\fm^n[G]$-modules on $X_{\et}$ to the category of 
 $R/\fm^n$-modules on $X_{\et}$ has an exact left adjoint, namely the functor which to an $R/\fm^n$-module
$M$ associates the $R/\fm^n[G]$-module given by the same $M$ with trivial $G$-action.)
Further $H^0({\sI^\bullet}^G)=\sF_n^G$ and $H^j({\sI^\bullet}^G)$ is $N$-torsion for all $j\neq 0$,
by a similar argument as above. It follows that kernel and cokernel of 
$H^i(X, \sF_n^G)\to H^i(\Gamma(X,{\sI^\bullet}^G))$ are killed by $N$.
On the other hand $H^i(\Gamma(X,{\sI^\bullet}^G))=H^i(\Gamma(X,\sI^\bullet)^G)$ and kernel and
cokernel of the natural map $H^i(\Gamma(X,\sI^\bullet)^G)\to H^i(\Gamma(X,\sI^\bullet))^G=H^i(X,\sF_n)^G$
are killed by $N$. All together we obtain a natural map
\[H^i(X, \sF_n^G)\to H^i(X, \sF_n)^G\]
whose kernel and cokernel are killed by $2N$. Taking the limit and tensoring with $E$ we arrive at the statement.

\end{proof}

\part{Ramification theory of $\ell$-adic sheaves}

\section{Grothendieck-Ogg-Shafarevich}\label{sec:GOS} 
\setcounter{subsubsection}{0}
In this section we explain Grothendieck's proof of the Grothendieck-Ogg-Shafarevich formula, partly following Katz.
References for this section are \cite[Exp. X]{SGA5}, \cite[Ch. 2]{Katz/Kloosterman}, see also \cite{Raynaud}.

\subsubsection{}Throughout this section we fix the following notation:
\begin{itemize}
\item $k$ is a perfect field of characteristic $p>0$ and  $\bar{k}$ an algebraic closure.
\item $\ell$ is a prime number different from $p$.
\item $C$ is a smooth proper and geometrically connected curve over $k$.
\item $U\subset C$ is a {\em strict} open subset (in particular it is affine).
\item $K=k(C)$ is the function field of $C$, $\bar{K}$ an algebraic closure and $K^{\rm sep}\subset \bar{K}$ 
         a separable closure.  We denote by $\eta: \Spec K\to C$ the generic point of $C$ and by
          $\bar{\eta}: \Spec \bar{K}\to C$ the induced geometric point over $\eta$.
 \item\label{iotax} For a closed point $x\in C$ we denote by 
         $K_x$ the completion of $K$ with respect to the valuation $v_x:K^\times\to \Z$ corresponding to $x$
             and by $K^{\rm sep}_x$ a separable closure. For all $x$ we choose an embedding
          $\iota_x: K^{\rm sep}\inj K^{\rm sep}_x$ over $K$.
\item $G={\rm Gal}(K^{\rm sep}/K)$ is the absolute Galois group of $K$ and for a closed point $x\in C$
         we denote by   $D_x:=D_x^{\iota_x}$ the image of the inclusion
          ${\rm Gal}(K^{\sep}_x/K_x)\inj G$ induced by $\iota_x$, i.e. $D_x$ is the decomposition 
          subgroup of $G$ with respect to $\iota_x$. We denote by $P_x\subset I_x\subset D_x$ 
          the wild inertia and inertia subgroups, respectively.
         (Notice that if we take a different $K$-embedding $\iota_{x,1}:K^{\rm sep}\inj K^{\rm sep}_x$, 
          then there exists a $\sigma \in G$ such that $\iota_{x,1}=\iota_{x}\circ\sigma$; thus the resulting
           decomposition group is conjugate to $D_x$, i.e. $D^{\iota_{x,1}}_x=\sigma^{-1}D_x\sigma$.
            Same with $I_x, P_x$. Hence $D_x, P_x, I_x$ are only well-defined up to conjugation by elements in
            $G$.)
 \end{itemize}

\setcounter{subsection}{1}\setcounter{subsubsection}{0}

\subsubsection{}\label{swan-lisse-sheaf}
Let $A$ be an $\ell$-adic coefficient ring (Convention
\ref{conv-l-adic-coefficients}) and $\sF$ a free lisse $A$-module on $U$.
By Theorem \ref{thm-sheaf-vs-rep}, we can identify $\sF$ with a representation of $\pi_1(U,\bar{\eta})$ on 
$\sF_{\bar{\eta}}$. The natural surjection $G\to \pi_1(U,\bar{\eta})$ hence induces a 
$G$-action on $\sF_{\bar{\eta}}$, which we can further restrict to a $P_x$-action, $x\in C$.
The action of $P_x$ on $\sF_{\bar{\eta}}$ factors over a finite quotient of $P_x$, see 
Lemma \ref{lemma:proPfactorsThroughFiniteQuotient}. Thus the Swan conductor of the $P_x$-representation
on $\sF_{\bar{\eta}}$ is defined (Definition \ref{def:Swan-conductor}) and denoted by
\[\Swan_x(\sF).\]
We observe that $\Swan_x(\sF)$ is independent of the chosen embedding $\iota_x:
K^{\rm sep}\inj K^{\rm sep}_x$. Indeed, if we choose a different embedding
there is a $\sigma\in G$ such that
we have to consider the restriction of the $G$-action on $\sF_{\bar{\eta}}$ to the group $\sigma^{-1}P_x\sigma$.
But then $\sigma$ maps the break decomposition of $\sF_{\bar{\eta}}$ (Definition \ref{def:breakDecomposition}) 
with respect to $\sigma^{-1}P_x\sigma$ isomorphically to the break decomposition of $\sF_{\bar{\eta}}$
with respect to $P_x$. Hence we obtain the same number $\Swan_x(\sF)$,
 when we compute it with respect to the $\sigma^{-1}P_x\sigma$-action.

We recall the following facts (see Remark \ref{rem:SwanAndBasechange}, Theorem \ref{thm:integralityOfSwanConductor}):
\begin{enumerate}
\item\label{swan-lisse-sheaf-bc} 
 Let $A'$ be an $\ell$-adic coefficient ring and $A\to A'$ a ring homomorphism. Then for all $x\in C$
\[\Swan_x(\sF\otimes_A A')=\Swan_x(\sF).\]
\item $\Swan_x(\sF)\in \Z_{\ge 0}$.
\item\label{swan-lisse-sheaf-tame} $\Swan_x(\sF)=0\Longleftrightarrow $ $\sF$ is tame at $x$, i.e.~$(\sF_{\bar{\eta}})^{P_x}=\sF_{\bar{\eta}}$.
        In particular, $\Swan_x(\sF)=0$ for all $x\in U$.
\item\label{swan-lisse-sheaf-exactness} 
         Let $0\to \sF'\to\sF\to \sF''\to 0 $ be an exact sequence of free lisse $A$-sheaves on $U$. Then
          \[\Swan_x(\sF)=\Swan_x(\sF')+\Swan_x(\sF'').\]
\end{enumerate}

\subsubsection{} Recall  that the compactly supported Euler characteristic of a  lisse $\Qlb$-sheaf $\sF$ on $U$ 
is defined to be 
\begin{align*}
\chi_c(\bar{U},\sF) & =\sum_i (-1)^i \dim_{\Qlb} H^i_c(\bar{U}, \sF)\\
     &     =-\dim_{\Qlb} H^1_c(\bar{U},\sF)+\dim_{\Qlb} H^2_c(\bar{U}, \sF),
\end{align*}
where $\bar{U}=U\otimes_k\bar{k}$ and we abuse notation denoting the pullback of $\sF$ to $\bar{U}$ still by $\sF$.

The {\em Grothendieck-Ogg-Shafarevich}\index{Grothendieck-Ogg Shafarevich theorem}
 formula is the following theorem.
\begin{theorem}\label{thm-GOS}
Let $\sF$ be a lisse $\Qlb$-sheaf on $U$. Then
\[\chi_c(\bar{U},\sF)=
\rk(\sF)\cdot\chi_c(\bar{U},\Qlb) -\sum_{x\in C\setminus U}[k(x):k]\cdot \Swan_x(\sF).\]
\end{theorem}

{\em Outline of the proof:} First we give a more general definition of the compactly supported Euler characteristic
which works for any free lisse $A$-sheaf $\sF$, where
 $A$ is any $\ell$-adic coefficient ring, and does not change when we replace $\sF$ by 
$\sF\otimes_A A'$, for any $A\to A'$. Using this and the corresponding fact for the Swan conductor we can 
reduce the theorem to a statement about lisse $\F_\lambda$-sheaves, where $\F_\lambda$ is a finite extension of $\F_\ell$.
The corresponding representation has finite monodromy group $G'$. Using facts from the representation theory in mixed characteristic
explained in Section \ref{subsubsec:RepMixedChar} we find a lisse $E$-sheaf $\sF'$, where $E$ is a finite extension of $\Q_\ell$, which
has the same finite monodromy group $G'$ and its compactly supported Euler characteristic and its Swan conductors are equal
to those of $\sF$. Thus we are reduced to the case where $\sF$ is a lisse $\bar{\Q}_\ell$-sheaf on $U$ which is trivialized 
by a finite \'etale cover $U'\to U$. In this case we can prove the theorem via a direct computation relying 
on the Lefschetz trace formula, the Hurwitz genus formula and the results from Section \ref{sec:swan}.
\\
\\
After developing some preliminaries, we will give a detailed proof at the end of this subsection.

\begin{lemma}\label{lem-swan-alg-closed}
Let $\bar{x}\in C\otimes_k\bar{k}$ be a closed point mapping to $x\in C$.
Let $A$ be an $\ell$-adic coefficient ring and $\sF$ a lisse $A$-sheaf on $U$.
Then 
\[\Swan_x(\sF)=\Swan_{\bar{x}}(\sF),\]
where here as everywhere we abuse notation and write $\sF$ on the right-hand side instead of $\sF_{|U\otimes_k\bar{k}}$.
\end{lemma}
\begin{proof}
The point $\bar{x}$ corresponds to a prime ideal in $k(x)\otimes_k \bar{k}$ which by Hensel's Lemma
corresponds uniquely to a prime ideal in $K_x\otimes_k \bar{k}$.
Denote by $L$ the completion of the function field of $U\otimes_k \bar{k}$
with respect to $\bar{x}$; it is the localization of $K_x\otimes_k \bar{k}$
with respect to the prime ideal corresponding to $\bar{x}$. 
Then $L$ is a complete discrete valuation field with residue field $\bar{k}$. 
Since $k$ is perfect, the canonical inclusion $K_x\inj L$ is unramified. 
It follows that $L$ is isomorphic to the maximal unramified extension of $K_x$ inside 
$K_x^{\rm sep}$, hence its absolute Galois group is isomorphic to the inertia group $I_x$.
Hence the restriction to $I_x$ of the representation  corresponding to $\sF$ is isomorphic
to the representation corresponding to $\sF_{|U\otimes_k \bar{k}}$.
The claim follows.
\end{proof}

\begin{lemma}\label{lem-proj-resolution}
Assume $k=\bar{k}$. Let $A$ be an $\ell$-adic coefficient ring and $\sF$ a lisse $A$-sheaf on $U$.
Then there exists a two term complex of $A$-modules
\[C(\sF): C^1(\sF)\to C^2(\sF)\]
with the following properties:
\begin{enumerate}
\item $H^i(C(\sF))= H^i_c(U, \sF)$, all $i$.
\item\label{lem-proj-resolution-free} If $\sF$ is free then $C^i(\sF)$ is a free $A$-module of finite rank.
\item\label{lem-proj-resolution-exact} The functor $\sF\mapsto C(\sF)$ is exact.
\item\label{lem-proj-resolution-bc} 
  If $A'$ is an $\ell$-adic coefficient ring and $A\to A'$ a ring homomorphism and $\sF$ is free, we have a canonical  
     isomorphism   $C(\sF)\otimes_A A' \cong C(\sF\otimes_A A')$.
\end{enumerate}
\end{lemma}
\begin{proof}We are following the proof given in \cite[Lem.~2.2.7]{Katz/Kloosterman}.
Denote by $j:U\inj C$ the inclusion of $U$ into $C$. Pick once and for all a closed point $P\in U$.
It is a closed subscheme of $U$ of codimension 1 which is smooth over $k$. 
Thus the long exact localization
sequence \ref{coh-properties}, \ref{coh-support} together with purity (Theorem \ref{thm-purity}) gives a long exact sequence
\[\ldots\to H^i(C, j_!\sF)\to H^i(C\setminus P, (j_!\sF)_{|C\setminus P})\to
        H^{i-1}(P, (j_!\sF)(-1)_{|P})\to\ldots. \]
By definition $H^i(C, j_!\sF)= H^i_c(U,\sF)$, which vanishes for $i\neq 1,2$ by 
Corollary \ref{cor-affine-vanishing-copact-support}. By \ref{coh-properties}, \ref{affine-vanishing}, we have
\eq{lem-proj-resolution1}{H^i(P, (j_!\sF)(-1)_{|P})=0,\quad \text{for } i\neq 0.}
From this  and \ref{coh-properties}, \ref{affine-vanishing} we get
\eq{lem-proj-resolution2}{H^i(C\setminus P, (j_!\sF)_{|C\setminus P})=0,\quad
              \text{for } i\neq 1.}
Set 
\[C^1(\sF):=H^1(C\setminus P, (j_!\sF)_{|C\setminus P}),\quad 
 C^2(\sF):=H^0(P, (j_!\sF)(-1)_{|P}). \]
By \eqref{lem-proj-resolution1} and \eqref{lem-proj-resolution2} the functors $\sF\mapsto C^i(\sF)$ are exact and
by the long exact sequence above we get a complex $C^1(\sF)\to C^2(\sF)$ whose cohomology groups are equal
to $H^i_c(U,\sF)$. It remains to prove \ref{lem-proj-resolution-free} and \ref{lem-proj-resolution-bc}.
Notice that both statements are obvious if $A$ is a field extension of $\Q_\ell$. Thus we may assume that
there exists a DVR $R$ finite over $\Z_\ell$ that surjects onto $A$. In this case we can represent $\sF$ as a
projective system $(\sF_n)$ and by definition $C^i(\sF)=\varprojlim_n C^i(\sF_n)$.
Since $C^i(\sF)$ is a finitely generated $A$-module it suffices to prove \ref{lem-proj-resolution-free}
for $A$ finite, i.e. $A=R/\fm^n$.  Furthermore let $\pi\in\fm$ be a local parameter of $R$.
Then we have an exact sequence $0\to\sF\xr{\cdot\pi^n}\sF\to\sF_n\to 0$, which yields
$C^i(\sF)/\pi^n\cong C^i(\sF_n)$. If $R\to R'$ is a finite extension we get
\[C^i(\sF)\otimes_R R'=\varprojlim_n(C^i(\sF)\otimes_{R/\fm^n} R'/{\fm'}^n )=
       \varprojlim_n(C^i(\sF_n)\otimes_{R/\fm^n} R'/{\fm'}^n ).\]
Thus in \ref{lem-proj-resolution-bc} it suffices to consider $A$  (and hence also $A'$) finite.

So let $A$ be a finite local ring with residue characteristic prime to $p$ and $\sF$ a lisse  $A$-sheaf
on $U$. If $M$ is a finitely generated $A$-module we find a resolution $P_\bullet\to M$ by free finitely 
generated $A$-modules. We can view $M$ and the $P_j$ as constant sheaves on $U$ and form the
complex of lisse $A$-sheaves $\sF\otimes_A P_\bullet$, which is augmented towards $\sF\otimes_A M$.
The $j$-th homology of this complex is by definition $\sTor_j^A(\sF, M)$, which is a lisse
$A$-sheaf. Since $P_\bullet$ is a complex of free $A$-modules, we obtain 
$C^i(\sF\otimes_A P_\bullet)=C^i(\sF)\otimes_A P_\bullet$. All together the exactness of the $C^i$ gives
\[C^i(\sTor_j^A(\sF, M))=H_j(C^i(\sF\otimes_A P_\bullet))=
   H_j(C^i(\sF)\otimes_A P_\bullet)={\sTor}_j^A(C^i(\sF), M).\]
In case $\sF$ is free, the left-hand side vanishes for $j\ge 1$ thus $C^i(\sF)$ 
is a finitely generated flat $A$-module, hence is free. Furthermore, the case $j=0$ gives the base change statement.
This finishes the proof.
\end{proof}
As an immediate consequence we get:
\begin{corollary}\label{cor-EulerChar-lisse-A-sheaf}
Assume $k=\bar{k}$.
Let $\sF$ be a free lisse $A$-sheaf on $U$ and $C(\sF)$ the complex from Lemma \ref{lem-proj-resolution}.
Then we define 
\[\chi_c(U,\sF):= -\rk_A(C^1(\sF))+\rk_A(C^2(\sF)).\]
We have:
\begin{enumerate}
\item If $H^i_c(U,\sF)$ is a free $A$-module, for $i=1,2$, then 
          \[\chi_c(U,\sF)=\sum_i (-1)^i \rk_A H^i_c(U, \sF).\]
\item\label{cor-EulerChar-lisse-A-sheaf-bc} 
          If $A'$ is an $\ell$-adic coefficient ring and $A\to A'$ a ring homomorphism, then
        \[\chi_c(U,\sF)= \chi_c(U,\sF\otimes_A A').\]
\item\label{cor-EulerChar-lisse-A-sheaf-exactness}
        If $0\to \sF'\to \sF\to \sF''\to 0$ is an exact sequence of free lisse $A$-sheaves on $U$, then
        \[\chi_c(U,\sF)=\chi_c(U,\sF')+\chi_c(U,\sF'').\]
\end{enumerate}
\end{corollary}

\subsubsection{}\label{Swan-representation-geometry}
Assume $k=\bar{k}$.
Let $U'\to U$ be a connected finite \'etale Galois covering with Galois group $G_{U'}={\rm Gal}(K'/K)$, where $K'=k(U')$.
We can extend it uniquely to a finite morphism $C'\to C$ between smooth proper and connected curves
over $k$. Let $x'\in C'\setminus U'$ be a closed point lying over 
$x\in C\setminus U$ and 
$G_{U',x'}=\{\sigma\in G_{U'}\,|\, \sigma(x')=x'\}={\rm Gal}(K'_{x'}/K_x)$ the decomposition group of 
 $G_{U'}$ at $x'$. By Theorem \ref{thm:integralityOfSwan} there exists a finitely generated and projective $\Z_\ell[G_{U',x'}]$-module
    \[\Sw_{G_{U',x'}}\]
   underlying the Swan representation of $G_{U',x'}$.
We define
\[\Sw_{G_{U'}, x}:= \Sw_{G_{U',x'}}\otimes_{\Z_{\ell}[G_{U',x'}]} \Z_\ell[G_{U'}].\]
and we denote its character by
\eq{Swan-representation-geometry:char}{\sw_{G_{U'}, x}: G_{U'}\to \Z_\ell.}

\begin{proposition}\label{prop-swan-formula}
In the situation above we have the following properties:
\begin{enumerate}
\item\label{prop-swan-formula-independence} The definition of $\Sw_{G_{U'}, x}$ is independent of the choice of $x'/x$.
\item\label{prop-swan-formula-sw} \[\sw_{G_{U'},x}(\sigma)= \begin{dcases} 
  \sum_{\genfrac{}{}{0pt}{}{y/x}{\sigma(y)=y}} (1-i_{G_{U',y}}(\sigma)),  & \text{if }\sigma\neq1,\\
  \left(\sum_{y/x} (1+v_{y}(\mathfrak{D}_{C'/C}))\right)- |G_{U'}|, &\text{if } \sigma=1,
 \end{dcases}\]
     where $y\in C'$ maps to $x$,  $i_{G_{U',y}}: G_{U',y}\to \Z$ is the function 
  defined in \eqref{iG} and $v_{y}(\mathfrak{D}_{C'/C})$ is the multiplicity at $y$ of the different
   $\mathfrak{D}_{B/\sO_{C,x}}$, where $B$ is the integral closure of $\sO_{C,x}$ in $K'$,
   see Definition \ref{def:different}.
\item\label{prop-swan-formula-cond-rep} Let $R$ be a complete DVR which is finite over $\Z_\ell$ 
         with maximal ideal $\fm$ and residue field $\F_\lambda$. 
Let $\sF$ be a free lisse $R$-sheaf and assume that $(\sF\otimes_R \F_\lambda)_{|U'}$ is trivial, i.e.
we can view $\sF_{\bar{\eta}}\otimes_R \F_\lambda$ as a $\F_{\lambda}[G_{U'}]$-module. 
    Then
  \[\Swan_x(\sF)= \Swan_x(\sF\otimes_R \F_\lambda)=   \dim_{\F_\lambda} 
 \Hom_{\F_\lambda[G_{U'}]}(\Sw_{G_{U'}, x}\otimes_{\Z_\ell} \F_\lambda,\sF_{\bar{\eta}}\otimes_{R} \F_\lambda ).\]
     \item\label{prop-swan-formula-Swan} 
       Let $\sF$ be a  lisse $\Qlb$-sheaf on $U$ such that $\sF_{|U'}$ is trivial, i.e. $\sF_{\bar{\eta}}$ is a 
         $\Qlb[G_{U'}]$-module.
     Then for all closed points $x\in C$
        \[\Swan_x(\sF)= 
            \frac{1}{|G_{U'}|}\sum_{\sigma\in G_{U'}} \sw_{G_{U'},x}(\sigma)\cdot\Tr(\sigma|\sF_{\bar{\eta}}).\]
\end{enumerate}
\end{proposition}
\begin{proof}
Denote by $\iota_{x'}: G_{U',x'}\inj G_{U'}$ the inclusion.
By Theorem \ref{thm:mainRepThmIntegrality},
\ref{thm:mainRepThmIntegrality-e-inj}, to prove
\ref{prop-swan-formula-independence} it suffices
 to check that the character of
\[\Ind_{\iota_{x'}}(\Sw_{G_{U',x'}})=\Sw_{G_{U',x'}}\otimes_{\Z_{\ell}[G_{U',x'}]} \Z_\ell[G_{U'}]\]
is independent of the choice of $x'$. Hence it suffices to prove \ref{prop-swan-formula-sw}.
To this end, notice that by Lemma \ref{lem:ind-char} the character of $\Ind_{\iota_{x'}}(\Sw_{G_{U',x'}})$
is given by
\[G_{U'}\ni\sigma\mapsto\chi(\sigma):= \frac{1}{|G_{U',x'}|} 
\sum_{\substack{\tau\in G_{U'}\\ \tau\sigma\tau^{-1}\in G_{U',x'}}} \sw_{G_{U',x'}}(\tau\sigma\tau^{-1}),\]
where $\sw_{G_{U',x'}}$ is the Swan character of $G_{U',x'}$, see Definition \ref{def:swan-char}.

{\em First case:} $\sigma\neq 1$.  We have 
\[\tau\sigma\tau^{-1}\in G_{U',x'}\Longleftrightarrow \sigma\in G_{U', \tau^{-1}(x')}=\tau^{-1}\circ G_{U',x'}\circ \tau\]
and
\[\sw_{G_{U',x'}}(\tau\sigma\tau^{-1})=\sw_{G_{U',\tau^{-1}(x')}}(\sigma).\] 
Since $G_{U'}$ acts transitively on the points over $x$, we have a bijection
of sets $G_{U'}/G_{U',x'}\xr{1:1} \{y/x\}\subset C'$. Hence
\[\chi(\sigma)=\sum_{\substack{y/x \\ \sigma(y)=y}} \sw_{G_{U',y}}(\sigma).\]
The formula in \ref{prop-swan-formula-sw} for $\sigma\neq 1$ thus follows from the definition of 
$\sw_{G_{U',y}}$, see Definition \ref{def:swan-char}. (Notice that since $k=\bar{k}$ we have $G_{U',x'}=G_{U',x'}^0$.)

{\em Second case:} $\sigma=1$. In this case
\mlnl{\chi(1)=\frac{|G_{U'}|}{|G_{U',x'}|} \cdot \sw_{G_{U',x'}}(1)=\sum_{y/x} \sw_{G_{U',y}}(1)\\
            =\sum_{y/x} \left((\sum_{\sigma\neq1} i_{G_{U',y}}(\sigma))- (|G_{U',y}|-1)\right).}
Since $\sum_{\sigma\neq1} i_{G_{U',y}}(\sigma)=v_y(\mathfrak{D}_{C'/C})$ (see
Proposition \ref{valuation-of-discriminant})
and $\sum_{y/x} |G_{U',y}|=|G_{U'}|$ we obtain the formula for  \ref{prop-swan-formula-sw}.

The formula in \ref{prop-swan-formula-cond-rep} follows directly from 
$\Swan_x(\sF)=b(\RES_{\iota_{x'}}(\sF_{\bar{\eta}}))$ (see Theorem \ref{thm:integralityOfSwanConductor}), the
definition of  $b(-)$ in Definition \ref{defn:SerresB} and that $\Ind_{\iota_{x'}}$ is left adjoint to 
$\RES_{\iota_{x'}}$, see \eqref{Ind-Res}.

Finally, let us prove \ref{prop-swan-formula-Swan}. First notice that if $P$ and $V$ are two finite dimensional $\Qlb$-representations 
of $G_{U'}$, then 
\[\Hom_{\Qlb[G_{U'}]}(P,V)= (P^\vee \otimes_{\Qlb} V)^{G_{U'}},\]
see Example \ref{ex:representationsFirstExamples}.
The $\Qlb$-dimension of the right-hand side is equal to the number of times the trivial
representation appears in $P^\vee \otimes_{\Qlb} V$.
Let $\chi_P$ (resp. $\chi_V$) be the character of $P$ (resp. $V$); the character of
$P^\vee$ is given by $\chi_{P^\vee}(\sigma)=\chi_{P}(\sigma^{-1})$. 
The character of $P^\vee \otimes_{\Qlb} V$ is given by $\chi_{P^\vee}\cdot \chi_V$. Hence 
by Corollary \ref{cor-irred-rep-in-rep} we obtain
\[\dim_{\Qlb} \Hom_{\Qlb[G_{U'}]}(P,V)= 
 \frac{1}{|G_{U'}|} \sum_{\sigma\in G_{U'}} \chi_{P}(\sigma^{-1})\cdot \chi_{V}(\sigma).\]
Therefore the statement for \ref{prop-swan-formula-Swan} follows from the formula
(cf. \ref{prop-swan-formula-cond-rep})
\[\Swan_x(\sF)=   \dim_{\Qlb} \Hom_{\Qlb[G_{U'}]}(\Sw_{G_{U'}, x}\otimes_{\Z_\ell} \Qlb,\sF_{\bar{\eta}}),\]
the discussion above and $\sw_{G_{U'}}(\sigma)=\sw_{G_{U'}}(\sigma^{-1})$, see \ref{prop-swan-formula-sw}.
\end{proof}

\begin{lemma}\label{lem-intersection-multiplicities}
In the situation of \ref{Swan-representation-geometry}  denote 
by $\Delta_{C'}\subset C'\times_{k} C'$ the diagonal and by $\Gamma_\sigma$
the graph of $\sigma$ acting on $C'$.
Then $\Gamma_\sigma$ and $\Delta_{C'}$ intersect properly and
\[(\Gamma_\sigma\cdot \Delta_{C'})=\sum_{x\in C\setminus U} 
\sum_{\genfrac{}{}{0pt}{}{x'/x}{\sigma(x')=x'}}  i_{G_{U',x'}}(\sigma),\quad \text{for all } \sigma\neq 1,\]
where the left-hand side is defined as in \eqref{thm-LFT1}. 
\end{lemma} 
\begin{proof}
Take $\sigma\in G_{U'}\setminus\{1\}$ and  $x'\in C'$ with $\sigma(x')=x'$.
(Notice that this implies $x'\in C'\setminus U'$ since $G_{U'}$ acts transitively and freely on the fibers over $U$.)
We have to show 
\[i_{G_{U',x'}}(\sigma)= {\rm length}(\sO_{C',x'}/ I_{\sigma}),\]
where $I_\sigma\subset \sO_{C',x'}$ is the ideal generated by the elements $a-\sigma^*(a)$, $a\in \sO_{C', x'}$. 
Denote by $x$ the image of $x'$ in $C$. Denote by $A$ the completion of $\sO_{C,x}$ and by 
$A'$ the completion of $\sO_{C',x'}$.  Then  the right-hand side of the equality above  is equal to
${\rm length}(A'/I_\sigma\cdot A')$.
Further by Theorem \ref{thm:generator} there exists a generator of $A'/A$, i.e.~an element 
$\alpha\in A'$ such that $A'=A[\alpha]$. It follows that $I_\sigma$ is the ideal generated
by $\alpha-\sigma^*(\alpha)$. Thus by definition:
${\rm length}(A'/I_\sigma\cdot A')= v_{x'}(\alpha-\sigma^*(\alpha))=i_{G_{U',x'}}(\sigma).$
\end{proof}   

\begin{notation}\label{not-EC}
Let $D$ be a smooth projective curve over an algebraically closed field  and 
$V\subset D$ a strict non-empty open subset. In Corollary \ref{cor-EulerChar-lisse-A-sheaf}, 
we defined $\chi_c(V,A)$ and in part \ref{cor-EulerChar-lisse-A-sheaf-bc} of this corollary we
saw that it is independent of the choice of $A$. We set
\[\chi_c(V):=\chi_c(V, A),\quad \chi(D):=\chi_c(V)+{\rm card}(D\setminus V)\]
where $A$ is some $\ell$-adic coefficient ring. 
We have an exact  sequence 
$\cdots\to H^i(D, A)\to H^i(D\setminus V, A)\to H^{i+1}_c(V, A)\to\cdots$. This shows that if
$A$ is a field then $\chi(D)=\sum_i (-1)^i \dim_A H^i(D,A)$.
\end{notation}

\begin{lemma}\label{lem-EulerChar-Hur}
In the situation of \ref{Swan-representation-geometry} we have
\[\chi(C')=2-2g(C')= |G_{U'}| \cdot \chi(C)
                                  -\sum_{x'\in C'\setminus U'} v_{x'}(\mathfrak{D}_{C'/C}),\]
where $v_{x'}(\mathfrak{D}_{C'/C})$  is as in Proposition \ref{prop-swan-formula}, \ref{prop-swan-formula-sw}
and $g(C')$ is the genus of the connected smooth projective curve
$C'$. In particular, $\chi_c(U')=2-2g(C')-\card(C'\setminus U')$.
\end{lemma}
\begin{proof}
Since we are over an algebraically closed field we have $\mu_\ell\cong \Z/\ell\Z$.
Hence $\chi(C)= \sum_{i=0}^2 (-1)^i\dim_{\F_\ell} H^i(C,\mu_\ell)$
and  \cite[Arcata III, Cor.~(3.5)]{SGA4.5} yields $\chi(C)=2-2g(C)$.
(This uses that if $J$ is the Jacobian of $C$ then $H^1(C,\mu_\ell)\cong J(k)[\ell]\cong (\Z/\ell\Z)^{2g}$.)
Now the lemma follows directly from the Hurwitz genus formula, see e.g. \cite[IV, Cor 2.4]{Ha77}.
\end{proof}

\begin{proof}[\bf Proof of Theorem \ref{thm-GOS}.]
	Given a closed point $x\in C$ the number of closed points $\bar{x}\in C\otimes_k \bar{k}$ lying over $x$ is $[k(x):k]$.
Hence by Lemma \ref{lem-swan-alg-closed} it suffices to consider the case $k=\bar{k}$.
Let $\sF$ be a lisse $\Qlb$-sheaf on $U$. By Theorem \ref{thm-sheaf-vs-rep} and Lemma \ref{lemma:lattice} 
we find a finite field extension $E/\Q_\ell$ with ring of integers
$R/\Z_\ell$ and a free lisse $R$-sheaf $\sF'$ such that $\sF=(\sF'\otimes_R E)\otimes_{E}\Qlb$. Thus by
\ref{swan-lisse-sheaf}, \ref{swan-lisse-sheaf-bc} and 
Corollary \ref{cor-EulerChar-lisse-A-sheaf}, \ref{cor-EulerChar-lisse-A-sheaf-bc}  it suffices to prove the following:
Let $\F_\lambda$ be a finite field of characteristic $\ell$ and $\sF$ a lisse $\F_\lambda$-sheaf of rank $r$, then 
\eq{GOS-Fl}{\chi_c(U,\sF)=
r\cdot\chi_c(U) -\sum_{x\in C\setminus U} \Swan_x(\sF).}
The sheaf $\sF$ corresponds to a homomorphism $\pi_1(U,\bar{\eta})\rightarrow\GL_r(\F_{\lambda})$;
since its target is a finite group
 we find a connected finite \'etale Galois cover $\pi: U'\to U$ with Galois group $G_{U'}$ which trivializes $\sF$.

Let $A$ be an $\ell$-adic coefficient ring. 
By Theorem \ref{thm-sheaf-vs-rep} we can identify the category of free lisse $A$-sheaves on $U$ which are trivial on $U'$ 
with the category of finitely generated $A[G_{U'}]$-modules which are free as $A$-modules.
We can restrict the functors $C^i$ from Lemma \ref{lem-proj-resolution-free}  to this category and obtain
functors 
\[T^i: (\text{fin. generated } A\text{-free } A[G_{U'}]\text{-modules})\to (\text{free }A\text{-modules})\]
with a natural transformation $T^1\to T^2$, such that for lisse $A$-sheaves 
$\sG$ which are trivialized by $\pi: U'\to U$ we get
\[H^i_c(U, \sG)= H^i(T^1(\sG_{\bar{\eta}})\to T^2(\sG_{\bar{\eta}})),\quad
	i\geq 0.\]
We define a function
\eq{GOS-nu}{\nu_A:  (\text{fin. generated } A\text{-free } A[G_{U'}]\text{-modules})\to \Z}
by
\mlnl{\nu_A(M):=
\rk_A (T^2(M))-\rk_A (T^1(M))-( \rk_A (M)\chi_c(U,\Qlb)
                                                                  - \sum_{x\in C\setminus U} \Swan_x(M)).}
Denote by $R_A(G_{U'})$ the Grothendieck group of the category on the left-hand side of \eqref{GOS-nu} 
(see Definition \ref{def:GrothendieckRing}).
Then it follows from  \ref{swan-lisse-sheaf}, \ref{swan-lisse-sheaf-exactness}
and Corollary \ref{cor-EulerChar-lisse-A-sheaf}, \ref{cor-EulerChar-lisse-A-sheaf-exactness}, 
that $\nu_A$ induces a well-defined group homomorphism
\[\nu_A: R_A(G_{U'})\to \Z.\]

Since we reduced to showing \eqref{GOS-Fl}, we have to prove that
$\nu_{\F_\lambda}=0$.
Let $E$ be a finite extension of $\Q_\ell$ with ring of integers $R$ and residue field $\F_\lambda$.
By Proposition \ref{prop:triangle} there is a homomorphism $d: R_E(G_{U'})\to R_{\F_{\lambda}}(G_{U'})$.
Recall that $d$ is constructed as follows: Let $V$ be an $E[G_{U'}]$-module, which is finite dimensional 
as an $E$-vector space. Take $M\subset V$ an $R[G_{U'}]$-submodule which is an $R$-lattice in $V$. 
Then $d([V])= [M\otimes_R\F_{\lambda}]$. It follows from \ref{swan-lisse-sheaf}, \ref{swan-lisse-sheaf-bc}
and Corollary \ref{cor-EulerChar-lisse-A-sheaf}, \ref{cor-EulerChar-lisse-A-sheaf-bc} that we obtain a commutative 
diagram
\[\begin{tikzcd}
			R_E(G_{U'})\ar[swap]{dd}{d}\ar{rd}{\nu_{E}}& ~\\
                           ~ &   \Z.\\
                   R_{\F_\lambda}(G_{U'}) \ar[swap]{ru}{\nu_{\F_\lambda}}   & ~
		\end{tikzcd}\]
By Theorem \ref{thm:mainRepThmIntegrality}, \ref{thm:mainRepThmIntegrality-d-surj} the map $d$ 
is surjective. Hence it suffices to show that $\nu_E=0$. Retranslating this
into lisse $E$-sheaves, we see that it suffices to prove  
equality \eqref{GOS-Fl} for $\sF$ a lisse $E$-sheaf of rank $r$ which is trivialized by $\pi: U'\to U$.
Set $V=\sF_{\bar{\eta}}$. It is an $E$-vector space of dimension $\dim_E V=r$ together with a $G_{U'}$-action.
We have (see Theorem \ref{thm-sheaf-vs-rep})
\[\pi^*\sF\cong V_{U'},\quad \sF\cong (\pi_*V_{U'})^{G_{U'}},\]
where $V_{U'}$ is the constant lisse $E$-sheaf on $U'$ defined by $V$.
Denote by $\bar{\pi}: C'\to C$ the unique finite morphism between smooth proper curves over $k$, which restricts
to $\pi: U'\to U$, and by $j: U\inj C$, $j':U'\inj C'$ the open immersions.
We obtain (via a direct computation)
\[j_!\sF\cong j_!(\pi_*V_{U'})^{G_{U'}}\cong (j_!\pi_*V_{U'})^{G_{U'}}\cong 
           (\bar{\pi}_*j'_!V_{U'})^{G_{U'}}.\]
By Lemma \ref{lem-group-action-coh} and \ref{coh-properties}, \ref{finite-map} we get
\[H^i_c(U,\sF)= H^i(C, (\bar{\pi}_*j'_!V_{U'})^{G_{U'}})=
  H^i(C, \bar{\pi}_*j'_!V_{U'})^{G_{U'}}=H^i_c(U', V_{U'})^{G_{U'}}.\]

  Next,  if $R$ is the ring of integers of $E$, then we find an $R$-lattice $M\subset V$ with $G_{U'}$-action such that $V\cong M\otimes_R E$ as $E[G_{U'}]$-modules.
If $(R/\fm^n)_{U'}\to \sI^\bullet$ is an injective resolution of the constant sheaf $R/\fm^n$ on $U'$, then 
$\sI^\bullet\otimes_{R/\fm^n} (M\otimes_R R/\fm^n)_{U'}$ is a resolution by injectives of the constant sheaf 
$(M\otimes_R R/\fm^n)$ which is compatible with the $G_{U'}$-action. 
This gives a $G_{U'}$-equivariant isomorphism 
$H^i_c(U', V_{U'})\cong H^i_c(U', E_{U'})\otimes_E V$. All together we obtain
\[H^i_c(U, \sF)=(H^i_c(U', E_{U'})\otimes_E V)^{G_{U'}}. \]
The dimension of the $E$-vector space $(H^i_c(U', E_{U'})\otimes_E V)^{G_{U'}}$
is equal to the number of times the trivial rank one representation is contained in
the $E[G_{U'}]$-module $H^i_c(U', E_{U'})\otimes_E V$. Hence Corollary \ref{cor-irred-rep-in-rep}
yields
\begin{align}\label{GOS1}
\chi_c(U,\sF)  &=  -\dim_E (H^1_c(U', E)\otimes_E V)^{G_{U'}}+ 
           \dim_E(H^2_c(U', E)\otimes_E V)^{G_{U'}}\\
       &=\frac{1}{|G_{U'}|}\sum_{\sigma\in G_{U'}} 
\Tr(\sigma|V)\left(\Tr(\sigma^*|H^2_c(U', E))-\Tr(\sigma^*|H^1_c(U', E))\right).\notag
\end{align}
Set $Y':=C'\setminus U'$ . The exact sequence 
$\cdots\to H^i(C', E)\to H^i(Y', E)\to H^{i+1}_c(U', E)\to\cdots$ from \ref{coh-properties}, \ref{compact-localization} yields 
a short exact sequence
\[0\to H^0(C', E)\to H^0(Y', E) \to H^1_c(U', E)\to H^1(C', E)\to 0\]
and an isomorphism
\[H^2_c(U',E)\cong H^2(C', E).\]
For $\sigma\in G_{U'}$ we obtain
\ml{GOS4}{\Tr(\sigma^*|H^2_c(U', E))-\Tr(\sigma^*|H^1_c(U', E))\\
     = - \Tr(\sigma^*|H^0(Y',E))+ \sum_{i=0}^2 (-1)^i \Tr(\sigma^*|H^i(C', E)).}
We can write
\[H^0(Y',E)=\bigoplus_{x'\in Y'} E_{x'},\]
where $E_{x'}=E$ and $\sigma\in G_{U'}$ acts via:
\[\sigma^*: E_{x'}= E\xr{\id} E=E_{\sigma(x')}.\]
Thus
\eq{GOS5}{
\Tr(\sigma^*|H^0(\bar{Y}',E))=
      \sum_{x\in Y}  \sum_{\genfrac{}{}{0pt}{}{x'/x}{\sigma(x')=x'}} 1,
          \quad \sigma\in G_{U'}.
}
Further, for $\sigma\neq 1$ the graph $\Gamma_\sigma$ and the diagonal $\Delta_{C'}$ intersect properly,
see Lemma \ref{lem-intersection-multiplicities}.
Hence the Lefschetz trace formula (Theorem \ref{thm-LTF}) yields
\eq{GOS6}{\sum_{i=0}^2 (-1)^i \Tr(\sigma^*|H^i(C', E))= (\Gamma_\sigma\cdot\Delta_{C'}), \quad \text{if }\sigma\neq 1.}
By Lemma \ref{lem-EulerChar-Hur}, 
equation \eqref{GOS5} and Proposition \ref{prop-swan-formula}, \ref{prop-swan-formula-sw} we get
\begin{align*}
\omit \span -\Tr(1^*|H^0(Y',E))+ \sum_{i=0}^2 (-1)^i \Tr(1^*|H^i(C',
E)) \\
&= -{\rm card}(Y')+ \chi(C')\\
\hspace{2cm}&=  -{\rm card}(Y')+     |G_{U'}| \cdot (\chi_c(U)+{\rm card}(Y))  
                                         -\sum_{x'\in Y'}
					 v_{x'}(\mathfrak{D}_{C'/C})\\
&=|G_{U'}| \cdot \chi_c(U)- 
  \sum_{x\in Y} \left(\left(\sum_{x'/x} 1+v_{x'}(\mathfrak{D}_{C'/C}))\right)-|G_{U'}|
  \right)\\
&=|G_{U'}| \cdot \chi_c(U)-\sum_{x\in Y} \sw_{G_{U',x}}(1).
				 \end{align*}
By \eqref{GOS5}, \eqref{GOS6}, Lemma \ref{lem-intersection-multiplicities} and 
Proposition \ref{prop-swan-formula}, \ref{prop-swan-formula-sw},  we have 
for $\sigma\in G_{U'}\setminus\{1\}$
\begin{align*}
	\omit \span -\Tr(\sigma^*|H^0(Y',E))+ \sum_{i=0}^2 (-1)^i \Tr(\sigma^*|H^i(C', E)) \\
\hspace{2cm}&= -\sum_{x\in Y}\sum_{\genfrac{}{}{0pt}{}{x'/x}{\sigma(x')=x'}} 1 +
      \sum_{x\in Y} \sum_{\genfrac{}{}{0pt}{}{x'/x}{\sigma(x')=x'}}
      i_{G_{U',x'}}(\sigma)\\
&= \sum_{x\in Y} \sum_{\genfrac{}{}{0pt}{}{x'/x}{\sigma(x')=x'}} (
i_{G_{U',x'}}(\sigma)-1  )\\
&=-\sum_{x\in Y} \sw_{G_{U'},x}(\sigma).\end{align*}
Hence by \eqref{GOS1}, \eqref{GOS4} and Proposition \ref{prop-swan-formula}, \ref{prop-swan-formula-Swan}
\begin{align*}
\chi_c(U,\sF) &=
                      \frac{1}{|G_{U'}|}\cdot
\Tr(1^*|V)\cdot\left (|G_{U'}| \cdot \chi_c(U)-\sum_{x\in Y}  \sw_{G_{U',x}}(1)\right) \\
&\phantom{=} - \frac{1}{|G_{U'}|}\cdot 
  \sum_{\sigma\in G_{U'}\setminus\{1\}} \left(\Tr(\sigma|V)\cdot\sum_{x\in Y}\sw_{G_{U',x}}(\sigma) \right)\\
   &= r\cdot \chi_c(U)-\sum_{x\in Y} \Swan_x(\sF).
\end{align*}
This finishes the proof.
\end{proof}

\begin{example}\label{ex:GOS-AS}
Let $k$ be an algebraically closed field of characteristic $p>0$,  $m$ a natural number with $(m,p)=1$
and $\psi: \F_p\to \Qlb^\times$ a group homomorphism. Let $\sL_{m,\psi}$ be the lisse rank $1$ sheaf  on 
$\A^1_k$ from Example \ref{ex:ArtinSchreier}. Denote by $\infty$ the point in the complement of 
$\A^1_k\subset\P^1_k$. Then $\Swan_\infty(\sL_{m,\psi})=m$, by Example \ref{ex:ASSwan}, 
and $\chi_c(\A^1_k)=1$ by Lemma \ref{lem-EulerChar-Hur}. Thus the Grothendieck-Ogg-Shafarevich 
formula gives
\[\chi_c(\A^1_k, \sL_{\psi})= 1-m.\]
\end{example}

\subsection{Cohomological description of the Swan representation in equal characteristic}\label{sec:GOS-CohSwan}

We fix the following situation: 
\begin{itemize}
\item $k=\bar{k}$ is an algebraically closed field of positive characteristic $p$. 
\item We write $\G_m=\G_{m,k}=\Spec k[t,\frac{1}{t}]$
      and denote by $\bar{\eta}\to \G_{m}$ a geometric point over the generic point. 
        We embed $\G_m\subset \P^1$ and write $\P^1\setminus \G_m=\{0,\infty\}$.
\item $K$ is a complete discrete valuation field containing $k$ with the property that $k$ maps isomorphically onto the residue field of $K$.
         The choice of a local parameter of $K$ induces an isomorphism of $K$ with the completion of $k(\P^1)$ at $\infty$,
         i.e. $K\cong k\llparen t^{-1}\rrparen$. {\em We fix such an isomorphism}.
\item Let $L/K$ be a finite Galois extension with Galois group $G$.
\end{itemize}

In the following we explain Katz' cohomological description of 
$\Sw_G$ (see Theorem \ref{thm:integralityOfSwan}).

\subsubsection{}\label{can-ext}Let $A$ be an $\ell$-adic coefficient ring (see Convention  \ref{conv-l-adic-coefficients}).
Let $\sF$ be a  lisse $A$-sheaf on $\G_m$. Following Katz, we define a {\em special}\index{special lisse sheaf on $\G_m$}  
lisse $A$-sheaf on $\G_m$ as follows: If $A$ is finite, we say $\sF$ is special
if it is tame at $0$ and if its {\em monodromy group}\index{monodromy group of a lisse sheaf}, i.e.,
the image under the corresponding representation $\pi_1(\G_m,\bar{\eta})\to \Aut(\sF_{\bar{\eta}})$, 
has a unique $p$-Sylow group. In general we can write $\sF=(\sF_n)\otimes_R A$ as in Convention \ref{conv-l-adic-coefficients} and
we say $\sF$ is special if $\sF_1$ is special.
By \cite[Cor.~1.5.7]{Katz/LocalToGlobal} the restriction functor induces an equivalence of categories  
\[(\text{special lisse } A\text{-sheaves on } \G_m )\xrightarrow{\simeq} (\text{lisse } A\text{-sheaves on } \Spec K ).\]
Composing the natural functor from the category of finitely generated $A[G]$- modules to the category of 
lisse  $A$-sheaves on  $\Spec K$, see Theorem \ref{thm-sheaf-vs-rep},
with the inverse of the above equivalence we obtain a fully faithful functor
\eq{can-ext1}{(-)^{\can}: (\text{fin.~gen.~} A[G]\text{-modules})\to 
(\text{special lisse } A\text{-sheaves on } \G_m ).}
If $M$ is a finitely generated $A[G]$-module, then the sheaf $M^{\can}$ is called the {\em canonical extension of $M$}
\index{canonical extension}. The name is justified by the fact that restricting the representation
corresponding to $M^{\can}$ to the decomposition group at $\infty$ gives back the original $A[G]$-module $M$.
(Here we view $G$ as a quotient of the absolute Galois group of $k\llparen t^{-1}\rrparen$ via the isomorphism
$K\cong k\llparen t^{-1}\rrparen$ fixed above.) The functor $(-)^{\can}$ commutes with direct sums, tensor products, duals, and
with changes of coefficients $A\to A'$.

\begin{theorem}[{\cite[Thm.~1.6.8]{Katz/LocalToGlobal}}]\label{thm:coh-swan}
Denote by $j: \G_m\inj \A^1$ the canonical open immersion.
Then  with the notation from Theorem \ref{thm:integralityOfSwan}
there is an isomorphism of $\Z_{\ell}[G]$-modules
\eq{coh-swan1}{H^1(\A^1, j_!(\Z_\ell [G]^{\can}))\cong \Sw_G.}
Here the $G$-module structure on the left-hand side is defined as follows:
 The left $G$-module $\Z_\ell[G]$ defines the constructible sheaf $j_!(\Z_\ell[G]^{\can})$ and 
the right $G$-module structure of $\Z_\ell[G]$ defines via functoriality a 
$G$-action on the sheaf $j_!(\Z_\ell[G]^{\can})$, which via
functoriality induces a $G$-action on the left-hand side of \eqref{coh-swan1}.
\end{theorem}

\begin{proof}
We sketch the proof following Katz.
For an arbitrary $\ell$-adic coefficient ring $A$ we define the functor
\eq{coh-swan2}{T: (\text{fin. gen. } A[G]-\text{modules})\to (\text{fin. gen. } A-\text{modules})}
by
\[T(M):= H^1(\A^1, j_{!}(M^{\can})).\]
We have $H^i(\A^1, j_{!}(M^{\can}))=0$, for $i\neq 1$. (This is clear for $i=0$ and follows for $i\ge 2$ from 
\ref{coh-properties}, \ref{affine-vanishing}.) Hence $T$ is an exact functor.
Furthermore, $A[G]$ acts left-$A[G]$-linearly on itself by right multiplication. Thus
$T(A[G])$ becomes an $A[G]$-module by functoriality. Using an argument similar to the proof of Lemma \ref{lem-proj-resolution} one can
show:
\begin{enumerate}
\item\label{coh-swan3} $T(M)=T(A[G])\otimes_{A[G]} M$, for all finitely generated $A[G]$-modules $M$.
\item\label{coh-swan4} $T(A[G])$ is a finitely generated projective $A[G]$-module. 
\item\label{coh-swan5} $T(A[G])\otimes_{A} A'=T(A'[G])$, where $A\to A'$ is any map of $\ell$-adic coefficient rings.
\end{enumerate}
By \ref{coh-swan4}, \ref{coh-swan5} 
and the uniqueness part of Theorem \ref{thm:integralityOfSwan}, \ref{item:integrality},
it suffices to show that $T(\Qlb[G])$ is the Swan representation.
By Proposition \ref{prop:Orthogonality}, \ref{item:hom} and Corollary \ref{cor-irred-rep-in-rep}, it suffices to show
that for any irreducible $\Qlb[G]$-module $W$ we have
\eq{coh-swan6}{\dim_{\Qlb}(T(\Qlb[G])\otimes_{\Qlb} W^\vee)^G= 
                         \dim_{\Qlb}(\Sw_G\otimes_{\Z_\ell} W^\vee)^G.}
Since dualizing induces a bijection on the set of irreducible representations we may replace $W^\vee$ by $W$.
The Swan representation is self dual (as one easily checks that $\sw_G(g)=\sw_G(g^{-1})$).
Hence 
\[\dim_{\Qlb}(\Sw_G\otimes_{\Z_\ell} W)^G=
\dim_{\Qlb} \Hom_{\Qlb[G]}(\Sw_G\otimes_{\Z_\ell}\Qlb, W)=\Swan(W),\]
where the second equality follows from Theorem \ref{thm:integralityOfSwanConductor}
 and Remark \ref{rem:independenceOfFiniteQuotient}.
On the other hand by \ref{coh-swan3} above
\begin{align*}
\dim_{\Qlb}(T(\Qlb[G])\otimes_{\Qlb} W)^G &= \dim_{\Qlb}(T(\Qlb[G])\otimes_{\Qlb[G]} W)\\
&= \dim_{\Qlb}H^1(\A^1, j_!W^{\can}) \\
&=-\chi(\A^1, j_!W^{\can}).
\end{align*}
Furthermore, $\chi(\A^1, j_!W^{\can})=\chi_c(\A^1, j_!W^{\can})=\chi_c(\G_m, W^{\can})$,
where the first equality follows e.g. from \cite[2.0.7]{Katz/Kloosterman} and the second equality holds by definition.
We have $\chi_c(\G_m,\Qlb)=0$. Since $W^{\can}$ is lisse on $\G_m$ and tame in $0$, the
formula of Grothendieck-Ogg-Shafarevich yields
$\chi_c(\G_m, W^{\can})= -\Swan_{\infty}(W^{\can})= -\Swan(W)$.
Putting everything together yields the theorem.
\end{proof}


\section{Higher dimensional ramification theory via restriction to curves}\label{RTvC}
\setcounter{subsection}{0}\setcounter{subsubsection}{0}
In this section we present the ramification theory of lisse sheaves on a higher dimensional smooth variety over a
perfect field of positive characteristic, following \cite[3.]{EK12}.

\setcounter{subsubsection}{0}
\subsubsection{}\label{conventions-hdrt}Throughout this section we fix the following notation:
\begin{itemize}
\item $k$ is a perfect field of characteristic $p>0$ and  $\bar{k}$ an algebraic closure.
\item $\ell$ is a prime number different from $p$.
\item A $k$-scheme is a scheme which is separated and of finite type over $k$.
\item For a $k$-scheme $X$ we denote by $\Cu(X)$ the set of normalizations 
         of closed integral 1-dimensional subschemes of $X$. In other words, $C\in \Cu(X)$
         is a smooth connected curve over $k$ with a morphism $\nu: C\to X$, which is birational
	 onto its image.
\item If $C$ is a smooth connected $k$-scheme of dimension 1, we denote 
          by $\widehat{C}$ the (up to isomorphism) unique smooth projective curve over $k$, which
          admits an open dense embedding $C\inj \widehat{C}$.
  \end{itemize}

\subsection{Lisse sheaves with bounded ramification}

\begin{definition}\label{def:divisorial-compactification}
Let $U$ be a normal $k$-scheme. A {\em divisorial compactification}\footnote{This is non-standard terminology.}
\index{divisorial compactification} of $U$
is a normal proper $k$-scheme $X$ together with an open dense immersion $j:U\inj X$ such that the complement
$X\setminus U$ is the support of an effective Cartier divisor.
\end{definition}

\begin{remark}\label{rmk:divisorial-compactifications-exist}
Let $U$ be a normal $k$-scheme.
\begin{enumerate}
\item A divisorial compactification of $U$ exists.
        Indeed, by Nagata's compactification theorem (see e.g. \cite[Thm.~4.1]{Conrad/Nagata}),
        we find a dense open immersion of $U$ into a proper $k$-scheme $Y$. 
        Let $\tilde{Y}$ be the blow-up of $Y$ in the closed subscheme $Y\setminus U$ 
        (with its reduced scheme structure). We obtain a dominant open immersion $U\inj \tilde{Y}$
      whose complement is the support  of the exceptional divisor $E$ of the blow-up, which is an effective Cartier divisor.
      Now let $\nu :X\to \tilde{Y}$ be the normalization of $\tilde{Y}$. Then $\nu$ is a finite morphism which is an isomorphism
      over $U$. We obtain an open dense embedding $U\inj X$  whose complement is the support of
        the effective Cartier divisor $\nu^*E$. Then $U\inj X$ is a divisorial compactification of $U$.
\item If $C$ is a smooth connected curve over $k$, then with the notation 
       from \ref{conventions-hdrt} the inclusion $C\inj \widehat{C}$
        is the (up to isomorphism) unique divisorial compactification of $C$.
\item Let $X$ be a divisorial compactification of $U$ and $C\in \Cu(U)$. Then by the valuative criterion for properness,
        the morphism $\nu: C\to U$ extends uniquely to a morphism $\nu_X: \widehat{C}\to X$.
        Furthermore $\widehat{C}$ is the normalization of its image in $X$, hence $\widehat{C}\in \Cu(X)$.
\end{enumerate}
\end{remark}

\begin{definition}\label{def:swan-divisor}
Let $C$ be a smooth connected curve over $k$, $A$ an $\ell$-adic coefficient ring in the sense of Convention \ref{conv-l-adic-coefficients}
and $\sF$ a lisse $A$-sheaf on $C$. We define the {\em Swan conductor of $\sF$}
\index{Swan conductor of a lisse sheaf on a curve} to be the following
effective Cartier divisor on $\widehat{C}$:
\[\Swan(\sF):=\sum_{x\in \widehat{C}} \Swan_x(\sF)\cdot [x].\]
Here $\Swan_x(\sF)$ is the Swan conductor of $\sF$ on $C$ at $x$, see \ref{swan-lisse-sheaf}.
Notice that $\Swan_x(\sF)\in\Z_{\ge 0}$ by Theorem \ref{thm:integralityOfSwanConductor}.
\end{definition}

\begin{definition}\label{def:ramification-of-sheaves}
Let $U$ be a normal $k$-scheme, $U\inj X$ a divisorial compactification and $D$ an effective Cartier divisor supported in 
$X\setminus U$.  Let $A$ be an $\ell$-adic coefficient ring and $\sF$ a lisse $A$-sheaf on $U$. 
Then we say that the {\em ramification of $\sF$ is bounded by $D$}
\index{bounded ramification of a sheaf} if the following condition is satisfied:
\eq{def:ramification-of-sheaves1}{\Swan(\nu^*\sF)\le \nu_X^*D,\quad
   \text{for all }C\in \Cu(U). }
Here $\nu: C\to U$ is the natural map and $\nu_X: \widehat{C}\to X$ its unique extension; 
the inequality takes place in the monoid of effective divisors on $\widehat{C}$.
\end{definition}

\begin{example}\label{ex:ramification-AS}
Let $\sL_{m,\psi}$ be the lisse sheaf on $\A^1\subset \P^1$ from Example \ref{ex:ArtinSchreier}.
Then the ramification of $\sL_{m,\psi}$ is bounded by $m\cdot\{\infty\}$, by Example \ref{ex:ASSwan}. 
\end{example}

\begin{remark}\label{rmk:ramification-depends-on-F1}
In the situation above we find a DVR $R$ which is finite over $\Z_\ell$ and a lisse $R$-sheaf $(\sF_n)_n$ on $U$
such that $A$ is an $R$-algebra and $\sF\cong (\sF_n)\otimes_R A$. Then for any $\nu:C\to U$ in $\Cu(U)$
we have $\nu^*\sF\cong (\nu^*\sF_n)\otimes_R A$ and by \ref{swan-lisse-sheaf}, \ref{swan-lisse-sheaf-bc}
\[\Swan(\nu^*\sF)=\Swan(\nu^*\sF_1).\]
Thus whether the ramification of $\sF$ is bounded or not by a given Cartier divisor $D$ depends only on $\sF_1$.
\end{remark}

\begin{remark}\label{rmk:ramification-exact-sequences}
In the situation of Definition \ref{def:ramification-of-sheaves} if $0\to\sF'\to\sF\to\sF''\to 0$ is an exact sequence
of lisse $A$-sheaves on $U$ and the ramification of $\sF'$ and $\sF''$ is bounded by $D'$ and $D''$, respectively,
the ramification of $\sF$ is bounded by $D'+D''$. This follows immediately from 
\ref{swan-lisse-sheaf}, \ref{swan-lisse-sheaf-exactness}.
\end{remark}

\begin{proposition}\label{prop:lisse-sheaf-has-bounded-ramification}
Let $U$ be a normal $k$-scheme and $U\inj X$ a divisorial compactification. Let $A$ be an $\ell$-adic coefficient ring
and  $\sF$ a free lisse $A$-sheaf on $U$. Then there exists an effective Cartier divisor $D$ on $X$ supported in $X\setminus U$
such that the ramification of $\sF$ is bounded by $D$.
\end{proposition}

Before we can prove the proposition we need a notion of discriminant in higher dimensions.
The following definition is due to Alexander Schmidt.
\begin{definition}\label{def:HigherDimDiscriminant}
Let $\pi: Y'\to Y$ be a finite and generically \'etale morphism of degree $n$ between connected  normal $k$-schemes.
Denote by $\Tr_{Y'/Y}: \pi_*\sO_{Y'}\to \sO_Y$ the trace map which is induced by $\Tr_{k(Y')/k(Y)}$.
We define the discriminant\index{discriminant in higher dimensions}
 $\sI(D_{Y'/Y})$ of $Y'/Y$ to be the sheaf associated to the following 
presheaf of ideals in $\sO_Y$ 
\[Y\supset U\mapsto
\sum_{\{x_1,\ldots, x_n\}}\det((\Tr_{Y'/Y}(x_ix_j))_{i,j})\cdot\sO_Y(U)\quad \subset \sO_Y(U),\]
where the sum is over all subset $\{x_1,\ldots, x_n\}\subset \sO_{Y'}(\pi^{-1}(U))$, which consist of 
$\sO_Y(U)$-linearly independent elements.
\end{definition}

\begin{lemma}\label{lem:PropertiesHigherDimDiscriminant}
In the situation of Definition \ref{def:HigherDimDiscriminant} the discriminant $\sI(D_{Y'/Y})$ has the following properties:
\begin{enumerate}
\item\label{lem:PropertiesHigherDimDiscriminant-coh} 
              $\sI(D_{Y'/Y})\subset \sO_Y$ is a coherent ideal sheaf.
\item\label{lem:PropertiesHigherDimDiscriminant-free} If $\pi: Y'\to Y$ is free and $e_1,\ldots, e_n$ is a basis
         of the $\sO_Y$-module $\pi_*\sO_{Y'}$, then $\sI(D_{Y'/Y})= \det((\Tr_{Y'/Y}(e_ie_j))_{i,j})\cdot\sO_Y$.
\item\label{lem:PropertiesHigherDimDiscriminant-dim1}
         Assume $Y=\Spec A$ and $Y'=\Spec B$ are smooth affine curves. Then
        $\sI(D_{Y'/Y})$ is the sheaf induced by the discriminant ideal $\mathfrak{d}_{B/A}\subset A$ 
      from Definition \ref{def:different}.
\item\label{lem:PropertiesHigherDimDiscriminant-et}
         If $U\subset Y$ is an open subset such that $\pi^{-1}(U)\to U$ is \'etale, then 
      $\sI(D_{Y'/Y})_{|U}=\sO_{U}$.
\item\label{lem:PropertiesHigherDimDiscriminant-semi-continuity}
      Let $\nu: C\to Y$ be a morphism from a smooth connected curve $C$ to $Y$ and  assume
       that $\nu(C)$ meets the locus over which $\pi$ is \'etale. Let $C'$ be 
       the normalization of an irreducible component of $C\times_Y Y'$, which is generically \'etale over $C$.
        We obtain a commutative  diagram 
\[\begin{tikzcd} C'\ar[swap]{d}{\pi_{C'}}\ar{r}{\nu'} & Y'\ar{d}{\pi}\\
                       C\ar{r}{\nu} & Y.
   \end{tikzcd}\]
 Then the image of $\sI(D_{Y'/Y})$ under $\nu^*: \nu^{-1}\sO_{Y}\to \sO_{C}$ is contained in $\sI(D_{C'/C})$,
       i.e.
\[\nu^*(\sI(D_{Y'/Y}))\subset \sI(D_{C'/C}).\]
\end{enumerate}
\end{lemma}
\begin{proof}
By definition $\sI(D_{Y'/Y})$ is an $\sO_Y$-submodule of $\sO_Y$ 
hence \ref{lem:PropertiesHigherDimDiscriminant-coh}  is automatic.
Now let $U=\Spec A$ be an affine open subset of $Y$ and write $\pi^{-1}(U)=\Spec B$. 
Let $\{x_1,\ldots, x_n\}$ and $\{y_1,\ldots y_n\}$ be two sets of $A$-linearly independent elements of $B$ and denote
$B_x=\sum_i A x_i$, $B_y=\sum_i A y_i$. Assume $B_x\subset B_y$ then 
\eq{lem:PropertiesHigherDimDiscriminant-1}{
        \det((\Tr_{B/A}(x_ix_j))_{i,j})\cdot A\subset\det((\Tr_{B/A}(y_iy_j))_{i,j})\cdot A.}
Indeed by assumption we find an $n\times n$-matrix $M$ with coefficients in $A$ such that 
(with the obvious vector notation) $x= M y$ and we obtain
\[       \det((\Tr_{B/A}(x_ix_j))_{i,j})= \det(M)^2\cdot \det((\Tr_{B/A}(y_iy_j))_{i,j}),\]
in particular \eqref{lem:PropertiesHigherDimDiscriminant-1} holds.
This immediately implies \ref{lem:PropertiesHigherDimDiscriminant-free} and using 
Lemma \ref{lem:DiscriminantFormula} also \ref{lem:PropertiesHigherDimDiscriminant-dim1}. 

To prove \ref{lem:PropertiesHigherDimDiscriminant-et} we can assume that $B/A$ is \'etale and free with basis
$e_1,\ldots, e_n$ and we have to show $\delta:=\det((\Tr_{B/A}(e_ie_j))_{i,j})\in A^\times$. 
For this it suffices to show that for any prime $\mathfrak{p}\subset A$ the element 
$\delta$ maps to a unit in the residue field $k(\mathfrak{p})$.
But since $B/A$ is \'etale, $B\otimes_A k(\mathfrak{p})$ is \'etale over $k(\mathfrak{p})$ 
with basis $\bar{e}_i=e_i\otimes 1$, $i=1,\ldots, n$, and 
the image of $\delta$ in $k(\mathfrak{p})$ equals
$\det((\Tr_{B\otimes_A k(\mathfrak{p})/k(\mathfrak{p})}(\bar{e}_i\bar{e}_j))_{i,j})$.
It follows from Proposition \ref{prop:separability}, that this element is not zero in $k(\mathfrak{p})$.
Hence \ref{lem:PropertiesHigherDimDiscriminant-et} follows.

Finally,  in the situation of \ref{lem:PropertiesHigherDimDiscriminant-semi-continuity}, we can assume
that $Y'\to Y$ is given by a finite ring extension $B/A$.  
If the image of $C$ is a point, then by assumption it lies  in the locus over which $\pi$ is \'etale 
and hence $\sI(D_{Y'/Y})$ is the unit ideal around this point. Thus we can assume that $\nu(C)$ is a curve
and therefore $C'\to C$ corresponds to a finite and generically \'etale extension of Dedekind domains $S/R$.
Let $x_1,\ldots, x_n\in B$ be a
sequence of $A$-linearly independent elements. We have to show that the image of 
$\det((\Tr_{B/A}(x_ix_j))_{i,j})$ under $A\to R$ is contained in $\mathfrak{d}_{S/R}$, which amounts to show
that
\eq{lem:PropertiesHigherDimDiscriminant-2}{
                      \det((\Tr_{B\otimes_A R/R}(\bar{x}_i\bar{x}_j))_{i,j})\in\mathfrak{d}_{S/R},}
where $\bar{x}_i$ denotes the image of $x_i$ in $B\otimes_A R$. 
Observe that all irreducible components of $Y'\times_Y C$ map surjectively to $C$.
Indeed, since $\pi:Y'\to Y$ is finite and surjective and $Y$ is normal (hence universally unibranch)
$\pi$ is universally open, by \cite[Cor.~(14.4.4), (i)]{EGA4}. Let $S_1=S, S_2, \ldots, S_r$ be the affine
coordinate rings of the normalizations of the irreducible components of $Y'\times_Y C$.
We get a natural map $B\otimes_A R\to \prod_i S_i$ which becomes an isomorphism
when tensored with $\otimes_A k(Y)$, since $Y'\times_Y C\to C$ is generically \'etale.
It follows that we obtain a commutative diagram 
\[\begin{tikzcd} B\otimes_A R\ar{r}{}\ar[swap]{dr}{\Tr} & \prod_i S_i\ar{d}{\sum_i\Tr_{S_i/R}}\\        
                                          ~ & R. 
   \end{tikzcd}\]
The images of $\bar{x}_1,\ldots, \bar{x}_n$ in $\prod_i S_i$ span an $R$-submodule of rank $n$.
Further taking a basis $\{e_{ij}\}_{j=1,\ldots, n_i}$ of $S_i/R$ we get the following basis of $\prod_i S_i$
\[(e_{11},0\ldots, 0), \ldots ,( e_{1,n_1},0\ldots, 0),\ldots, (0,\ldots, 0, e_{r, 1}),\ldots, (0,\ldots, 0, e_{r,n_r}).\]
Therefore the formula \eqref{lem:PropertiesHigherDimDiscriminant-1} (applied with $\{y_i\}$ the above basis)
yields
\[\det((\Tr_{B\otimes_A R/R}(\bar{x}_i\bar{x}_j))_{i,j})\cdot A 
                                 \subset \mathfrak{d}_{S_1/R}\cdots \mathfrak{d}_{S_r/R}\subset \mathfrak{d}_{S_1/R}.\]
This proves \eqref{lem:PropertiesHigherDimDiscriminant-2} and finishes the proof of the lemma.
\end{proof}

\begin{remark}
A more intrinsic way to define the discriminant ideal $\sI(D_{Y'/Y})$ is the following:
There is a well defined $\sO_Y$-linear morphism
\[\phi_{Y'/Y}: \bigwedge^n(\pi_*\sO_{Y'})\otimes_{\sO_Y} \bigwedge^n(\pi_*\sO_{Y'})\to \sO_Y,\]
which sends $(x_1\wedge\ldots \wedge x_n)\otimes (y_1\wedge\ldots\wedge y_n)$ to $\det((\Tr_{Y'/Y}(x_iy_j))_{i,j})$,
see e.g. \cite[\S 1, 9]{Bourbaki/AlgebreIX}. 
For an $\sO_Y$-module $\sM$ denote by $\Gamma(\sM)$ the divided power algebra of $\sM$, see \cite[III, 1. Def.]{Roby}.
Recall that it is a graded $\sO_Y$-algebra which is locally generated by elements of the form
$x^{[n]}$, $x\in \sM$, $n\ge 1$, which have degree $n$ and satisfy $n!x^{[n]}=(x^{[1]})^n$. 
Denote by $\Gamma^2(\sM)$ the $\sO_Y$-module of homogeneous elements of degree 2.
There is a canonical  $\sO_Y$-linear map $\gamma_2: \Gamma^2(\sM)\to \sM\otimes_{\sO_Y}\sM$
satisfying $\gamma_2(x^{[2]})=x\otimes x$ and if $2$ is invertible in $\sO_Y$ it 
induces an isomorphism with the symmetric tensors $\Gamma^2(\sM)\xr{\simeq} (\sM\otimes_{\sO_Y}\sM)^{\Sigma_2}$,
see \cite[Prop.~III.1, Prop.~III.3]{Roby}.
Define $\Phi_{Y'/Y}:=\phi_{Y'/Y}\circ \gamma_2: \Gamma^2(\bigwedge^n\pi_*\sO_{Y'})\to \O_Y$.
Then we have  $\sI(D_{Y'/Y})= \image(\Phi_{Y'/Y})$, as follows directly from the definition.
\end{remark}

\begin{proof}[Proof of Proposition \ref{prop:lisse-sheaf-has-bounded-ramification}]
We can assume that $U$ is connected; let  $\bar{\eta}\to U$ be a geometric point over the generic point of $U$.
By Remark \ref{rmk:ramification-depends-on-F1} it suffices to consider the case in which $A$ is a finite field
of characteristic $\ell$.
In particular $\sF_{\bar{\eta}}$ is finite (as a set). Hence the representation
$\pi_1(U,\bar{\eta})\to {\rm Aut}(\sF_{\bar{\eta}})$ defined by $\sF$ factors
over a finite quotient, i.e.~there exists a finite connected \'etale Galois cover $\pi: U'\to U$ which trivializes $\sF$.
Let $\pi_X: X'\to X$ be the normalization of $X$ inside the function field of $U'$. Then $\pi_X$ is finite,
$\pi_X^{-1}(U)=U'$ and $X'$ is a divisorial compactification of $U'$. 
Let $\sI(D_{X'/X})$ be the discriminant of $X'/X$ (Definition \ref{def:HigherDimDiscriminant}). Since
$\sI(D_{X'/X})_{|U}=\sO_U$ we find an effective Cartier divisor $D$ on $X$ with 
$\sO_X(-D)\subset \sI(D_{X'/X})$.
We claim:
\eq{prop:lisse-sheaf-has-bounded-ramification1}{\text{The ramification of }\sF \text{ is bounded by } 
   \rk(\sF)\cdot  D.}
Take $\nu:C\to U$ in $\Cu(U)$. Let $C'$ be a connected component of $C\times_U U'$.
The projection maps induce maps $C'\to C$ and $C'\to U'$, which after compactification yield
a commutative diagram
\[\begin{tikzcd}[column sep=1.5cm]
			\widehat{C}'\ar[swap]{d}{\pi_{\widehat{C}}}\ar{r}{\nu_{X'}}& X'\ar{d}{\pi_{X}}\\
                          \widehat{C}\ar{r}{\nu_{X}} &   X.
		\end{tikzcd}\] 
 By Lemma \ref{lem:PropertiesHigherDimDiscriminant}, \ref{lem:PropertiesHigherDimDiscriminant-semi-continuity}, we have
\[\sO_{\widehat{C}}(-D_{|\widehat{C}})\subset \nu_X^*\sI(D_{X'/X})\subset \sI(D_{\widehat{C}'/\widehat{C}}).\]
Denote by $D_{\widehat{C}'/\widehat{C}}$ the effective divisor on $\widehat{C}$ given by $\sI(D_{\widehat{C}'/\widehat{C}})$;
it is the discriminant divisor. We obtain
\eq{prop:lisse-sheaf-has-bounded-ramification2}{\sum_{x'/x}[x':x]\cdot v_{x'}(\mathfrak{D}_{\widehat{C}'/\widehat{C}})
                = \text{mult}_x(D_{\widehat{C}'/\widehat{C}})\le \text{mult}_x(D_{|\widehat{C}}),}
where we use the notation from Proposition \ref{prop-swan-formula}.

On the other hand $C'\to C$ is a connected finite \'etale Galois covering, say with Galois group $G_{C'}$,
trivializing $\sE:=\nu^*\sF$. Let $\bar{\xi}\to C$ be a geometric point over the generic point of $C$
and $x\in \widehat{C}\setminus C$ a closed point.
Then 
\begin{align*}
\Swan_x(\sE) & = \dim_A  \Hom_{A[G_{C'}]}(\Sw_{G_{C'}, x}\otimes_{\Z_\ell} A,\sE_{\bar{\xi}})&&
                         \text{by  \ref{prop-swan-formula}\ref{prop-swan-formula-cond-rep}, 
                            \ref{rem:independenceOfFiniteQuotient}}\\
                             & =  \dim_A \left(\IntHom_A(\Sw_{G_{C'}, x}\otimes_{\Z_\ell} A,\sE_{\bar{\xi}})^{G_{C'}}\right)&&
                                        \text{by  \ref{ex:representationsFirstExamples}\ref{item:internalHom}}\\
                             &\le  \dim_A \IntHom_A(\Sw_{G_{C'}, x}\otimes_{\Z_\ell} A,\sE_{\bar{\xi}})\\
                             & =\rk(\sE)\cdot \dim_A (\Sw_{G_{C'}, x}\otimes_{\Z_\ell} A)\\
                             &= \rk(\sE)\cdot \sw_{G_{C'},x}(1)&& \text{by \eqref{Swan-representation-geometry:char}}\\
                            &=\rk(\sF)\cdot \left(\left(\sum_{x'/x}
			    [x':x]\left(1+v_{x'}\left(\mathfrak{D}_{\widehat{C}'/\widehat{C}}\right)\right)\right)- |G_{C'}|\right)&&
                            \text{by   \ref{prop-swan-formula}\ref{prop-swan-formula-sw}}\\
                             &\le \rk(\sF)\cdot\sum_{x'/x} [x':x]\cdot v_{x'}(\mathfrak{D}_{\widehat{C}'/\widehat{C}})&&
                            \text{by \eqref{eq:nef-formular}, \ref{ex-etale}\ref{ex-etale:unramified}}\\
                            &\le \rk(\sF) \cdot\text{mult}_x(D_{|\widehat{C}})&& 
                                               \text{by \eqref{prop:lisse-sheaf-has-bounded-ramification2}.}
\end{align*}
Hence $\Swan(\nu^*\sF)\le \rk(\sF) \cdot D_{|\widehat{C}}$ which proves 
the claim \eqref{prop:lisse-sheaf-has-bounded-ramification1} and the proposition.
\end{proof}

As a consequence of the Grothendieck-Ogg-Shafarevich theorem we obtain 
a universal bound for the cohomolgy with compact support of a lisse $\Qlb$-sheaf on a smooth curve over $k$ 
with fixed rank and ramification bounded by a divisor of fixed degree. More precisely:

\begin{definition}\label{def:complexity}
Let $C$ be a smooth and geometrically connected curve over $k$. 
Let $d\ge 0$ be a natural number.
Then we define
\[\sC_d:= 2g(\widehat{C})+2d+1,\]
where $g(\widehat{C})$ is the genus of $\widehat{C}$. 
In case $D$ is an effective divisor of degree $\deg(D)\ge 0$ on $\widehat{C}$ we set
\[\sC_D:=\sC_{\deg D}\]
and call it the {\em complexity of $D$}\index{complexity of a divisor}.
\end{definition}

\begin{proposition}\label{prop:coh-dim-bound}
Let $C$ be a smooth geometrically connected curve over $k$. Let $r\ge 1$ and $d\ge 0$ be natural numbers.
Then
\[\dim_{\Qlb} H^0_c(C\otimes_k \bar{k},\sF)+\dim_{\Qlb} H^1_c(C\otimes_k \bar{k},\sF)\le r\cdot \sC_d\]
for all lisse $\Qlb$-sheaves $\sF$ on $C\otimes_k \bar{k}$ of rank $r$ 
whose ramification is bounded by $r\cdot D$, where $D$ is an effective divisor of degree $\deg(D)=d$ and with  support ${\rm supp}(D)=\widehat{C}\otimes_k \bar{k}\setminus C\otimes_k \bar{k}$.
\end{proposition}
\begin{proof}
Observe that 
\eq{prop:coh-dim-bound0}{\dim_{\Qlb} H^0(C\otimes_k\bar{k},\sF)=
                                     \dim_{\Qlb} \sF_{\bar{\eta}}^{\pi_1(\widehat{C}\otimes_k\bar{k},\bar{\eta})}\le r,}
where $\sF_{\bar{\eta}}$ denotes the $\pi_1(\widehat{C}\otimes_k\bar{k},\bar{\eta})$-representation
corresponding to $\sF$, see Theorem \ref{thm-sheaf-vs-rep}.

First assume $C=\widehat{C}$ is projective and hence $D=0$. Let $x\in \widehat{C}\otimes_k \bar{k}$ be a 
closed point and denote by $U$ the complement. Notice that $\Swan_x(\sF)=0$.
The exact sequence 
\[\cdots \rightarrow  H^i(\widehat{C}\otimes_k\bar{k},\sF)\to H^i(x, \sF)\to
	H^{i+1}_c(U,\sF)\to\cdots,\] Theorem \ref{thm-GOS}
and Lemma \ref{lem-EulerChar-Hur} give
\begin{align}\label{prop:coh-dim-bound1}
\sum_i (-1)^i\dim_{\Qlb}H^i(\widehat{C}\otimes_k \bar{k},\sF)& =
                                                               \chi_c(U,\sF)+ \dim_{\Qlb} H^0(x,\sF)\\
   & =r \chi_c(U)+r\notag\\
   &= r\chi(\widehat{C}\otimes_k \bar{k})\notag\\
   &= r(2-2g(\widehat{C})).\notag
\end{align}
Furthermore by Poincar\'e duality (Theorem \ref{thm-PD}) and \eqref{prop:coh-dim-bound0}
\[\dim_{\Qlb} H^2(\widehat{C}\otimes_k \bar{k},\sF)=\dim_{\Qlb} H^0(\widehat{C}\otimes_k \bar{k},\sF^\vee(1))\le r.\]
Thus \eqref{prop:coh-dim-bound0} and \eqref{prop:coh-dim-bound1} yield
\[\dim_{\Qlb} H^0_c(C\otimes_k \bar{k},\sF)+\dim_{\Qlb} H^1_c(C\otimes_k \bar{k},\sF)\le r\cdot \sC_0.\]

Now assume the inclusion $C\subset \widehat{C}$ is strict, i.e.~$C$  is affine and the ramification
of $\sF$ is bounded by $D$ with ${\rm supp}(D)=\widehat{C}\otimes_k \bar{k}\setminus C\otimes_k\bar{k}$.
Then by Corollary \ref{cor-affine-vanishing-copact-support}, Poincar\'e duality \ref{thm-PD} and \eqref{prop:coh-dim-bound0}
\[\dim_{\Qlb}H^0_c(C\otimes_k \bar{k},\sF)=0, \quad
\dim_{\Qlb}H^2_c(C\otimes_k\bar{k},\sF)=\dim_{\Qlb}H^0(C\otimes_k\bar{k},\sF^\vee(1))\le r.\]
Thus by Grothendieck-Ogg-Shafarevich \ref{thm-GOS} and Lemma \ref{lem-EulerChar-Hur}
\begin{align*}\omit\span \dim_{\Qlb} H^1_c(C\otimes_k \bar{k},\sF) \\
	\hspace{2cm}&=r\cdot (2g(\widehat{C})-2)+r\cdot{\rm card}({\supp}(D)) +\dim_{\Qlb}
H^2_c(C\otimes_k\bar{k},\sF)\\
&\phantom{=}\quad+\sum_{x\in \,{\rm supp}(D)}\Swan_x(\sF)\\
&\le r\cdot(2g(\widehat{C})-2) +r d+ r+r d\\
&= r(2g(\widehat{C})+2d-1)\\
&< r\cdot
\sC_d.\end{align*}
\end{proof}

\begin{definition}\label{def:tame-lisse-sheaf}
Let $U$ be a normal $k$-scheme, $A$ an $\ell$-adic coefficient ring and $\sF$ a lisse $A$-sheaf on $U$.
We say that $\sF$ is {\em tame}\index{tame lisse sheaf} if for any $\nu:C\to U$ in $\Cu(U)$
we have $\Swan(\nu^*\sF)=0$.
\end{definition}

\begin{remark}\label{rmk:tame-bounded-by-zero}
We immediately see that $\sF$ is tame if and only if there exists a divisorial compactification $U\inj X$ 
with respect to which the ramification of $\sF$ is bounded by the zero divisor. Moreover, this is equivalent to
the ramification of $\sF$ being bounded by the zero divisor with respect to all divisorial compactifications.
\end{remark}

The following proposition is a direct consequence of Theorem \ref{thm:differentNotionsOfTameness}
due to Kerz and Schmidt.
\begin{proposition}[Kerz-Schmidt]\label{prop:tame-sheaf-vs-rep}
Let $U$ be a normal connected scheme over $k$, $\bar{x}\to U$ a geometric point and $A$ an $\ell$-adic coefficient ring.
Let  $\pi_1^{\tame}(U,\bar{x})$ be the tame fundamental group of $U$ at $\bar{x}$, 
see Definition \ref{def:tameFundamentalGroup}. Then the equivalence of categories from Theorem \ref{thm-sheaf-vs-rep}
restricts to an equivalence between the following full subcategories
\[(\text{tame } A\text{-sheaves on }U)\xr{\simeq} (A\text{-representations of }\pi_1^{\tame}(U,\bar{x})).\]
\end{proposition}
\begin{proof}
In Theorem \ref{thm-sheaf-vs-rep} we constructed two quasi-inverse functors
\[(\text{lisse } A\text{-sheaves on }U)\substack{F \\ \longrightarrow\\ \longleftarrow\\G}
                         (A\text{-representations of }\pi_1(U,\bar{x}))\]
It suffices to show $F$ and $G$ induce functors between the full subcategories from the statement.
In view of the construction of $F$ and $G$ it suffices to consider the case where $A$ is a finite ring.
Let $\sF$ be a tame $A$-sheaf on $U$. Then we have to show that the $\pi_1(U,\bar{x})$-representation
$F(\sF)=\sF_{\bar{x}}$ factors over the continuous quotient map
 $\pi_1(U,\bar{x})\twoheadrightarrow \pi_1^{\tame}(U,\bar{x})$,
see Proposition \ref{prop:basic-properties-tame-fund-group}, \ref{prop:basic-properties-tame-fund-group-surj}.
By Proposition \ref{prop:loc-constant-sheaf-rep} the sheaf $\sF$ is represented by a 
finite \'etale $U$-group scheme $U_{\sF}$. The stalk $\sF_{\bar{x}}$ is equal to the fiber
of $U_\sF\to U$ at $\bar{x}$. By the  definition of $\pi_1^{\tame}(U,\bar{x})$ it suffices 
to prove that $U_\sF\to U$ is tame, see Definition \ref{defn:tameness}, \ref{defn:tamenessInGeneral}. By
Theorem \ref{thm:differentNotionsOfTameness} it suffices to check that for any smooth connected $k$-curve $C$
and any map $\nu:C\to U$ the base changed morphism $U_\sF\times_U C\to C$
is tamely ramified along $\widehat{C}\setminus C$. If $\nu(C)$ is a point there is nothing to show.
Otherwise $\nu$ factors over the normalization $\widetilde{\nu(C)}\to\nu(C)\subset U$.
If the base change of  $U_\sF\to U$  over $\widetilde{\nu(C)}$ is tame, then so is the base change over $C$,
see e.g. \cite[II, Cor.~7.8]{Neukirch/1999}. Thus we can assume that $\nu:C\to
U$ is in $\Cu(U)$ (see \ref{conventions-hdrt} for the Definition of
$\Cu(U)$).
Notice that $U_{\nu^*\sF}:=U_{\sF}\times_U C\to C$ is a finite \'etale morphism 
representing $\nu^*\sF$.  Since $\sF$ is tame we have $\Swan_y(\nu^*\sF)=0$ for all $y\in \widehat{C}$, i.e.
for $\bar{\eta}\to C$ a geometric point over the generic point of $C$ and $P_y$ the wild inertia group at $y$
we have $U_{\sF}\times_U\bar{\eta}=\sF_{\bar{\eta}}=\sF_{\bar{\eta}}^{P_y}= (U_{\sF}\times_U\bar{\eta})^{P_y}$
(see \ref{swan-lisse-sheaf}, \ref{swan-lisse-sheaf-tame}). This means that $U_{\sF}\times_U C\to C$ is tamely ramified.

It remains to show that if $M$ is an $A$-representation of $\pi_1^{\tame}(U,\bar{x})$ then the
lisse sheaf $G(M)$ is a tame sheaf on $U$. But by construction of $\pi_1^{\tame}(U,\bar{x})$ there is
a finite \'etale Galois cover $P\to U$ which is tame such that $G(M)_{|P}$ is trivial. It follows that
$G(M)$ is represented by a finite \'etale $U$-group-scheme $U_{G(M)}$ sitting inside
$P\to U_{G(M)}\to U$. It follows that $U_{G(M)}\to U$  is tame. Hence for $\nu:C\to U$ in $\Cu(U)$,
the base change $U_{G(M)}\times_U C\to C$ is tame (by Theorem \ref{thm:differentNotionsOfTameness}).
This means $\Swan(\nu^*G(M))=0$, so $G(M)$ is tame.
\end{proof}

\begin{remark}\label{rmk:twist-by-tame-character}
\begin{enumerate}
\item Let $A$ and $U$ be as above and $\sF$ a lisse $A$-sheaf on $U$ and $\sL$ a tame $A$-sheaf on $U$, which is free
of rank 1. Let $U\inj X$ be a divisorial compactification and $D$ an effective Cartier divisor supported in $X\setminus U$. 
Then the ramification of $\sF$ is  bounded by $D$  if and only if the ramification of $\sF\otimes \sL$ is bounded
by $D$. Indeed if $U$ is a curve and $x\in X\setminus U$ a point at infinity, then the restriction of 
the representation corresponding to $\sL$ to the wild inertia group at $x$ is trivial; in particular
$\Swan_x(\sF)=\Swan_x(\sF\otimes_A\sL)$. The general case follows from the curve case.
\item If $\sL$ is a  lisse $A$-sheaf  on $\Spec k$, then the pullback $\sL_{|U}$ to $U$ is tame.
\end{enumerate}
\end{remark}

\subsubsection{Weil sheaves}(See \cite[(1.1.7)-(1.1.12)]{Deligne/WeilII}.)
\begin{enumerate}
\item Let $\F_q$ be the finite field with $q=p^n$ elements and fix an algebraic closure $\F$.
          Denote by $F\in \Gal(\F/\F_q)$ the {\em geometric Frobenius} element given by
        $F(a)=a^{-q}$. We define the {\em Weil group $W(\F/\F_q)$ of $\F_q$} \index{Weil group}
        to be the subgroup of $\Gal(\F/\F_q)$ generated by $F$.
        Notice that under the isomorphism $\Gal(\F/\F_q)\xr{\simeq}\hat{\Z}$
	which sends $F$ to $1$,
       the group $W(\F/\F_q)$ is identified with $\Z$. It is equipped with the discrete topology.
\item Let $X$ be a geometrically connected $\F_q$-scheme and $\bar{x}\to X\otimes_{\F_q}\F$ a geometric point. 
       Then there is a short exact sequence (see \cite[IX, Thm.~6.1]{SGA1})
       \[0\to \pi_1(X\otimes_{\F_q} \F,\bar{x})\to \pi_1(X,\bar{x})\to \Gal(\F/\F_q)\to 0.\]
        We define the {\em Weil group of $X$ at $\bar{x}$} to be the topological group
       \[W(X,\bar{x})= \pi_1(X,\bar{x})\times_{\Gal(\F/\F_q)} W(\F/\F_q).\]
        Notice that since $W(\F/\F_q)$ has the discrete topology the subgroup 
      \[\pi_1(X\otimes_{\F_q} \F,\bar{x})=\pi_1(X,\bar{x})\times_{\Gal(\F/\F_q)} \{0\}\subset W(X,\bar{x})\]
        is open and closed.
      (In particular, $W(X,\bar{x})$ is {\em not} equipped with the induced topology from its inclusion
       into $\pi_1(X,\bar{x})$, since $\pi_1(X\otimes_{\F_q} \F,\bar{x})\subset \pi_1(X,\bar{x})$ is not
      open, the quotient group being infinite.) We obtain a short exact sequence of topological groups
      \[0\to \pi_1(X\otimes_{\F_q} \F,\bar{x})\to W(X,\bar{x})\to W(\F/\F_q)\to 0.\]
\item Let $X$ and $\bar{x}\to X\otimes_{\F_q}\F$ be as above. 
          The action of $W(\F/\F_q)$ on $\F$ induces an action of $W(\F/\F_q)$ on $X\otimes_{\F_q}\F$.
         By definition a {\em Weil sheaf on $X$}\index{Weil sheaf}
         is a constructible $\Qlb$-sheaf $\sF$ on $X\otimes_{\F_q} \F$ together with
        an action of $W(\F/\F_q)$ on $\sF$, i.e.~morphisms $\sF(\sigma):\sF\to \sigma^*\sF$, such that
       $\sF(\id)=\id$ and $\tau^*\sF(\sigma)\circ\sF(\tau)=\sF(\tau\circ\sigma)$.
       We say that $\sF$ is a \emph{lisse Weil sheaf} if the underlying $\Qlb$-sheaf on $X\otimes_{\F_q}\F$ is lisse.
     Morphisms between Weil sheaves are morphisms between the underlying $\Qlb$-sheaves, 
     which are compatible with the $W(\F/\F_q)$-action in the obvious sense.
           
       We record the following properties:
       \begin{enumerate}[label={(\roman*)}]
           \item There are natural functors
                   \mlnl{( \text{constructible } \Qlb\text{-sheaves on }X)\inj (\text{Weil sheaves on }X)\\
                                     \to (\text{constructible } \Qlb\text{-sheaves on }X\otimes_{\F_q}\F).}
           Here the first functor is given by pulling back a constructible sheaf
            on $X$ to $X\otimes_{\F_q} \F$ with $W(\F/\F_q)$ acting via restriction
              of the natural $\Gal(\F/\F_q)$-action; this functor is fully faithful. The second functor is 
              given by forgetting the $W(\F/\F_q)$-action. 
            \item Furthermore one can show that there is an equivalence of categories 
                    \[(\text{lisse } \text{Weil sheaves on }X)\xr{\simeq} (\Qlb\text{-representations of }W(X,\bar{x})), \quad
                                                  \sF\mapsto \sF_{\bar{x}}.\]
                     Here we define a $\Qlb$-representations of $W(X,\bar{x})$ as a group
                    homomorphism $W(X,\bar{x})\to \Aut_{\Qlb}(V)$ with $V$ a finite dimensional $\Qlb$-vector space,
                   which factors over a continuous homomorphism $W(X,\bar{x})\to \Aut_{E}(V_E)$, where
                    $E$ is a finite field extension of $\Q_\ell$ and $V_E$ is a finite dimensional $E$-vector space 
                    with an isomorphism $V_E\otimes_E \Qlb\cong V$.
         \end{enumerate}
\item Let $U$ be a normal $\F_q$-scheme,  $U\inj X$ a divisorial compactification and $D$ an effective Cartier
         divisor on $X$ supported in $X\setminus D$. Let $\sF$ be a lisse Weil sheaf on $U$.
         Then we say that the {\em ramification of $\sF$ is bounded by $D$} if the ramification of 
          the underlying lisse $\Qlb$-sheaf on $U\otimes_{\F_q}\F$ is bounded by $D\otimes_{\F_q}\F$.
\end{enumerate}

\begin{remark}
Notice that (up to isomorphism) a lisse Weil sheaf of rank 1 on $\Spec \F_q$ is just a group homomorphism
$W(\F/\F_q)=\Z\to \Qlb^\times$, i.e. it corresponds uniquely to an element in $\Qlb^\times$.
\end{remark}

\begin{theorem}[Deligne, see {\cite{EK12}}]\label{thm-finiteness-ram-sheaves}
Let $U$ be a {\em smooth} and connected $\F_q$-scheme, $U\inj X$ a divisorial compactification and 
$D$ an effective Cartier divisor supported in $X\setminus U$. Let $r\ge 1$ be a natural number.
Then, up to twist with a lisse rank $1$ Weil sheaf coming from $\Spec \F_q$,
there are only finitely many irreducible lisse Weil sheaves of rank $r$ on $U$ whose
ramification is bounded by $D$. 
\end{theorem}

In fact Deligne  proves a more general finiteness result (see Theorem \ref{thm:finiteness-ram-skeleton}). 
We need some preparations to state it.

\subsection{Skeleton sheaves}

\subsubsection{Semi-simplification}\label{Semisimplification}
Let $G$ be a group (possibly infinite), 
$E$ a field and $V$ a finite dimensional $E$-vector space with $G$ action, i.e.~a representation of $G$.
We can view $V$ as a left $E[G]$-module  and $\dim_E V<\infty$ implies that it is artinian and noetherian.
Therefore $V$ has finite length as $E[G]$-module, i.e.~there is a chain of $E[G]$-submodules
\[0=V_0\subset V_1\subset \ldots\subset V_r=V\]
such that the successive subquotients $V_i/V_{i-1}$ are  simple (or irreducible) $E[G]$-modules.
Such a chain is called a decomposition series of $V$.
By the Jordan-H\"older Theorem (see e.g. \cite[I, \S 4.7, Thm.~6]{Bourbaki/AlgebraI}) the length
$r$ is independent of the chosen decomposition series as is
the sequence $V_r/V_{r-1},\ldots, V_1/V_0$, at least up to permutation and isomorphism.
Therefore we can associate to $V$ the $E[G]$-module
\[V^{\rm ss}=\oplus_{i=1}^r V_i/V_{i-1}.\]
The isomorphism class of this module depends only on $V$ and it is called
the {\em semi-simplification of $V$}\index{semi-simplification}. 
Note that if $E$ is a finite extension of $\Q_{\ell}$ and if $G$ is a topological group which acts continuously on $V$ equipped with the $\ell$-adic topology (cf.~\ref{rep-pi1}), then the action of $G$
on the subspaces $V_i$ is automatically continuous. The semi-simplification $V^{\rm ss}$ therefore is a continuous semi-simple $G$-representation.

In particular if $\sF$ is a lisse $\Qlb$-sheaf on a connected $k$-scheme $X$ (resp. a lisse Weil sheaf), then we can 
view it as an  $\Qlb[\pi_1(X)]$-module (resp.~$\Qlb[W(X)]$-module) and can talk about
its semi-simplification $\sF^{\rm ss}$. Notice that $\sF^{\rm ss}$ is a direct sum of irreducible lisse
$\Qlb$-sheaves (resp. lisse Weil sheaves).

\begin{notation}\label{not-R-for-Weil}
Let $U$ be a smooth and connected $\F_q$-scheme and $r\ge 1$ a natural number.
We set
\[\sR_r(U)=\{\text{isomorphism classes of lisse rank $r$ Weil sheaves on $U$}\}/ \sim_{\rm ss},\]
where $\sim_{\rm ss}$ is the equivalence relation defined by
\[\sF\sim_{\rm ss}\sG:\Longleftrightarrow \sF^{\rm ss}\cong \sG^{\rm ss}.\]
Let $U\inj X$ be a divisorial compactification and $D$ an effective Cartier divisor supported on the complement.
We set
\mlnl{\sR_r(U,D):=\sR_r(X,D)\\
                :=\{[\sF]\in \sR_r(U)\,|\, \text{the ramification of }\sF^{\rm ss}\text{ is bounded by }D \},}
here $[\sF]$ denotes the class of a lisse Weil sheaf $\sF$ of rank $r$ in $\sR_r(U)$.
\end{notation}

\begin{definition}\label{defn:skeleton-sheaves}
Let $U$ be a smooth connected $\F_q$-scheme.
\begin{enumerate}
\item  A {\em skeleton sheaf}\index{skeleton sheaf} $\sF$
on $U$ of rank $r$ is a collection
$\sF=(\sF_{C})_{C\in\Cu(U)}$, with $\sF_C\in \sR_r(C)$, which is compatible in the sense
\[{\sF_C}_{|(C\times_U C')_{\rm red}}={\sF_{C'}}_{|(C\times_U C')_{\rm red}}, \quad \text{for all }C,C'\in \Cu(U).\]
We denote by $\sV_r(U)$ the set of rank $r$ skeleton sheaves on $U$.
\item\label{defn:skeleton-sheaves-irreducible} We say that a skeleton sheaf $\sF$ on $U$ 
             is {\em irreducible}\index{irreducible!skeleton sheaf}
         if it cannot be written as a direct sum $\sF_1\oplus\sF_2$, with $\sF_i\in\sV_{r_i}(U)$, $r_i\ge 1$.
       (For $C\in \Cu(U)$ we have a well-defined notion of direct sum in 
          $\bigcup_{r\ge 0}\sR_r{C}$ and for $\sF_i=(\sF_{i,C})_{C\in\Cu(U)} \in\sV_{r_i}(U)$, $r_i\ge 1$,  
          we define   $\sF_1\oplus \sF_2= (\sF_{1,C}\oplus\sF_{2,C})_{C\in \Cu(U)}$;
             clearly $\sF_1\oplus \sF_2$ lies in $\sV_{r_1+r_2}(U)$.)
\item Let $U\inj X$ be a divisorial compactification and $D$ an effective Cartier divisor supported in $X\setminus U$. 
We say that a skeleton sheaf $\sF$ has {\em ramification bounded by $D$} if
\[\sF_C\in \sR_r(\widehat{C},\nu_X^*D),\quad \text{for all } \nu:C\to U \text{ in }\Cu(U).\]
We denote by $\sV_r(U,D)$ or $\sV_r(X,D)$ the set of rank $r$ skeleton sheaves on $U$ 
whose ramification is bounded by $D$. 
\end{enumerate}
\end{definition}

\begin{remark}
Skeleton sheaves were introduced by Deligne and Drinfeld motivated by work of Wiesend, Kerz and Schmidt.
See \cite[2.2]{EK12} for a discussion where the name is coming from.
\end{remark}

\begin{proposition}\label{prop-Weil-to-skeleton}
Let $U$ be a smooth connected $\F_q$-scheme and $r\ge 1$.
There is a well-defined map
\[{\rm sk}:\sR_r(U)\to \sV_r(U),\quad [\sF]\mapsto {\rm sk}([\sF]):=([\sF_{|C}])_{C\in \Cu(U)}.\]
It has the following properties:
\begin{enumerate}
\item\label{sk-ram} 
     If $U\inj X$ is a divisorial compactification and $D$ an effective Cartier divisor supported on $X\setminus U$,
          then ${\rm sk}$ restricts to a map $\sR_r(X,D)\to \sV_r(X,D)$.
\item\label{sk-inj} The map ${\rm sk}$ is injective.
\item\label{sk-irred}
        A lisse Weil sheaf $\sF$ on $U$ is irreducible if and only if ${\rm sk}([\sF])$ is irreducible in the sense of
        Definition \ref{defn:skeleton-sheaves}, \ref{defn:skeleton-sheaves-irreducible}.
\end{enumerate}
 \end{proposition}
\begin{proof}
For the well-definedness of ${\rm sk}$ we have to show that when we have two lisse Weil sheaves 
$\sF$, $\sG$ of rank $r$ on $U$ with $\sF^{\rm ss}=\sG^{\rm ss}$, then 
$(\sF_{|C})^{\rm ss}=(\sG_{|C})^{\rm ss}$ for all $C\to U$ in $\Cu(U)$.
But if we identify $\sF$ with a representation of $\pi_1(U)$ and 
$0=\sF_0\subset\sF_1\subset \ldots\subset \sF_r=\sF$ is a decomposition series of $\sF$, then 
the chain $0=\sF_{0|C}\subset \sF_{1|C}\subset\ldots\subset \sF_{r|C}=\sF_{|C}$
can be refined to a composition series of $\sF_{|C}$. Hence
\[(\sF_{|C})^{\rm ss}\cong ((\sF^{\rm ss})_{|C})^{\rm ss},\]
which immediately gives the well-definedness.

Part \ref{sk-ram} of the Proposition follows directly from the definitions. The parts  \ref{sk-inj}
and \ref{sk-irred} require some preliminary interludes. We start with part
\ref{sk-inj}, for which we recall some well-known facts: 

\subsubsection{Dirichlet density}\label{Dirichlet-density}
We follow \cite{Serre/ZetaL}. Let $Y$ be an irreducible $\F_q$-scheme of dimension $n\ge 1$. 
Let $||Y||$ be the set of closed points in $Y$. Then the sum
$\sum_{y\in ||Y||} \frac{1}{q^{\deg(y)\cdot s }}$ converges absolutely
 for $s\in \C$ with ${\rm Re}(s)>n$, where $\deg(y)=[\F_q(y):\F_q]$. 
(This is a consequence of the fact that the zeta function $\zeta(Y,s)$ of $Y$  converges in this domain.)
Let $M\subset ||Y||$ be a subset. We say that $M$ has {\em Dirichlet
density}\index{Dirichlet density} $\delta$ if 
\[\lim_{s\to n} \left(\frac{\sum_{y\in M} \frac{1}{q^{\deg(y)\cdot s}}}{ \log(\frac{1}{s-n})}\right)=\delta.\]
It follows from the fact that the zeta function of $Y$ has a simple pole at $s=n$, that the Dirichlet density of
$||Y||$ equals 1.

\subsubsection{}\label{Frobenius class}
Let $Y$ be an irreducible $\F_q$-scheme and $\bar{\eta}\to Y$ a geometric point over the generic point of $Y$.
Let $K$ be the residue field of the generic point of $Y$ and $K\inj \bar{K}$ the inclusion corresponding to 
$\bar{\eta}\to\eta$. Let $\F$ be an algebraic closure of $\F_q$ and denote by $F\in \Gal(\F/\F_q)$ the geometric Frobenius 
($F(a)=a^{-q}$). For $y\in ||Y||$ we denote by $F_y\in \pi_1(Y,\bar{\eta})$ the following
conjugation class, called the {\em Frobenius class at $y$}\index{Frobenius class at a point}:
Denote by $K_y$ the completion of $K$ at $y$ and choose an algebraic closure $\bar{K}_y$ with an embedding
$\bar{K}\inj \bar{K}_y$. Denote by $I_y\subset \Gal(\bar{K}_y/K_y)$ the inertia group. We have natural maps
\[\Gal(\F/k(y))\cong \Gal(\bar{K}_y/K_y)/I_y\to \pi_1(Y,\bar{\eta})\]
and we define $F_y$ to be the conjugacy class of the image of $F^{\deg(y)}\in \Gal(\F/k(y))$ in $\pi_1(Y,\bar{\eta})$.
If we choose a different embedding $\bar{K}\inj \bar{K}_y$ then the image of $F^{\deg(y)}$
in $\pi_1(U,\bar{\eta})$ will differ by conjugation, so that the class $F_y$ is independent of all the choices.

Notice that if we have a representation $\rho:\pi(Y,\bar{\eta})\to \Aut_A(M)$ where $M$ is a finite free module 
over some ring $A$, then we can talk about the trace and the characteristic polynomial of $\rho(F_y)$.

\begin{theorem}[Artin-Chebotarev density, see {\cite[Thm.~7]{Serre/ZetaL}}]\label{thm:Chebotarev}
Let $U$ be an irreducible $\F_q$-scheme, $\bar{\eta}\to U$ a geometric point over its generic point
 and $V\to U$ a pointed finite \'etale Galois covering with Galois group $G$.
Let $S\subset G$ be a subset of $G$ which is stable under conjugation.
Then the set
\[\{u\in ||U|| \,\,|\, F_u \text{ is mapped into } S \text{ under } \pi_1(U,\bar{\eta})\twoheadrightarrow G \}\]
has Dirichlet density equal to $|S|/|G|$.
\end{theorem}

\begin{corollary}\label{cor-Frobenius-dense}
Let $\bar{\eta}\to U$ be as in Theorem \ref{thm:Chebotarev} above. Then
the subset $\bigcup_{u\in ||U||} F_u$ is dense in $\pi_1(U,\bar{\eta})$.
\end{corollary}
\begin{proof}
Since $||U||$ has Dirichlet density 1, Theorem \ref{thm:Chebotarev} yields that $S:=\bigcup_{u\in ||U||} F_u$ 
maps surjectively to any finite quotient of $\pi_1(U,\bar{\eta})$. Since $\pi_1(U,\bar{\eta})$ is profinite the
statement follows.
\end{proof}

The following Proposition is a generalization of Corollary \ref{cor-irred-rep-in-rep}.
\begin{proposition}[{\cite[\S 12, No.~1, Prop.~3]{Bourbaki/AlgebraVIII}}]\label{prop:semi-simple-trace}
Let $E$ be a field of characteristic zero and $A$ an $E$-algebra with 1 (not necessarily commutative).
Let $M$ and $M'$ be two semi-simple (left) $A$-modules which are finite dimensional as $E$-vector spaces.
Then $M$ and $M'$ are isomorphic as $A$-modules if and only if $\Tr_{M/E}(a)=\Tr_{M'/E}(a)$, for all $a\in A$.
Here $\Tr_{M/E}(a)$ denotes the trace of the $E$-linear map $M\to M$, $m\mapsto a\cdot m$.
\end{proposition}

{\bf Proof of Proposition \ref{prop-Weil-to-skeleton}, \ref{sk-inj}}.
First of all notice that for every $u\in ||U||$ there is a curve $\nu:C\to U$ in $\Cu(U)$ and a point $u_1\in C$
such that $\nu$ induces an isomorphism $k(u)\xr{\simeq} k(u_1)$. Indeed we can take 
a regular sequence $t_1,\ldots, t_d\in \sO_{U,u}$ which generates the maximal ideal.
Then the vanishing locus of $t_1,\ldots, t_{d-1}$ defines a smooth curve $C_0$ in an open neighborhood of $u\in U$
and we can take $\nu: C\to U$ as the normalization of the closure of $C_0$ in $U$. 
 
Now let $\sF$ and $\sG$ be two semi-simple lisse Weil sheaves on $U$ of rank $r$ and assume 
$\sF_{|C}\cong \sG_{|C}$, for all $\nu : C\to U$. By the above we get
$\Tr_{\sF/\Qlb}(F_u)= \Tr_{\sG/\Qlb}(F_u)$ for all $u\in ||U||$.
(Here we identify $\sF$ and $\sG$ with their representations and use the notation from \ref{Frobenius class}.)
Therefore, there exists a finite field extension $E/\Q_\ell$ such that 
$\Tr_{\sF/\Qlb}$, $\Tr_{\sG/\Qlb}:\pi_1(U)\to E$ are continuous class functions which
by Corollary \ref{cor-Frobenius-dense} coincide on a dense open subset. Since the $\ell$-adic topology on $E$ is 
Hausdorff this implies $\Tr_{\sF/\Qlb}=\Tr_{\sG/\Qlb}$. Now Proposition \ref{prop:semi-simple-trace}
with $A=E[\pi_1(U)]$ gives $\sF=\sG$.\qed
\\
\\
\indent Part \ref{sk-irred} of Proposition \ref{prop-Weil-to-skeleton} will be a consequence of a series
of lemmas and propositions. We follow closely the arguments in \cite[Prop. B.1]{EK12}:

\begin{lemma}\label{lem:surj-on-ab-vs-surj}
Let $\ell$ be a prime number and $K\to H$ a morphism of profinite groups with $H$  a pro-$\ell$ group.
Denote by $H^{\rm ab}$ the abelianization of $H$ in the category of profinite groups; it is
the quotient of $H$ by the closure of the commutator subgroup.
If the induced map $K\to H^{\rm ab}/\ell H^{\rm ab}$ is surjective, then so is $K\to H$.
\end{lemma}
\begin{proof}[Proof (cf. {\cite[I.4.2]{Serre/GaloisCohomology}}).]
We may assume without loss of generality that we have an inverse system of maps 
between finite groups $K_i\to H_i$, $i\in I$,  which in the limit gives the map $K\to H$.
(Indeed we know $H=\varprojlim_U H/U$, where the limit is over all normal open subgroups $U\subset H$.
We denote by $K_U$ the preimage of $U$ in $K$ and can replace $K$ by $\varprojlim_U K/K_U$.)
Since $H^{\rm ab}= \varprojlim_i H_i^{\rm ab}$, we are reduced to the case in which
$K$ is a finite group and $H$ a finite $\ell$-group. Now assume the map $K\to H^{\rm ab}/\ell$ is surjective.
Since $H^{\ab}$ is a finitely generated $\Z/\ell^N\Z$-module (for some $N>>0$), Nakayama's Lemma
implies that $K\to H^{\rm ab}$ is surjective.

Let $\ell^r$ be the order of $H$, then by \cite[I, \S 6.5, Thm.~1]{Bourbaki/AlgebraI} there exists a sequence
\eq{lem:surj-on-ab-vs-surj1}{H=H^1\supset H^2\supset\ldots\supset H^{r+1}=\{1\}}
 such that $[H, H^i]\subset H^{i+1}$, $1\le i\le r$,
and $H^i/H^{i+1}$ is cyclic of order $\ell$. Denote by $K'\subset H$ the image of $K$  in $H$.
Assume $K'$ is strictly contained in $H$. Then there exists a minimal $i\le r$ such that   
$N:=K'\cdot H^{i+1}\subsetneq K'\cdot H^i=H$. Then $N\subset H$ is a normal subgroup. 
(Indeed it suffices to see that $H^i$ normalizes $N$. It clearly normalizes $H^{i+1}$ and 
for $k\in K'$ and $h\in H^i$ we have  $h\cdot k\cdot h^{-1}= k \cdot [k^{-1}, h]\in N$.)
Further $H/N$ is a non-zero quotient of $H^i/H^{i+1}$ hence is cyclic of order $\ell$.
Thus the quotient map $H\to H/N$ factors over $H^{\rm ab}\to H/N$. But $K$ maps to $1$ in $H/N$, 
a contradiction to the surjectivity of $K\to H^{\rm ab}$. Thus $K'=H$.
\end{proof}
  
\begin{lemma}[{\cite[Lem.~B.2]{EK12}}]\label{lem-irred-curve}
Let $U$ be a smooth connected $\F_q$-scheme and $\sF$ an irreducible lisse $\Qlb$-sheaf on $U$.
Then there exists a connected finite \'etale  Galois cover $U'\to U$ with the following property:
For any $\nu: C\to U$ in $\Cu(U)$ such that $C\times_U U'$ is connected the pullback $\nu^*\sF$ is irreducible.
\end{lemma}
\begin{proof}
Take a DVR $R$ finite over $\Z_\ell$ and a free lisse $R$-sheaf $(\sF_n)_n$ such that $\sF=(\sF_n)\otimes_R\Qlb$.
Let $\F_\lambda=R/\fm$ be the residue field of $R$ and denote by $\rho: \pi_1^{\et}(U)\to \GL_r(R)$ 
the $R$-representation corresponding to $(\sF_n)$, where $r=\rank \sF$. 
Let $G=\rho(\pi_1^{\et}(U))$ be the image of $\rho$ and set
\[H_1:=\ker\bigg(\pi_1^{\et}(U)\xrightarrow{\rho}\GL_r(R)\to \GL_r(\F_\lambda)\bigg),\quad 
 H_2:=\bigcap_{\varphi\in\Hom_{\rm cont}(H_1,\Z/\ell)}\ker(\varphi).\]
Then $H_1$ is an open normal subgroup of $\pi_1^{\et}(U)$, hence there exists a finite \'etale Galois cover 
$U_{H_1}\to U$ with $\pi_1^{\et}(U_{H_1})=H_1$. 
Now the group $H_1^{\rm ab}/\ell$ is profinite, in particular compact, and its Pontryagin dual is equal to 
\[\Hom_{\rm cont}(H_1^{\rm ab}/\ell,\R/\Z)=\Hom_{\rm cont}(\pi_1^{\et}(U_{H_1}),\Z/\ell)
 = H^1(U_{H_1, \et}, \Z/\ell).\]
Here the second equality follows from the fact that for a finite abelian group $A$ the first cohomology group
$H^1(U_{\rm et}, A)$ is in bijection to the set of $A$-torsors over $U$ 
(see e.g. \cite[III, \S4]{Milne/EtaleCohomologyBook}).
In particular $H_1^{\rm ab}/\ell$ is a finite group, cf.~\ref{coh-properties}, \ref{coh-finite}.
Further, by definition $H_2\subset H_1$ is  a closed normal subgroup.
We have $H_2=\ker(H_1\to H_1^{\rm ab}/\ell)$. (Clearly $\supset$; for the other inclusion it suffices
to see that for any non-zero element  $\sigma\in H_1^{\rm ab}/\ell$ there exists a continuous map
$H_1^{\rm ab}/\ell\to \Z/\ell$ sending $\sigma$ to a non-zero element. The construction of such a map is easily
achieved by considering a filtration as in \eqref{lem:surj-on-ab-vs-surj1}.)
Hence $H_1/H_2=H^{\rm ab}_1/\ell$ and therefore $H_2$ is an open normal subgroup of $H_1$.
Denote by $U'\to U$ the associated Galois cover. 

Now assume $\nu: C\to U$ is in $\Cu(U)$ and $C':=C\times_U U'$ is connected.
Then $C'\to C$ is a Galois cover with Galois group $\pi^{\et}(U)/H_2$.
Denote by $H_2':=\rho(H_2)\subset H_1':=\rho(H_1)$ 
the images of  $H_i$ in $G$. Notice that $H_1'$ is a normal subgroup and is contained in
$1+\fm\End_R(R^r)$, hence is pro-$\ell$. Denote by $\rho_C$ the composition (well-defined up to conjugation)
$\rho_C:\pi_1^{\et}(C)\to \pi_1^{\et}(U)\to G$. By the above $\pi^{\et}_1(U)/H_2$ is a quotient of $\pi_1^{\rm et}(C)$.
Hence $\rho_C$ maps onto $G/H'_2$ and a fortiori onto $G/H'_1$. Let $K\subset \pi_1^{\et}(C)$ 
be the preimage of $H'_1/H_2'={H'_1}^{\rm ab}/ \ell$. Then $K\to H_1'$ is a map of profinite groups which is surjective by Lemma \ref{lem:surj-on-ab-vs-surj}. All together we see that $\rho_C$ is surjective.
Since $\rho$ is irreducible so is $\rho_C$ and this finishes the proof.
\end{proof}

We will need the following version of the Hilbert irreducibility theorem.

\begin{theorem}[see e.g. {\cite[Cor.~A.2]{Drinfeld}}]\label{thm-Hilbert-irre}
Let $C$ be a smooth projective curve over $\F_q$ with function field $K$. Let $x\in C$ be a closed point
and denote by $K_x$ the completion of $K$ along $x$. Let $V\subset \A^n_K$ be a non-empty open subset
and $V'\to V$ a connected finite \'etale covering.
Then the set of $K$-rational points  $y\in V(K)$ which do not split in $V'$ (i.e.~$y\times_V V'$ is a single point)
is dense in $V(K_x)$. Here we equip $V(K_x)\subset K_x^n$ with the subspace topology.  
\end{theorem}

\begin{lemma}\label{lem:curve-con-points}
Let $k$ be a perfect field and $y,y'\in \A^d_k$ two closed points. Then there exists an integral curve
$\Gamma\subset \A^d_k$ with $y,y'$ contained in the smooth locus of $\Gamma$.
\end{lemma}
\begin{proof}
In case $y,y'$ are $k$-rational points we can simply take $\Gamma$ to be the line $L$ connecting $y$ and $y'$.
In general we find a finite Galois extension $k'/k$ such that $y, y'$ split completely in $\A^d_{k'}$, i.e.
\[y\times_{\Spec k} \Spec k'=\{y_1,\ldots,y_n\},\quad y'\times_{\Spec k}\Spec k'=\{y'_1,\ldots, y'_{n'}\}\]
with $k'$-rational points $y_i, y'_j\in \A^d_{k'}$. Denote by $L_{i,j}$ the line in $\A^d_{k'}$ connecting
$y_i$ with $y'_j$. Set and let $\Gamma'=\cup_{i,j} L_{i,j}$; it is a curve in $\A^d_{k'}$, which is smooth in the points
$y_i, y_j'$. Since $\Gal(k'/k)$ acts via permutation on $\{y_i\}$ and on $\{y'_j\}$ it also permutes
the $L_{i,j}$. In particular $\Gamma'$ is stable under the $\Gal(k'/k)$ action and by Galois descent 
(a special case of Theorem \ref{thm:descent}) there exists a curve $\Gamma_0\subset \A^d_{k}$ with
$\Gamma_0\times_k k'=\Gamma'$.  Let $\Gamma\subset \Gamma_0$ be an irreducible component (with its reduced scheme structure.)
Then $\Gamma\times_k k'$ is a closed subcurve  of $\Gamma'$, hence is the union of certain $L_{i,j}$
and therefore contains certain $y_i, y'_j$. It follows that $\Gamma$ contains $y,y'$ and that these lie in the smooth locus of $\Gamma$.
 \end{proof}

\begin{proposition}[{\cite[Prop.~B.1]{EK12}}]\label{prop-irred-curve}
Let $U$ be a smooth connected $\F_q$-scheme, $x\in U$ a closed point and $\sF\in \sR_r(U)$ irreducible.
Then there exists a curve $\nu: C\to U$ in $\Cu(U)$ such that $x\in \nu(C)$ and $\nu^*\sF$ is irreducible.
\end{proposition}
\begin{proof}
By \cite[Prop.~(1.3.14)]{Deligne/WeilII} there exists a character $\chi\in \sR_1(\F_q)$ such that
$\sF\otimes\chi$ is a lisse $\Qlb$-sheaf. Thus we may assume from the beginning that $\sF$
is an irreducible lisse $\Qlb$-sheaf on $U$. Let $U'\to U$ be the finite \'etale Galois cover from 
Lemma \ref{lem-irred-curve}. Then it suffices to construct a curve $\nu: C\to U$ in $\Cu(U)$ such that
$U'\times_U C$ is connected and $x\in\nu(C)$.

To this end we proceed as follows.  By Noether normalization (see e.g. \cite[Cor.~16.18]{Eisenbud})
there exists a finite and generically \'etale $\F_q$-morphism $f: U\to \A^d_{\F_q}$.
Set $y:=f(x)\in \A^d_{\F_q}$. 
Denote by $V\subset \A^d_{\F_q}$ a non-empty open subset such that $f^{-1}(V)\to V$ is finite \'etale.
Let $y'\in V$ be a closed point. Then by Lemma \ref{lem:curve-con-points} we find 
 $\mu: \Gamma \to \A^n_{\F_q}$ in $\Cu(\A^n_{\F_q})$, such that $\mu$ is an isomorphism over $y,y'$.
Let $K$ be the function field of $\Gamma$ and $K_y$ the completion along $y$.
Since $\Gamma$ is smooth the choice of a local parameter at $y$ yields a unique isomorphism
of $\F_q$-algebras $K_y\cong k(y)\llparen t\rrparen$ (see e.g. \cite[II, \S4]{Serre/LocalFields}). 
On the other hand $k(y)$ is a finite extension of $\F_q$ hence of the form $k(y)=\F_q[x]/f$ for some irreducible
polynomial $f$. Thus there is a closed point $z\in \A^1_k$ with $k(z)=k(y)$ and 
 the completion of $k(\A^1)$ at $z$ is isomorphic to $K_y$, i.e.~$k(\A^1)_z\cong K_y=k(z)\llparen t\rrparen$.

Let $\A^d_{\F_q}\to \A^1_{\F_q}$ be a linear projection such that the composition 
$\Gamma\to \A^d_{\F_q}\to \A^1_{\F_q}$ is finite.
Denote by $\A^{d-1}_{k(\A^1)}$ the base change of this projection along $\Spec k(\A^1)\to \A^1_{\F_q}$ 
(inclusion of the generic point) and 
by $V_{k(\A^1)}\subset \A^{d-1}_{k(\A^1)}$ the corresponding base change of $V$.
Since $\Gamma$ is finite over $\A^1$ and $\mu(\Gamma)\cap V\ni y'$ the map $\mu$ induces a morphism
$\Spec k(\A^1)_z=\Spec K_y\to  V_{k(\A^1)}$, i.e.~an element $\mu_z\in V_{k(\A^1)}(k(\A^1)_z)$.
By Theorem \ref{thm-Hilbert-irre} we find a point $v\in V_{k(\A^1)}(k(\A^1))$ which is $t$-adically arbitrarily close to 
$\mu_z$ and which does not split in $U'_{k(\A^1)}$. 
Since $\mu_z$ spreads out to $\Spec (k(z)\llbracket t\rrbracket)=\Spec \sO_{K_y}\to \Spec \sO_{\A^d_{\F_q},y}$, 
where the closed point maps to $y$, we can achieve that $v$ spreads out to a map 
$\Spec \sO_{\A^1, z}\to \Spec \sO_{\A^d_{\F_q},y}$, which sends the closed point to $y$.
Then $v$ defines a point in $V$ whose closure in $\A^d_{\F_q}$ is a curve containing $y$ and over which in $U'$
-- a fortiori in $U$ -- lies exactly one point. Let $u\in U$ be the unique point lying over $v$.
By the going-down theorem (see e.g. \cite[Thm.~9.4, (ii)]{Matsumura}) there exists a curve in $U$ whose generic point lies
over $v$ and which contains $x$; since $u$ is the only point in $U$ mapping to $v$ we get that the closure of $u$ in $U$
contains $x$. Let $\nu: C\to U$ be the normalization of $\overline{\{u\}}$. Then $x\in \nu(C)$ and 
$U'\times_U C$ is irreducible by the choice of $v$. This finishes the proof.
 \end{proof}

{\bf Proof of Proposition \ref{prop-Weil-to-skeleton}, \ref{sk-irred}.}
Let  $\sF$ be  a lisse Weil sheaf of rank $r$ on $U$.
If $\sF$ is not irreducible, then the class $[\sF]$ in $\sR_r(U)$ can be written as a non-trivial sum
$[\sF]=[\sF_1]\oplus [\sF_2]$. Hence ${\rm sk}([\sF])={\rm sk}([\sF_1])\oplus {\rm sk}([\sF_2])$ is not irreducible in
$\sV_r(U)$. If $\sF$ is irreducible, then for all closed points $x\in U$ there exists a curve $\nu_x: C_x\to U$ such that
$x\in\nu_x(C_x)$ and $\nu_x^*\sF$ is irreducible, by Proposition \ref{prop-irred-curve}.
Assume ${\rm sk}([\sF])=\sF_1\oplus \sF_2$, with $\sF_i\in \sV_{r_i}(U)$. Then it follows that either
$\nu^*_x\sF_1=0$ for all $x\in U$ or $\nu^*_x\sF_2=0$ for all $x\in U$.
(Indeed the set $Z_i$ of all $x\in U$ such that $\nu_x^*\sF_i=0$ is closed; since $U=Z_1\sqcup Z_2$ and 
$U$ is connected we have either $Z_1=U$ or $Z_2=U$.)
Since the $\sF_i$ are lisse we obtain that either $\sF_1=0$ or $\sF_2=0$. Hence  ${\rm sk}([\sF])$ is irreducible. 
This finishes the proof of Proposition \ref{prop-Weil-to-skeleton}.
\end{proof}

In view of Proposition \ref{prop-Weil-to-skeleton}, Theorem \ref{thm-finiteness-ram-sheaves}
is a consequence of the following theorem.

\begin{theorem}[Deligne]\label{thm:finiteness-ram-skeleton}
Let $U$ be a smooth connected $\F_q$-scheme, $U\inj X$ a divisorial compactification and 
$D$ an effective Cartier divisor supported in $X\setminus U$. Then the set of irreducible
skeleton sheaves $\sF\in \sV_r(X, D)$ is finite up to twists by elements from $\sR_1(\Spec \F_q)$
and its cardinality does not depend on the choices of $\ell\neq p$.
\end{theorem}

For a proof of this Theorem, which relies on the Langlands correspondence \cite{Lafforgue}
and Weil II \cite{Deligne/WeilII}, we refer the reader to \cite{EK12}.

\begin{question}[Deligne, see {\cite[Question 1.2]{EK12}}]
Let $U$ be a smooth connected $\F_q$-scheme, $U\inj X$ a divisorial compactification 
and $D$ an effective Cartier divisor supported in $X\setminus U$.
Is the injection ${\rm sk}: \sR_r(X,D)\inj \sV_r(X, D)$ actually a bijection?
\end{question}

\begin{remark}
The above question has a positive answer for $r=1$, see \cite[Cor.~V]{Kerz-Saito/Chow-modulus}. 
\end{remark}


\section{Grothendieck-Ogg-Shafarevich in higher dimensions following Kato-Saito}\label{higher-GOS}
\setcounter{subsection}{0}\setcounter{subsubsection}{0}
In this section we give a short account of the generalization of the Grothendieck-Ogg-Shafarevich formula
to higher dimension due to Kato and Saito.

\setcounter{subsubsection}{0}
\subsubsection{}\label{conventions-hdGOS}Throughout this section we fix the following notation:
\begin{itemize}
\item $k$ is a perfect field of characteristic $p>0$ and  $\bar{k}$ an algebraic closure.
\item $\ell$ is a prime number different from $p$.
\item A $k$-scheme is a scheme which is separated and of finite type over $k$.
\item If $X$ is a $k$-scheme we denote by $CH_0(X)$ the Chow group of zero-cycles on $X$ modulo
 rational equivalence, i.e.~it is the free abelian group on the set of closed points of $X$ modulo
the subgroup generated by elements of the form $\div(f)_C$, where $C\subset X$ is a closed integral subscheme of dimension 1,
         $f\in k(C)^\times$ and $\div(f)_C$ denotes the divisor on $C$ associated to $f$, which we view as a zero-cycle on $X$. 
 \end{itemize}

\subsection{Higher dimensional version of Grothendieck-Ogg-Shafarevich}

\subsubsection{Intersection product with the log-diagonal}
Let $U$ be a smooth connected $k$-scheme of dimension $d$. Let $j:U\inj X$ and $j': U\inj X'$ be two open embeddings 
of $U$ into integral proper $k$-schemes. If there exists a morphism
$\pi: X'\to X$ such that $\pi\circ j'=j$, then $\pi$ is unique and induces a proper morphism
$\pi: X'\setminus U\to X\setminus U$ and hence also a pushforward
\[\pi_*: CH_0(X'\setminus U)\to CH_0(X\setminus U),\]
which sends the class of a closed point $x'\in X'\setminus U$ to the class of $[x':\pi(x')]\cdot \pi(x')$.
We obtain a projective system indexed by open embeddings $j:U\inj X$ as above.
We set
\[CH_0(\partial U):=\varprojlim_j CH_0(X\setminus U),\quad 
                                                  CH_0(\partial U)_\Q:=\varprojlim_j (CH_0(X\setminus U)\otimes_\Z\Q).\]
The degree map $\deg_k: CH_0(X\setminus U)\to CH_0(\Spec k)=\Z$  is induced by pushforward along the
structure map $X\to \Spec k$ and hence induces well-defined degree maps
\[\deg_k: CH_0(\partial U)\to \Z,\quad \deg_k:CH_0(\partial U)_\Q\to \Q.\]
Notice that if $U$ is a curve with unique smooth compactification $j: U\inj C$, then
$CH_0(\partial U)=CH_0(C\setminus U)$ and $CH_0(\partial U)_\Q=CH_0(C\setminus U)\otimes_\Z \Q$.

Let $V\to U$ be a finite \'etale Galois cover with Galois group $G$. 
Denote by $\Delta_V\subset V\times_U V$ the diagonal. Notice that 
\[(V\times_U V)\setminus\Delta_V=\bigsqcup_{\sigma\in G\setminus\{1\}} \Gamma_\sigma,\]
where $\Gamma_\sigma$ is the graph of $\sigma: V\to V$. 
In \cite[Thm 3.2.3]{Kato} Kato-Saito define for each $\sigma\neq 1$ an element  
\[(\Gamma_\sigma,\Delta_{\bar{V}})^{\log}\in CH_0(\partial V)_\Q,\]
which is called {\em the intersection product of $\Gamma_\sigma$  with the log-diagonal}.
It has the following properties:
\begin{enumerate}
\item The subscheme $\Gamma_\sigma\subset V\times_U V$ defines a map
     \[\Gamma^*_\sigma: H^i_c(V\otimes_k\bar{k},\Q_\ell)\xr{p_1^*}
                       H^i_c(\Gamma_\sigma\otimes_k\bar{k},\Q_{\ell})\xr{p_{2*}} 
                  H^i_c(V\otimes_k \bar{k},\Q_\ell).\]
Here the $p_i:\Gamma_\sigma\to V$, $i=1,2$, are induced by the two projection maps $V\times_U V\to V$ 
and the pushforward $p_{2*}$ is induced by
\[p_{2*}(\Q_{\ell,\Gamma_\sigma})=(\Q_{\ell,V})^{|G|}\xr{\sum} \Q_{\ell,V} .\]
      Then (see \cite[Prop. 3.2.4]{Kato})
             \[\sum_{i=0}^{2d} (-1)^i\Tr(\Gamma_\sigma^*|H^i_c(V\otimes_k\bar{k},\Q_\ell))= 
        \deg_k(\Gamma_\sigma,\Delta_{\bar{V}})^{\log}.\]
\item Assume there exist compactifications $V\inj Y$ and $U\inj X$ and an {\em \'etale} morphism
       $Y\to X$ extending the covering $V\rightarrow U$. Then the image of
       $(\Gamma_\sigma,\Delta_V)^{\log}$ under the projection
        \eq{intersection-with-diagonal-proj}{   CH_0(\partial V)_\Q  \to CH_0(Y\setminus V)\otimes_\Z\Q}
          is zero (see \cite[Cor. 3.3.4]{Kato}).
\item Let $V\inj Y$ and $U\inj X$  be compactifications and 
       $Y\to X$ a morphism extending $V\to U$. 
       Assume $X$ is smooth and $X\setminus U$ is a strict normal crossings divisor. 
       If $Y\to X$ is tamely ramified along $X\setminus U$ 
        (in the sense of Definition \ref{defn:tameness},\ref{def:tameness-wrt-emb}), then  
        $(\Gamma_\sigma,\Delta_V)^{\log}$ maps to zero under the projection
        \eqref{intersection-with-diagonal-proj} (see \cite[Prop. 3.3.5, 2.]{Kato}).
\end{enumerate}

\subsubsection{Swan character class and Swan class}
Let $U$ be a smooth connected $k$-scheme.
\begin{enumerate}
\item\label{swan-char} Let $\pi: V\to U$ be a finite \'etale Galois covering with Galois group $G$. 
                                 Then in \cite[Def.~4.1.1]{Kato} the {\em Swan character} $s_{V/U}(\sigma)\in CH_0(\partial V)$, 
                              $\sigma\in G$, is defined by
\[ s_{V/U}(\sigma)=\begin{cases} -(\Gamma_\sigma,\Delta_V)^{\log},&\text{if } \sigma\neq 1\\
                                        \sum_{\sigma\neq 1} (\Gamma_\sigma,\Delta_V)^{\log}, & \text{if }\sigma=1.\end{cases}\]
\item Let $\sF$ be a lisse $\F_\lambda$-sheaf on $U$, for some finite extension 
        $\F_\lambda$ of $\F_\ell$. Assume $\sF$ is trivialized by a Galois
	covering $\pi: V\to U$ as in \ref{swan-char}.
        Then in \cite[Def.~4.2.1]{Kato} the {\em Swan class} $\Sw_{V/U}(\sF)\in CH_0(\partial V)_\Q$ is defined 
          by 
\[\Sw_{V/U}(\sF):= \sum_{\sigma\in G_{(p)}}\left( \dim_{\F_\ell}M^\sigma- 
                   \frac{\dim_{\F_\ell} M^{\sigma^p}/M^\sigma}{p-1}\right)\cdot s_{V/U}(\sigma).\]
           Here $M$ is the $G$-representation corresponding to $\sF$, $M^\sigma=\{m\in M\,|\,\sigma(m)=m\}$
            and   $G_{(p)}=\{\sigma\in G\,|\, \sigma \text{ has }p\text{-power order}\}$.
 \item Let $A$ be an $\ell$-adic coefficient ring and $\sF$ a lisse $A$-sheaf on $U$.
          Then in \cite[Def.~4.2.8]{Kato} the {\em Swan class} $\Sw_U(\sF)\in CH_0(\partial U)_\Q$ is defined as follows:
          We find a DVR $R$ finite over $\Z_\ell$ and a lisse $R$-sheaf $(\sF_n)$ such that
         $\sF=(\sF_n)\otimes_R A$. We find a Galois cover $\pi: V\to U$ as in \ref{swan-char} trivializing $\sF_1$
	 and we define
\[\Sw_U(\sF):=\frac{1}{|G|} \cdot \pi_*\Sw_{V/U}(\sF_1).\]
          This definition is independent of the choice of $R$, $(\sF_n)$, and $\pi$.
 \end{enumerate}

\begin{remark}
\begin{enumerate}
\item Considering Lemma \ref{lem-intersection-multiplicities}, the definition of $s_{V/U}(\sigma)$ above
      should remind you of the definition     of the character $a_G$ in Theorem \ref{thm:ExistenceOfArtinComplex}.
\item If $U$ is a smooth curve with smooth compactification  $X$ and $\sF$ is a lisse $A$-sheaf on $U$, then
          $\Sw_U(\sF)=\sum_{x\in X\setminus U} \Swan_x(\sF)[x]$ in 
            $CH_0(\partial U)_\Q=CH_0(X\setminus U)\otimes_Z\Q$, see \cite[Introduction]{Kato}.
\item In \cite[Conj.~4.3.7]{Kato} it is conjectured that $\Sw_U(\sF)$ always lies in the image 
          of $CH_0(\partial U)\to CH_0(\partial U)_\Q$. In dimension 1 this is a consequence of the Hasse-Arf theorem
         and the remark above. In dimension 2 this conjecture is proved in \cite[Cor.~5.1.7]{Kato}.
\end{enumerate}
\end{remark}

The higher dimensional version of the Grothendieck-Ogg-Shafarevich formula now takes the following form:
\begin{theorem}[{\cite[4.2.9]{Kato}}]\label{thm:higherGOS}
Let $U$ be a smooth and connected $k$-scheme of dimension $d$ and $\sF$ a lisse $\Qlb$-sheaf on $U$.
Then 
\mlnl{\chi_c(U\otimes_k \bar{k},\sF):=\sum_{i=0}^{2d}(-1)^i\dim_{\Qlb} H^i_c(U\otimes_k \bar{k},\sF)\\
         =\rk(\sF)\cdot\chi_c(U,\Qlb)-\deg_k(\Sw_U(\sF)). }
\end{theorem}

We conclude with some remarks on more recent developments. For a more extensive overview we refer to \cite{Xiao}.

\subsection{Outlook on ramification theory following Abbes-Saito}
\begin{enumerate}
	\item In \cite{AbbesSaito07}, Abbes and Saito give a refinement of Theorem \ref{thm:higherGOS}.
		If $\pi: X\to\Spec k$ is a $k$-scheme and $A$ an $\ell$-adic coefficient ring, then
		the dualizing complex $K_X:=\pi^!(A_{\Spec k})$ is defined as an object in the 
		derived category of complexes of $A$-modules of finite tor-dimension with constructible cohomology.
		(In case $\pi$ is smooth of pure dimension $d$, $K_X\cong A(d)[2d]$.)
		Let $\sF$ be a constructible sheaf of free $A$-modules. Then in \cite[Def.~2.1.1]{AbbesSaito07}
		the {\em characteristic class of $\sF$} is defined as an element $C(\sF)\in H^0(X, K_X)$.
		It has the property that if $\pi$ is proper, then 
		\[\Tr_\pi(C(\sF))=\chi(X_{\bar{k}},\sF)\in A\]
		where $\Tr_\pi: H^0(X, K_X)\to H^0(\Spec k, A)=A$ is the trace map.

		Now assume $U$ is a smooth $k$-scheme and that $j:U\inj X$ is a compactification.
		Let $\sF$ be a free lisse $A$-sheaf on $U$ which is of Kummer type with respect to $X$
		(see \cite[Def.~3.1.1]{AbbesSaito07}). Then (\cite[Thm.~3.3.1]{AbbesSaito07})
		\eq{finer-higherGOS}{C(j_!\sF)= \rk(\sF)\cdot C(j_!A_U)-
			[\Sw_U^{\text{naive}}(\sF)]\quad \text{in }H^0(X,K_X),}
		where $\Sw_U^{\text{naive}}(\sF)$ is the naive Swan class (it is conjectured to be equal 
		to $\Sw_U(\sF)$ and it is known that they have the same degree) and 
		$[\Sw_U^{\text{naive}}(\sF)]$ denotes the cycle class of $\Sw_U^{\text{naive}}(\sF)$ in $H^0(X,K_X)$.
		In particular applying $\Tr_\pi$ we get back Theorem \ref{thm:higherGOS} for $\sF$ as above.
		Actually, Theorem \ref{thm:higherGOS} can be reproved in its full generality using the 
		equality \eqref{finer-higherGOS}, see \cite[Cor.~3.3.2]{AbbesSaito07}.
	\item\label{item:logSwanConductor} We comment on the relation between the ramification theory via curves which is described in Section \ref{RTvC}
		and the ramification theory developed by Abbes-Saito in \cite{AbbesSaito02}, \cite{AbbesSaito03},
		\cite{Saito09}, \cite{AbbesSaito11}. 

		Let $K$ be a complete discrete valuation field with arbitrary residue field,
		$K^{\rm sep}$ a separable closure of $K$ and $G=\Gal(K^{\rm sep}/K)$ the absolute Galois group.
		Then in \cite{AbbesSaito02} a {\em logarithmic ramification filtration} 
		$(G^r_{\log})_{r\in \Q_{\ge 0}}$ is defined, which has the following properties:
		\begin{enumerate}[label={(\roman*})]
			\item $G^s_{\log}\subset G^r_{\log}$, for $s<r$.
			\item $G_{\log}^0=I=$ inertia subgroup of $G$, 
				$G_{\log}^{0+}:=\overline{\bigcup_{r>0}G_{\log}^r}=P=$ wild inertia subgroup.
			\item Set $G_{\log}^{r+}:=\overline{\bigcup_{s>r}G_{\log}^s}$.  
				If $K$ has characteristic $p>0$ and ``comes from geometry'', then the quotient
				$G^r_{\log}/G^{r+}_{\log}$ is abelian and is killed by $p$ for all $r>0$.
			\item\label{item:logSwanConductor:d} In particular, if $\ell$ is a prime different from $p$, then according to Lemma
				\ref{lemma:proPfactorsThroughFiniteQuotient},
				every $\ell$-adic representation $V$ of
				$G_{\log}^{0+}$ factors through a finite
				quotient, and as in Proposition
				\ref{prop:breakDecomposition} we obtain a break
				decomposition $V=\bigoplus_{x\in \Q_{\geq 0}}
				V(x)$. One then defines the \emph{logarithmic Swan
					conductor} $\Sw^{\log}(V):=\sum_{x\in
					\Q_{\geq 0}} x\dim V(x)$.  \item If the
				residue field of $K$ is perfect, the filtration
				$G_{\log}^r$ coincides with the classical
				ramification filtration from Definition
				\ref{defn:upperNumberingInfiniteExtension}, and
				for an $\ell$-adic representation $V$ of $G$ we
				have $\Sw(V)=\Sw^{\log}(V)$.
		\end{enumerate}

		Now let $U$ be a smooth connected $k$-scheme and denote by
		$\bar{\eta}\to U$ a geometric point over the generic point of
		$U$. Assume that we have a dominant open embedding $j:U\inj X$
		into a smooth $k$-scheme such that the complement $X\setminus U$
		with its reduced structure is a strict normal crossings divisor
		$D=\cup_i D_i$. Denote the generic points of $D$ by $\xi_i\in
		D_i$ and by $G_{\xi_i}$ the decomposition subgroups at $\xi_i$
		of the absolute Galois group of $k(X)$ (well-defined up to
		conjugation).  Let $\sF$ be a free lisse $A$-sheaf on $U$.
		Define the \emph{logarithmic Swan conductor of $\mathcal{F}$ with respect to $X$ as} 
		\[\Sw_{X}^{\log}(\mathcal{F}):=\sum_{i}\Sw_{\xi_i}^{\log}(\mathcal{F})\cdot D_i,\]
		where $\Sw^{\log}_{\xi_i}(\mathcal{F})$ is the logarithmic Swan
		conductor of the representation $\mathcal{F}|_{G_{\xi}}$ in the
		sense of \ref{item:logSwanConductor},
		\ref{item:logSwanConductor:d} above. Note that this is
		well-defined, even though $G_{\xi_i}$ is only well-defined up to
		conjugation.

		In \cite[after Remark 3.10]{EK12} the following conjecture is
		formulated (see also \cite[Conj.~B]{Barrientos}):
		\begin{conjecture}
			Let $U$ be a smooth connected $k$-scheme and $\sF$ a lisse $A$-sheaf on $U$. Let
			$j:U\inj X$ be a divisorial compactification and $D$ an effective Cartier divisor supported in $X\setminus U$.
			Then the following two statements are equivalent:
			\begin{enumerate}
				\item The ramification of $\sF$ is bounded by $D$ in the sense of Definition \ref{def:ramification-of-sheaves}.
				\item For all $k$-morphisms $h: V\to X$ fitting into a commutative diagram 
					\[\begin{tikzcd}    
							~ & V\ar{d}{h}\\
							U\ar{ur}{j'}\ar[swap]{r}{j} & X,
						\end{tikzcd}\]
					such that $V$ is a smooth $k$-scheme, $j'$ is an open
					dominant embedding 
					and the complement $V\setminus U$ is a strict normal crossings divisor, we have 
					\[\Sw_V^{\log}(\sF)\le h^*D.\]
			\end{enumerate}
		\end{conjecture}
		For $D=0$ this is Proposition \ref{prop:tame-sheaf-vs-rep};  for $X$ smooth and $\sF$ of rank $1$
		this conjecture is proved in \cite[Thm.~7.1]{Barrientos}. 
\end{enumerate}

\providecommand{\bysame}{\leavevmode\hbox to3em{\hrulefill}\thinspace}
\providecommand{\MR}{\relax\ifhmode\unskip\space\fi MR }
\providecommand{\MRhref}[2]{%
  \href{http://www.ams.org/mathscinet-getitem?mr=#1}{#2}
}
\providecommand{\href}[2]{#2}

\input{madrid.ind}
\newpage


\end{document}